\documentclass{memo-l}

\usepackage{amscd}
\usepackage{amssymb}
\usepackage{url}
\usepackage{nicefrac}
\usepackage[pdftex]{graphicx}%
\usepackage[outercaption]{sidecap}
\usepackage{amsfonts,epigraph}%
\usepackage{booktabs} 
\usepackage[all,cmtip]{xy}
\newtheorem{theorem}{Theorem}[section]
\newtheorem{corollary}[theorem]{Corollary}
\newtheorem{lemma}[theorem]{Lemma}
\newtheorem{proposition}[theorem]{Proposition}

\newtheorem{properties}[theorem]{Properties}

\theoremstyle{definition}
\newtheorem{definition}[theorem]{Definition}
\newtheorem{example}[theorem]{Example}

\newtheorem{ass}[theorem]{Assumption}

\newtheorem{choices-notations}[theorem]{Choices and Notations}

\newtheorem{rem}[theorem]{Remark}

\theoremstyle{remark}
\newtheorem{data}{Data\!\!}

\newtheorem{note}{Note\!\!}

\numberwithin{section}{chapter} \numberwithin{equation}{chapter} \numberwithin{figure}{chapter}

\newcommand{\exendproof}{\renewcommand{\qed}{\relax}\end{proof}}

\newsavebox{\SmallMathBox}

\DeclareRobustCommand*{\nicefrac}[2]{\ifmmode\mathnicefrac{#1}
{ #2}%
  \else\textnicefrac{#1}{#2}\fi}
\newcommand*{\textnicefrac}[2]{\check@mathfonts%
\mbox{\raisebox{.5ex}{\fontsize\sf@size\z@\selectfont#1}\kern-.
1em%
/\kern-.1em\raisebox{- .25ex}{\fontsize\sf@size\z@\selectfont#2} }}
\newcommand*{\mathnicefrac}[2]{%
  \mathchoice
    {\m@fr@c{\scriptstyle}{#1}{#2}}
    {\m@fr@c{\scriptstyle}{#1}{#2}}
    {\m@fr@c{\scriptscriptstyle}{#1}{#2}}
    {\m@fr@c{\scriptscriptstyle}{#1}{#2}}}

\def\bbb{{\boldsymbol{\beta}}}

\def\Ci{C^\infty}
\def\dpa{\partial}

\def\fequal#1{\stackrel{#1}{=}}

\def\into{\hookrightarrow}
\def\lla{\langle}

\def\noi{\noindent}
\newcommand{\norm}[1]{\lVert#1\rVert}

\def\ol{\overline}

\def\rra{\rangle}
\def\sqm1{\sqrt{-1}}

\def\tand{\mbox{\ \rm  and }}

\def\too{\longrightarrow}

\def\wh{\widehat}
\def\wt{\widetilde}
\def\x{\times}

\def\={\cong}
\def\>{\supset}
\def\<{\subset}
\def\ii{^{-1}}
\def\12{\frac{1}{2}}
\def\0{^{\circ}}

\def\AA{{\mathbb A}}

\def\CC{{\mathbb C}}
\def\EE{{\mathbb E}}

\def\MM{{\mathbb M}}
\def\NN{{\mathbb N}}

\def\RR{{\mathbb R}}

\def\ZZ{{\mathbb Z}}

\def\Aa{{\mathcal A}}
\def\Bb{{\mathcal B}}
\def\Cc{{\mathcal C}}

\def\Ff{{\mathcal F}}
\def\Gg{{\mathcal G}}
\def\Hh{{\mathcal H}}

\def\Kk{{\mathcal K}}
\def\Ll{{\mathcal L}}
\def\Mm{{\mathcal M}}

\def\Pp{{\mathcal P}}

\def\Ss{{\mathcal S}}
\def\Tt{{\mathcal T}}

\def\GGG{{\mathfrak G}}

\def\A{\Aa}
\def\C{\CC}
\renewcommand*{\d}{\delta}
\newcommand*{\D}{\Delta}
\def\e{\varepsilon}
\def\f{\varphi}

\def\g{\gamma}

\def\K{\Kk}
\def\la{\lambda}

\def\m{\mu}

\def\R{\RR}
\def\s{\sigma}
\def\Si{\Sigma}

\def\w{\omega}
\def\W{\Omega}

\def\Z{\ZZ}

\DeclareMathOperator{\ad}{ad}

\DeclareMathOperator{\can}{can}

\DeclareMathOperator{\defi}{def}
 \DeclareMathOperator{\dist}{dist}
\DeclareMathOperator{\diag}{diag}
\DeclareMathOperator{\dom}{dom}

\DeclareMathOperator{\GL}{GL} \DeclareMathOperator{\gl}{gl} \DeclareMathOperator{\Graph}{graph}
\DeclareMathOperator{\Grass}{Grass} \DeclareMathOperator{\Green}{Green}

\DeclareMathOperator{\Hom}{Hom}

\DeclareMathOperator{\Index}{index}

  \DeclareMathOperator{\Lin}{Lin}

\DeclareMathOperator{\Mas}{Mas} \DeclareMathOperator{\mmax}{max}

\DeclareMathOperator{\nuli}{nul} \DeclareMathOperator{\ran}{im} \DeclareMathOperator{\range}{im}
  
 \DeclareMathOperator{\romS}{S} \DeclareMathOperator{\sa}{sa}
\DeclareMathOperator{\SF}{sf}
\DeclareMathOperator{\sign}{sign}

  \DeclareMathOperator{\Span}{span}

 \DeclareMathOperator{\supp}{supp}

\usepackage{amsmidx}
\makeindex{authorindex} \makeindex{symbolindex} \makeindex{subjectindex}
\newcommand{\auindex}[1]{\index{authorindex}{#1}}
\newcommand{\symindex}[1]{\index{symbolindex}{#1}}
\newcommand{\subindex}[1]{\index{subjectindex}{#1}}

\usepackage{hyperref}
\hypersetup{%
   colorlinks=true,
   bookmarks=true,
   pdftitle={Maslov index},
   bookmarksopenlevel=0,
   pdfauthor={Zhu and Booss-Bavnbek},
   pdfsubject={research monograph on index theory},
   pdfkeywords={Maslov index, symplectic, Banach space}%
}
\usepackage{memhfixc}

\begin{document}
\setcounter{page}{1} \setcounter{tocdepth}{2} \setcounter{figure}{0} 

\renewcommand*{\labelenumi}{%
   (\roman{enumi})}

\frontmatter

\title{The Maslov index in symplectic Banach spaces\\
}

\author[Bernhelm Booss-Bavnbek]{Bernhelm Boo{\ss}-Bavnbek}
\address{Department of Science and Environment/IMFUFA\\ Roskilde
University, 4000 Ros\-kilde, Denmark}
\email{booss@ruc.dk}
\thanks{}

\author{Chaofeng Zhu}
\address{Chern Institute of Mathematics and LPMC\\
Nankai University\\
Tianjin 300071, P. R. China} \email{zhucf@nankai.edu.cn}
\thanks{Corresponding author: CZ [\texttt{zhucf@nankai.edu.cn}]}

\date{September 3, 2014}



\dedicatory{Dedication text (use \\[2pt] for line break if necessary)}

\begin{abstract}
We consider a curve of \subindex{Fredholm pair!of Lagrangian subspaces!curve}Fredholm pairs of
\subindex{Lagrangian subspaces}Lagrangian subspaces in a fixed \subindex{Symplectic form!weak}
\subindex{Banach space!with varying weak symplectic structure}Banach space with continuously
varying \textit{weak} symplectic structures. Assuming vanishing index, we obtain intrinsically a
continuously varying splitting of the total Banach space into pairs of symplectic subspaces. Using
such decompositions we define the Maslov index of the curve by symplectic reduction to the classical
finite-dimensional case. We prove the transitivity of repeated symplectic reductions and obtain the
invariance of the Maslov index under \subindex{Symplectic reduction} symplectic reduction, while
recovering all the standard properties of the Maslov index.

As an application, we consider curves of elliptic operators which have varying \subindex{Principal
symbol}principal symbol, varying \subindex{Domain!maximal}maximal domain and are \subindex{Elliptic
differential operators!of not-Dirac type}not necessarily of Dirac type. For this class of operator
curves, we derive a desuspension spectral flow formula for varying well-posed boundary conditions
on manifolds with boundary and obtain the splitting of the spectral flow on \subindex{Partitioned
manifold}partitioned manifolds.
\end{abstract}

\maketitle

\tableofcontents
\listoffigures
%

\chapter*{Preface}

The purpose of this Memoir is to establish a universal relationship between incidence geometries in
finite and infinite dimensions. In finite dimensions, counting incidences is nicely represented by
the \textit{Maslov index}. It counts the dimensions of the intersections of a pair of curves of
Lagrangian subspaces in a symplectic finite-dimensional vector space. The concept of the Maslov
index is non-trivial: in finite dimensions, the Maslov index of a loop of pairs of Lagrangians does
not necessarily vanish. In infinite dimensions, counting incidences is nicely represented by the
\textit{spectral flow}. It counts the number of intersections of the spectral lines of a curve of
self-adjoint Fredholm operators with the zero line. In finite dimensions, the spectral flow is
trivial: it vanishes for all loops of Hermitian matrices.

Over the last two decades there have been various, and in their way successful attempts to
generalize the concept of the Maslov index to curves of Fredholm pairs of Lagrangian subspaces in
strongly symplectic Hilbert space, to establish the correspondence between Lagrangian subspaces and
self-adjoint extensions of closed symmetric operators, and to prove spectral flow formulae in
special cases, namely for curves of Dirac type operators and other curves of closed symmetric
operators with bounded symmetric perturbation and subjected to curves of self-adjoint Fredholm
extensions (i.e., well-posed boundary conditions). While these approaches vary quite substantially,
they all neglect the essentially finite-dimensional character of the Maslov index, and,
consequently, break down when one deals with operator families of \textit{varying} maximal domain.
Quite simply, there is no directly calculable Maslov index when the symplectic structures are weak
(i.e., the symplectic forms are not necessarily generated by anti-involutions $J$) and vary in an
uncontrolled way.

\smallskip

In this Memoir we show a way out of this dilemma. We develop the classical method of symplectic
reduction to yield an \textit{intrinsic} reduction to finite dimension, induced by a given curve of
Fredholm pairs of Lagrangians in a fixed Banach space with varying symplectic forms. From that
reduction, we obtain an \textit{intrinsic} definition of the Maslov index in symplectic Banach
bundles over a closed interval. This Maslov index is calculable and yields a general spectral flow
formula. In our application for elliptic systems, say of order one on a manifold $M$ with boundary
$\Si$, our fixed Banach space (actually a Hilbert space) is the Sobolev space
$H^{1/2}(\Si;E|_{\Si})$ of the traces at the boundary of the $H^1(M;E)$ sections of a Hermitian
vector bundle $E$ over the whole manifold. For $H^{1/2}(\Si;E|_{\Si})$, we have a family of
continuously varying weak symplectic structures induced by the principal symbol of the underlying
curve of elliptic operators, taken over the boundary in normal direction. That yields a symplectic
Banach bundle which is the main subject of our investigation.

Whence, the message of this Memoir is: The Maslov index belongs to finite dimensions. Its most
elaborate and most general definitions can be reduced to the finite-dimensional case in a natural
way. The key for that - and for its identification with the spectral flow - is the concept of
Banach bundles with weak symplectic structures and intrinsic symplectic reduction. From a technical
point of view, that is the main achievement of our work.

\aufm{Bernhelm Boo{\ss}-Bavnbek\\
Chaofeng Zhu}


\mainmatter
%

\addtocontents{toc}{\medskip\noi} 
\chapter*{Introduction}

\subsection*{Upcoming and continuing interest in the Maslov index} Since the legendary work of
\auindex{Maslov,\ V.P.}V.P. Maslov \cite{Maslov} in the mid 1960s and the supplementary
explanations by \auindex{Arnol'd,\ V.I}V. Arnol'd \cite{Arnold:1967}, there has been a continuing
interest in the \subindex{Maslov index}Maslov index for \subindex{Continuously varying!Fredholm
pairs of Lagrangian subspaces}curves of pairs of Lagrangians in symplectic space. As explained by Maslov and
Arnol'd, the interest arises from the study of \subindex{Dynamical systems}dynamical systems in
classical mechanics and related problems in \subindex{Morse theory}Morse theory. This same index
occurs as well in certain asymptotic formulae for solutions of the Schr{\"o}dinger equations. For a
systematic review of the basic vector analysis and geometry and for the physics background, we
refer to \auindex{Arnol'd,\ V.I}Arnol'd \cite{Arnold:1989} and \auindex{de Gosson,\ M.}M. de Gosson
\cite{Go01}.

The \subindex{Morse theory!Morse index theorem}Morse index theorem expresses the Morse index of a
geodesic by the Maslov index. Later, \auindex{Yoshida,\ T.}T. Yoshida \cite{Yo91} and
\auindex{Nicolaescu,\ L.}L. Nicolaescu \cite{Ni95,Ni97} expanded the view by embracing also
spectral problems for \subindex{Elliptic differential operators!Dirac type operators}Dirac type
operators on \subindex{Partitioned manifold}partitioned manifolds and thereby stimulating some
quite new research in that direction. For a short review, we refer to our Section \ref{s:history}
below.

\subsection*{Weak symplectic forms on Banach manifolds}
Early in the 1970s, \auindex{Chernoff,\ P.R.}\auindex{Marsden,\ J.E.}P.~Chernoff, J. Marsden
\cite{CheMa74} and \auindex{Weinstein,\ A.}A. Weinstein \cite{Weinstein:1971} called attention to
the practical and theoretical importance of \subindex{Banach manifold!symplectic forms on
it}symplectic forms on Banach manifolds. See \auindex{Swanson,\ R.C.}R.C. Swanson
\cite{Swanson1,Swanson2,Swanson:1980} for an elaboration of the achievements of that period
regarding \subindex{Banach space!linear symplectic structures on it}linear symplectic structures on
Banach spaces. It seems, however, that rigorous and operational definitions of the Maslov index of
\subindex{Continuously varying!Fredholm pairs of Lagrangian subspaces}curves of
\subindex{Lagrangian subspaces!in spaces of infinite dimension}Lagrangian subspaces in spaces of
infinite dimension was not obtained until 25 years later. Our \auindex{Boo{\ss}--Bavnbek,\
B.}\auindex{Zhu,\ C.}\cite[Section 3.2]{BooZhu:2013} gives an account and compares the various
definitions.

\begin{figure}
\includegraphics[scale=0.95]{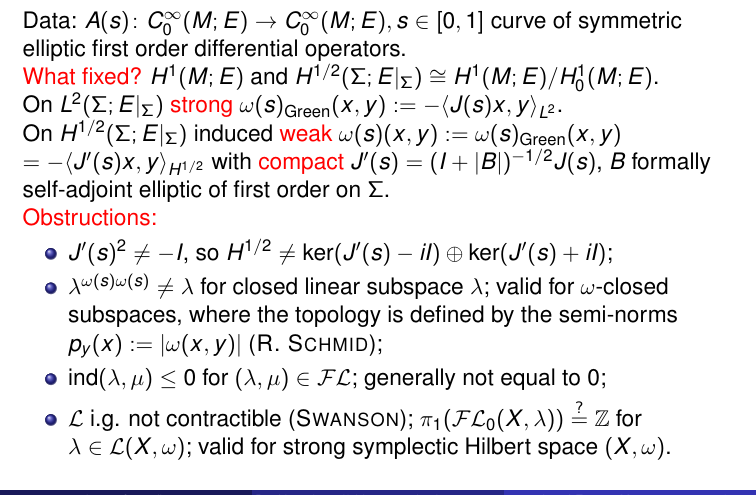}
\caption{Why going weak and what obstructions to circumvent?}\auindex{Swanson,\
R.C.}\auindex{Schmid,\ R.}\symindex{J@$J$  operator associated to symplectic
form}\subindex{Operator!associated to symplectic form}\subindex{Symplectic form!associated almost
complex generating operator}\subindex{Obstructions to straight forward generalization!generator $J$
non-invertible} \subindex{Obstructions to straight forward generalization!no symplectic splitting}
\subindex{Obstructions to straight forward generalization!double annihilator not the identity on
closed subspaces} \subindex{Obstructions to straight forward generalization!non-vanishing index of
Fredholm pair of Lagrangian subspaces} \subindex{Obstructions to straight forward
generalization!$\Ll$ not contractible}\subindex{Obstructions to straight forward
generalization!fundamental group of Fredholm Lagrangian Grassmannian unknown}\subindex{Obstructions
to straight forward generalization!no meaningful Maslov cycle}\subindex{Fredholm Lagrangian
Grassmannian} \label{f:obstructions}
\end{figure}

At the same place we emphasized a couple of rather serious obstructions (see Figure
\ref{f:obstructions}) to applying these concepts to arbitrary systems of \subindex{Elliptic
differential operators!of non-Dirac type}elliptic differential equations of non-Dirac type:
Firstly, some of the key section spaces for studying \subindex{Boundary value problems}boundary
value problems (the \subindex{Sobolev spaces!with weak symplectic forms}Sobolev space
\symindex{H@$H^{1/2}(\Si;E\mid_{\Sigma})$ weak symplectic Sobolev space}$H^{1/2}(\Si;E|_{\Si})$
containing the traces over the boundary $\Si=\partial M$ of sections over the whole manifold
\symindex{M@$M, M(s)$ smooth compact manifold!with boundary $\Si,\Si(s)$}$M$) are not carrying a
\subindex{Symplectic form!strong}strong symplectic structure, but are naturally equipped with a
\subindex{Symplectic form!weak}weak structure not admitting the rule \symindex{J@$J$  operator
associated to symplectic form}\subindex{Operator!associated to symplectic form}\subindex{Symplectic
form!associated almost complex generating operator} $J^2=-I$. Secondly, in
\auindex{Boo{\ss}--Bavnbek,\ B.}\auindex{Zhu,\ C.}\cite{BooZhu:2013} our definition of the Maslov
index in weak symplectic spaces requires a \subindex{Symplectic splitting}symplectic splitting
which does not always exist, is not canonical, and therefore, in general, not obtainable in a
continuous way for continuously varying symplectic structures. Recall that a \textit{symplectic
splitting} of a symplectic Banach space $(X,\w)$ is a decomposition $X=X^-\oplus X^+$ with $-i\w$
negative, respectively, positive definite on $X^{\mp}$ and vanishing on $X^-\times X^+$. Thirdly, a
priori, a \subindex{Symplectic reduction!finite-dimensional}symplectic reduction to finite
dimensions is not obtainable for weak symplectic structures in the setting of \cite{BooZhu:2013}.

An additional incitement to investigate weak symplectic structures comes from a stunning
observation of \auindex{Witten,\ E.}E. Witten (explained by \auindex{Atiyah,\ M.F.}M.F. Atiyah in
\cite{Atiyah:1985} in a heuristic way). He considered a weak presymplectic form on the loop space
\symindex{M@$\operatorname{Map}(S^1,M)$ loop space}\subindex{Geodesic!loop space of a
manifold}$\operatorname{Map}(S^1,M)$ of a finite-dimensional closed orientable Riemannian manifold
$M$ and noticed that a (future) thorough understanding of the infinite-dimensional symplectic
geometry of that loop space ``should lead rather directly to the \subindex{Index!Atiyah-Singer
Index Theorem!via infinite-dimensional symplectic geometry}index theorem for Dirac operators"
(l.c., p. 43). Of course, restricting ourselves to the linear case, i.e., to the geometry of
Lagrangian subspaces instead of \subindex{Lagrangian manifold}Lagrangian manifolds, we can only
marginally contribute to that program in this Memoir.

\subsection*{Symplectic reduction}\subindex{Symplectic reduction!history, origin and meaning}
In their influential paper \cite[p. 121]{Marsden-Weinstein}, \auindex{Marsden,\
J.E.}\auindex{Weinstein,\ A.}J. Marsden and A. Weinstein describe
the purpose of symplectic reduction in the following way:%
{\small
\begin{quotation}
``... when we have a \subindex{Symplectic geometry!symplectic manifold}symplectic manifold on which
a group acts symplectically, we can reduce this \subindex{Symplectic reduction!reduced phase
space}phase space to another symplectic manifold in which, roughly speaking, the symmetries are
divided
out."%
\end{quotation}}%
and {\small
\begin{quotation}
``When one has a \subindex{Hamiltonian system}Hamiltonian system on the phase space which is
invariant under the group, there is a Hamiltonian system canonically induced on the reduced phase
space."
\end{quotation}}
The basic ideas go back to the work of \auindex{Hamel,\ G.}G. Hamel \cite{Hamel:1904,Hamel:1978}
and \auindex{Carath\'eodory,\ C.}C. Carath{\'e}odory \cite{Caratheodory} in \subindex{Dynamical
systems}dynamical systems at the beginning of the last century, see also \auindex{Souriau,\
J.-M.}J.-M. Souriau \cite{Souriau}. For \subindex{Symplectic reduction!in low-dimensional
geometry}symplectic reduction in low-dimensional geometry see the monographs by
\auindex{Donaldson,\ S.K.}\auindex{Kronheimer,\ P.B.}S.K. Donaldson and P.B. Kronheimer, and by
\auindex{McDuff,\ D.}\auindex{Salamon,\ D.}D. McDuff and D. Salamon \cite{Donaldson-Kronheimer,
McDuff-Salamon}.

Our aim is less intricate, but not at all trivial: Following \auindex{Nicolaescu,\ L.}L. Nicolaescu
\cite{Ni97} and \auindex{Boo{\ss}--Bavnbek,\ B.}\auindex{Furutani,\ K.}K. Furutani \cite{BoFu99}
(joint work with the first author) we are interested in the finite-dimensional reduction of
\subindex{Fredholm pair!of Lagrangian subspaces}Fredholm pairs of Lagrangian \textit{linear}
subspaces in infinite-dimensional Banach space. The \subindex{Symplectic reduction!general
procedure}general procedure is well understood, see also \auindex{Kirk,\ P.}\auindex{Lesch,\ M.}P.
Kirk and M. Lesch in \cite[Section 6.3]{KiLe00}: let $W\< X$ be a closed \subindex{Co-isotropic
subspaces}co-isotropic subspace of a symplectic Banach space \symindex{X@$(X,\w)$ symplectic vector
or Banach space}$(X,\w)$. Then \symindex{W@$W/W^{\w}$ reduced symplectic space}$W/W^{\w}$ inherits
a symplectic form from $\w$ such that
\[\symindex{R@$R_W(\la)$ symplectic reduction of isotropic $\la$ along co-isotropic $W$}
 R_W(\la)\ :=\  \frac{(\la +W^{\w})\cap W}{W^{\w}}\
\<\ \frac W{W^{\w}} \text{ isotropic for $\la$ isotropic}.
\]
Here $W^\w$ denotes the annihilator of $W$ with respect to the symplectic form $\w$ (see Definition
\ref{d:symplectic-space}c).

In general, however, the reduced space $R_W(\la)$ does not need to be Lagrangian in $W/W^{\w}$ even
for Lagrangian $\la$ unless we have $W^\w\<\la\< W$ (see Proposition \ref{p:red-fredholm}). In
\auindex{Nicolaescu,\ L.}\auindex{Boo{\ss}--Bavnbek,\ B.}\auindex{Furutani,\ K.}\cite{Ni97,BoFu99}
a closer analysis of the \subindex{Symplectic reduction!reduction map}reduction map $R_W$ is given
within the setting of strong symplectic structures; with emphasis on the topology of the space of
Fredholm pairs of Lagrangians; and for fixed $W$. Now we drop the restriction to strong symplectic
forms; our goal is to define the Maslov index for \subindex{Continuously varying!Fredholm pairs of
Lagrangian subspaces}continuous \subindex{Fredholm pair!of Lagrangian subspaces!curve}curves $s\to
(\la(s),\mu(s))$ of Fredholm pairs of Lagrangians with respect to continuously varying symplectic
forms $\w(s)$; and, at least locally (for $s\in (t-\e,t+\e)$ around $t\in [0,1]$), we let the pair
$(\la(t),\mu(t))$ induce  the reference space $W(t)$ for the \subindex{Symplectic
reduction}symplectic reduction and the pair $(\la(s),\mu(s))$ induce the reduction map
$R^{(s)}_{W(t)}$ in a natural way. The key to finding the reference spaces $W(t)$ and defining a
suitable reduction map $R_{W(t)}$ is our Proposition \ref{p:reduction-prop}. It is on
decompositions of symplectic Banach spaces, naturally induced by a given Fredholm pair of
Lagrangians of vanishing index. It might be, as well, of independent interest. The assumption of
vanishing index is always satisfied for Fredholm pairs of Lagrangian subspaces in strong symplectic
Hilbert spaces, and by additional global analysis arguments in our applications as well.

Thus for each path $\{(\la(s),\mu(s))\}_{s\in[0,1]}$ of Fredholm pairs of Lagrangian subspaces of
vanishing index, we receive a finite-dimensional symplectic reduction \textit{intrinsically}, i.e.,
without any other assumption. The reduction transforms the given path into a path of pairs of
Lagrangians in finite-dimensional symplectic space. The main part of the Memoir is then to prove
the invariance under symplectic reduction and the independence of choices made. That permits us a
conservative view in this Memoir. Instead of defining the Maslov index in infinite dimensions via
spectral theory of unitary generators of the Lagrangians as we did in \auindex{Boo{\ss}--Bavnbek,\
B.}\auindex{Long,\ Y.}\auindex{Zhu,\ C.}\cite{BooZhu:2013}, we elaborate the concept of the
\subindex{Maslov index!in finite dimensions}Maslov index in finite dimensions and reduce the
infinite-dimensional case to the finite-dimensional case, i.e., we take the symplectic reduction as
our beginning for re-defining the Maslov index instead of deploring its missing.

\begin{figure}
\includegraphics[scale=0.95]{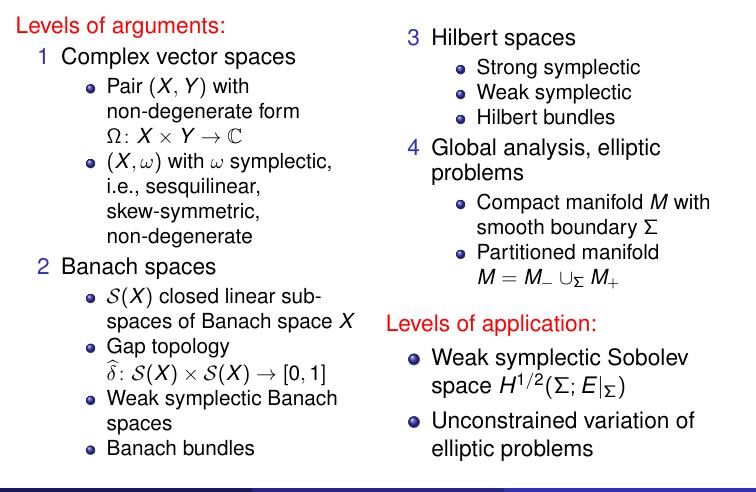}
\caption{Natural levels of treatment}\subindex{Levels of treatment}\label{f:natural-levels}
\end{figure}

\subsection*{Structure of presentation}
This Memoir is divided into four chapters and one appendix. The first three chapters present a
rigorous definition of the Maslov index in Banach bundles by symplectic reduction. In Chapter
\ref{s:banach}, we fix the notation and establish our key technical device, namely the mentioned
\subindex{Banach space!natural decomposition}natural decomposition of a symplectic Banach space
into two symplectic spaces, induced by a pair of co-isotropic subspaces with finite codimension of
their sum and finite dimension of the intersection of their annihilators. We introduce the
\subindex{Symplectic reduction}symplectic reduction of arbitrary linear subspaces via a fixed
co-isotropic subspace $W$ and prove the transitivity of the symplectic reduction when replacing $W$
by a larger co-isotropic subspace $W'$. For Fredholm pairs of Lagrangian subspaces of vanishing
\subindex{Fredholm pair!index}index, that yields an identification of the two naturally defined
symplectic reductions. In Chapter \ref{s:maslov-hilbert}, we recall and elaborate the
\subindex{Maslov index!in finite dimensions}\subindex{Maslov index!in strong symplectic Hilbert
space}Maslov index in strong symplectic Hilbert space, particularly in finite dimensions, to prove
the invariance of our definition of the Maslov index under different symplectic reductions. In
Chapter \ref{s:maslov-general}, we investigate the symplectic reduction to finite dimensions for a
given path of Fredholm pairs of Lagrangian subspaces in fixed \subindex{Banach space!with varying
weak symplectic structure}Banach space with varying symplectic structures and define the Maslov
index in the general case via \subindex{Symplectic reduction!finite-dimensional}finite-dimensional
symplectic reduction. In Section \ref{ss:sympl-invariance}, we show that the Maslov index is
invariant under symplectic reduction in the general case. For a first review of the entangled
levels of treatment see Figure \ref{f:natural-levels}.

Chapter \ref{s:gsff} is devoted to an application in \subindex{Global analysis}global analysis.
We summarize the predecessor formulae, we prove a
wide generalization of the \subindex{Yoshida-Nicolaescu's Theorem}Yoshida-Nicolaescu spectral flow
formula, namely the identity \subindex{Spectral flow}\subindex{Maslov index!spectral flow
formula}Maslov index=spectral flow, both in general terms of \subindex{Banach bundle}Banach bundles
and for \subindex{Elliptic differential operators!on smooth manifolds with boundary}elliptic
differential operators of arbitrary positive order on smooth manifolds with boundary. That involves
weak symplectic Hilbert spaces like the Sobolev space $H^{1/2}$ over the boundary. Applying
substantially more advanced results we derive a corresponding spectral flow formula in all
\subindex{Sobolev spaces}Sobolev spaces $H^{\sigma}$ for $\sigma\ge 0$, so in particular in the
familiar strong symplectic $L^2$.

In the Appendix \ref{s:closed-subspaces} on \subindex{Banach space!closed subspace}closed subspaces
in Banach spaces, we address the \subindex{Banach space!continuity of operations}continuity of
operations of linear subspaces. In \subindex{Gap topology}gap topology, we prove some sharp
estimates which might be of independent interest. E.g., they yield the following basic convergence
result for sums and intersections of permutations in the space $\Ss(X)$ of closed linear subspaces
in a Banach space $X$ in Proposition \ref{p:close-to} (\auindex{Neubauer,\ G.}\cite[Lemma 1.5 (1),
(2)]{Ne68}): Let $(M'_j)_{j=1,2.\dots}$ be a sequence in $\Ss(X)$ converging to $M\in\Ss(X)$ in the
gap topology, shortly \symindex{zz@$\to$ converging to}$M'\to M$, let similarly $N'\to N$ and $M+N$
be closed. Then $M'\cap N'\to M\cap N$ iff $M'+N'\to M+N$. For each of the three technical main
results of the Appendix, some applications are given to the \subindex{Global analysis!of elliptic
problems on manifolds with boundary}global analysis of elliptic problems on manifolds with
boundary.


\subsection*{Relation to our previous results}
With this Memoir we conclude a series of our mutually related previous approaches to
\subindex{Symplectic geometry}symplectic geometry, \subindex{Dynamical systems}dynamical systems,
and \subindex{Global analysis}global analysis; in chronological order
\auindex{Boo{\ss}--Bavnbek,\ B.}\auindex{Chen,\ G.}%
\auindex{Furutani,\ K.}\auindex{Otsuki,\ N.}\auindex{Lesch,\ M.}%
\auindex{Phillips,\ J.}\auindex{Long,\ Y.}\auindex{Zhu,\ C.}\cite{BoFu98, BoFu99, ZhLo99, Zh00,
Zh01, BoFuOt01, BoLePh01, Zh05, BoZh05, BoLeZh08, BCLZ, BooZhu:2013}.

The model for our various approaches was developed in joint work with K. Furutani and N. Otsuki in
\cite{BoFu98, BoFu99, BoFuOt01}. Roughly speaking, there we deal with a \textit{strong} symplectic
Hilbert space $(X,\lla\cdot,\cdot\rra,\w)$, so that $\w(x,y)=\lla Jx,y\rra$ with $J^*=-J\tand
J^2=-I$, possibly after continuous deformation of the inner product $\lla \cdot,\cdot\rra$. Then
the space \subindex{Lagrangian Grassmannian}\symindex{L@$\Ll(X,\w)$ Lagrangian
Grassmannian}$\Ll(X,\w)$ of all Lagrangian subspaces is contractible and, for fixed
$\la\in\Ll(X,\w)$, the fundamental group of the \subindex{Fredholm Lagrangian Grassmannian}
\symindex{FL@$\Ff\Ll(X,\w)$ Fredholm Lagrangian Grassmannian}Fredholm Lagrangian Grassmannian
$\Ff\Ll(X,\w,\la)$ of all Fredholm pairs $(\la,\mu)$ with $\mu\in\Ll(X,\w)$ is cyclic, see
\auindex{Boo{\ss}--Bavnbek,\ B.}\auindex{Furutani,\ K.}\cite[Section 4]{BoFu99} for an elementary
proof. By the induced \subindex{Symplectic splitting}{\em symplectic splitting} $X=X^+\oplus X^-$
with $X^{\pm}:= \ker(J\mp iI)$ we obtain
\begin{enumerate}\symindex{g@$\Graph(A)$ graph of operator $A$}
\item[(1)]  $\forall \la\in\Ll(X,\w) \ \exists U\colon X^+\to X^-$ unitary with $\la=\Graph(U)$;

\item[(2)] $(\la,\mu)\in\Ff\Ll(X,\w)\iff UV\ii -I_{X^-}\in\Ff(X)$; and
\item[(3)] $\Mas\{\la(s),\mu(s)\}_{s\in [0,1]}:=-\SF_{(0,\infty)}\bigl\{U_sV_s\ii\bigr\}_{s\in [0,1]}$ well
defined.
\end{enumerate}\symindex{Mas@$\Mas(\la(s),\mu(s))_{s\in [0,1]}$ Maslov index}\symindex{sfl@$\SF_{\ell}$
spectral flow through gauge curve $\ell$}%
Here \symindex{F@$\Ff(X)$ space of bounded Fredholm operators on $X$}$\Ff(X)$ denotes the space of
bounded Fredholm operators on $X$ and \symindex{FL@$\Ff\Ll(X,\w)$ Fredholm Lagrangian
Grassmannian}$\Ff\Ll(X,\w)$ the set of Fredholm pairs of Lagrangian subspaces of $(X,\w)$ (see
Definition \ref{d:fredholm-lag-pair}).

This setting is suitable for the following application in operator theory: Let $\mathcal{H}$ be a
complex separable \subindex{Hilbert space}Hilbert space and $A$ a \subindex{Closed symmetric
operator}closed symmetric operator. We extend slightly the frame of the \auindex{Birman,\
M.S.}\auindex{Krein,\ M.G.}\auindex{Vishik,\ M.I.}\subindex{Birman-Kre\u\i n-Vishik
theory}{B}irman-{K}re\u\i n-{V}ishik theory of \subindex{Extension!self-adjoint}self-adjoint
extensions of semi-bounded operators (see the review \cite{Alonso-Simon} by \auindex{Alonso,\
A.}\auindex{Simon,\ B.}A. Alonso and B. Simon). Consider the space \symindex{\beta@$\beta(A)$ space
of abstract boundary values}$\beta(A):=\dom(A^*)/\dom(A)$ of \subindex{Abstract boundary
values}abstract boundary values. It becomes a \subindex{Hilbert space!strong symplectic}strong
symplectic Hilbert space with
\[
 \omega(\gamma(x),\gamma(y))\ :=\ \lla x,A^*y\rra-\lla A^*x,y\rra,%
\]
and the projection \symindex{\gamma@$\gamma$ abstract trace map}$\gamma\colon
\dom(A^*)\to\beta(A)$, $x\mapsto [x]:=x+\dom(A)$. The inner product $\lla \gamma(x),\gamma(y)\rra$
is induced by the graph inner product $\lla x,y\rra_{\Gg} := \lla x,y\rra + \lla A^*x,A^*y\rra$
that makes $\dom(A^*)$ and, consequently, $\beta(A)$ to Hilbert spaces. Introduce the
\subindex{Cauchy data space!abstract (or reduced)}abstract \textit{Cauchy data space}
\symindex{CD@$\operatorname{CD}(A)$ abstract Cauchy data space}$\operatorname{CD}(A):=
\left(\ker(A^*)+\dom(A)\right)/\dom(A) = \{\gamma(x)\mid x\in\ker A^*\}$. From \auindex{Neumann,
von,\ J.}von Neumann's \subindex{von Neumann program}famous \cite{Neu} we obtain
the correspondence%
\[\subindex{Lagrangian subspaces!and self-adjoint extensions}\symindex{AD@$A_D$ extension
(realization) of operator $A$ with domain $D$} A_D \text{  self-adjoint extension $\ \iff\ [D]\<
\beta(A)$ Lagrangian},
\]
for $\dom(A) \< D \< \dom(A^*)$. Now let $A_D$ be a \subindex{Extension!self-adjoint
Fredholm}\textit{self-adjoint Fredholm extension}, $\{C(s)\}_{s\in [0,1]}$ a $C^0$
\subindex{Continuously varying!bounded self-adjoint operators}curve in
\symindex{Bs@$\mathcal{B}^{\sa}(\mathcal {H})$ space of bounded self-adjoint
operators}$\mathcal{B}^{\sa}(\mathcal {H})$, the space of bounded self-adjoint operators, and
assume \subindex{Weak inner Unique Continuation Property (wiUCP)}\textit{weak inner Unique
Continuation Property (UCP)}, i.e., $\ker(A^*+C(s)+\e)\cap \dom(A)=\{0\}$ for small positive $\e$.
Then, \auindex{Boo{\ss}--Bavnbek,\ B.}\auindex{Furutani,\ K.}\cite{BoFu98} shows that

\begin{enumerate}

\item[(1)] $\{\operatorname{CD}(A+C(s)),\gamma(D)\}_{s\in [0,1]}$ is a continuous curve of Fredholm pairs of Lagrangians
in the gap topology, and

\item[(2)] $\SF\{(A+C(s))_D\}_{s\in [0,1]} =
-\Mas\{\operatorname{CD}(A+C(s)),\gamma(D)\}_{s\in [0,1]}$.
\end{enumerate}

On one side, the approach of \auindex{Boo{\ss}--Bavnbek,\ B.}\auindex{Furutani,\ K.}\cite{BoFu98}
has considerable strength: It is ideally suited both to \subindex{Hamiltonian system}Hamiltonian
systems of ordinary differential equations of first order over an interval $[0,T]$ with varying
lower order coefficients, and to \subindex{Continuously varying!Dirac type operators}curves of
\subindex{Dirac type operators!curves of}Dirac type operators on a Riemannian partitioned manifold
or manifold $M$ with boundary $\Si$ with fixed Clifford multiplication and Clifford module (and so
fixed principal symbol), but symmetric bounded perturbation due to varying affine connection
(background field). Hence it explains \auindex{Nicolaescu,\ L.}\subindex{Yoshida-Nicolaescu's
Theorem}Nicolaescu's Theorem (see below Section \ref{s:history}) in purely functional analysis
terms and elucidates the decisive role of weak inner UCP. For such curves of Dirac type operators,
the $\beta$-space remains fixed and can be described as a subspace of the distribution space
$H^{-1/2}(\Si)$ with ``half" component in $H^{1/2}(\Si)$. As shown in \auindex{Boo{\ss}--Bavnbek,\
B.}\auindex{Furutani,\ K.}\cite{BoFu99}, the Maslov index constructed in this way is invariant
under finite-dimensional symplectic reduction. Moreover, the approach admits varying boundary
conditions and varying symplectic forms, as shown in \auindex{Boo{\ss}--Bavnbek,\
B.}\auindex{Lesch,\ M.}\auindex{Phillips,\ J.}\auindex{Zhu,\ C.}\cite{BoLePh01,BoZh05} and can be
generalized to a \subindex{Spectral flow!spectral flow formula}spectral flow formula in the common
$L^2(\Si)$ as shown in \auindex{Furutani,\ K.}\auindex{Otsuki,\ N.}\cite{BoFuOt01}.

Unfortunately, that approach has severe limitations since it excludes \subindex{Varying maximal
domain}varying maximal domain: there is no $\beta$-space when variation of the highest order
coefficients is admitted for the \subindex{Continuously varying!elliptic differential
operators}curve of elliptic differential operators.

The natural alternative (here for first order operators) is to work with the Hilbert space
\symindex{H@$H^{1/2}(\Si;E\mid_{\Si})$
weak symplectic Sobolev space}%
\[
H^{1/2}(\Si;E|_{\Si})\ \cong\ H^1(M;E)/H^1_0(M;E)%
\]
which remains fixed as long as we keep our underlying Hermitian vector bundle $E\to M$ fixed. So,
let $A(s)\colon \Ci_0(M;E)\to\Ci_0(M;E),\/s\in [0,1]$ be a \subindex{Continuously varying!elliptic
differential operators}\subindex{Elliptic differential operators!symmetric!curve}curve of symmetric
elliptic first order differential operators. Green's form for $A(s)$ induces on $L^2(\Si;E|_{\Si})$
a \textit{strong} symplectic form $\w(s)_{\operatorname{Green}}(x,y):=-\lla J(s)x,y\rra_{L^2}$. On
\symindex{H@$H^{1/2}(\Si;E\mid_{\Si})$ weak symplectic Sobolev space}$H^{1/2}(\Si;E|_{\Si})$ the
induced symplectic form \symindex{\omega s@$\w(s)_{\operatorname{Green}}$ by Green's form induced
symplectic form}$\w(s)(x,y):=\w(s)_{\operatorname{Green}}(x,y)= -\lla J'(s)x,y\rra_{H^{1/2}}$ is
\textit{weak}. To see that, we choose a formally self-adjoint elliptic operator $B$ of first order
on $\Si$ to generate the metric on \symindex{H@$H^{1/2}(\Si;E\mid_{\Si})$ weak symplectic Sobolev
space}$H^{1/2}$ according to \subindex{G{\aa}rding's Inequality}G{\aa}rding's Theorem. Then we find
$J'(s)=(I+|B|)^{-1/2} J(s)$, which is a compact operator and so not invertible. This we emphasized
already in our \auindex{Boo{\ss}--Bavnbek,\ B.}\auindex{Long,\ Y.}\auindex{Zhu,\ C.}\cite{BoZh04}
where we raised the following questions:

\begin{description}
\item[Q1] How to define $\Mas(\la(s),\mu(s))_{s\in [0,1]}$ for curves of Fredholm pairs of
Lagrangian subspaces?

\item[Q2] How to calculate?
\item[Q3] What for?\qquad

\item[Q4] Dispensable? Non-trivial example?
\end{description}

Questions Q3 and Q4 are addressed below in  Chapter \ref{s:gsff} (see also our
\auindex{Boo{\ss}--Bavnbek,\ B.}\auindex{Zhu,\ C.}\cite{BoZh04}). There we point to the necessity
to work with the weak symplectic Hilbert space $H^{1/2}(\Si)$. Such work is indispensable when we
are looking for \subindex{Spectral flow!spectral flow formula}spectral flow formulae for
\subindex{Partitioned manifold}partitioned manifolds with curves of \subindex{Continuously
varying!elliptic differential operators}\subindex{Elliptic differential operators!of not-Dirac
type}elliptic operators which are \textit{not} of Dirac type.

To answer questions Q1 and Q2, we recall the following list of obstructions and open problems,
partly from \auindex{Boo{\ss}--Bavnbek,\ B.}\auindex{Zhu,\ C.}\cite{BoZh04} (see also Figures
\ref{f:obstructions}, \ref{f:examples}). For simplicity, we specify for Hilbert spaces instead of
Banach spaces:

Let $(X,\w)$ be a fixed complex Hilbert space with weak symplectic form $\w(x,y)=\lla Jx,y\rra$,
and $(X(s),\w(s)),s\in[0,1]$ a \subindex{Continuously varying!symplectic
forms}\subindex{Continuously varying!Hilbert inner products}curve of weak symplectic Hilbert
spaces, parametrized over the interval $[0,1]$ (other parameter spaces could be dealt with). Then
in general we have in difference to strong symplectic forms:

\begin{enumerate}\subindex{Symplectic splitting}\subindex{Annihilator!double}
\subindex{Obstructions to straight forward generalization!generator $J$ non-invertible}
\renewcommand{\labelenumi}{(\Roman{enumi})}
\item $J^2\ne -I$;
\subindex{Obstructions to straight forward generalization!no symplectic splitting}
\item so, in general $X\ne X^-\oplus X^+$ with $X^{\pm}:=\ker(J\mp iI)$; more generally, our
Example \ref{ex:no-splitting} shows that there exist strong symplectic Banach spaces that do not
admit any symplectic splitting;
\subindex{Obstructions to straight forward generalization!symplectic splitting not continuously
varying even when existing}
\item in general, for continuously varying $\w(s)$ it does not hold that $X^{\mp}(s)$ is continuously varying;
\subindex{Obstructions to straight forward generalization!double annihilator not the identity on
closed subspaces}
\item as shown in our Example \ref{ex:double-annihilator}, we have $\la^{\w\w}\supsetneqq \la$ for some closed linear subspaces $\la$;
according to our Lemma \ref{l:double-annihilator}, the double annihilator, however, is the identity
for $\w$-closed subspaces, where the topology is defined by the semi-norms $p_y(x):=|\w(x,y)|$
(based on \auindex{Schmid,\ R.}R. Schmid, \cite{Schmid:1987});
\subindex{Obstructions to straight forward generalization!non-vanishing index of Fredholm pair of
Lagrangian subspaces}
\item by Corollary \ref{c:fred-pair-index-vanishing} we have $\Index(\la,\mu)\le 0$ for $(\la,\mu)\in \Ff\Ll(X)$;
our Example \ref{ex:negative-index} shows that there exist Fredholm pairs of Lagrangian subspaces
with truly negative index; hence, in particular, the concept of the \textit{Maslov cycle}
\symindex{M@$\Mm(X,\w,\la_0)$ Maslov cycle}\subindex{Symplectic geometry!Maslov
cycle}$\Mm(X,\w,\la_0):=\Ff\Ll(X,\la_0)\setminus\Ff\Ll^0(X,\la_0)$ of a fixed Lagrangian
subspace $\la_0$ (comprising all Lagrangians that form a Fredholm pair with $\la_0$ but do not
intersect $\la_0$ transversally) is invalidated: \subindex{Obstructions to straight forward
generalization!no meaningful Maslov cycle}we can no longer conclude complementarity of $\mu$ and
$\la_0$ from $\mu\cap\la_0=\{0\}$;
\subindex{Obstructions to straight forward generalization!$\Ll$ not contractible}
\item in general, the space $\Ll(X,\w)$ is \subindex{Counterexamples!$\Ll$ not contractible}
not contractible and even not connected according to Swanson's arguments for counterexamples
\auindex{Swanson,\ R.C.}\cite[Remarks after Theorem 3.6]{Swanson:1980}, based on \auindex{Douady,\
A.}A. Douady, \cite{Douady:1965};
\subindex{Obstructions to straight forward generalization!fundamental group of Fredholm Lagrangian
Grassmannian unknown}\subindex{Fredholm Lagrangian Grassmannian}
\item \symindex{\pi@$\pi_1(\Ff\Ll(\la,\cdot))$ fundamental group of Fredholm Lagrangian Grassmannian}
$\pi_1(\Ff\Ll_0(X,\la))\fequal{?}\ZZ$ for $\la\in\Ll(X,\w)$; valid for strong symplectic Hilbert
space $(X,\w)$.
\end{enumerate}%

\begin{figure}
\includegraphics[scale=0.95]{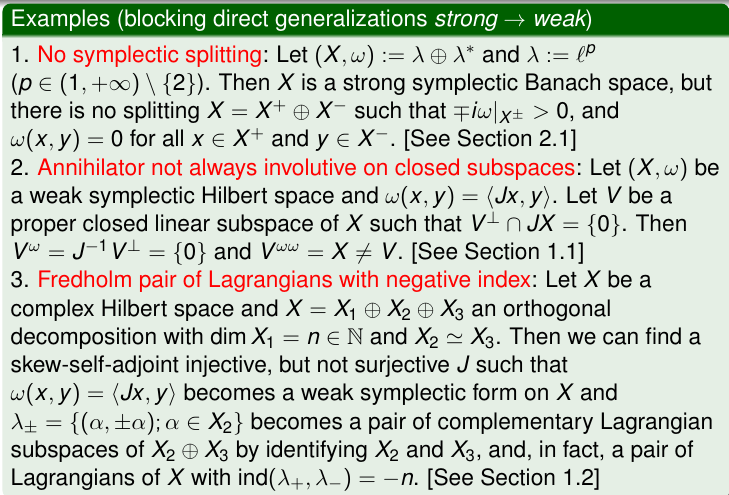}
\caption{Three counterexamples}\label{f:examples}
\end{figure}

\subsection*{Limited value of our previous pilot study}
Anyway, our previous \auindex{Boo{\ss}--Bavnbek,\ B.}\auindex{Zhu,\ C.}\cite{BooZhu:2013} deals
with a continuous family of \subindex{Continuously varying!weak symplectic forms}weak symplectic
forms $\w(s)$ on a curve of \subindex{Continuously varying!Banach spaces}Banach spaces $X(s)$,
$s\in[0,1]$. It gives a definition of the \subindex{Maslov index}Maslov index for a path
$(\lambda(s),\mu(s))_{s\in [0,1]}$ of Fredholm pairs of Lagrangian subspaces of index $0$ under the
assumption of a \subindex{Continuously varying!symplectic splittings}continuously varying
symplectic splitting $X=X^+(s)\oplus X^-(s)$. The definition is inspired by the careful
distinctions of planar intersections in \auindex{Long,\ Y.}\auindex{Zhu,\ C.}\cite{ZhLo99, Zh00,
Zh01, Zh05}. Then it is shown that all nice properties of the Maslov index are preserved for this
general case. However, that approach has four serious \subindex{Obstructions to straight forward
generalization}drawbacks which render this definition incalculable:

\begin{enumerate}\subindex{Symplectic splitting}\subindex{Annihilator!double}
\renewcommand{\labelenumi}{\arabic{enumi}.}
\item In Section \ref{ss:maslov-unitary}, our Example \ref{ex:no-splitting} provides a strong symplectic Banach space that
does not admit a symplectic splitting.

\item Even when a single \subindex{Symplectic splitting!continuously varying}symplectic splitting is
guaranteed, there is no way to establish such splitting for families in a continuous way (see also
our obstruction III above).

\item The Maslov index, as defined in \cite{BooZhu:2013} becomes independent of the choice of the splitting only for strong symplectic
forms.

\item That construction admits finite-dimensional \subindex{Symplectic reduction}symplectic reduction
only for strong symplectic forms.

\end{enumerate}

To us, our \auindex{Boo{\ss}--Bavnbek,\ B.}\auindex{Zhu,\ C.}\cite{BooZhu:2013} is a highly
valuable pilot study, but the preceding limitations explain why in this Memoir we begin again from
scratch. For that purpose, an encouraging result was obtained in \auindex{Boo{\ss}--Bavnbek,\
B.}\auindex{Zhu,\ C.}\cite{BoLeZh08} combined with \auindex{Chen,\ G.}\cite{BCLZ}: the continuous
variation of the \subindex{Calder{\'o}n projection}Calder{\'o}n projection in $L^2(\Si)$ for a
\subindex{Continuously varying!elliptic differential operators}curve of elliptic differential
operators of first order. We shall use this result in our Section \ref{ss:gsff}.

\subsection*{Acknowledgements} {\small We thank Prof. K. Furutani (Tokyo), Prof. M. Lesch
(Bonn), and Prof. R. Nest (Copenhagen) for inspiring discussions about this subject. Last but not least,
we would like to thank the referees of this paper for their critical reading and very helpful comments and suggestions.
The second author was partially supported by NSFC (No.11221091 and No. 11471169), LPMC of MOE of China.}


%
%
%

\part{Maslov index in symplectic Banach spaces}\label{part:1}

\addtocontents{toc}{\medskip\noi}
\chapter{General theory of symplectic analysis in Banach
spaces}\label{s:banach} We fix the notation and establish our key technical device in Proposition
\ref{p:reduction-prop} and Corollary \ref{c:complementary-lagrangian}, namely a natural
decomposition of a fixed symplectic vector space into two symplectic subspaces induced by a
single Fredholm pair of Lagrangians of index $0$. Reversing the order of the Fredholm pair, we obtain
an alternative symplectic reduction. We establish the transitivity of symplectic reductions in Lemma
\ref{l:red-transitive} and Corollary \ref{c:inner-red}. In Proposition \ref{p:cal-red}, we show that the two
natural symplectic reductions coincide by establishing Lemma \ref{l:red-index}. As we shall see later in Section
\ref{s:maslov-general}, that yields the symplectic reduction to finite dimensions for a given path
of Fredholm pairs of Lagrangian subspaces of index 0 in a fixed Banach space with varying
symplectic structures and the invariance of the Maslov index under different symplectic reductions.

Our assumption of vanishing index is trivially satisfied in \textit{strong} symplectic Hilbert
space. More interestingly and inspired by and partly reformulating previous work by
\auindex{Schmid,\ R.}R. Schmid, \auindex{Bambusi,\ D.}and D. Bambusi \cite{Schmid:1987, Bambusi},
we obtain in Lemma \ref{l:double-annihilator} a delicate condition for making the annihilator an
involution. In Corollary \ref{c:fred-pair-index-vanishing} we show that the index of a Fredholm pair of Lagrangian subspaces
can not be positive. In Corollary \ref{c:negative-sum-index} we derive a necessary and sufficient
condition for its vanishing for \textit{weak} symplectic forms and in the concrete set-up of our
global analysis applications in Section \ref{s:gsff}. In order to emphasize the intricacies of weak
symplectic analysis, it seems worthwhile to clarify in Lemma \ref{l:double-annihilator} a
potentially misleading formulation in \auindex{Schmid,\ R.}\cite[Lemma 7.1]{Schmid:1987}, and in
Remark \ref{r:bambusi-error}, to isolate an unrepairable error in \auindex{Bambusi,\ D.}\cite[First
claim of Lemma 3.2, pp.\/3387-3388]{Bambusi}, namely the wrong claim that the double annihilator is
the identity on all closed subspaces of reflexive weak symplectic Banach spaces.

To settle some of the ambiguities around weak symplectic forms once and for all, we provide two
counterexamples in Examples \ref{ex:double-annihilator} and \ref{ex:negative-index}. The first
gives a closed subspace where the double annihilator is not itself. The second gives a Fredholm
pair of Lagrangians with negative index.

\section{Dual pairs and double annihilators}{\label{ss:ann} Our point of departure is recognizing
the difficulties of dealing with both \textit{varying} and \textit{weak} symplectic structures, as
explained in our \auindex{Boo{\ss}--Bavnbek,\ B.}\auindex{Zhu,\ C.}{\cite{BooZhu:2013}. As shown
there, a direct way to define the Maslov index in that context requires a \subindex{Continuously
varying!symplectic splittings}continuously varying symplectic splitting. As mentioned in the
Introduction, neither the existence nor a continuous variation of such a splitting is guaranteed.
Consequently, that definition is not very helpful for calculations in applications.

To establish an intrinsic alternative, we shall postpone the use of the symplectic structures to
later sections and do as much as possible in the rather \textit{neutral} category of linear
algebra. A first taste of the use of purely algebraic arguments of linear algebra for settling open
questions of symplectic geometry is the making of a kind of \subindex{Annihilator!pair annihilator
of linear algebra}annihilator. For the true annihilator concept of symplectic geometry see below
Definition \ref{d:symplectic-space}.c.

Already here we can explain the need for technical innovations when dealing with weak symplectic
structures instead of hard ones. To give a simple example, let us consider a complex symplectic
Hilbert space $(X,\lla,\rra,\w)$ with $w(x,y)=\lla Jx,y\rra$ for all $x,y\in X$ where $J\colon X\to
X$ is a bounded, injective and skew-self-adjoint operator (for details see below Section
\ref{s:basic-concepts}). Then we get at once $\la^{\w}=(J\la)^{\bot}$ and $\la^{\w\w}\>\ol{\la}$
for all linear subspaces $\la\< X$. We denote the orthogonal complement by the common orthogonality
exponent $\bot$ and the symplectic annihilator by the exponent $\w$. Now, if we are in the
\textit{strong} symplectic case, we have $J$ surjective and $J^2=-I$, possibly after a slight
deformation of the inner product. In that case, we have immediately
\[
\la^{\w\w}=\bigl(J\left((J\la)^\bot\right)\bigr)^\bot = (\la^\bot)^\bot =\ol{\la}.
\]
Hence the double annihilator is the identity on the set of closed subspaces in strong symplectic
Hilbert space, like in the familiar case of finite-dimensional symplectic analysis. Moreover, from
that it follows directly that the index of a Fredholm pair of Lagrangians (see Definition
\ref{d:fredholm-lag-pair} and Corollaries \ref{c:fred-pair-index-vanishing} and
\ref{c:negative-sum-index}) vanishes in strong symplectic Hilbert space.

The preceding chain of arguments breaks down for the double annihilator in \textit{weak} symplectic
analysis, and we are left with two basic technical problems:

\begin{enumerate}
\item[(1)] when do we have precisely $\la^{\w\w}=\ol{\la}$, and consequently,

\item[(2)] when are we guaranteed the vanishing of the index of a Fredholm pair of Lagrangian subspaces?

\end{enumerate}

As mentioned above, we are not the first who try to determine the precise conditions for the
annihilator of an annihilator not to become larger than the closure of the original space. We are
indebted to the previous work by \auindex{Schmid,\ R.}R. Schmid \cite[Arguments of the proof of
Lemma 7.1]{Schmid:1987} and \auindex{Bambusi,\ D.}D. Bambusi \cite[Arguments around Lemmata 2.7 and
3.2]{Bambusi}. They suggested to apply a wider setting and address the \subindex{Annihilator!pair
annihilator of linear algebra}pair-annihilator concept of linear algebra. We shall follow - and
modify - some of their arguments and claims.

\begin{definition}\label{d:pair-annihilator}
Let $X$, $Y$ be two complex vector spaces. Denote by $\R$, $\C$ and $\Z$ the sets of real numbers,
complex numbers and integers, respectively. Let $h\colon \C\to\C$ be a $\R$-linear isomorphism. Let
\symindex{\Omega@$\Omega:X\times Y\to\C$ generalized $\R$-linear form on pair of  complex vector
spaces}$\Omega\colon X\times Y\to\C$ be a $\R$-linear map with $\Omega(ax,by)=ah(b)\Omega(ax,by)$
for all $a,b\in\C$ and $(x,y)\in X\times Y$.
\newline (a) For each of the subspaces $\la\subset X$ and $\mu\subset Y$, we define the right and left {\em annihilators} of $\la$
and $\mu$ as real linear subspaces of $X$ and $Y$ by
\begin{align}
\symindex{\la@$\la^{\Omega,r}, \mu^{\Omega,l}$ right and left annihilators of subspaces in
components of a pair of  complex vector spaces with generalized bilinear
form}\subindex{Annihilator!left and right}
\label{e:r-annihilator}\la^{\Omega,r}\ :&=\ \{y\in Y;\Omega(x,y)=0,\forall x\in\la\},\\
\label{e:l-annihilator}\mu^{\Omega,l}\ :&=\ \{x\in X;\Omega(x,y)=0,\forall y\in\mu\}.
\end{align}
\newline (b)
The form $\Omega$ is said to be {\em non-degenerate in} $X$ (\textit{in} $Y$) if
$X^{\Omega,r}=\{0\}$ ($Y^{\Omega,l}=\{0\}$). The form $\Omega$ is said to be just {\em
non-degenerate} if $X^{\Omega,r}=\{0\}$ and $Y^{\Omega,l}=\{0\}$. In that case one says that $X,Y$
form an {\em algebraic $\R$-dual pair} (see also \auindex{Pedersen,\ G.K.}Pedersen
\cite[2.3.8]{Pedersen}).
\newline (c)
We have the \textit{reduced form}\symindex{\Omega a@$\tilde\Omega$ induced non-degenerate reduced
form for pairs of vector spaces}
\[
\wt\Omega\colon X/Y^{\Omega,l}\times Y/X^{\Omega,r}\too\C%
\]
defined by $\wt\Omega(x+Y^{\Omega,l},y+X^{\Omega,r}):=\Omega(x,y)$ for each $(x,y)\in X\times Y$.
\newline (d) The \symindex{\Omega b@$\Omega^b$ annihilator map}\subindex{Annihilator!annihilator map}\textit{annihilator map} $\Omega^b\colon Y\to\Hom(X,\C)$ is the $\R$-linear map defined
by $\Omega^b(y)(x):=\Omega(x,y)$ for all $x\in X$.
\end{definition}

\begin{note}
By definition, the reduced form $\wt\Omega$ is always non-degenerate, since%
\begin{align*}
(X&/Y^{\W,l})^{\wt \W,r} \\
&= \ \{y+X^{\W,r}\,;\, \wt\W\left(x+Y^{\W,l},y+X^{\W,r}\right)=\W(x,y)=0\,\,\forall x\in X\}\\
&=\  X^{\W,r}\ =\ \{0\}\quad\text{ in $Y/X^{\W,r}$,}
\end{align*}
making the form $\wt\W$ non-degenerate in $X/Y^{\W,l}$. Similarly, we obtain
$\left(Y/X^{\W,r}\right)^{\wt \W,l}= Y^{\W,l}$, making the form $\wt\W$ non-degenerate in
$Y/X^{\W,r}$.
\end{note}

We list a few immediate consequences. First of all, we have $\ker_{\R}\Omega^b=X^{\Omega,r}$, as
real vector spaces. Then we have $\la+Y^{\Omega,l}\subset(\la^{\Omega,r})^{\Omega,l}$, and
$\la_1^{\Omega,r}\supset\la_2^{\Omega,r}$ if $\la_1\subset\la_2\subset X$. From that we get
$\la^{\Omega,r}\supset((\la^{\Omega,r})^{\Omega,l})^{\Omega,r}\supset \la^{\Omega,r}$, hence
\begin{equation}\label{e:three-Omega}\la^{\Omega,r}=((\la^{\Omega,r})^{\Omega,l})^{\Omega,r}.
\end{equation}

The following lemma generalizes our \auindex{Boo{\ss}--Bavnbek,\ B.}\auindex{Zhu,\ C.}\cite[Lemma
5, Corollary 1]{BooZhu:2013}. We shall use it below in the proof of Lemma \ref{l:negative-index} to
establish the general result that the index of Fredholm pairs of Lagrangians in symplectic Banach
space always is non-positive. The results also hold for the bilinear forms on vector spaces over a field.

\begin{lemma}\label{l:negative} (a) If $\dim X/Y^{\Omega,l}<+\infty$, we have%
\[
\dim Y/X^{\Omega,r}\ =\ \dim X/Y^{\Omega,l}.%
\]

\noi (b) Let $\la\subset\mu\subset X$ be two linear subspaces. If $\dim\mu/((\la^{\Omega,r})^{\Omega,l}\cap\mu)<+\infty$, we have%
\[
\dim\la^{\Omega,r}/\mu^{\Omega,r}\ =\mu/((\la^{\Omega,r})^{\Omega,l}\cap\mu) \le\ \dim \mu/(\la+Y^{\Omega,l}\cap\mu).
\]
The equality holds if and only if $(\la^{\Omega,r})^{\Omega,l}\cap\mu=\la+Y^{\Omega,l}\cap\mu$.

\noi (c) Let $\la\subset X$ be a linear subspace. If $\dim(\la+Y^{\Omega,l})/Y^{\Omega,l}<+\infty$, we have%
\[
\dim(\la+Y^{\Omega,l})/Y^{\Omega,l}\ =\ \dim Y/\la^{\Omega,r} \ \tand\
\la+Y^{\Omega,l}\ =\ (\la^{\Omega,r})^{\Omega,l}.%
\]
\end{lemma}

\begin{proof} (a) If $\dim X<+\infty$ and $X^{\Omega,r}=\{0\}$, $\Omega^b$ is injective.
Then we have $2\dim Y=\dim_{\R}Y\le\dim_{\R}\Hom(X,\C)=2\dim X$. So we have $\dim Y\le\dim X$.

If $\Omega$ is non-degenerate, we have $\dim X\le\dim Y$ and $\dim X=\dim Y$. Applying the argument
for $\wt\Omega$, we have $\dim X/Y^{\Omega,l}=\dim Y/X^{\Omega,r}$.
\newline (b) Let $f=\Omega|_{\mu\times\la^{\Omega,r}}$.
Then we have
\[
(\mu)^{f,r}\ =\ \mu^{\Omega,r}\cap\la^{\Omega,r}=\mu^{\Omega,r} \text{ and }(\la^{\Omega,r})^{f,l}\ =\
(\la^{\Omega,r})^{\Omega,l}\cap\mu.
\]
Note that $(\la^{\Omega,r})^{\Omega,l}\cap\mu\supset\la+Y^{\Omega,l}\cap\mu$. By (a), we get our results.
\newline (c) Note that $(Y^{\Omega,l})^{\Omega,r}=Y$ and $(\la+Y^{\Omega,l})^{\Omega,r}=\la^{\Omega,r}\cap Y=\la^{\Omega,r}$. By (b) we have
\[
\dim(\la+Y^{\Omega,l})/Y^{\Omega,l}\ =\ \dim Y/\la^{\Omega,r}\ \ge\ \dim (\la^{\Omega,r})^{\Omega,l}/Y^{\Omega,l}.\]
Since
$\la+Y^{\Omega,l}\subset (\la^{\Omega,r})^{\Omega,l}$, we have
$\la+Y^{\Omega,l}=(\la^{\Omega,r})^{\Omega,l}$.
\end{proof}

Assume that $\Omega$ is non-degenerate in $Y$. Then the
\symindex{F@$\Ff$ family of semi-norms on first factor of paired vector spaces}
\symindex{py@$p_y(\cdot)$ family of semi-norms}
family of semi-norms
$\Ff:=\{p_y(x):=|\Omega(x,y)|, \, x\in X\}_{y\in Y}$ is {\em separating}, i.e., for $x\neq x'$ in
$X$, there is a $y\in Y$ such that $p_y(x-x')\neq 0$. We shall denote the topology on $X$ induced
by the family $\Ff$ by
\symindex{TT@$\Tt_\W$ weak topology induced by pairing $\W$ ($\W$-topology)}
$\Tt_\W$ and call it the {\em weak topology induced by $\W$} or shortly the {\em
$\W$-topology}. By \auindex{Pedersen,\ G.K.|bind}\cite[1.5.3 and 3.4.2]{Pedersen} $(X,\Tt_{\W})$
becomes a Hausdorff separated, locally convex, topological vector space. The following two lemmata
are proved implicitly by \auindex{Schmid,\ R.}\cite[Arguments of the proof of Lemma
7.1]{Schmid:1987}. Clearly, we have

\begin{lemma}\label{l:dual-space}
Assume that $\Omega$ is non-degenerate in $Y$. Then the real linear map $\Omega^b$ maps $Y$ onto
$(X,\Tt_{\W})^*$.
\end{lemma}

Then the \subindex{Hahn-Banach Extension Theorem}Hahn-Banach Theorem yields
\begin{lemma}[R. Schmid, 1987]\label{l:double-annihilator}
\subindex{Schmid's Lemma}
Assume that $\Omega$ is non-degenerate in $Y$ and $\la$ is a closed linear
subspace of $(X,\Tt_{\W})$. Then we have
\begin{equation}\label{e:double-annihilator}\subindex{Annihilator!double annihilator
of linear algebra} \la=(\la^{\Omega,r})^{\Omega,l}.
\end{equation}
\end{lemma}

For later use it is worth noting the following extension of Schmid's Lemma which is the weak and
corrected version of \cite[Lemma 3.2]{Bambusi}.

\begin{lemma}\label{l:double-annihilator1}
\subindex{Schmid's Lemma}
Assume that $X,Y,\W$ as above and $\Omega$ non-degenerate in $Y$ and bounded in $X$. Assume that
$X$ is a reflexive Banach space. Then $\Omega^b(Y)$ is dense in $X^*$ and we have
\begin{equation}\label{e:double-annihilator1}\la=(\la^{\Omega,r})^{\Omega,l}
\quad\text{for any linear and $\Tt_{\W}$-closed subspace $\la\< X$}.
\end{equation}
\end{lemma}

\medskip
\section{Basic symplectic concepts}\label{s:basic-concepts}
Before defining the Maslov index in symplectic Banach space by symplectic reduction to the
finite-dimensional case, we recall the basic concepts and properties of symplectic functional
analysis.

\begin{definition}\label{d:symplectic-space}
Let $X$ be a complex vector space.%
\newline (a) A mapping
\[\symindex{\w@$\w: X\times X\too \C$ symplectic form}\subindex{Symplectic form!weak}
  \w\colon X\times X\too \C
\]
is called a {\em symplectic form} on $X$, if it is sesquilinear, skew-symmetric, and
non-degenerate, i.e.,

\begin{enumerate}

\item $\w(x,y)$ is linear in $x$ and conjugate linear in $y$;

\item $\w(y,x)\ =\ -\ol{\w(x,y)}$;

\item $X^{\w} \ :=\  \{x\in X ;\w(x,y)\ =\ 0\text{ for all $y\in X$}\} \ =\  \{0\}$.

\end{enumerate}
Then we call $(X,\w)$ a {\em symplectic vector space}.\subindex{Symplectic structures!symplectic
vector space}
\newline (b) Let $X$ be a complex Banach space and $(X,\omega)$ a symplectic vector space. $(X,\omega)$
is called {\em (weak) symplectic Banach space}, if $\omega$ is bounded, i.e., $|\w(x,y)| \leq
C\|x\| \|y\|$ for all $x,y\in X$.\subindex{Symplectic structures!symplectic Banach space}
\newline (c) The {\em annihilator} of a subspace ${\la}$ of $X$ is defined by
\[\subindex{Annihilator}
{\la}^{\w} \ :=\  \{y\in X ; \w(x,y)\ =\ 0 \quad\text{ for all $x\in {\la}$}\} .
\]
\newline (d) A subspace ${\la}$ is called \subindex{Symplectic subspaces}{\em symplectic}, \subindex{Isotropic subspaces}
{\em isotropic}, \subindex{Co-isotropic subspaces}{\em co-isotropic}, or \subindex{Lagrangian
subspaces}{\em Lagrangian} if
\[
\la\cap{\la}^{\w}\ =\ \{0\}\,,\quad{\la} \,\<\, {\la}^{\w}\,,\quad {\la}\,\>\, {\la}^{\w}\,,\quad
{\la}\,\ =\ \, {\la}^{\w}\,,
\]
respectively.
\newline (e) \subindex{Lagrangian Grassmannian}The {\em Lagrangian Grassmannian}
\symindex{L@$\Ll(X,\w)$ Lagrangian Grassmannian}$\Ll(X,\w)$ consists of all Lagrangian subspaces of
$(X,\w)$. We write $\Ll(X):=\Ll(X,\w)$ if there is no confusion.
\end{definition}

\begin{rem}\label{r:bambusi-error}
(a) Let $(X,\w)$ be a complex weak symplectic Banach space. By definition (see below), the form
$\w\colon X\times X\to \C$ is non-degenerate. Then we have \subindex{Topologies on symplectic
Banach spaces (norm, weak, $\w$-weak)}three topologies on $X$: the {\em norm-topology}, the
canonical {\em weak topology} induced from the family $X^*$ of continuous functionals on $X$, and
the {\em $\w$-induced weak topology} $\Tt_{\w}$. The weak topology is weaker than the norm
topology; and the $\w$-induced topology is weaker than the weak topology. So, a closed subset $V\<
X$ is not necessarily weakly closed or closed in $(X,\Tt_{\w})$: the set $V$ can have more
accumulation points in the weak topology and even more in the $\w$-induced weak topology than in
the norm topology. A standard example is the unit sphere that is not weakly closed in infinite
dimensions (see, e.g., \auindex{Brezis,\ H.}H. Brezis \cite[Example 1, p. 59]{Brezis:2011}.
Fortunately, by \cite[Theorem 3.7]{Brezis:2011} every norm-closed {\em linear} subspace is weakly
closed. Hence it is natural (but erroneous) to suppose that the difference between the three
topologies does not necessarily confine severely the applicability of Schmid's Lemma, namely to
linear subspaces.

\noi (b) It seems that \auindex{Bambusi,\ D.}D. Bambusi in \cite[Lemmata 2.7,3.2]{Bambusi} supposed
erroneously that in reflexive Banach space all norm-closed subspaces are not only weakly closed but
also $\w$-weakly closed. Rightly, in spaces where that is valid, Schmid's Lemma is applicable (or
can be reproved independently).

\noi (c) Recall that a Banach space $X$ is \textit{reflexive} if the isometry
\[
\iota\colon  X\too X^{**}, \text{ given by $\iota(x)(\f):=\f(x)$ for $x\in X,\ \f\in X^*$}
\]
is surjective, i.e., its range is the whole bidual space $X^{**}$. Typical examples of reflexive
spaces are all Hilbert spaces and the $L^p$-spaces for $1<p<\infty$, but not $L^1$.

\noi (d) Unfortunately, in general the claim of \cite[Lemma 3.2]{Bambusi} (the validity of the
idempotence of the double annihilator for closed linear subspaces in complex reflexive symplectic
Banach space) is not correct. \subindex{Obstructions to straight forward generalization!double
annihilator not the identity on closed subspaces}If it was correct, then, e.g., in (automatically
reflexive) weak symplectic Hilbert space $(X,\lla\cdot,\cdot\rra, \w)$, the double annihilator
$\la^{\w\w}$ of every closed subspace $\la$ should coincide with $\la$. However, here is a
counterexample: Let $\left(X,\lla\cdot,\cdot\rra\right)$ be a complex Hilbert space and $J\colon
X\to X$ a bounded injective skew-self-adjoint operator. Then $\omega\colon X\times X\to\C$ defined
by $\omega(x,y):=\lla Jx,y\rra$ is a symplectic form on $X$. So $\ran J$ is dense in $X$. For $V\<
X$ closed subspace, denote by $V^\perp$ the orthogonal complement of $V$ with respect to the inner
product on $X$, and by \symindex{\la 2@$\la^{\omega}$ symplectic complement (annihilator) of
subspace $\la$}$V^{\omega}$ the symplectic complement (i.e., the annihilator) of $V$. Then we have
\begin{equation}\label{e:annihi-perp}
V^{\omega}\ =\ (JV)^{\perp}\ =\ J^{-1}\left(V^{\perp}\right).
\end{equation}
Now assume that $\ran J\ne X$ (like in the weak symplectic Sobolev space
$X:=H^{1/2}(\Si;E|_{\Si})$, as explained in the Introduction). Let $x\in X\setminus\ran J$ and set
$V:=\left(\Span\{x\}\right)^{\perp}$. Then we have  $J^{-1}\left(V^{\perp}\right)=\{0\}$, hence
$V^{\omega}=\{0\}$ and $V^{\omega\omega}=X\ne V$. That contradicts the first part of Equation (13) in
\cite[Lemma 3.2]{Bambusi}.

\noi (e) The preceding example contradicts \cite[Equation (11)]{Bambusi}, as well: For any closed
subspace $V\< X$ we have $J\left((JV)^{\perp}\right)\subset V^{\perp}$. Then Bambusi's Equation
(11) is equivalent to
\[
\overline{J\left((JV)^{\perp}\right)}\ =\ V^{\perp}.
\]
For our concrete example $V:=\left(\Span\{x\}\right)^{\perp}$, however, we obtain%
\[
J\left((JV)^{\perp}\right)\cap V^{\perp}=\{0\}\ \tand \ V^{\perp}=\Span\{x\}.
\]
Thus \cite[Equation (11)]{Bambusi} is incorrect.

\noi (f) For any Lagrangian subspace $\la$ in a complex symplectic Banach space $(X,\w)$ we have
$\la^{\w\w}=\la$ by definition. That follows also directly from the identity \eqref{e:three-Omega},
and, alternatively, from Schmid's Lemma, since a Lagrangian subspace is always $\w$-closed.
\end{rem}

The counterexample of the preceding Remarks d and e can be generalized in the following form.

\begin{example}[Closed subspaces different from their double annihilators]\label{ex:double-annihilator}
\subindex{Obstructions to straight forward generalization!double annihilator not the identity on
closed subspaces}\subindex{Counterexamples!double annihilator not the identity on closed subspaces}
 \subindex{Symplectic structures!symplectic Hilbert space}
Let $(X,\w)$ be a \subindex{Hilbert space!weak symplectic}weak symplectic Hilbert space and
$\w(x,y)=\lla Jx,y\rra$. Let $V$ be a proper closed linear subspace of $X$ such that $V^{\bot}\cap
JX=\{0\}$. Then $V^{\w}=J^{-1}V^{\bot}=\{0\}$ and $V^{\w\w}=X \supsetneqq  V$.
\end{example}

\begin{rem}\label{r:double-annihilator1}
(a) By definition, each one-dimensional subspace in real symplectic space is isotropic, and there
always exists a Lagrangian subspace in finite-dimensional real symplectic Banach space, namely the
maximal isotropic subspace. However, there are complex symplectic Hilbert spaces
\subindex{Counterexamples!no Lagrangian subspace}without any Lagrangian subspace. That is, in
particular, the case if $\dim X^+\ne \dim X^-$ in $\NN\cup \{\infty\}$ for a single (and hence for
all) symplectic splittings. More generally, we refer to A. \subindex{Weinstein's
Theorem}Weinstein's Theorem \auindex{Weinstein,\ A.}\cite{Weinstein:1971}(see also
\auindex{Swanson,\ R.C.}R.C. Swanson, \cite[Theorem 2.1 and Corollary]{Swanson:1980}) that relates
the existence of complemented Lagrangian subspaces to the generalized \auindex{Darboux,\
J.G.}\subindex{Darboux property}Darboux property, recalled below at the end of Subsection
\ref{ss:product-spaces}.

(b) As in the finite-dimensional case, the basic geometric concept in infinite-dimensional
symplectic analysis is the  Lagrangian subspace, i.e., a linear subspace which is isotropic and
co-isotropic at the same time. Contrary to the finite-dimensional case, however, the common
definition of a Lagrangian as a {\em maximal} isotropic space or an isotropic space of {\em half}
dimension becomes inappropriate.

(c) In symplectic Banach spaces, the annihilator ${\la}^{\w}$ is closed for any linear subspace
$\la$, and we have the trivial inclusion
\begin{equation}\label{e:double-annih}\subindex{Annihilator!double annihilator in symplectic Banach spaces}%
{\la}^{\w\w} \> \ol{\la}.
\end{equation}
In particular, all Lagrangian subspaces are closed, and trivially, as emphasized in Remark
\ref{r:bambusi-error}.f, we have an equality in the preceding \eqref{e:double-annih}.
\end{rem}

If $X$ is a complex Banach space, each symplectic form $\w$ induces a uniquely defined mapping
$J\colon X\to X^{\ad}$ such that\symindex{J@$J$ operator associated to symplectic form}
\begin{equation}\label{e:almost-complex}
\w(x,y) \ =\  (Jx,y) \quad\text{ for all $x,y\in X$},
\end{equation}
where we set $(Jx,y):=(Jx)(y)$. The induced mapping $J$ is a bounded, injective mapping $J\colon
X\to X^{\ad}$ where \symindex{X@$X^{\ad}$ space of all continuous complex-conjugate functionals on
$X$}$X^{\ad}$ denotes the (topological) dual space of continuous complex-conjugate linear
functionals on $X$.

\begin{definition} Let $(X,\omega)$ be a symplectic Banach space.
If $J$ is also surjective (hence with bounded inverse), the pair $(X,\w)$ is called a {\em strong
symplectic Banach space}.\subindex{Symplectic form!strong}
\end{definition}

\begin{lemma}\label{l:double-strong}Let $(X,\omega)$ be a strong symplectic Banach space, and $\lambda\subset X$ be a linear subspace.
Then we have $\lambda^{\omega\omega}=\overline{\lambda}$.
\end{lemma}

\begin{proof} Since $(X,\omega)$ be a strong symplectic Banach space, $J$ is surjective. So the weak topology of $X$ is the same as the $\omega$-induced weak topology. By \cite[Problem III.1.34]{Ka76}, $\overline{\lambda}$ is weakly closed. So it is also $\omega$-closed. By Lemma \ref{l:double-annihilator}, we have $\lambda^{\omega\omega}=\overline{\lambda}$.
\end{proof}

We have taken the distinction between {\em weak} and {\em strong} symplectic structures from P.
Chernoff and J. Marsden \auindex{Chernoff,\ P.R.}\auindex{Marsden,\ J.E.}\cite[Section 1.2, pp.
4-5]{CheMa74}. If $X$ is a Hilbert space with symplectic form $\w$, we identify $X$ and $X^*$. Then
the induced mapping $J$ defined by $\w(x,y)=\lla Jx,y\rra$ is a bounded, skew-self-adjoint operator
(i.e., $J^* =-J$) on $X$ with $\ker J=\{0\}$. As in the strong symplectic case, we then have that
$\la\< X$ is Lagrangian if and only if $\la^\perp=J\la$\,. As explained above, in Hilbert space, a
main difference between weak and strong is that we can assume $J^2=-I$ in the strong case (see
\auindex{Boo{\ss}--Bavnbek,\ B.}\auindex{Zhu,\ C.}\cite[Lemma 1]{BooZhu:2013} for the required
smooth deformation of the inner product), but not in the weak case. The importance of such an
anti-involution is well-known from symplectic analysis in finite dimensions and exploited in strong
symplectic Hilbert spaces, but, in general, it is lacking in weak symplectic analysis.

We recall the key concept to symplectic analysis in infinite dimensions:

\begin{definition}\label{d:fredholm-lag-pair}
The space of \emph{Fredholm pairs} of Lagrangian subspaces of a symplectic vector space
$(X,\omega)$ is defined by\symindex{FL@$\Ff\Ll(X,\w)$ Fredholm Lagrangian
Grassmannian}\subindex{Fredholm Lagrangian Grassmannian!of symplectic vector space}
\begin{multline}\label{e:fp-lag-alg}
\Ff\Ll(X)\ :=\ \{(\lambda,\mu)\in\Ll(X)\x\Ll(X);  \dim (\lambda\cap\mu)  <+\infty \text{
and}\\
\dim X/(\lambda+\mu)<+\infty\}
\end{multline}
with
\symindex{index@$\Index(\la,\mu)$ index of Fredholm pair}\subindex{Fredholm pair!index}
\subindex{Index!of Fredholm pair}
\begin{equation}\label{e:fp-lag-index}
\Index(\la,\mu)\ :=\ \dim(\la\cap\mu) - \dim X/(\la+\mu).
\end{equation}
For $k\in\Z$ we define\symindex{FLk@$\Ff\Ll_k(X),\Ff\Ll(X,\mu),\Ff\Ll_k(X,\mu),\Ff\Ll_0^k(X,\mu)$
subspaces of the Fredholm Lagrangian Grassmannian}
\begin{equation}\label{e:fl-index-k}
\Ff\Ll_k(X)\ :=\ \{(\lambda,\mu)\in\Ff\Ll(X); \Index(\la,\mu)=k\}.
\end{equation}
For $k\in\Z$ and $\mu\in\Ll(X)$ we define
\begin{align}
\label{e:fl-mu}\Ff\Ll(X,\mu):&=\ \{\la\in\Ll(X);(\la,\mu)\in\Ff\Ll(X)\},\\
\label{e:fl-mu-k}\Ff\Ll_k(X,\mu):&=\ \{\la\in\Ll(X);(\la,\mu)\in\Ff\Ll_k(X)\},\\
\label{e:trans-mu}\Ff\Ll_0^k(X,\mu):&=\ \{\la\in\Ff\Ll_0(X,\mu);\dim(\la\cap\mu)=k\}.
\end{align}
\end{definition}

What do we know about the index of Fredholm pairs of Lagrangian subspaces in the weak symplectic
case? Here we give another proof for the fact (proved before in our \cite[Proposition
1]{BooZhu:2013}) that Fredholm pairs of Lagrangian subspaces in symplectic vector spaces never can
have positive index.

\begin{lemma}\label{l:negative-index} Let $(X,\w)$ be a symplectic vector space and $\la_1,\ldots,\la_k$ linear subspaces of $X$.
Assume that $\dim X/(\sum_{j=1}^k\la_j)<+\infty$. Then the following holds.
\newline (a) We have
\begin{equation}\label{e:negative-index1}\dim(\bigcap_{j=1}^k\la_j^{\w})\le\dim X/(\sum_{j=1}^k\la_j).\end{equation}
The equality holds if and only if $\sum_{j=1}^k\la_j=(\sum_{j=1}^k\la_j)^{\w\w}$.
\newline (b) If $\la_j$ is isotropic for each $j$, we have
\begin{equation}\label{e:negative-index2}\dim(\bigcap_{j=1}^k\la_j)\le\dim X/(\sum_{j=1}^k\la_j).\end{equation}
The equality holds if and only if $\bigcap_{j=1}^k\la_j=\bigcap_{j=1}^k\la_j^{\w}$ and
$\sum_{j=1}^k\la_j=(\sum_{j=1}^k\la_j)^{\w\w}$.
\end{lemma}

\begin{proof} (a) Since $\bigcap_{j=1}^k\la_j^{\w}=(\sum_{j=1}^k\la_j)^{\w}$, our result follows from Lemma \ref{l:negative}.b.
\newline (b) By (a) and $\bigcap_{j=1}^k\la_j\subset\bigcap_{j=1}^k\la_j^{\w}$.
\end{proof}

\begin{corollary}[Fredholm index never positive]\label{c:fred-pair-index-vanishing}
a) Let $X$ be a complex vector space with symplectic form $\omega$. Then each Fredholm pair
$(\lambda,\mu)$ of Lagrangian subspaces of $(X,\omega)$ has negative index or is of index $0$.

\noi b) If $(X,\w)$ is a strong symplectic Banach space, then we have
\begin{align}\label{e:double-anni-identity}\subindex{Annihilator!double annihilator in symplectic
Banach spaces}\subindex{Index!of Fredholm pair!vanishing for Lgrangian subspaces in strong
symplectic Banach space}
V^{\w\w}\ &=\ V\quad\text{ for each closed subspace $V\< X$, and}\\
\Index(\la,\mu)\ &=\ 0\quad\text{ for each $(\lambda,\mu)\in
\Ff\Ll(X,\w)$}.\label{e:vanishing-index}
\end{align}
\end{corollary}

\begin{proof} (a) is immediate from the Lemma. To derive (b) from the Lemma, we shall summarize a
couple of \subindex{Symplectic structures!symplectic Banach space!elementary identities}elementary
concepts and identities about symplectic Banach spaces:

\noi For \eqref{e:double-anni-identity} we recall from \eqref{e:almost-complex} that any
symplectic form $\w$ on a complex Banach space $X$ induces a uniquely defined bounded, injective
mapping $J\colon X\to X^{\ad}$ such that $\w(x,y)=(Jx)(y)$ for all $x,y\in X$. Here
\symindex{X@$X^{\ad}$ space of all continuous complex-conjugate functionals on $X$}$X^{\ad}$
denotes the space \symindex{B@$\Bb^{\ad}(X,Y)$ space of all continuous linear-complex-conjugate
operators from $X$ to $Y$}$\Bb^{\ad}(X,\C)$ of all continuous complex-conjugate functionals on $X$.
For linear subspaces $W\< X$ and $Z\< X^{\ad}$, we set \symindex{\lambda_bot@$\lambda^\bot$ dual
complement of subspace $\la$}$W^\bot:=\{e\in X^{\ad}; e(x)=0 \text{ for all } x\in W\}$ and
$Z^\bot:=\{x\in X; e(x)=0 \text{ for all } e\in Z\}$, as usual. By the \subindex{Hahn-Banach
Extension Theorem}Hahn-Banach extension theorem, we have
\begin{equation}\label{e:bot-bot}
W^{\bot\bot}\ =\ \overline{W}\ \tand\ Z^{\bot\bot}\ =\ \overline{Z}.
\end{equation}
Moreover, we have the following elementary identities
\begin{equation}\label{e:annihi-perp2}
W^{\omega}\ =\ (JW)^{\perp}\ =\ J^{-1}\left(W^{\perp}\right).
\end{equation}
They correspond exactly to the identities of \eqref{e:annihi-perp}, given there only for $X$
symplectic Hilbert space.

Recall that we call $\w$ strong, if $J$ is surjective, i.e., an isomorphism. That we assume now. Then we have%
\begin{equation}\label{e:double-anni-involutive-identity-proof}
V^{\w\w}\ \fequal{i}\ J^{-1}(V^{\w\bot})\ \fequal{ii}\ J^{-1}((JV)^{\bot\bot})\ \fequal{iii}\
J^{-1}(JV)\ \fequal{iv}\ V.
\end{equation}
The identities (i) and (ii) follow from \eqref{e:annihi-perp2} and are valid also in the weak case,
while we for identity (iii) need that $J$ is bounded \textit{and} surjective, hence $JV$ is closed
by the \subindex{Open Mapping Theorem}Open Mapping Theorem. Identity (iv) is a trivial consequence
of the injectivity of $J$ and so valid also in the weak case. That proves
\eqref{e:double-anni-involutive-identity-proof}. In particular, we have $(\la+\mu)^{\w\w}=\la+\mu$ and so by
(b) of the Lemma $\dim(\la\cap\mu)=\dim X/(\la+\mu)$. In general, i.e., for weak symplectic form,
we have $\dim(\la\cap\mu)=\dim X/(\la^{\w}+\mu^{\w})$ which does not suffice to prove the vanishing
of the index.
\end{proof}

\begin{rem}\label{r:vanishing-index}
(a) The Corollary has a wider validity. Let $(\la,\mu)$ be a Fredholm pair of isotropic subspaces.
Then we have by Lemma \ref{l:negative-index}.b $\Index(\la,\mu)\le 0$. If $\Index(\la,\mu)=0$,
$\la$ and $\mu$ are Lagrangians (see \auindex{Boo{\ss}--Bavnbek,\ B.}\auindex{Zhu,\
C.}\cite[Corollary 1 and Proposition 1]{BooZhu:2013}).
\newline (b) To obtain $\Index(\la,\mu)=0$ from Lemma \ref{l:negative-index}.b for strong symplectic Banach spaces,
it was crucial that we have
\[
(\la+\mu)^{\w\w}=\la+\mu \ \tand\ \la=\la^{\w}\ \tand\  \mu=\mu^{\w}.
\]
For Lagrangian subspaces the last two equations are satisfied by definition, and the first is our
\eqref{e:double-anni-identity}, valid for strong symplectic $\w$. More generally, by Lemma
\ref{l:double-annihilator}, the first equation is satisfied if the space $\la+\mu$ is $\w$-closed,
i.e., closed in the weak topology $\Tt_\w$ (see above). In a symplectic Banach space $(X,\w)$ all
Lagrangian subspaces are norm-closed, weakly closed and $\w$-weakly closed at the same time, as
emphasized in Remark \ref{r:bambusi-error}. Since $\la$, $\mu$ are norm-closed and $\dim
X/(\la+\mu)<+\infty$, $\la+\mu$ is norm-closed by \auindex{Boo{\ss}--Bavnbek,\
B.}\auindex{Furutani,\ K.}\cite[Remark A.1]{BoFu99} and \auindex{Kato,\ T.}\cite[Problem
4.4.7]{Ka95}. However, that does not suffice to prove that $\la+\mu$ is $\w$-closed, see Remark
\ref{r:double-annihilator1}.c.
\newline (c) Our \eqref{e:vanishing-index} is well known for strong symplectic Hilbert
spaces (follow, e.g., the arguments of \auindex{Boo{\ss}--Bavnbek,\ B.}\auindex{Furutani,\
K.}\cite[Corollary 3.7]{BoFu98}). Below, in Example \ref{ex:negative-index} we give a Fredholm pair
of Lagrangian subspaces in a weak symplectic Hilbert space with negative index. Hence, we can not
take the vanishing of the index for granted for weak symplectic forms, neither in Hilbert spaces -
contrary to the well established vanishing of the index of closed (not necessarily bounded)
self-adjoint Fredholm operators in Hilbert space (\auindex{Bleecker,\
D.}\auindex{Boo{\ss}--Bavnbek,\ B.}\cite[p. 43]{BlBo13}). That may appear a bit strange: Below in
Section \ref{ss:sa-relations}, we shall consider closed operators as special instances of
\subindex{Closed linear relation}closed linear relations. Then, e.g., a closed self-adjoint
Fredholm operator $A$ in a Hilbert space $(X,\lla\cdot,\cdot\rra)$ is a self-adjoint Fredholm
relation, i.e., the pair $(\Graph A,X\times\{0\})$ is a Fredholm pair of Lagrangian subspaces of
the Hilbert space $X\times X$ with the canonical strong symplectic structure
\begin{equation}\label{e:canonical-product-sympform}
\symindex{\omega_can@$\w_{\can}$ canonical strong symplectic form on $X\times X^*$}
\subindex{Symplectic form!canonical and strong on $X\times X^*$}
\w_{\can}\colon (X\times X)\times (X\times X)\too\CC,\quad
(x_1,y_1,x_2,y_2)\mapsto\lla x_1,y_2\rra-\lla x_2,y_1\rra.
\end{equation}
That yields an alternative, namely symplectic proof of the vanishing of the index of a closed
self-adjoint Fredholm operator in Hilbert space, since $\Index A=\Index(\Graph A,X\times\{0\})$ by
\eqref{e:clr-index} and $(\Graph A,X\times\{0\})$ a Fredholm pair of Lagrangian subspaces of
$(X\times X,\w_{\can})$. The preceding arguments generalize immediately for any
\subindex{Fredholm operator!closed self-adjoint}\subindex{Index!of closed self-adjoint Fredholm
operator}
closed self-adjoint Fredholm operator $A\colon X\to X^*$ with $\dom A\< X$ and $X$ reflexive
complex Banach space. We only need to reformulate the canonical strong symplectic form in
\eqref{e:canonical-product-sympform} on the Banach space $X\times X^*$, replacing $X\times X$ by
$X\times X^*$ and $\lla x,y\rra$ by $y(x)$. That yields a strong symplectic form if and only if $X$
is reflexive. For examples of self-adjoint Fredholm operators in ``non-Hilbertable"
\subindex{Banach space!non-Hilbertable}Banach spaces we refer to self-adjoint extensions of the
Laplacian in $L^p$-spaces appearing with convex \subindex{Hamiltonian system!self-adjoint
extensions of the Laplacian}Hamiltonian systems in \auindex{Ekeland,\ I.}I. Ekeland \cite[p.
108]{Ekeland:1990}. Later in Section \ref{ss:product-spaces}, for our applications we shall
introduce a new (and weak) concept of a Fredholm operator $A\colon X\to Y$ in Banach spaces $X,Y$
that is ``self-adjoint" relative to a weak symplectic structure on $X\times Y$ induced by a
non-degenerate sesquilinear form $\W\colon X\times Y\to \C$. A priori, we can not exclude
\subindex{Counterexamples!non-vanishing index for self-adjoint Fredholm operators in Banach spaces
with non-degenerate sesquilinear forms}.
\newline (d) In view of our Example \ref{ex:negative-index}, we shall need special
assumptions below in Chapter \ref{s:gsff} to exclude intractable complications with index
calculations for arbitrary Fredholm relations and ``self-adjoint" Fredholm operators (e.g., see the
assumptions of Proposition \ref{p:sum-0}, Assumption \ref{a:reduced-space} (iv), and Assumption
\ref{a:asff} (iv)).
\newline (e) In our applications, we shall deal only with Fredholm pairs of Lagrangians where the vanishing
of the index is granted by arguments of global analysis or simply because the underlying form is
strong symplectic.
\end{rem}

Here is an example which shows that the index of a Fredholm pair of Lagrangian subspaces in weak
symplectic Banach space need not vanish.

\begin{example}[Fredholm pairs of Lagrangians with negative index]\label{ex:negative-index}
\subindex{Obstructions to straight forward generalization!non-vanishing index of Fredholm pair of
Lagrangian subspaces}\subindex{Hilbert space!weak symplectic}
\subindex{Counterexamples!non-vanishing index of Fredholm pair of Lagrangian subspaces}
Let $X$ be a complex Hilbert space and $X=X_1\oplus X_2\oplus X_3$ an orthogonal decomposition with
$\dim X_1=n\in\NN$ and $X_2\simeq X_3$. Then we can find a bounded skew-self-adjoint injective, but
not surjective $J\colon X\to X$ such that $\w(x,y)=\lla Jx,y\rra$ becomes a weak symplectic form on
$X$. Let $J$ be of the form
\begin{eqnarray*}
J=i\left(\begin{array}{ccc}A_{11} & A_{12} & \bar k A_{12} \\A_{21} & A_{22}& 0\\kA_{21} & 0&
-A_{22}\end{array}\right),
\end{eqnarray*}
where $k\in\CC, k\ne\pm 1$, $\ran A_{21}\cap \ran A_{22}=\{0\}$ and $\ker A_{21}=\ker
A_{22}=\{0\}$.

Set $V=X_2\oplus X_3$. We identify the vectors in $X_2$ and $X_3$. Then the pair $(\la_+,\la_-)$
with $\la_{\pm}:=\{(\alpha,\pm\alpha);\alpha\in X_2\}$ becomes a Fredholm pair of Lagrangian
subspaces of $(V,\w|_V)$ with $\la_+\cap\la_-=\{0\}$ and
$$
V\ =\ \la_{+}\oplus\la_{-}.
$$
We claim that $J^{-1}(X_1\oplus\la_{\pm})\subset V$. In fact, let $(x_1,x_2,x_3)\in
J^{-1}(X_1\oplus\la_{\pm})$. Then there is an $\alpha\in X_2$ such that
$A_{21}x_1+A_{22}x_2=\alpha$ and $kA_{21}x_1-A_{22}x_3=\pm\alpha$. So $(1\mp
k)A_{21}x_1+A_{22}(x_2\pm x_3)=0$. Since $\ran A_{21}\cap\ran A_{22}=0$ and $\ker A_{21}=0$, we
have $x_1=0$.

Note that $\la_{\pm}^{\bot}=X_1\oplus\la_{\mp}$ and $\la_{\pm}^{\w}\cap V=\la^{\pm}$. Then we have
$\la_{\pm}^{\w}=J^{-1}(X_1\oplus\la_{\mp})\subset V $ and $\la_{\pm}^{\w}=\la_{\pm}^{\w}\cap
V=\la^{\pm}$. So $\la_{\pm}$ are Lagrangian subspaces of $(X,\w)$. Then, by definition of $J$ they
form a Fredholm pair of Lagrangians of $X$ with $\Index(\la_+,\la_-)=-n$.
\end{example}

\begin{corollary}\label{c:negative-sum-index} Let $(X,\w)$ be a symplectic vector space and $\la,\mu$ two linear subspaces. Assume that
\[\dim X/(\la+\mu)<+\infty\text{ and }\dim X/(\la^{\w}+\mu^{\w})<+\infty.\]
Then the following holds.
\newline (a) $(\la,\mu)$ and $(\la^{\w},\mu^{\w})$ are Fredholm pairs, and we have
\begin{equation}\label{e:negative-sum-index}\Index(\la,\mu)\ +\ \Index(\la^{\w},\mu^{\w})\ \le\ 0.
\end{equation}
(b) The equality holds in (\ref{e:negative-sum-index}) if and only if $\la+\mu=(\la+\mu)^{\w\w}$,
$\la^{\w}+\mu^{\w}=(\la^{\w}+\mu^{\w})^{\w\w}$, and $\la\cap\mu=\la^{\w\w}\cap\mu^{\w\w}$.
\end{corollary}

\begin{proof} (a) By Lemma \ref{l:negative-index}, we have
\begin{align}
\label{e:la-mu1}&\dim(\la^{\w}\cap\mu^{\w})\le\dim X/(\la+\mu)<+\infty,\\
\label{e:la-mu2}&\dim(\la\cap\mu)\le\dim(\la^{\w\w}\cap\mu^{\w\w})\le\dim
X/(\la^{\w}+\mu^{\w})<+\infty.
\end{align}
Then $(\la,\mu)$ and $(\la^{\w},\mu^{\w})$ are Fredholm pairs, and we have
\begin{align*}\Index(\la,\mu)+\Index(\la^{\w},\mu^{\w})\ &=\ \dim(\la\cap\mu)-\dim X/(\la+\mu)\\
&\qquad+\dim(\la^{\w}\cap\mu^{\w})-\dim X/(\la^{\w}+\mu^{\w})\\
&=\ \dim(\la\cap\mu)-\dim X/(\la^{\w}+\mu^{\w})\\
&\qquad+\dim(\la^{\w}\cap\mu^{\w})-\dim X/(\la+\mu)\le 0.
\end{align*}
\newline (b) By the proof of (a), the equality in (\ref{e:negative-sum-index}) holds if and only if $\dim(\la^{\w}\cap\mu^{\w})=\dim X/(\la+\mu)$
and $\dim(\la\cap\mu)=\dim(\la^{\w\w}\cap\mu^{\w\w})=\dim X/(\la^{\w}+\mu^{\w})$. Since
$\la\cap\mu\subset \la^{\w\w}\cap\mu^{\w\w}$, by Lemma \ref{l:negative-index}, the equality in
(\ref{e:negative-sum-index}) holds if and only if $\la+\mu=(\la+\mu)^{\w\w}$,
$\la^{\w}+\mu^{\w}=(\la^{\w}+\mu^{\w})^{\w\w}$, and $\la\cap\mu=\la^{\w\w}\cap\mu^{\w\w}$.
\end{proof}

\section[Intrinsic decomposition of a symplectic vector space]
{Natural decomposition of $X$ induced by a Fredholm pair of Lagrangian subspaces with vanishing
index} \label{ss:intrinsic-decomposition}

The following lemmata are the key to the definition of the Maslov index in symplectic Banach spaces
by symplectic reduction to the finite-dimensional case. For technical reasons, in this section,
Fredholm pairs of Lagrangians are always assumed to be of index $0$.

We begin with some general facts.

\begin{lemma}\label{l:sym-subspace} Let $(X,\w)$ be a symplectic vector space and $X_0, X_1$ two linear subspaces with $X=X_0+X_1$.
Assume that $X_0\subset X_1^{\w}$. Then we have $X_0=X_1^{\w}$, $X_1=X_0^{\w}$, $X=X_0\oplus X_1$,
and $X_0,X_1$ are symplectic.
\end{lemma}

\begin{proof} Since $X_0\subset X_1^{\w}$, we have $X_1\subset X_1^{\w\w}\subset X_0^{\w}$. Since $X=X_0+X_1$, there holds
\[X_1\cap X_1^{\w}\subset X_0^{\w}\cap X_1^{\w}=(X_0+X_1)^{\w}=\{0\}.\]
So $X_1$ is symplectic, and we have $X_1^{\w}=X_1^{\w}\cap(X_0+X_1)=X_0+X_1^{\w}\cap X_1=X_0$ and
$X_1\cap X_0=X_1\cap X_1^{\w}=\{0\}$. Hence we have $X=X_0\oplus X_1$. Since $X_1\subset X_0^{\w}$
and $X=X_0+X_1$, we have $X_1=X_0^{\w}$ and $X_0$ is symplectic.
\end{proof}

\begin{lemma}\label{l:w-quotient} Let $(X,\w)$ be a symplectic vector space and $\la,V$ two linear subspaces. Assume that $\dim V<+\infty$. Then we have
\begin{equation}\label{e:w-quotient}
\dim\la/(\la\cap V^{\w})\le\dim V.
\end{equation}
The equality holds if and only if $\la+V^{\w}=X$. In this case we have $\la^{\w}\cap V=\{0\}$.
\end{lemma}

\begin{proof} By \auindex{Boo{\ss}--Bavnbek,\ B.}\auindex{Zhu,\ C.}\cite[Corollary 1]{BooZhu:2013}, we have $\dim X/V^{\w}=\dim V$. Hence we have
\[\dim\la/(\la\cap V^{\w})=\dim(\la+V^{\w})/V^{\w}\le\dim X/V^{\w}=\dim V.\]
The equality holds if and only if $\la+V^{\w}=X$. In this case we have $\la^{\w}\cap
V=(\la+V^{\w})^{\w}=\{0\}$.
\end{proof}

Now we turn to our key observation.\subindex{Natural decomposition of a symplectic vector space}
\subindex{Symplectic reduction!decomposition induced by a pair of co-isotropic subspaces}

\begin{proposition}\label{p:reduction-prop}  Let $(X,\omega)$ be a symplectic vector space. Let $(\lambda,\mu)$ be a pair of
co-isotropic subspaces with $\dim\la_0=\dim X/(\la+\mu)<+\infty$, where
$\lambda_0=\lambda^{\w}\cap\mu^{\w}$. Let $V$ be a linear subspace of $X$  with
$X=V\oplus(\lambda+\mu)$. Let $\lambda_1=V^{\omega}\cap\lambda$ and $\mu_1=V^{\omega}\cap\mu$. Let
$X_0=\lambda_0+V$ and $X_1=\lambda_1+\mu_1$. Then the following holds.
\newline (a) $V^{\omega}\oplus\lambda_0=X$.
\newline (b) $X_0=\lambda_0\oplus V$, $\lambda=\lambda_0\oplus\lambda_1$ and $\mu=\lambda_0\oplus\mu_1$. $X_1=\lambda_1\oplus\mu_1$ if $\la$ and $\mu$ are Lagrangian subspaces of $X$.
\newline (c) $\la_1=\la\cap X_1$, $\mu_1=\mu\cap X_1$ and $\la+\mu=\la_0+X_1$.
\newline (d) $X_1=X_0^{\omega}=V^{\w}\cap(\la+\mu)$, $X_0=X_1^{\omega}$, $X=X_0\oplus X_1$, and $X_0$ and $X_1$ are
symplectic.
\newline (e) The subspace $\la_0$ is a Lagrangian subspace of $X_0$. $\lambda_1,\mu_1$ are Lagrangian subspaces of $X_1$ if $\la$ and $\mu$ are Lagrangian subspaces of $X$.
\end{proposition}

\begin{proof} (a) Since $X=V\oplus(\lambda+\mu)$, we have $V\cap\la_0=\{0\}$
and $V^{\omega}\cap\lambda_0=\{0\}$. By \auindex{Boo{\ss}--Bavnbek,\ B.}\auindex{Zhu,\
C.}\cite[Corollary 1]{BooZhu:2013}, we have $\dim X/V^{\w}=\dim V=\dim\la_0$. So we have
$X=V^{\w}\oplus\la_0$.
\newline (b) Note that
\begin{align*}
\dim\lambda_0\ &\le\ \dim\lambda/(V^{\omega}\cap\lambda)\ =\ \dim(V^{\omega}+\lambda)/V^{\omega}\\
&\le\ \dim X/V^{\omega}\ \le\ \dim V\ =\ \dim\lambda_0.
\end{align*}
We have $\la_1\cap\mu_1=V^{\omega}\cap\la\cap\mu=V^{\omega}\cap\lambda_0=\{0\}$ if $\la$ and $\mu$
are Lagrangian subspaces of $X$. So (b) holds.
\newline (c) Since $X_1=\la_1+\mu_1\subset V^{\w}$, we have $\la\cap X_1\subset\la_1\subset\la\cap X_1$.
So $\la_1=\la\cap X_1$ holds. Similarly we have $\mu_1=\mu\cap X_1$. By (b) we have
\[\la+\mu=\la_0+\la_1+\la_0+\mu_1=\la_0+\la_1+\mu_1=\la_0+X_1.\]
\newline (d) Since $X=X_0+X_1$, our claim follows from Lemma \ref{l:sym-subspace} and the fact
\[X_0^{\w}=V^{\w}\cap\la_0^{\w}\supset V^{\omega}\cap(\lambda+\mu)\supset X_1.
\]
\newline (e) By definition, $\la_0$ is isotropic. Moreover, $\dim \la_0=\12\dim X_0$. So $\la_0$ is
Lagrangian in  $X_0$.

Now assume that $\la$ and $\mu$ are Lagrangian subspaces of $X$. Note that $\lambda_1$ and $\mu_1$
are isotropic. Since $X_1=\lambda_1\oplus\mu_1$, by \auindex{Boo{\ss}--Bavnbek,\ B.}\auindex{Zhu,\
C.}\cite[Lemma 4]{BooZhu:2013}, $\lambda_1$ and $\mu_1$ are Lagrangian subspaces of $X_1$.
\end{proof}

\begin{corollary}\label{c:complementary-lagrangian}
Let $(X,\omega)$ be a symplectic vector space. Let $(\lambda,\mu)$ be a Fredholm pair of Lagrangian
subspaces of index $0$. Then there exists a Lagrangian subspace $\wt\mu\< X$ such that
$X=\lambda\oplus\wt\mu$ and $\dim\mu/(\mu\cap\wt\mu)=\dim\wt\mu/(\mu\cap\wt\mu)=\dim(\la\cap\mu)$.
\end{corollary}

\begin{proof} By Proposition \ref{p:reduction-prop}, $X_0$ is symplectic and $\lambda_0$ is a Lagrangian subspace of $X_0$.
Choose a Lagrangian $\wt V$ of $X_0$ with $X_0=\la_0\oplus\wt V$. Then set $\wt\mu:=\wt V\oplus
\mu_1$.
\end{proof}

\begin{lemma}\label{l:embedding-finite} Let $(X,\w)$ be a symplectic vector space and $\la$ an isotropic subspace of $X$.
Assume that $\dim\la=n<+\infty$. Then there exists a $2n$ dimensional symplectic subspace $X_0$
such that $\la$ is a Lagrangian subspace of $X_0$, $X_0=X_0^{\w\w}$ and $X=X_0\oplus X_0^{\w}$.
\end{lemma}

\begin{proof} Since $\dim\la=n<+\infty$, by \auindex{Boo{\ss}--Bavnbek,\ B.}\auindex{Zhu,\ C.}\cite[Corollary 1]{BooZhu:2013}
we have $\la^{\w\w}=\la$ and $\dim X/\la^{\w}=n$. Take an $n$ dimensional linear subspace $V$ of
$X$ such that $X=V\oplus\la^{\w}$. Since $\la\subset\la^{\w}$, we have
\[\la^{\w}\cap(\la+V)=\la+\la^{\w}\cap V=\la.\]
Since $\dim V=n<+\infty$, by \auindex{Boo{\ss}--Bavnbek,\ B.}\auindex{Zhu,\ C.}\cite[Corollary
1]{BooZhu:2013} we have $V^{\w\w}=V$ and $\dim X/V^{\w}=n$. Set $X_0:=\la+V$. Then we have
\[X_0\cap X_0^{\w}=(\la+V)\cap\la^{\w}\cap V^{\w}=\la\cap V^{\w}=(\la^{\w}+V)^{\w}=\{0\}.\]
By \auindex{Boo{\ss}--Bavnbek,\ B.}\auindex{Zhu,\ C.}\cite[Corollary 1]{BooZhu:2013}, $\dim
X/X_0^{\w}=\dim X_0=2n$ and $X_0^{\w\w}=X_0$. So we have $X=X_0\oplus X_0^{\w}$. Since $\dim\la=n$
and $\la$ is isotropic, $\la$ is a Lagrangian subspace of $X_0$.
\end{proof}

\begin{lemma}\label{l:parametrize-isotropic} Let $\e$ be a positive number.
Let $(X,\w(s))$, $s\in(-\e,\e)$ be a family of symplectic Banach space with \subindex{Continuously
varying!symplectic forms}continuously varying $\w(s)$. Let $X_0(s)$, $s\in(-\e,\e)$ be a
\subindex{Continuously varying!linear subspaces}continuous family of linear subspaces of dimension
$2n<+\infty$ such that $(X_0(0),\w(0)|_{X_0(0)})$ is symplectic. Let $\la(0)$ be a Lagrangian
subspace of $(X_0(0),\w(0)|_{X_0(0)})$. Then there exist a $\delta\in(0,\e)$ and a continuous
family of linear subspaces $\la(s)$, $s\in(-\delta,\delta)$ such that $(X_0(s),\w(s)|_{X_0(s)})$ is
symplectic and $\la(s)$ is a Lagrangian subspace of $(X_0(s),\w(s)|_{X_0(s)})$ for each
$s\in(-\delta,\delta)$.
\end{lemma}

\begin{proof} Since $\dim X_0(s)=2n<+\infty$, by the proof of \auindex{Kato,\ T.}\cite[Lemma III.1.40]{Ka95}, there exists a closed subspace
$X_1$ such that $X=X_0(0)\oplus X_1$. By Proposition \ref{p:closed-spaces-dimensions}, there exists a $\delta_1\in(0,\e)$ such that
$X=X_0(s)\oplus X_1$ for each $s\in(-\delta_1,\delta_1)$. By \auindex{Kato,\ T.}\cite[Lemma I.4.10]{Ka95}, there exist a $\delta_2\in(0,\delta_1)$ and
a continuous family $U(s)\in\Bb(X)$, $s\in[0,\delta_2)$ of such that $U(s)^{-1}\in\Bb(X)$ and $U(s)X_0(0)=X_0(s)$.

Since $X_0(0)$ is symplectic, there exists a $\delta_3\in(0,\delta_2)$ such that the form
\[U(s)^*(\w(s)|_{X_0(s)})(x,y):=\w(s)(U(s)x,U(s)y),\quad x,y\in X_0(0)\]
is a symplectic form on $X_0(0)$ for each $s\in[0,\delta_3)$. Then the form $\w(s)|_{X_0(s)}$ is symplectic.

We give $X_0(0)$ an inner product $\lla\cdot,\cdot\rra$. Let $J_0(s)\in\GL(X_0(0))$ be the
operators that define the symplectic structures $U(s)^*(\w(s)|_{X_0(s)})$. Then there exists a continuous family $T(s)\in\GL(X_0(0))$,
$s\in(-\delta,\delta)$ with $\delta\in(0,\delta_3)$ such that $T(s)^*J_0(s)T(s)=J_0(0)$, where
we set $\la(s):=U(s)T(s)\la(0)$ and our result follows.
\end{proof}

\section{Symplectic reduction of Fredholm pairs}\label{ss:symplectic-red1}
We recall the general definition of symplectic reduction.

\begin{definition}\label{d:sympl-red}
Let $(X,\omega)$ be a symplectic vector space and $W$ a co-isotropic subspace.
\newline (a) The space $W/W^{\omega}$ is a symplectic vector space with induced symplectic structure
\begin{equation}\label{e:red-structure}
\wt{\omega}(x+W^{\omega},y+W^{\omega})\ :=\ \omega(x,y) \text{ for all $x,y\in W$}.
\end{equation}
We call
\symindex{W@$W/W^{\omega}$ reduced symplectic space} \symindex{\omega tilde@$\tilde\omega$ induced
form on symplectic reduction}
$(W/W^{\omega}, \wt{\omega})$ the {\em
symplectic reduction} of $X$ via $W$.
\newline (b) Let $\lambda$ be a linear subspace of $X$.
\subindex{Symplectic reduction}The {\em symplectic reduction} of $\lambda$ via $W$ is defined by%
\begin{equation}\label{e:red-subspace}\symindex{R@$R_W(\la)$ symplectic reduction of isotropic $\la$
along co-isotropic $W$} \subindex{Symplectic reduction!reduction map}
R_W(\lambda)=R_W^{\w}(\la)\ :=\ \bigl((\lambda+W^{\omega})\cap W\bigr)/W^{\omega}\ =\ \bigl(\lambda\cap  W+W^{\omega}\bigr)/W^{\omega}.%
\end{equation}
\end{definition}

Clearly, $R_W(\la)$ is isotropic if $\la$ is isotropic. If $W^{\w}\subset\la\subset W$ and $\la$ is
Lagrangian, $R_W(\la)$ is Lagrangian. We have the following lemma.

\begin{lemma}\label{l:red-lagrangian} Let $(X,\w)$ be a symplectic vector space with isotropic subspace $W_0$.
Let $\la\supset W_0$ be a linear subspace. Then $\la$ is a Lagrangian subspace of $X$ if and only
if $W_0^{\w\w}\subset\la\subset W_0^{\w}$ and $R_{W_0^{\w}}(\la)$ is a Lagrangian subspace of
$W_0^{\w}/W_0^{\w\w}$.
\end{lemma}

\begin{proof} By (\ref{e:three-Omega}) we have $W_0^{\w\w\w}=W_0^{\w}$. Since $W_0\subset W_0^{\w}$, $W_0^{\w\w}\subset W_0^{\w}$.

If $\la\in\Ll(X)$ and $\la\supset W_0$, we have $\la\subset W_0^{\w}$ and $W_0^{\w\w}\subset\la$.
Then we get $R_{W_0^{\w}}(\la)=\la/W_0^{\w\w}$ and $(\la/W_0^{\w\w})^{\wt\w}=(\la^{\w}\cap
W_0^{\w})/W_0^{\w\w}=\la/W_0^{\w\w}$, i.e., $R_{W_0^{\w}}(\la)\in\Ll(W_0^{\w}/W_0^{\w\w})$.

Assume that $W_0^{\w\w}\subset\la\subset W_0^{\w}$, we have $W_0^{\w\w}\subset\la^\w\subset
W_0^{\w}$. If $R_{W_0^{\w}}(\la)\in\Ll(W_0^{\w}/W_0^{\w\w})$, we have
\[\la/W_0^{\w\w}=(\la/W_0^{\w\w})^{\wt\w}=(\la^{\w}\cap W_0^{\w})/W_0^{\w\w}=\la^{\w}/W_0^{\w\w}.\]
So we get $\la=\la^\w$, i.e., $\la\in\Ll(X)$.
\end{proof}

\begin{lemma}[Transitivity of symplectic reduction]\label{l:red-transitive}
\subindex{Symplectic reduction!transitivity}
Let $(X,\w)$ be a symplectic vector space with two
co-isotropic subspaces $W_1\subset W_2$, hence clearly $W_1/W_2^{\w}\< W_2/W_2^{\w}$ with
$(W_1/W_2^{\w})^{\w_2} = W_1^{\w}/W_2^{\w}$, where $\w_2$ denotes the symplectic form on
$W_2/W_2^{\w}$ induced by $\w$. Then the following holds.
\newline (a) Denote by $K_{W_1,W_2}\colon W_1/W_2^{\w}\too W_1/W_1^{\w}$ the map induced by $I_{W_1}$, where $I_W$ denotes the identity map on a
space $W$. Then $K_{W_1,W_2}$ induces a symplectic isomorphism
\begin{equation}\label{e:induced-K}
\wt K_{W_1,W_2}\colon (W_1/W_2^{\w})/(W_1^{\w}/W_2^{\w})\too W_1/W_1^{\w},
\end{equation}
such that the following diagram becomes commutative:
\begin{equation}\label{dia:induced-sympl-iso}
\xymatrixcolsep{6.5pc}%
\xymatrix{%
 W_1  \ar[r]^{[\cdot\, +W_2^{\w}]} \ar[d]^{I_{W_1}}%
 &W_1/W_2^{\w} \ar[r]^{\!\![\cdot\, +W_1^{\w}/W_2^{\w}]} \ar[d]_{K_{W_1,W_2}}%
 &(W_1/W_2^{\w})/(W_1^{\w}/W_2^{\w}) \ar[ld]_{\cong}^{\wt{K}_{W_1,W_2}}%
 \\%
W_1 \ar[r]_{[\cdot\, +W_1^{\w}]}  & W_1/W_1^{\w}%
& %
}
\end{equation}
\newline (b) For a linear subspace of $\la$ of $X$, we have
\begin{equation}\label{e:trans-red}
R_{W_1/W_2^{\w}}(R_{W_2}(\la))=\wt K_{W_1,W_2}^{-1}(R_{W_1}(\la)).
\end{equation}
Differently put, the following diagram is commutative:
\begin{equation}\label{dia:induced-sympl-iso-K}
\xymatrixcolsep{7pc}%
\xymatrix{%
 \Lin(X)  \ar[r]^{R_{W_1}} \ar[d]_{R_{W_2}}%
 &\Lin(W_1/W_1^{\w}) \ar[d]^{(\wt{K}_{W_1,W_2})\ii}%
 \\%
\Lin(W_2/W_2^{\w}) \ar[r]_{R_{W_1/W_2^{\w}}}  & \Lin\bigl((W_1/W_2^{\w})/(W_1^{\w}/W_2^{\w})\bigr)%
}
\end{equation}\symindex{RW1@$R_{W_1/W_2^{\w}}$ inner symplectic reduction}\subindex{Symplectic
reduction!inner} Here \symindex{Lin@$\Lin(X)$ set of linear subspaces of the vector space
$X$}$\Lin(X)$ denotes the set of linear subspaces of the vector space $X$.
\end{lemma}

\begin{proof} (a) Since $W_1\subset W_2$ and they are co-isotropic, we have $W_2^{\w}\subset W_1^{\w}\subset W_1\subset W_2$. So $K_{W_1,W_2}$
is well-defined. Since $\ker K_{W_1,W_2}=W_1^{\w}/W_2^{\w}$, $\wt K_{W_1,W_2}$ is a linear
isomorphism. By Definition \ref{d:sympl-red}, $\wt K_{W_1,W_2}$ is a symplectic isomorphism.

(b) Note that
\begin{align*}
R_{W_2}(\la)\cap (W_1&/W_2^{\w})+W_1^{\w}/W_2^{\w}=\bigl((\la\cap W_2+W_2^{\w})\cap W_1+W_1^{\w}\bigr)/W_2^{\w}\\
&=(\la\cap W_1+W_2^{\w}+W_1^{\w})/W_2^{\w}=(\la\cap W_1+W_1^{\w})/W_2^{\w}.
\end{align*}
So (\ref{e:trans-red}) holds.
\end{proof}

\begin{corollary}\label{c:inner-red} Let $(X,\w)$ be a symplectic vector space with a co-isotropic subspace $W$, a Lagrangian subspace $\mu$
and two linear spaces $V, \la$. Assume that $\dim W^{\w}\cap\mu=\dim X/(W+\mu)=\dim V<+\infty$,
$X=V\oplus(W+\mu)$ and $W^{\w}\cap\mu\subset\la\subset W+\mu$. Set $X_0:=W^{\w}\cap\mu+V$ and
$X_1:=V^{\w}\cap W+V^{\w}\cap\mu$. Denote by $P_1\colon X\to X_1$ defined by $X=X_0\oplus X_1$ (see
Proposition \ref{p:reduction-prop}). Then the following holds.
\newline (a) $W\cap X_1=V^{\w}\cap W$, $W^{\w}\cap X_1=W^{\w}\cap V^{\w}$, $\mu\cap X_1=V^{\w}\cap\mu$, $\la=W^{\w}\cap\mu+\la\cap X_1$, and $(W\cap X_1)^{\w}=W^{\w}+V=X_0+W^{\w}\cap X_1$.
\newline (b) $P_1$ induces a symplectic isomorphism%
\[
\wt P_1\colon (W+\mu)/(W^{\w}\cap\mu)\too X_1\ \tand\ \wt P_1(R_{W+\mu}(\la))\ =\ \la\cap X_1.%
\]
\newline (c) Denote by $R^{X_1}_{V^{\w}\cap W}(\la\cap X_1)$ the symplectic reduction of $\la\cap X_1$ in $X_1$ via $V^{\w}\cap W$.
Define $\wt L_{W,W+\mu}\colon (W\cap X_1)/(W^{\w}\cap X_1)\to W/W^{\w}$ by $\wt
L_{W,W+\mu}(x+W^{\w}\cap X_1)=x+W^{\w}$ for all $x\in W\cap X_1$. Then  the following diagram is
commutative
\begin{equation}\label{dia:induced-sympl-iso-L}
\xymatrixcolsep{7pc}%
\xymatrix{%
 \Lin_{W,\mu}(X)  \ar[r]^{R_{W}} \ar[d]_{\cap X_1}%
 &\Lin(W/W^{\w}) \ar[d]_{\cong}^{(\wt{L}_{W,W+\mu})\ii}%
 \\%
\Lin(X_1) \ar[r]_{R^{X_1}_{V^{\w}\cap W}}  & \Lin\bigl((W\cap X_1)/(W^{\w}\cap X_1)\bigr)%
}
\end{equation}
and, in particular, we have
\begin{equation}\label{e:red-p1}
R^{X_1}_{V^{\w}\cap W}(\la\cap X_1)=\wt L^{-1}_{W,W+\mu}(R_W(\la)).
\end{equation}
Here $\Lin_{W,\mu}(X):=\{\la\in\Lin(X); W^{\w}\cap\mu\< \la\< W+\mu\}$.
\newline (d) $W$ is complemented (see Remark \ref{r:complemented}) in $X$ if and only if $W\cap X_1$
is complemented in $X_1$. In the case of a Banach space we require all the appeared subspaces to be
closed.
\newline (e) $W^{\w}$ is complemented in $W$ if and only if $W^{\w}\cap X_1$ is complemented in $W\cap X_1$.
In the case of a Banach space we require all the appeared subspaces to be closed.
\end{corollary}

\begin{proof} (a) By Proposition \ref{p:reduction-prop}, we have $W\cap X_1=V^{\w}\cap W$, $\mu\cap X_1=V^{\w}\cap\mu$, $W+\mu=W^{\w}\cap\mu+X_1$,
and $X=V^{\w}+W=V^{\w}+\mu$. Since $W+V=W+X_0$ and $X_1=X_0^{\w}$, we have
\[W^{\w}\cap V^{\w}=(W+V)^{\w}=(W+X_0)^{\w}=W^{\w}\cap X_1.\]

Since $W^{\w}\cap\mu\subset\la\subset W+\mu$, we have
\[\la=\la\cap(W+\mu)=\la\cap(W^{\w}\cap\mu+X_1)=W^{\w}\cap\mu+\la\cap X_1.\]

Note that $W=W^{\w}\cap\mu\oplus W\cap X_1$. By Lemma \ref{l:negative}.b we have $\dim(W\cap
X_1)^{\w}/W^{\w}\le\dim(W^{\w}\cap\mu)=\dim V$. Moreover, we have
\[(W\cap X_1)^{\w}\supset W^{\w}+X_1^{\w}=W^{\w}+X_0=W^{\w}+V.\]
Since $W^{\w}\cap V=(W+V^{\w})^{\w}=\{0\}$, we have
\[(W\cap X_1)^{\w}=W^{\w}+V=X_0+W^{\w}\cap X_1.\]

\noi (b) Since $W^{\w}\cap\mu$ is isotropic and $W+\mu=W^{\w}\cap\mu\oplus X_1$, $P_1$ induces a
symplectic isomorphism $\wt P_1\colon (W+\mu)/(W^{\w}\cap\mu)\to X_1$. Since
$W^{\w}\cap\mu\subset\la\subset W+\mu$, we have $R_{W+\mu}(\la)=\la/(W^{\w}\cap\mu)$. So it holds
that $\wt P_1(R_{W+\mu}(\la))=\la\cap X_1$.

\noi (c) Let $\wt K_{W,W+\mu}$ denote the symplectic isomorphism defined by (\ref{e:induced-K}).
Note that $\wt K_{W,W+\mu}=\wt L_{W,W+\mu}$ under the symplectic isomorphism $\wt P_1$. So
(\ref{e:red-p1}) follows from (b) and Lemma \ref{l:red-transitive}.

\noi (d) If $W\cap X_1$ is complemented in $X_1$, there exists a linear subspace $M_1$ such that
$X_1=W\cap X_1\oplus M_1$. Since $\dim X_0<+\infty$, there exists a linear subspace $M_0$ such that
$X_0=W^{\w}\cap\mu\oplus M_0$. Take $M=M_0\oplus M_1$ and we have $X=W\oplus M$.

Conversely, if $W$ is complemented in $X$, there exists a linear subspace of $M$ such that
$X=W\oplus M$. By (a), we have $W=W\cap X_1\oplus W^{\w}\cap\mu$. So we have
\[X_1=X_1\cap(W\cap X_1+W^{\w}\cap\mu+M)=W\cap X_1\oplus X_1\cap(W^{\w}\cap\mu+M).\]

\noi (e) If $W^{\w}\cap X_1$ is complemented in $W\cap X_1$, there exists a linear subspace $N_1$
such that $W\cap X_1=W^{\w}\cap X_1\oplus N_1$. Then we have $W=W^{\w}\cap\mu\oplus W^{\w}\cap
X_1\oplus N_1=W^{\w}\oplus N_1$.

Conversely, if $W^{\w}$ is complemented in $W$, there exists a linear subspace of $N$ such that
$W=W^{\w}\oplus N$. By (a), we have
\begin{align*}
W\cap X_1&=(W\cap X_1\oplus W^{\w}\cap\mu)\cap X_1\\
&=(W^{\w}\cap X_1\oplus N\oplus W^{\w}\cap\mu)\cap X_1\\
&=(W^{\w}\cap X_1)\oplus (N\oplus W^{\w}\cap\mu)\cap X_1.\qedhere
\end{align*}
\exendproof

\begin{rem}\label{r:complemented}  A linear subspace $M$ of a vector space $X$ is called \subindex{Complemented subspace}
\symindex{Sc@$\Ss^c(X)$ set of closed complemented subspaces in Banach space $X$}{\em complemented}
in $X$ if there exists another linear subspace $N$ of $X$ such that $X=M\oplus N$. In Banach space
we require $M,N$ to be closed and write $M\in\Ss^c(X)$. Note that any linear subspace in a vector
space is complemented by \subindex{Zorn's Lemma}Zorn's lemma. Our Corollary \ref{c:inner-red} (d),
(e) is not trivial if either $X$ is a Banach space or one does not want to use Zorn's lemma.
\end{rem}

To ensure that \subindex{Symplectic reduction!invariance of class of pairs of Fredholm Lagrangian
subspaces of index 0}symplectic reduction does not lead us out of our class of pairs of Fredholm
Lagrangian subspaces of index 0, we prove Proposition \ref{p:red-fredholm} further below.

\begin{lemma}\label{l:red-index}
Let $X$ be a vector space and $W_1\subset W_2$, $\lambda,\mu$ four linear subspaces of $X$. For
each linear subspace $V$, set $R(V):=\bigl(V\cap W_2+W_1\bigr)/W_1$. Assume that
$W_1\subset\lambda\subset W_2$. Then $(\lambda,\mu)$ is a Fredholm pair of subspaces of $X$ if and
only if $(R(\lambda),R(\mu))$ is a
Fredholm pair of subspaces of $W_2/W_1$, $\dim(\mu\cap W_1)<+\infty$ and $\dim X/(W_2+\mu)<+\infty$. In this case it holds that%
\begin{align*}
\dim (R(\lambda)\cap R(\mu))\ =&\ \dim(\lambda\cap\mu)-\dim(\mu\cap W_1),\\
\dim(W_2/W_1)/(R(\lambda)+R(\mu))\ =&\ \dim X/(\lambda+\mu)-\dim X/(W_2+\mu),\\
\Index(R(\lambda),R(\mu))\ =&\ \Index(\lambda,\mu)\\
&-\dim(\mu\cap W_1)+\dim X/(W_2+\mu).
\end{align*}
\end{lemma}

\begin{proof} Since $W_1\subset\lambda\subset W_2$, we have
\begin{align*}
R(\lambda)\cap R(\mu)&=(\lambda/W_1)\cap\bigl(((\mu+W_1)\cap W_2)/W_1\bigr)=(\lambda\cap\mu+W_1)/W_1\\
&\cong(\lambda\cap\mu)/(\lambda\cap\mu\cap
W_1),%
\end{align*}
and
\begin{align*}
(W_2/W_1)&/(R(\lambda)+R(\mu))\cong\ W_2/(\lambda+\mu\cap W_2)= W_2/((\lambda+\mu)\cap
W_2)\\
&=(W_2+\lambda+\mu)/(\lambda+\mu)=(W_2+\mu)/(\lambda+\mu)\\
&\cong(X/(\lambda+\mu))/(X/(W_2+\mu)).%
\end{align*}
So our lemma follows.
\end{proof}

Now we can prove the basic calculation rule of symplectic reduction\subindex{Symplectic
reduction!basic algebraic calculation rule}:

\begin{proposition}[Symplectic quotient rule]\label{p:sum-index-0}
\subindex{Symplectic reduction!quotient rule}\subindex{Symplectic quotient rule}
Let $(X,\w)$ be a
symplectic vector space and $\la$, $\mu$, $W$ subspaces. Assume that $\la\subset W$, $\mu=\mu^{\w}$
and
\begin{equation}\label{e:sum-index-0}\Index(\la,\mu)+\Index(\la^{\w},\mu)=0.\end{equation}
Then we have $\dim(W^{\w}\cap\mu)=\dim X/(W+\mu)<+\infty$ and $W+\mu=W^{\w\w}+\mu$.
\end{proposition}

\begin{proof} Since $\la\subset W$, we have $W^{\w}\subset\la^{\w}$. Since $\mu=\mu^{\w}$, we have $(W+\mu)^{\w}=W^{\w}\cap\mu\subset W+\mu$. Denote by $\wt\w$ the symplectic structure on $(W+\mu)/(W^{\w}\cap\mu)$. Then we have
\begin{align*}
&\la\cap(W+\mu)=\la,\qquad \la^{\w}+W^{\w}\cap\mu=\la^{\w},\\
&(\la+W^{\w}\cap\mu)^{\w}\cap(W+\mu)=\la^{\w}\cap(W+\mu),\\
&R_{W+\mu}(\la)=(\la+W^{\w}\cap\mu)/(W^{\w}\cap\mu),\\
&R_{W+\mu}(\la^{\w})=(\la^{\w}\cap(W+\mu))/(W^{\w}\cap\mu)=(R_{W+\mu}(\la))^{\wt\w},\\
&R_{W+\mu}(\mu)=\mu/(W^{\w}\cap\mu)=(R_{W+\mu}(\mu))^{\wt\w}.
\end{align*}
By Lemma \ref{l:red-index} and (\ref{e:sum-index-0}) we have
\begin{align*}
\Index(\la,\mu)&=\Index(R_{W+\mu}(\la),R_{W+\mu}(\mu))\\
&\quad+\dim(\la\cap W^{\w}\cap\mu)-\dim X/(W+\mu),\\
\Index(\la^{\w},\mu)&=\Index(R_{W+\mu}(\la^{\w}),R_{W+\mu}(\mu))\\
&\quad+\dim(W^{\w}\cap\mu)-\dim X/(\la^{\w}+W+\mu).
\end{align*}
Note that $(\la^{\w}+W+\mu)^{\w}=\la^{\w\w}\cap W^{\w}\cap\mu\supset\la\cap W^{\w}\cap\mu$. By
Lemma \ref{l:negative-index} and Corollary \ref{c:negative-sum-index} we have
\begin{align*}
&\Index(R_{W+\mu}(\la),R_{W+\mu}(\mu))+\Index(R_{W+\mu}(\la^{\w}),R_{W+\mu}(\mu))\le 0,\\
&\dim(\la\cap W^{\w}\cap\mu)\le \dim X/(\la^{\w}+W+\mu),\\
&\dim(W^{\w}\cap\mu)\le\dim X/(W+\mu).
\end{align*}
By (\ref{e:sum-index-0}), the above three inequalities take equalities.

By (\ref{e:three-Omega}), we have $W^{\w\w\w}=W^{\w}$. Apply the above result to $W^{\w\w}$, we
have $\dim(W^{\w}\cap\mu)=\dim X/(W^{\w\w}+\mu)$. Since $W\subset W^{\w\w}$, we have
$W+\mu=W^{\w\w}+\mu$.
\end{proof}

The following proposition is inspired by \auindex{Boo{\ss}--Bavnbek,\ B.}\auindex{Furutani,\
K.}\cite[Proposition 3.5]{BoFu98}. It gives a natural sufficient condition for preserving the
Lagrangian property under symplectic reduction.

\begin{proposition}\label{p:red-fredholm}
\subindex{Symplectic reduction!invariance of class of pairs of Fredholm Lagrangian subspaces of
index 0}
Let $(X,\omega)$ be a symplectic vector space with a co-isotropic subspace $W$. Let
$(\la,\mu)$ be a Fredholm pair of Lagrangian subspaces of $X$ with index $0$. Assume that
$W^{\omega}\subset\la\subset W$. Then we have $\dim(W^{\w}\cap\mu)=\dim X/(W+\mu)<+\infty$,
$W+\mu=W^{\w\w}+\mu$, and $(R_W(\la),R_W(\mu))$ is a Fredholm pair of Lagrangian subspaces of
$W/W^{\omega}$ with index $0$.
\end{proposition}

\begin{proof} By Proposition \ref{p:sum-index-0} we have $\dim(W^{\w}\cap\mu)=\dim X/(W+\mu)<+\infty$ and $W+\mu=W^{\w\w}+\mu$.

By Lemma \ref{l:red-index}, $(R_W(\la),R_W(\mu))$ is a Fredholm pair of subspaces of
$W/W^{\omega}$, $\dim(W^{\w}\cap\mu)<+\infty$, and $\dim X/(W+\mu)<+\infty$. Since $\la$ and $\mu$
are Lagrangian subspaces of $X$, $R_W(\la)$ and $R_W(\mu)$ are isotropic subspaces of
$W/W^{\omega}$. By Lemma \ref{l:negative-index}, we have $\dim(W^{\w}\cap\mu)\le\dim X/(W+\mu)$ and
$\Index(R_W(\la),R_W(\mu))\le 0$. By Lemma \ref{l:red-index}, we have $\dim(W^{\w}\cap\mu)=\dim
X/(W+\mu)$ and $\Index(R_W(\la),R_W(\mu))=0$. By \auindex{Boo{\ss}--Bavnbek,\ B.}\auindex{Zhu,\
C.}\cite[Proposition 1]{BooZhu:2013}, $R_W(\la)$ and $R_W(\mu)$ are Lagrangian subspaces of
$W/W^{\omega}$.
\end{proof}

\begin{lemma}\label{l:sum-whole-space}\footnote{Added in proof: The lemma was discovered by Li Wu and the second author \auindex{Wu,\ L.}\auindex{Zhu,\
C.}\cite{WuZh}} Let $(X,\omega)$ be a symplectic vector space with a finite-dimensional linear subspace $V$.
Let $\la$ be a Lagrangian subspace of $X$. Assume that
$V\cap\la=\{0\}$. Then we have $V^{\w}+\la=X$.
\end{lemma}

\begin{proof} Set $W:=V+\la$. Then $W^{\w}=V^{\w}\cap\la$. By Lemma \ref{l:w-quotient}, we have $\dim\la/W^{\w}\le\dim V$.
Since $V\cap\la=\{0\}$, we have $\dim W/W^{\w}=\dim V+\dim\la/W^{\w}$.

Since $W^{\w}\subset\la\subset W$, $R_W(\la)=\la/W^{\w}$ is a Lagrangian subspace of $W/W^{\w}$.
Then we have
\[\dim\la/W^{\w}=\frac{1}{2}\dim W/W^{\w}=\dim V.\]
By Lemma \ref{l:w-quotient} we have $V^{\w}+\la=X$.
\end{proof}

The following proposition gives us a \subindex{Symplectic reduction!new understanding}new
understanding of the symplectic reduction.\subindex{Symplectic reduction!invariance of two natural
choices}

\begin{proposition}\label{p:cal-red} Let $(X,\w)$ be a symplectic vector space and $\la_0,V$ linear subspaces.
Let $\la$ and $\mu$ be Lagrangian subspaces. Set $\la_1:=V^{\w}\cap\la$, $\mu_1:=V^{\w}\cap\mu$,
$X_0:=\la_0+V$ and $X_1:=\la_1+\mu_1$. Assume that
\begin{equation}\label{e:direct-sum-assumption}
X=\la_0\oplus V\oplus\la_1\oplus\mu_1=\la\oplus(V+\mu_1)=\mu\oplus(V+\la_1).
\end{equation}
Denote by \symindex{P0@$P_0,P_0(s),P_0(t,s)$ decomposition projections}$P_0\colon X\to X_0$ the
projection defined by $X=X_0\oplus X_1$. Then the following holds.
\newline (a) There exist $A_1\in\Hom(\la_0,V),A_2\in\Hom(\la_0,\mu_1),B_1\in\Hom(\la_0,V)$ and $B_2\in\Hom(\la_0,\la_1)$ such that
\begin{align}
\label{e:expression-la}\la&=\{x_0+A_1x_0+x_1+A_2x_0;x_0\in\la_0,x_1\in\la_1\},\\
\label{e:expression-mu}\mu&=\{y_0+B_1y_0+B_2y_0+y_1;y_0\in\la_0,y_1\in\mu_1\},
\end{align}
where \symindex{Hom@$\Hom(X,Y)$ space of linear maps from $X$ to $Y$}$\Hom(X,Y)$ denotes the linear
maps from $X$ to $Y$.
\newline (b) The linear maps $P_0|_{(V+\la)}$ and $P_0|_{(V+\mu)}$ induce linear isomorphisms
\symindex{Tl@$T_l,T_r$ reduction isomorphisms}$T_l\colon (V+\la)/\la_1\to X_0$ and $T_r\colon
(V+\mu)/\mu_1\to X_0$ respectively, and
\begin{equation}
\label{e:dim-intersection}\dim(\la\cap\mu)=\dim (P_0(\la)\cap P_0(\mu)).
\end{equation}
\newline (c) We have
\begin{align}
\label{e:red-la}T_l(R_{V+\la}(\la))&=T_r(R_{V+\mu}(\la))=P_0(\la),\\
\label{e:red-mu}T_l(R_{V+\la}(\mu))&=T_r(R_{V+\mu}(\mu))=P_0(\mu).
\end{align}
\newline (d) Denote by \symindex{\omega l@$\w_l, \w_r$ decomposition symplectic structures}$\w_l$ the symplectic structure of $X_0$ induced by $T_l$ from $(V+\la)/\la_1$ and $\w_r$ the symplectic structure of $X_0$ induced by $T_r$ from $(V+\mu)/\mu_1$. Then we have
\begin{align*}
\w_l(x_0+v,x_0^{\prime}+v^{\prime})&=\w(x_0+v,x_0^{\prime}+v^{\prime})-\w(x_0+A_1x_0,x_0^{\prime}+A_1x_0^{\prime})\\
=\w_r(x_0+v,x_0^{\prime}+v^{\prime})&=\w(x_0+v,x_0^{\prime}+v^{\prime})-\w(x_0+B_1x_0,x_0^{\prime}+B_1x_0^{\prime})
\end{align*}
for all $x_0,x_0^{\prime}\in\la_0$ and $v,v^{\prime}\in V$, where $A_1$ and $B_1$ are chosen like in (a). If either $\la_0\subset\la_1^{\w}$ or
$\la_0\subset\mu_1^{\w}$, we have $\w_l=\w_r=\w|_{X_0}$.
\newline (e) Assume that $V$ is isotropic. Then the sesquilinear form $Q(x_0,x_0^{\prime}):=\omega(x_0,(A_1-B_1)x_0^{\prime})$ on $\la_0$
is an Hermitian form. We call the form \symindex{Q@$Q,Q(s),Q(\la,t)$ intersection
form}\subindex{Intersection form}$Q$ the {\em intersection from of $(\la,\mu)$ on $\la_0$ at $V$}.
If $\la_0=\mu$ and $V$ is a Lagrangian subspace $W$ of $X$, we set $Q(\mu,W;\la):=Q$ (see
\auindex{Duistermaat,\ J.J.}\cite[(2.4)]{Du76}).
\newline (f) For all $x_0\in\la\cap\mu$ and $x_0^{\prime}\in\la\cap\la_0$, we have
$\w(x_0,A_1x_0^{\prime})=0$.
\newline (g) We have $X=V+\la+\mu=V^{\w}+\la=V^{\w}+\mu$.
\end{proposition}

\begin{proof} (a) Note that $\la_1=V^{\w}\cap\la\subset\la$ and $\mu_1=V^{\w}\cap\mu\subset\mu$. Our claim follows from the assumptions.
\newline (b) By (a) we have
\[V+\la=\{x_0+v+x_1+A_2x_0;x_0\in\la_0,v\in V,x_1\in\la_1\}.\]
So $P_0|_{(V+\la)}$ induces a linear map $T_l\colon (V+\la)/\la_l\to X_0$. Clearly, $T_l$ is a linear
isomorphism. Similarly we get that the map $P_0|_{(V+\mu)}$ induces a linear isomorphism $T_r\colon
(V+\mu)/\mu_1\to X_0$. The equation (\ref{e:dim-intersection}) follows from Lemma
\ref{l:red-index}.
\newline (c) By (a) and (b) we have $T_l(R_{V+\la}(\la))=P_0(\la)$. Note that
\[\mu\cap(V+\la)=\{x_0+B_1x_0+B_2x_0+A_2x_0;x_0\in\la_0\}.\]
By (a) and (b) we have $T_l(R_{V+\la}(\mu))=P_0(\mu)$. Similarly we get the result for $T_r$.
\newline (d) Since $\la_1=(V+\la)^{\w}$ and $\mu_1=(V+\mu)^{\w}$, $(V+\la)/\la_1$ and $(V+\mu)/\mu_1$ are symplectic vector spaces. Let $x_0,x_0^{\prime}\in\la_0$ and $v,v^{\prime}\in V$ be vectors in $X$. By (a) and (b), we have
\begin{align}
\nonumber\w_l(x_0+v,x_0^{\prime}+v^{\prime})=&\w(x_0+v+A_2x_0,x_0^{\prime}+v^{\prime}+A_2x_0^{\prime})\\
\nonumber=&\w(x_0+v,x_0^{\prime}+v^{\prime})+\w(x_0+v,A_2x_0^{\prime})\\
\nonumber &+\w(A_2x_0,x_0^{\prime}+v^{\prime})\\
\label{e:wl2}=&\w(x_0+v,x_0^{\prime}+v^{\prime})+\w(x_0,A_2x_0^{\prime})+\w(A_2x_0,x_0^{\prime}).
\end{align}
So we have $\w_l=\w|_{X_0}$ if $X_0=X_1^{\w}$. Note that $A_1x_0, A_1x_0^{\prime}\in V$. Then we
have
\begin{align*}
0&=\w(x_0+A_1x_0+A_2x_0,x_0^{\prime}+A_1x_0^{\prime}+A_2x_0^{\prime})\\
&=\w(x_0+A_1x_0,x_0^{\prime}+A_1x_0^{\prime})+\w(x_0,A_2x_0^{\prime})+\w(A_2x_0,x_0^{\prime}).
\end{align*}
Thus it holds
\[
\w_l(x_0+v,x_0^{\prime}+v^{\prime})=\w(x_0+v,x_0^{\prime}+v^{\prime})-\w(x_0+A_1x_0,x_0^{\prime}+A_1X_0^{\prime}).
\]
Similarly we get the expression for $\w_r$. Since $P_0(\mu)=T_l(R_{V+\la}(\mu))$ is isotropic in
$(X_0,\w_l)$, we have
\[\w(x_0+B_1x_0,x_0^{\prime}+B_1X_0^{\prime})=\w(x_0+A_1x_0,x_0^{\prime}+A_1X_0^{\prime})\]
for all $x_0,x_0^{\prime}\in\la_0$ and $v,v^{\prime}\in V$. So we have $\w_l=\w_r$.

If $\la_0\subset\mu_1^{\w}$, by (\ref{e:wl2}) we have $\w_l=\w_r=\w|_{X_0}$. Similarly, we have
$\w_l=\w_r=\w|_{X_0}$ if $\la_0\subset\la_1^{\w}$.
\newline (e) By (d).
\newline (f) Since $x_0\in\la\cap\mu$, $A_1x_0^{\prime}+A_2x_0^{\prime}\in\la$ and $A_2x_0^{\prime}\in\mu$, we have
\[0=\w(x_0, A_1x_0^{\prime}+A_2x_0^{\prime})=\w(x_0, A_1x_0^{\prime}).\]
\newline (g) Since $X=\la+V+\mu_1\subset\la+V+\mu\subset X$, we have $X=V+\la+\mu$. Since $\la=la^{\w}$, $\mu=\mu^{\w}$, $V\cap\la=V\cap\mu=\{0\}$ and $\dim V<+\infty$, by Lemma \ref{l:sum-whole-space} we have $X=V^{\w}+\la=V^{\w}+\mu$.
\end{proof}

\addtocontents{toc}{\medskip\noi} \chapter{The Maslov index in strong symplectic Hilbert
space}\label{s:maslov-hilbert}

As explained in the Introduction, the goal of this Memoir is to provide a calculable definition of
the Maslov index in weak symplectic Banach (or Hilbert) spaces. Later in Chapter
\ref{s:maslov-general} we shall achieve that in an intrinsic way, namely by providing a natural
symplectic reduction to the finite-dimensional case, based on the novel decomposition and reduction
techniques introduced in the preceding Chapter \ref{s:banach}. To get through with that plan, we
have to bring the - in principle - well understood definition and calculation of the
\subindex{Maslov index!in finite dimensions} Maslov index in finite dimensions (or, similarly, in
strong symplectic Hilbert space) into a form suitable to receive the symplectic reduction from the
weak infinite-dimensional setting. That is what this chapter is about.

\section{The Maslov index via unitary generators}\label{ss:maslov-unitary}
\subindex{Maslov index!via unitary generators}\subindex{Symplectic structures!symplectic Hilbert
space}

In \auindex{Boo{\ss}--Bavnbek,\ B.} \auindex{Furutani,\ K.}\cite{BoFu98} K. Furutani, jointly with
the first author of this Memoir, explained how the Maslov index of a curve of Fredholm pairs of
Lagrangian subspaces in strong symplectic Hilbert space can be defined and calculated as the
spectral flow of a corresponding curve of unitary operators through a control point on the unit
circle of $\CC$. In this section we give a slight reformulation and simplification, adapted to our
application. Moreover, we show why this approach can \textit{not} be generalized to weak symplectic
Banach spaces nor to weak symplectic Hilbert spaces immediately.

Let \symindex{X@$\mathbb{X}$ Hilbert or Banach bundle}$p\colon \mathbb{X}\to[0,1]$ be a
\subindex{Hilbert bundle}Hilbert bundle with fibers $X(s):=p^{-1}(s)$ for each $s\in[0,1]$. Let
$(X(s),\w(s))$, $s\in[0,1]$ be a family of \subindex{Hilbert space!strong symplectic}strong
symplectic Hilbert spaces with \subindex{Continuously varying!Hilbert inner products}continuously
varying Hilbert inner product $\lla\cdot,\cdot\rra_s$ and \subindex{Continuously varying!symplectic
forms}continuously varying symplectic form $\w(s)$. For a rigorous definition of the terms
\textit{Hilbert bundle} and \textit{continuous variation} we refer to our Appendix
\ref{ss:banach-bundles}. As usual, we assume that we can write $\w(s)(x,y)= \lla J(s)x,y\rra_s$
with invertible $J(s)\colon X(s)\to X(s)$ and $J(s)^*=-J(s)$. The fiber bundle $\mathbb{X}$ is
always trivial. So we can actually assume that $X(s)\equiv X$. By \auindex{Kato,\ T.}\cite[Lemma
I.4.10]{Ka95} and Lemma \ref{l:connectness-local}, the set of closed subspaces is a
\subindex{Hilbert manifold}Hilbert manifold and can be identified locally with bounded
linear maps from a closed subspace to its orthogonal complement.

\begin{note} Let $N\subset M\subset X$ be closed linear subspaces. Note that we then have the useful
rules $M/N\cong N^{\bot_M}=N^\bot\cap M$.
\end{note}

Denote by \symindex{X+@$X^{\mp}(s)$ symplectic splitting} $X^{\mp}(s)$ the positive (negative)
eigenspace of $iJ(s)$. Together they yield a \subindex{Symplectic splitting!canonical in strong
symplectic Hilbert space}spectral decomposition of $X$. Then the Hermitian form $-i\w(s)$ is
negative definite, respectively, positive definite on the subspaces $X^{\mp}(s)$ and we have a
symplectic splitting $X=X^-(s)\oplus X^+(s)$\/.

\begin{definition}[Oriented Maslov index in strong symplectic Hilbert space] \label{d:mas-hilbert}
\subindex{Symplectic structures!symplectic Hilbert space}
Let $\{\lambda(s),\mu(s)\}_{s\in[0,1]}$ be a path of \subindex{Continuously varying!Fredholm pairs
of Lagrangian subspaces}Fredholm pairs of Lagrangian subspaces of $(X,\w(s))$. Let \subindex{Unitary
generators}\symindex{U@$U,V$ unitary generators}$U(s),V(s)\colon X^-(s)\to X^+(s)$ be generators
for $(\lambda(s),\mu(s))$, i.e., $\lambda(s)=\Graph(U(s))$ and $\mu(s)=\Graph(V(s))$ (see
\auindex{Boo{\ss}--Bavnbek,\ B.}\auindex{Zhu,\ C.}\cite[Proposition 2]{BooZhu:2013}). Then
$U(s)V(s)^{-1}$ is a \subindex{Continuously varying!unitary operators}continuous family of unitary
operators on continuous families of Hilbert spaces $X^+(s)$ with \subindex{Hilbert space!inner
product induced by symplectic form on positive component}Hilbert structure $-i\w(s)|_{X^+(s)}$\/,
and $U(s)V(s)^{-1}-I_{X^+(s)}$ is a family of \subindex{Fredholm operator}\subindex{Continuously
varying!Fredholm operators with index $0$}Fredholm operators with index $0$. Denote by
\symindex{l@$\ell_{\pm}$ co-oriented gauge curve}$\ell_{\pm}$ the curve $(1-\e,1+\e)$ with real
$\e\in (0,1)$ and with upward (downward) \subindex{Co-orientation}co-orientation. The
\subindex{Maslov index!in strong symplectic Hilbert space}{\em oriented Maslov index}
\symindex{Mas@$\Mas_{\pm}\{\la(s),\mu(s)\}$ oriented Maslov index}$\Mas_{\pm}\{\la(s),\mu(s)\}$ of
the path \subindex{Continuously varying!Fredholm pairs of Lagrangian
subspaces}$(\lambda(s),\mu(s))$, $s\in[0,1]$ is defined by
\begin{align}\label{e:positive-mas-hilbert}\Mas\{\la(s),\mu(s)\}=\Mas_+\{\la(s),\mu(s)\}&=-\SF_{\ell_-}\{U(s)V(s)^{-1}\},\\
\label{e:negative-mas-hilbert}\Mas_-\{\la(s),\mu(s)\}&=\SF_{\ell_+}\{U(s)V(s)^{-1}\}.
\end{align}
Here we refer to \auindex{Boo{\ss}--Bavnbek,\ B.}\auindex{Zhu,\ C.}\cite[Definition 2.1]{Zh01} and
\cite[Definition 13]{BooZhu:2013} for the definition of the \subindex{Spectral flow}spectral flow
\symindex{sfl@$\SF_{\ell}$ spectral flow through gauge curve $\ell$}$\SF_{\ell}$.
\end{definition}

The following simple example shows that the preceding definition of the Maslov index can not be
generalized literally to symplectic Banach spaces or weak symplectic Hilbert spaces. It shows that
there exist strong symplectic Banach spaces that do not admit a symplectic splitting in the
preceding sense. That may seem to contradict \subindex{Symplectic splitting!non-existence}
\subindex{Zorn's Lemma}Zorn's Lemma. However, in a symplectic Banach space $(X,\w)$ Zorn's Lemma
can only provide the existence of a maximal subspace $X^+$ where the form $-i\w$ is positive
definite. Then $-i\w$ is negative definite on $X^-:=(X^+)^{\w}$ and vanishing on $X^+\times X^-$.
However, one can not show that $X=X^+\oplus X^-$. Denote by $V:=X^+\oplus X^-$, then
$V^{\w}=\{0\}$. We see from it that $\ol{V}^{\Tt} =X$, where $\Tt$ denotes the locally convex
topology defined by $\w$.

\begin{example}[Symplectic splittings do not always exist]\label{ex:no-splitting}
\subindex{Obstructions to straight forward generalization!no symplectic splitting}
\subindex{Counterexamples!no symplectic splitting}
Let $(X,\w):=\la\oplus\la^*$ and $\la:=\ell^p$ with $p\in(1,+\infty)$ and $p\ne 2$.
Then $X$ is a strong symplectic Banach space, but there is no splitting $X=X^+\oplus X^-$ such that
$\mp i\w|_{X^{\pm}}>0$, and $\w(x,y)=0$ for all $x\in X^+$ and $y\in X^-$. Otherwise we could
establish an inner product on $X$ that makes $X$ a Hilbert space.

Moreover, even when a symplectic splitting exists, there is no way to establish such splitting for
families of symplectic Banach spaces in a continuous way, as emphasized in the Introduction.

\end{example}

\section{The Maslov index in finite dimensions}\label{ss:finite-dimensions}

Consider the special case $\dim X=2n<+\infty$. Note that the eigenvalues of $U(s)V(s)^{-1}$ are on
the unit circle $S^1$. Recall that each map in $C\left([0,1],S^1\right)$ can be lifted to a map
$C\left([0,1],\R\right)$. By \auindex{Kato,\ T.}\cite[Theorem II.5.2]{Ka95}, there are $n$
continuous functions $\theta_1,\ldots, \theta_n\in C([0,1],\R)$ such that the eigenvalues of the
operator $U(s)V(s)^{-1}$ for each $s\in [0,1]$ (counting algebraic multiplicities) have the form
\[ \subindex{Parametrization of eigenvalues on the circle}
e^{i\theta_j(s)},\ j=1,\dots,n.
\]

Denote by $[a]$ the integer part of $a\in\R$. Define
\begin{equation}\label{e:function-E}
E(a):=\begin{cases} a,& a\in\Z\\
[a]+1, & a\notin \Z. \end{cases}
\end{equation}

In this case, we have

\begin{align}
\label{e:positive-mas-finite}\Mas_+\{\lambda(s),\mu(s);s\in [0,1]\}&=
\sum_{j=1}^n\left(E\bigl(\frac{\theta_j(1)}{2\pi}\bigr) -
E\bigl(\frac{\theta_j(0)}{2\pi}\bigr)\right),\\
\label{e:negative-mas-finite}\Mas_-\{\lambda(s),\mu(s);s\in [0,1]\}&= \sum_{j=1}^n
\left([\frac{\theta_j(1)}{2\pi}] - [\frac{\theta_j(0)}{2\pi}]\right).
\end{align}

By definition, $\Mas_{\pm}\{\lambda(s),\mu(s);s\in [0,1]\}$ is an integer that does not depend on
the choices of the arguments $\theta_j(s)$. By \auindex{Boo{\ss}--Bavnbek,\ B.}\auindex{Zhu,\
C.}\cite[Proposition 6]{BooZhu:2013}, it does not depend on the particular choice of the paths of
the symplectic splittings.

\section[Properties of the Maslov index in symplectic Hilbert space]{Properties of the Maslov
index in strong symplectic Hilbert space}\label{ss:mas-property-Hilbert}
 \subindex{Symplectic structures!symplectic Hilbert space}
From the properties of the spectral flow, we get all the basic properties of the Maslov index for
strong symplectic Hilbert spaces (see S. E. Cappell, R. Lee, and E. Y. Miller \auindex{Cappell,\
S.E.}\auindex{Lee,\ R.}\auindex{Miller,\ E.Y.}\cite[Section 1]{CaLeMi94} for a more comprehensive
list). Be aware that the proof of Proposition \ref{p:maslov-properties}.d is less trivial (see
\auindex{Zhu,\ C.}\cite[Corollary 4.1]{Zh05}).

The properties of the following list will first be used for establishing a rigorous and calculable
concept of the Maslov index in weak symplectic Banach space. For the Maslov index defined in that
way by symplectic finite-dimensional reduction, we shall later recover the full list of valid
properties in Theorem \ref{t:main} for the general case.

\begin{proposition}[Basic properties of the Maslov
index]\label{p:maslov-properties}\subindex{Maslov index!basic properties!in strong symplectic
Hilbert spaces} Let $p\colon\mathbb{X}\to [0,1]$, $p_j\colon\mathbb{X}_j\to [0,1]$, $j=1,2$, be
Hilbert bundles with continuously varying strong symplectic forms $\w(s)$ (respectively, $\w_j(s)$)
on $X(s):=p\ii(s)$ (respectively $X_j(s):=p_j\ii(s)$). Let $(\la,\mu),(\la_1,\mu_1),(\la_2,\mu_2)$
be curves of Fredholm pairs of Lagrangians in $\mathbb{X}, \mathbb{X}_1,\mathbb{X}_2$. Then we
have:
\subindex{Maslov index!basic properties!invariance under homotopies}

\noi (a) The Maslov index is {\em invariant under homotopies} of curves of Fredholm pairs of
Lagrangian subspaces with fixed endpoints. In particular, the Maslov index is invariant under {\em
re-parametrization} of paths.

\subindex{Maslov index!basic properties!additivity under catenation}
\noi (b) The Maslov index is additive under {\em catenation}, i.e.,
\[
\Mas_{\pm}\{\la, \mu\} = \Mas_{\pm}\{\la|_{[0,a]},\mu|_{[0,a]}\} +
\Mas_{\pm}\{\la|_{[a,1]},\mu|_{[a,1]}\}\,,
\]
for any $a\in[0,1]$.

\subindex{Maslov index!basic properties!additivity under direct sum}
\noi (c) The Maslov index is additive under {\em direct sum}, i.e.,
\[
\Mas_{\pm}\{\la_1\oplus\la_2, \mu_1\oplus \mu_2\} = \Mas_{\pm}\{\la_1,\mu_1\} +
\Mas_{\pm}\{\la_2,\mu_2\}\,,
\]
where $\{{\la}_j(s)\},\{{\mu}_j(s)\}$ are paths of Lagrangian subspaces in $(X_j,\w_j(s))$, $j=1,2$
and $\la_1\oplus\la_2$ is a path of subspaces in $(X_1\oplus X_2, \w_1(s)\oplus\w_2(s))$.

\subindex{Maslov index!basic properties!natural under symplectic action}
\noi (d) The Maslov index is {\em natural} under symplectic action: given a second Hilbert bundle
$\mathbb{X}'=\{X'(s)\}$, a path of symplectic structures $\w'(s)$ on $X'(s)$, and a path of bundle
isomorphisms $\{L(s)\in\Bb\left(X(s),X'(s)\right)\}$ such that \subindex{Symplectic
action}$L(s)^*(\w'(s))=\w(s)$, then we have
\[
\Mas_{\pm}\{\la(s),\mu(s);\w(s)\}= \Mas_{\pm}\{L(s)\la(s),L(s)\mu(s);\w'(s)\}.
\]

\subindex{Maslov index!basic properties!vanishing for constant intersection dimension}
\noi (e) The Maslov index {\em vanishes}, if $\dim (\la(s)\cap \mu(s))$ is a constant function on $s\in
[0,1]$.

\subindex{Maslov index!basic properties!flipping}
\noi (f) {\em Flipping}. We have
\begin{align*}
\Mas_+\{\la(s),\mu(s)\}&+\Mas_+\{\mu(s),\la(s)\}\\
&=\Mas_+\{\la(s),\mu(s)\}-\Mas_-\{\la(s),\mu(s)\}\\
&= \dim(\la(0)\cap\mu(0)) -\dim(\la(1)\cap\mu(1))\/,
\end{align*}
and $\Mas_{\pm}\{\la(s),\mu(s);\w(s)\}=\Mas_{\pm}\{\mu(s),\la(s);-\w(s)\}$.

\subindex{Maslov index!basic properties!local range}
\noi (g) {\em Local range}. There exists an $\e>0$ such that
\begin{equation*}0\le\Mas_+\{\la(s),\mu(s);\w(s)\}\le\dim(\la(0)\cap\mu(0))-\dim(\la(1))\cap\mu(1),
\end{equation*}
if the variation of the curves $\la,\mu$ of Lagrangians and the variation of the symplectic forms
$\w(s)$ is sufficiently small, namely if in the notations of Appendix \ref{ss:gap-topology}
\[\hat\delta(\la(s),\la(0)), \hat\delta(\mu(s),\mu(0)),\|\w(s)-\w(0)\|<\e, \text{ for all $s\in [0,1]$.}\]
\end{proposition}

\begin{note} To preserve the preceding basic properties for the weak symplectic case, we have to
make the explicit assumption of the vanishing of the index of the relevant Fredholm pairs of
Lagrangian subspaces in properties (a) and (g), see Assumption \ref{a:banach-fredholm-lagrange} and
Theorem \ref{t:main}.
\end{note}

\smallskip

We have the following lemma (see \auindex{Robbin,\ J.}\auindex{Salamon,\ D.}J. Robbin and D.
Salamon, \cite[Theorem 2.3, Localization]{RoSa93} for the constant symplectic structure case).

\begin{lemma}\label{l:local-mas-finite} Let $(X,\w(s))$ be a continuous family of $2n$ dimensional symplectic vector spaces
with Lagrangian subspaces $\la_0$, $\mu_0$ such that $X=\la_0\oplus\mu_0$. Let
$A(s)\in\Hom(\la_0,\mu_0)$, $s\in[0,1]$ be a path of linear maps such that $\la(s)=\Graph(A(s))$ is
a Lagrangian subspace of $(\C^{2n},\w(s))$ for each $s\in[0,1]$. Define
\symindex{Q@$Q,Q(s),Q(\la,t)$ intersection form}\subindex{Intersection
form}$Q(s)(x,y)=\w(s)(x,A(s)y)$ for all $s\in[0,1]$, $x\in\la_0$ and $y\in\mu_0$. Then $Q(s)$ is an
Hermitian form on $\la_0$ and we have
\begin{align}
\label{e:local-mas-finite-plus}\Mas_+\{\la(s),\la_0;s\in[0,1]\}&=m^+(Q(1))-m^+(Q(0)),\\
\label{e:local-mas-finite-minus}\Mas_-\{\la(s),\la_0;s\in[0,1]\}&=m^-(Q(0))-m^-(Q(1)),
\end{align}
where \subindex{Morse theory!Morse index}\symindex{m@$m^{\pm}(Q)$, $m^0(Q)$ positive (negative)
Morse index and nullity of Hermitian form $Q$}$m^{\pm}(Q)$, $m^0(Q)$ denote the positive (negative)
Morse index and the nullity of $Q$ respectively for an
Hermitian form $Q$.
\end{lemma}

\begin{proof} Clearly, $\la(s)$ is Lagrangian if and only if $Q(s)$ is Hermitian.
By choosing a frame, we can assume that $X=\C^{2n}$, $\la_0=\C^n\times\{0\}$ and
$\mu_0=\{0\}\times\C^n$. Let $J(s)$ be defined by $\w(s)(x,y)=\lla J(s) x,y\rra$ for each
$s\in[0,1]$. Then we have $J(s)=\left(
\begin{array}{cc}0&-K(s)^*\\K(s)&0\end{array}\right)$ for some $K(s)\in\GL(n,\C)$. Set $T(s):=\diag(K(s)^{-1},I_n)$. Then we have $T(s)^*J(s)T(s)=J_{2n}:=\left(
\begin{array}{cc}0 & -I_n \\I_n & 0 \\ \end{array}\right)$.
By Proposition \ref{p:maslov-properties}.d, we can assume that $J(s)=J_{2n}$. Then we have
$X^{\pm}=\{(x,\mp ix);x\in\C^n\}$. The generator of $\la(s)$ is the map $(x,ix)\mapsto
(U(s)x,-iU(s)x)$, $x\in\C^n$. So $U(s)=(I_n+iA(s))(I_n-iA(s))^{-1}$. We have $U(s)=0$ if $A(s)=0$.
Note that $A(s)$ is a continuous family of self-adjoint operators. By the definition of the
spectral flow we have
\begin{align*}
\Mas_+\{\la(s),\la_0\}&=-\SF_{\ell_-}\{U(s)\}=-\SF\{-A(s)\}\\&
=m^+(A(1))-m^+(A(0))=m^+(Q(1))-m^+(Q(0)).
\end{align*}
Similarly we have (\ref{e:local-mas-finite-minus}).
\end{proof}

The following proposition is a slight generalization of \auindex{Boo{\ss}--Bavnbek,\
B.}\auindex{Furutani,\ K.}\cite[Theorem 4.2 and Remark 5.1]{BoFu99}, where it was shown for the
first time that the Maslov index is preserved under certain symplectic reductions. It was that
result that inspired us to base our new definition of the Maslov index in weak symplectic infinite
dimensional spaces on the concept of symplectic reduction. From a technical point of view, the
following very general proposition for strong symplectic structures together with its modifications
for weak symplectic structures in Section \ref{ss:sympl-invariance} is one of the main achievements
in this Memoir. Note that the arguments depend on our novel intrinsic decomposition techniques of
Section \ref{ss:intrinsic-decomposition} in the preceding chapter.

\begin{proposition}[Invariance of symplectic reduction in strong symplectic Hilbert space]\label{p:local-hilbert-red}
 \subindex{Symplectic structures!symplectic Hilbert space}
\subindex{Symplectic reduction!invariance in strong symplectic Hilbert space}
Let $(X,\w(s))$, $s\in(-\e,\e)$ be a family of strong symplectic Hilbert spaces with continuously
varying symplectic form $\w(s)$, where $\e>0$. Let $(\lambda(s),\mu(s))$, $s\in(-\e,\e)$ be a path
of Fredholm pairs of Lagrangian subspaces of $(X,\w(s))$. Let $V(s)$ be a path of
finite-dimensional subspaces of $X$ with $X=V(0)\oplus(\la(0)+\mu(0))$. Then there exists a
$\delta\in(0,\e)$ such that
\[X=V(0)+\la(s)+\mu(s)=V(s)^{\w(s)}+\la(s)=V(s)^{\w(s)}+\mu(s)\]
for all $s\in (-\delta,\delta)$, and
\begin{align}\subindex{Calculation of the Maslov index}
\nonumber\Mas_{\pm}&\{\la(s),\mu(s);s\in[s_1,s_2]\}\\
\label{e:maslov-local-hilbert1}&=\Mas_{\pm}\bigl\{R^{\w(s)}_{V(s)+\la(s)}(\la(s)),R^{\w(s)}_{V(s)+\la(s)}(\mu(s));s\in[s_1,s_2]\bigr\}\\
\label{e:maslov-local-hilbert2}&=\Mas_{\pm}\bigl\{R^{\w(s)}_{V(s)+\mu(s)}(\la(s)),R^{\w(s)}_{V(s)+\mu(s)}(\mu(s));s\in[s_1,s_2]\bigr\}
\end{align}
for all $[s_1,s_2]\subset(-\delta,\delta)$.
\end{proposition}

\begin{proof} Set $\la_0(0):=\la(0)\cap\mu(0)$, $\la_1(s):=V(s)^{\w(s)}\cap\la(s)$,
$\mu_1(s):=V(s)^{\w(s)}\cap\mu(s)$, $X_1(s):=\la_1(s)+\mu_1(s)$, and $X_0(s):=X_1(s)^{\w(s)}$. By
Proposition \ref{p:reduction-prop} we have
\[X=\la_0(0)\oplus V(0)\oplus\la_1(0)\oplus\mu_1(0),\]
$X_0(0)=\la_0(0)+V(0)$, $\la_0(0)\in\Ll(X_0(0))$, and $X_1(0)=X_0(0)^{\w}$.

By Appendix \ref{ss:continuity-of-operations} and Proposition \ref{p:reduction-prop}.g, there exists a
$\delta_1\in(0,\e)$ such that
\begin{align*}
X&=V(s)+\la(s)+\mu(s)=V(s)^{\w(s)}+\la(s)=V(s)^{\w(s)}+\mu(s)\\
&=\la(s)\oplus(V(s)+\mu_1(s))=\mu(s)\oplus(V(s)+\la_1(s)),
\end{align*}
$X_1(s)=\la_1(s)\oplus\mu_1(s)$, and $X=X_0(s)\oplus X_1(s)$ for all $s\in (-\delta,\delta)$. Set
$X_0(s):=X_1(s)^{\w(s)}$. Then we have $V(s)\subset X_0(s)$. Since $X_0(s)$ is a finitely dimensional symplectic space, by Lemma \ref{l:parametrize-isotropic},
there exist a $\delta\in(0,\delta_1)$ and a path $\la_0(s)\in\Ll(X_0(s))$, $s\in(-\delta,\delta)$ such that $X_0(s)=\la_0(s)\oplus V(s)$.

Denote by $P_0(s)\colon X\to X_0(s)$ the projection defined by $X=X_0(s)\oplus X_1(s)$. Denote by $\w_l(s)$ the symplectic form of $X_0(s)$ defined by
Proposition \ref{p:reduction-prop}.d. Since $\la_0(s)\subset\la_1(s)^{\w(s)}$, by Proposition \ref{p:reduction-prop}.d we have $\w_l(s)=\w(s)|_{x_0(s)}$. By
Proposition \ref{p:reduction-prop}.c,d and Proposition \ref{p:maslov-properties}.c,d,e, we have
\begin{align*}
\Mas_{\pm}&\{\la(s),\mu(s);s\in[s_1,s_2]\}\\
&=\Mas_{\pm}\{P_0(s)(\la(s)),P_0(s)(\mu(s));s\in[s_1,s_2]\}\\
&\quad +\Mas_{\pm}\{\la_1(s),\mu_1(s);s\in[s_1,s_2]\}\\
&=\Mas_{\pm}\bigl\{T_l(R^{\w(s)}_{V(s)+\la(s)}(\la(s))),T_l(R_{V(s)+\la(s)}(\mu(s)));s\in[s_1,s_2]\bigr\}\\
&=\Mas_{\pm}\bigl\{R^{\w(s)}_{V(s)+\la(s)}(\la(s)),R^{\w(s)}_{V(s)+\la(s)}(\mu(s));s\in[s_1,s_2]\bigr\}.
\end{align*}
Note that by Proposition \ref{p:reduction-prop} and Appendix \ref{ss:continuity-of-operations}, the
Maslov indices in the above calculations are well-defined. The equality
(\ref{e:maslov-local-hilbert2}) follows similarly.
\end{proof}

\addtocontents{toc}{\medskip\noi}  \chapter{The Maslov index in Banach bundles over a closed
interval}\label{s:maslov-general}

\section[The Maslov index by symplectic reduction]{The Maslov index by symplectic reduction to a finite-dimensional subspace}\label{ss:maslov-definition}

We need the following simple fact for the definition of the Maslov index.

\begin{lemma}\label{l:v-w-continuous} Let $X$ be a Banach space with continuously varying symplectic structures $\w(s)$, $s\in[0,1]$. Let $V(s)$, $s\in[0,1]$ be a path of finite-dimensional subspaces of $X$. Then $V(s)^{\w(s)}$, $s\in[0,1]$ is a $C^0$ path.
\end{lemma}

\begin{proof} Since the problem is local, we only need to consider $|s-s_0|<<1$ for each fixed $s_0\in[0,1]$.
Note that the symplectic form $\w(s)$ on $X$ induces a uniquely defined bounded, injective
mapping $J(s)\colon X\to X^{\ad}$ such that $\w(s)(x,y)=(J(s)x)(y)$ for all $x,y\in X$. Then $J(s)$ is continuously varying. By (\ref{e:annihi-perp2}), we have $V(s)^{\w(s)}=(J(s)V(s))^{\perp}$.
Since $V(s)$, $s\in[0,1]$ is a path of finite-dimensional subspaces of $X$, $J(s)V(s)$ is continuously varying. By Lemma \ref{l:basic-complemented}, $J(s)V(s)$ is complemented and it is defined by a projection $P(s)$. By Lemma \ref{l:connectness-local} we can make $P(s)$ continuously varying on $s$. Then $P(s)^*$ is continuously varying. By Corollary \ref{c:perp-complemented-continuous}, $\ran(I- P(s)^*)=V(s)^{\w(s)}$ is continuously varying on $s$. Now by Lemma \ref{l:ck-complemented}, $V(s)^{\w(s)}$, $s\in[0,1]$ is a $C^0$ path.
\end{proof}

For this section, we fix some data and notations and make the following assumption:

\begin{ass}\label{a:banach-fredholm-lagrange}
Let $p\colon \mathbb{X}\to [0,1]$ be a
\subindex{Banach bundle}
Banach bundle. Denote by $X(s):=p^{-1}(s)$ the fiber of $p$ at $s\in[0,1]$. Let
$\{\w(s)\}_{s\in[0,1]}$ be a continuous family of symplectic structures with $\w(s)$ acting on
$X(s)\times X(s)$. Let $\{(\lambda(s),\mu(s))\}_{s\in[0,1]}$ be a path of Fredholm pairs of
Lagrangian subspaces of $(X(s),\omega(s))$ of index $0$.
\end{ass}

Here for a fiber bundle $p\colon \mathbb{X}\to[0,1]$, a \subindex{Path in Banach bundle}path
$c(s)$, $s\in[0,1]$ of $\mathbb{X}$ is a continuous map $c\colon [0,1]\to \mathbb{X}$ such that
$c(s)\in p^{-1}(s)$ for each $s\in[0,1]$. We refer to \auindex{Zaidenberg,\ M.G.}\auindex{Krein,\
S.G.}\auindex{Kucment,\ P.A.}\auindex{Pankov,\ A.A.}\cite{ZKKP:1975} for the concept of Banach
bundles; see also our summary in the Appendix \ref{ss:banach-bundles}. The fiber bundle
$\mathbb{X}$ is always trivial. So we can actually assume that $X(s)\equiv X$. By \auindex{Kato,\
T.}\cite[Lemma I.4.10]{Ka95} and Lemma \ref{l:connectness-local}, the set of complemented closed
subspaces is a \subindex{Banach manifold}Banach manifold and can be identified locally with the
general linear group \symindex{B@$\Bb^\times(X)$ general linear group of bounded invertible
operators on $X$}$\Bb^\times(X)$ of bounded invertible operators of $X$.

As shown in Example \ref{ex:negative-index}, the assumption of vanishing index is a restriction for
Fredholm pairs of Lagrangian subspaces in weak symplectic structures, even when the fibres are
Hilbert spaces.

\smallskip

To define the
\subindex{Maslov index!in symplectic Banach bundles!purely formal definition}
Maslov index via finite-dimensional symplectic reduction, we begin with a purely formal definition.

We make Assumption \ref{a:banach-fredholm-lagrange} and the following choices and notations.

\begin{choices-notations}\label{cn:formal-definition}
By the definition of Fredholm pairs, for each $t\in[0,1]$, there exists $V(t)\subset X(t)$ such
that $V(t)\oplus(\lambda(t)+\mu(t))=X(t)$. Set $\la_0(t):=\la(t)\cap\mu(t)$ and
$X_0(t):=\la_0(t)\oplus V(t)$. Then there exists for each $t$ a $\delta(t)>0$ such that
\begin{enumerate}
\item[(1)] there exists a \subindex{Banach bundle!local frame}\symindex{Lt@$L(t,s)$ local frame}local frame $L(t,s)\colon X(t)\to X(s)$,
$s\in(t-\delta(t),t+\delta(t))\cap[0,1]$ of the bundle $\mathbb{X}$,
\item[(2)] $X(s)=L(t,s)V(t)+\la(s)+\mu(s)=(L(t,s)V(t))^{\omega(s)}+\lambda(s)$ for all
$s\in(t-\delta(t),t+\delta(t))\cap[0,1]$, and
\item[(3)] we have
\begin{align}\label{e:direct-sum}
X(s)&=L(t,s)X_0(t)\oplus\la_1(t,s)\oplus\mu_1(t,s)\\
\nonumber&=\la(s)\oplus(L(t,s)V(t)+\mu_1(t,s))\\
\nonumber&=\mu(s)\oplus(L(t,s)V(t)+\la_1(t,s))
\end{align}
for all $s\in(t-\delta(t),t+\delta(t))\cap[0,1]$, where
$\la_1(t,s):=(L(t,s)V(t))^{\w(s)}\cap\la(s)$ and $\mu_1(t,s):=(L(t,s)V(t))^{\w(s)}\cap\mu(s)$.
\end{enumerate}

Denote by $X_1(t,s):=\la_1(t,s)+\mu_1(t,s)$. Denote by
\symindex{P0@$P_0,P_0(s),P_0(t,s)$ decomposition projections}
$P_0(t,s)\colon X(s)\to L(t,s)X_0(s)$ the projection defined by $X(s)=L(t,s)X_0(t)\oplus X_1(t,s)$.
Denote by \symindex{\omega l@$\w_l, \w_r$ decomposition symplectic structures}$\w_l(t,s)=\w_r(t,s)$
the symplectic structure defined by Proposition \ref{p:cal-red}.d. We have the finite-dimensional
vector space $\{(X_0(t),L(t,s)^*(\w_l(t,s))\}_{s\in [0,1]}$ with continuously varying symplectic
structure for fixed $t\in[0,1]$.
\end{choices-notations}

\begin{definition}[Maslov index by symplectic reduction]\label{d:maslov-banach}
\subindex{Maslov index!in symplectic Banach bundles!purely formal definition}
\symindex{Mas@$\Mas(\la(s),\mu(s))_{s\in [0,1]}$ Maslov index}
With the notations and choices above, let $0=a_0<a_1<...<a_n=1$ be a partition with
$[a_k,a_{k+1}]\subset(t_k-\delta(t_k),t_k+\delta(t_k))$ for some $t_k\in[0,1]$, $k=0,\dots n-1$.
Define
\begin{multline}\label{e:maslov-banach}\subindex{Maslov index!segmental}
\symindex{Mas@$\Mas_{\pm}\{\lambda(s),\mu(s)\}$ oriented Maslov index}
\Mas_{\pm}\bigl\{\lambda(s),\mu(s);s\in[0,1]\bigr\}\ :=\ \sum_{k=0}^{n-1}
\Mas^{\w_l(t_k,s)}_{\pm}\\
\bigl\{L(t_k,s)^{-1}P_0(t_k,s)(\lambda(s)),L(t_k,s)^{-1}P_0(t_k,s)(\mu(s));s\in[a_k,a_{k+1}]\bigr\},
\end{multline}
where $\Mas^{\w_l(t_k,s)}_{\pm}\{\dots\}$ denotes the oriented Maslov index for the specified
Fredholm pair of symplectically reduced Lagrangian subspaces in the finite-dimensional complex
vector space $X_0(t_k)$ with continuously varying induced symplectic structures $\w_l(t_k,s)$. That
oriented Maslov index was introduced in Definition \ref{d:mas-hilbert}. We call $\Mas_{\pm}$ the
{\em positive (negative) Maslov index}. We call the positive Maslov index $\Mas:=\Mas_+$ the {\em
Maslov index}.
\end{definition}

To lift the formal concepts of Definition \ref{d:maslov-banach} to a useful definition of the
Maslov index in Banach spaces, we prove the following theorem:

\begin{theorem}[Main Theorem]\label{t:main}
\subindex{Maslov index!in symplectic Banach bundles!independence of choices} \subindex{Maslov
index!in symplectic Banach bundles!Main Theorem}\subindex{Maslov index!basic properties!in
symplectic Banach bundles}\subindex{Main Theorem}
Under Assumption \ref{a:banach-fredholm-lagrange}, the mappings $\Mas_{\pm}$ are well-defined
(i.e., independent of the choices) and the common properties of the Maslov index (listed in
Proposition \ref{p:maslov-properties}) are preserved.
\end{theorem}

\begin{rem}\label{r:mas} By the definition of the spectral flow, our definition coincides with that in Definition \ref{d:mas-hilbert},
and more generally, \auindex{Boo{\ss}--Bavnbek,\ B.}\auindex{Zhu,\ C.}\cite[Definition
7]{BooZhu:2013} in their special cases. Our definition of the Maslov index generalizes the ideas in
\auindex{Swanson,\ R.C.}\auindex{Boo{\ss}--Bavnbek,\ B.}\auindex{Furutani,\
K.}\cite{BoFu99,Swanson1}.
\end{rem}

We firstly show that Theorem \ref{t:main} is true in the
\subindex{Main Theorem!local validity}
local case. For sufficiently small parameter variation, that follows from the homotopy invariance
of the Maslov index. This property was established  in Proposition \ref{p:maslov-properties}.a for
strong symplectic Hilbert spaces, and so for finite dimensions, i.e., it is valid in our case after
symplectic reduction. Besides the application of the established homotopy invariance, the point of
the following lemma is the intrinsic decomposition introduced in Section
\ref{ss:intrinsic-decomposition}: Roughly speaking, for curves of Fredholm pairs of Lagrangian
subspaces of index 0, the technique of intrinsic decomposition permits to extend the choice of a
single complementary space $V(s_0)$ for $\la(s_0)+\mu(s_0)$ in $X(s_0)$ to a continuous
decomposition of the induced finite-dimensional subspace $X_0(s_0):=(\la(s_0)\cap\mu(s_0))\oplus
V(s_0)$ of $X(s_0)$. While we can identify all Banach spaces $X(s)$ ($s\in [0,1]$) with one fixed
Banach space $X$, we can not identify the subspaces $X_0(s)$ ($s\in [0,1]$) with one
finite-dimensional subspace $X_0\< X$, in general and even not locally. The reason is that the
dimension of the intersection $\la(s)\cap\mu(s)$ is upper semi-continuous in any $s_0\in [0,1]$,
and so must be the codimension of the sum $\la(s)+\mu(s)$, having the jumps at the same parameters
like the intersection dimension due to the vanishing index.

The following lemma shows how we can get around that difficulty \textit{locally}, namely assuming
\eqref{e:plus-condition2}-\eqref{e:plus-condition3}.

\begin{lemma}[Two-parameter Maslov index]\label{l:local-mas-def}
\subindex{Maslov index!with two parameters}
We make Assumption \ref{a:banach-fredholm-lagrange} with $X(s)=X$. We assume that there exists a
finite-dimensional subspace  $V$ of $X$ that is a supplement in $X$ (not necessarily transversal)
to all sums $\la(s)+\mu(s)$. More precisely, we assume that
 \begin{align}
\label{e:plus-condition2}X&=X_0\oplus\la_1(s)\oplus\mu_1(s)\\
\label{e:plus-condition3}&=\la(s)\oplus(V+\mu_1(s)) =\mu(s)\oplus(V+\la_1(s))
\end{align}
for all $s\in[0,1]$, where $X_0:=V\oplus\la_0$, $\la_0$ is a finite-dimensional subspace of $X$,
$\la_1(s):=V^{\w(s)}\cap\la(s)$ and $\mu_1(s):=V^{\w(s)}\cap\mu(s)$.

Denote by $X_1(s):=\la_1(s)+\mu_1(s)$. Denote by
\symindex{P0@$P_0,P_0(s),P_0(t,s)$ decomposition projections}
$P_0(s)\colon X\to X_0$ the projection defined by $X=X_0\oplus X_1(s)$. Denote by \symindex{\omega
l@$\w_l, \w_r$ decomposition symplectic structures}$\w_l(s)=\w_r(s)$ the symplectic structure
defined by Proposition \ref{p:cal-red}.d. By definition of Fredholm pairs, for each $t\in[0,1]$,
there exists $V(t)$ such that $V(t)\oplus(\lambda(t)+\mu(t))=X$. Then there exists a $\delta(t)>0$
for each $t\in[0,1]$ such that, for all $[s_1,s_2]\subset(t-\delta(t),t+\delta(t))\cap[0,1]$,
\newline
(a) the properties (2) and (3) in Choices and Notations \ref{cn:formal-definition} are satisfied
with $L(t,s)=I_X$, and
\newline
(b) we have the following two-parameter formula:
\begin{align}
\nonumber\Mas_{\pm}&\bigl\{P_0(s)(\la(s)),P_0(s)(\mu(s));\w_l(s);s\in[s_1,s_2]\bigr\}\\
\label{e:maslov-local}&=\Mas_{\pm}\bigl\{P_0(t,s)(\lambda(s)),P_0(t,s)(\mu(s));\w_l(t,s);s\in[s_1,s_2]\bigr\},
\end{align}
where $P_0(t,s)$ and $\w_l(t,s)$ are given by Definition \ref{d:maslov-banach}.
\end{lemma}

\begin{proof} Let $s\in[0,1]$. By Proposition \ref{p:cal-red}.g, we have $X=V+\la(s)+\mu(s)=V^{\omega(s)}+\lambda(s)=V^{\w(s)}+\mu(s)$. Set
$W_l(s):=V+\la(s)$. Then we have $W_l(s)^{\w(s)}=\la_1(s)$. By Proposition \ref{p:cal-red}.b, we
have a linear isomorphism $T_l(s)\colon W_l(s)/\la_1(s)\to X_0$ induced by $P_0$. So $X_0$ is
symplectic. Denote by $\tilde\w(s)$ the induced symplectic structure on $W_l(s)/\la_1(s)$. By
Proposition \ref{p:cal-red}.d, the symplectic structure on $X_0$ induced from $\tilde\w(s)$ by
$T_l(s)$ is given by $\w_l(s)$.

Let $t\in[0,1]$. Since $X=V+\la(t)+\mu(t)$, there exists a linear subspace $V(t)^{\prime}$ of $V$
such that $V(t)^{\prime}\oplus(\lambda(t)+\mu(t))=X$ for each $t\in[0,1]$. Recall that
$\la_0(t)=\la(t)\cap\mu(t)$. For $s\in[0,1]$ with small $|s-t|$, we set
\begin{align*}&\la_1(t,s))^{\prime}:=(V(t)^{\prime})^{\w(s)}\cap\la(s),\quad\mu_1(t,s))^{\prime}:=(V(t)^{\prime})^{\w(s)}\cap\mu(s),\\
&X_0(t)^{\prime}:=V(t)^{\prime}\oplus\la_0(t),\quad
X_1(t,s)^{\prime}:=\la_1(t,s))^{\prime}\oplus\mu_1(t,s))^{\prime}.
\end{align*}
Denote by $P_0(t,s)^{\prime}$ the projection onto $X_0(t)^{\prime}$ defined by
\begin{equation}\label{e:decomposition-prime}
X=X_0(t)^{\prime}\oplus X_1(t,s).
\end{equation}
We denote by $\w_l(t,s)^{\prime}$ the symplectic structure on $X_0(t)$ defined by
\eqref{e:decomposition-prime} and Proposition \ref{p:cal-red}.d.

Set
\begin{align*}
\tilde W(t,s)\ :&=\ R^{\w(s)}_{V+\la(s)}(V(t)^{\prime}+\la(s))\ =\ \frac{V(t)^{\prime}+\la(s)}{\la_1(s)},\\
\tilde V(t,s)\ :&=\ R^{\w(s)}_{V+\la(s)}(V(t)^{\prime})\ =\ \frac{V(t)^{\prime}+\la_1(s)}{\la_1(s)},\\
\tilde \la(s)\ :&=\ R^{\w(s)}_{V+\la(s)}(\la(s))\ =\ \frac{\la(s)}{\la_1(s)},\\
\tilde \mu(s)\ :&=\ R^{\w(s)}_{V+\la(s)}(\mu(s))\ =\ \frac{\mu(s)\cap(V+\la(s))
+\la_1(s)}{\la_1(s)}.
\end{align*}
Since $V(t)^{\prime}\subset V$, by Lemma \ref{l:lin-alg} we have
\begin{align*}
\tilde V(t,t)&+\tilde\la(t)+\tilde\mu(t)\ =\ \frac{V(t)^{\prime}+\la(t)+\mu(t)\cap (V+\la(t))}{\la_1(t)}\\
&=\ \frac{(V(t)^{\prime}+\la(t)+\mu(t))\cap (V+\la(t))}{\la_1(t)}\\
&=\ \frac{X\cap (V+\la(t))}{\la_1(t)}\ =\ \frac{V+\la(t)}{\la_1(t)}.
\end{align*}
By Proposition \ref{p:cal-red}, Lemma \ref{l:red-transitive} and Proposition
\ref{p:local-hilbert-red}, there exists a $\delta_1(t)>0$ for each $t$ such that, for
$[s_1,s_2]\subset(t-\delta_1(t),t+\delta_1(t))\cap[0,1]$, we have
\begin{align*}
\Mas_{\pm}&\bigl\{P_0(s)(\la(s)),P_0(s)(\mu(s));\w_l(s);s\in[s_1,s_2]\bigr\}\\
&=\ \Mas_{\pm}\bigl\{\tilde\la(s),\tilde\mu(s);\tilde\w(s);s\in[s_1,s_2]\bigr\}\\
&=\ \Mas_{\pm}\bigl\{R^{\tilde\w(s)}_{\tilde W(t,s)}(\tilde\la(s)),R^{\tilde\w(s)}_{\tilde W(t,s)}(\tilde\mu(s));s\in[s_1,s_2]\bigr\}\\
&=\ \Mas_{\pm}\bigl\{R^{\w(s)}_{V(t)^{\prime}+\la_1(s)}(\la(s)),R^{\w(s)}_{V(t)^{\prime}+\la_1(s)}(\mu(s));s\in[s_1,s_2]\bigr\}\\
&=\
\Mas_{\pm}\bigl\{P_0(t,s)^{\prime}(\lambda(s)),P_0(t,s)^{\prime}(\mu(s));\w_l(t,s)^{\prime};s\in[s_1,s_2]\bigr\}.
\end{align*}

For any closed subspace $N\< X$ we denote by
\symindex{G@$G(X,N)$ set of closed subspaces of $X$ complementary to $N$}
$G(X,N)$ the (possibly empty) set of closed subspaces of $X$ that are transversal and complementary
to $N$ in $X$. By Lemma \ref{l:connectness-local}, it is an open affine space. Hence in our case
there exists a path $f(t,\cdot)\colon [0,1]\to G(X,\la(t)+\mu(t))$ with $f(t,0)=V(t)$ and
$f(t,1)=V(t)^{\prime}$ for each $t\in[0,1]$. For $t,a\in [0,1]$ and $s$ "close" to $t$ (to be
specified at once), set $\la_1(t,a,s):=f(t,a)^{\w(s)}\cap\la(s)$, and
$\mu_1(t,a,s):=f(t,a)^{\w(s)}\cap\mu(s)$. By Lemma \ref{l:sum-whole-space}, we have $X=f(t,a)^{\w(t)}+\la(t)=f(t,a)^{\w(t)}+\mu(t)$.
By Proposition \ref{p:reduction-prop} and Appendix
\ref{ss:continuity-of-operations}, there exists $\delta(t)\in(0,\delta_1(t))$ for each $t$ such
that
\begin{align*}
X\ &=\ \la_0(t)\oplus f(t,a)\oplus\la_1(t,a,s)\oplus\mu_1(t,a,s)\\
&=\ \la(s)\oplus\left(f(t,a)+\mu_1(t,a,s)\right)\ =\ \mu(s)\oplus\left(f(t,a)+\la_1(t,a,s)\right)
\end{align*}
for all $s\in(t-\delta(t),t+\delta(t))\cap[0,1]$. That proves properties (2) and (3) of our
Choices and Notations \ref{cn:formal-definition}, i.e., our claim (a).

To (b), we observe that in our case the symplectic reduction does not change the dimension of the
intersection of Lagrangian subspaces. By Lemma A.4.5, we can find a path connecting $\la_0$ and
$\la_0^{\prime}$ in $G(X, V\oplus X_1(t))$. By Proposition \ref{p:maslov-properties}.a, the left
hand side of \eqref{e:maslov-local} remains unchanged in a small interval $[s_1,s_2]$ if we replace
$V(t)$ by $V(t)^{\prime}$ and $\la_0$ by $\la_0^{\prime}$. Then (\ref{e:maslov-local}) holds.
\end{proof}

\begin{note}
We emphasize that for fixed $t$, $\{f(t,a)\}_{a\in [0,1]}$ is a path of finite-dimensional
subspaces of $X$. For each $t,a\in[0,1]$, $f(t,a)$ satisfies that $X=f(t,a)\oplus(\la(t)+\mu(t))$,
$f(t,0)=V(t)^{\prime}$, and $f(t,1)=V(t)$. So by the homotopy invariance and the vanishing of the
Maslov index (in the finite-dimensional case), the Maslov index is unchanged along the path
$f(t,a)$ (fix $t$). The difference between $V(t)$ and $V(t)^{\prime}$ is that $V(t)^{\prime}\subset
V$, while $V(t)$ can vary widely.
\end{note}

\begin{proof}[Proof of Theorem \ref{t:main}]
By taking a common refinement of the partitions, the first part of the Theorem follows from Lemma
\ref{l:local-mas-def}. The second part of the Theorem is a repetition of the list of properties
given in Proposition \ref{p:maslov-properties} for the case of a strong symplectic Hilbert space.
The validity in the general case follows from the proposition and our definition of the Maslov
index.
\end{proof}

We have the following lemma from (\auindex{Boo{\ss}--Bavnbek,\ B.}\auindex{Zhu,\ C.}\cite[Lemma
8]{BooZhu:2013}:

\begin{lemma}\label{l:boxplus}
Let $(X,\w)$ be a symplectic vector space. Let \symindex{\Delta@$\D$ diagonal (i.e., the canonical
Lagrangian) in product symplectic space}$\D$ denote the diagonal (i.e., the canonical Lagrangian)
in the product symplectic space $(X\oplus X,(-\w)\oplus \w)$, and $\la,\mu$ are linear subspaces of
$(X,\w)$. Then
\[
(\la,\mu)\in\Ff\Ll(X) \iff (\la\oplus\mu,\D)\in\Ff\Ll(X\oplus X)
\]
and
\[
\Index(\la,\mu)=\Index(\la\oplus\mu,\D),
\]
where $\la\oplus\mu:=\{(x,y); x\in \la,y\in\mu\}$.
\end{lemma}

The following proposition generalizes \auindex{Boo{\ss}--Bavnbek,\ B.}\auindex{Zhu,\
C.}\cite[Proposition 4 (b)]{BooZhu:2013}.

\begin{proposition}\label{p:boxplus} Denote by $\D(s)$ the diagonal of $X(s)\times X(s)$.
Under Assumption \ref{a:banach-fredholm-lagrange}, we have
\begin{align}\label{e:maslov1}
\Mas&\{\la(s)\oplus\mu(s),\D(s);\w(s)\oplus(-\w(s))\}
=\Mas\{\la(s),\mu(s);\w(s)\}\\
&=\Mas\{\mu(s),\la(s);-\w(s)\}
\label{e:maslov3}\\
&=\Mas\{\D(s),\la(s)\oplus\mu(s);(-\w(s))\oplus \w(s)\}.\label{e:maslov4}
\end{align}
\end{proposition}

\begin{proof} By \auindex{Boo{\ss}--Bavnbek,\ B.}\auindex{Zhu,\ C.}\cite[Proposition 4 (b)]{BooZhu:2013},
our results hold in the finite-dimensional case. The general case follows from the definition of
the Maslov index.
\end{proof}

\section{Calculation of the Maslov index}\label{ss:cal-maslov}
\subindex{Calculation of the Maslov index}
We start with the general case. We fix our data and make some choices.

\begin{data}
Let $\e>0$ be a positive number. Let $X$ be a (complex) Banach space with continuously varying
symplectic structure $\omega(s)$, $s\in(-\e,\e)$. Let $(\lambda(s),\mu(s))$ be a path of Fredholm
pairs of Lagrangian subspaces of $(X,\omega(s))$ of index $0$. Let $V(s)$ and $\la_0(s)$ be a path
of finite-dimensional subspaces of $X$ with $\la_0(0)=\la(0)\cap\mu(0)$ and
$V(0)\oplus(\lambda(0)+\mu(0))=X$. Set $\la_1(s):=V(s)^{\w(s)}\cap\la(s)$,
$\mu_1(s):=V(s)^{\w(s)}\cap\mu(s)$, $X_0(s):=\la_0(s)+V(s)$, $X_1(s):=\la_1(s)+\mu_1(s)$.
\end{data}

\begin{theorem}[Intrinsic decompositions and representations]\label{t:mas-local} There
exists a $\delta>0$ such that for each $s\in(-\delta,\delta)$ and each subinterval
$[s_1,s_2]\subset(-\delta,\delta)$ we have the following intrinsic decompositions, representations,
and formulae:
\newline (a) $X=V(s)+\la(s)+\mu(s)=V(s)^{\omega(s)}+\lambda(s)=V(s)^{\w(s)}+\mu(s)$;
\newline (b) $X=\la_0(s)\oplus V(s)\oplus \la_1(s)\oplus\mu_1(s)=V(s)^{\w(s)}\oplus\la_0(s)$;
\newline (c) $\la(s)$ and $\mu(s)$ are expressed by
\begin{align}
\label{e:local-la}\la(s)&=\Graph\Bigl(\left(\begin{array}{cc}A_1(s) & 0 \\A_2(s) & 0 \\
\end{array}\right)\colon \la_0(s)\oplus\la_1(s)\to V(s)\oplus \mu_1(s)\Bigr),\\
\label{e:local-mu}\mu(s)&=\Graph\Bigl(\left(\begin{array}{cc}B_1(s) & 0 \\B_2(s) & 0
\\\end{array}\right)\colon \la_0(s)\oplus\mu_1(s)\to V(s)\oplus \la_1(s)\Bigr)
\end{align}
with continuous families $\{A_j(s)\}_{s\in[0,1]},\,\{B_j(s)\} _{s\in[0,1]}$
of linear operators and $A_j(0)= B_j(0) =0$ for $j=1,2$;
\newline (d)
Setting
\begin{align}
\label{e:local-wl}\w_l(s):&=\w(s)|_{X_0(s)}-\left(\begin{array}{cc}I_{\la_0(s)} & 0 \\A_1(s) & 0 \\\end{array}\right)^*(\w(s)|_{X_0(s)}), \text{ and}\\
\label{e:local-wr}\w_r(s):&=\w(s)|_{X_0(s)}-\left(\begin{array}{cc}I_{\la_0(s)} & 0 \\B_1(s) & 0
\\\end{array}\right)^*(\w(s)|_{X_0(s)}),
\end{align}
we obtain $\w_l(s)=\w_r(s)$; moreover, the images $P_0(s)(\mu(s))$ are Lagrangian subspaces of the
symplectic vector space $(X_0(s),\w_l(s))$, where $P_0(s)\colon X\to X_0(s)$ denotes the projection
defined by $X=X_0(s)\oplus X_1(s)$ like before.
\newline (e) The following equalities hold for the \subindex{Maslov index!segmental}segmental Maslov
indices and intersection dimensions:
\begin{multline}
\Mas_{\pm}\{\la(s),\mu(s);s\in[s_1,s_2]\}\\
\label{e:mas-local}=\Mas_{\pm}\{P_0(s)(\la(s)),P_0(s)(\mu(s));\w_l(s);s\in[s_1,s_2]\},
\end{multline}
and
\begin{equation} \label{e:dim-intersection-local} \dim(\la(s)\cap\mu(s))\ =\ \dim
\bigl(P_0(s)(\la(s))\cap P_0(s)(\mu(s))\bigr).
\end{equation}
\end{theorem}

\begin{proof} (a) By Proposition \ref{p:cal-red}.g and Appendix \ref{ss:continuity-of-operations}.

(b), (c), (d) By (a), Proposition \ref{p:cal-red} and Appendix \ref{ss:continuity-of-operations}.

(e) If $V(s)\equiv V(0)$, our result follows from the definition of the Maslov index and
Proposition \ref{p:cal-red}.c.

In the general case, by \auindex{Kato,\ T.}\cite[Lemma I.4.10]{Ka95} there exists a path of bounded
invertible map $L(s)\in\Bb(X)$ such that $L(s)X_0=X_0(s)$ with $L(0)=I_X$. By Proposition
\ref{p:maslov-properties}.d we have
\begin{align}&\nonumber\Mas_{\pm}\{P_0(s)(\la(s)),P_0(s)(\mu(s));\w_l(s);s\in[s_1,s_2]\}\\
\label{e:finite-homotopy}=&\Mas_{\pm}\{L(s)^{-1}P_0(s)(\la(s)),L(s)^{-1}P_0(s)(\mu(s));
L(s)^*\w_l(s);s\in[s_1,s_2]\}.
\end{align}

Note that $L(s)$, $P_0(s)$ and $\w_l(s)$ depend \subindex{Continuously varying}continuously on
$\la_0(s)$, $V(s)$, $\la(s)$, $\mu(s)$ and $\w(s)$. Replacing $\la_0(s)$ by $\la_0(ts)$ and $V(s)$
by $V(ts)$ for $t\in[0,1]$, we get a homotopy of the right hand side of (\ref{e:finite-homotopy}).
Note that in our case $\dim(P_0(s)(\la(s))\cap P_0(s)(\mu(s)))=\dim(\la(s)\cap\mu(s))$. Then our
result follows from the special case and Proposition \ref{p:maslov-properties}.
\end{proof}

By Corollary \ref{c:complementary-lagrangian} and \auindex{Neubauer,\ G.}\cite[Lemma 0.2]{Ne68}, we
have a path $\la_0(s)\subset\mu(s)$ with $\la_0(0)=\la(0)\cap\mu(0)$. By Lemma
\ref{l:parametrize-isotropic} and Proposition \ref{p:reduction-prop}.d, we have a Lagrangian path
$V(s)\in X_0(s)$ of $X_0(s)$. We have the following corollary.

\begin{corollary}\label{c:fixed-mu-local} Assume that $\la_0(s)\subset\mu(s)$ as in the data before Theorem \ref{t:mas-local} and
let $\delta>0$ be found correspondingly. Then for each $s\in(-\delta,\delta)$ and
$[s_1,s_2]\subset(-\delta,\delta)$, we have $\w_l(s)=\w(s)|_{X_0(s)}$, and
\begin{align}
\nonumber \Mas_{\pm}&\{\la(s),\mu(s);s\in[s_1,s_2]\}\\
\label{e:fixed-mu-local1}&=\Mas_{\pm}\{P_0(s)(\la(s)),\la_0(s);\w(s)|_{X_0(s)};s\in[s_1,s_2]\}.
\end{align}
\end{corollary}

\begin{proof} In this case we have $B_1(s)=B_2(s)=0$, $P_0(s)(\mu(s))=\la_0(s)$ and $\w_l(s)=\w(s)|_{X_0(s)}$.
By Theorem \ref{t:mas-local}, our results follow.
\end{proof}

\begin{proposition}\label{p:mas-local1} Assume that $V(s)$ is isotropic in Theorem \ref{t:mas-local}.
 Let $\delta$ be as given there. Then for each $s\in(-\delta,\delta)$ and $[s_1,s_2]\subset(-\delta,\delta)$ we have an Hermitian
 form \symindex{Q@$Q,Q(s),Q(\la,t)$ intersection form}\subindex{Intersection form}$Q(s)$ such that
\begin{align}
\label{e:mas-local-plus1} \Mas_+\{\la(s),\mu(s);s\in[s_1,s_2]\}&=m^+(Q(s_2))-m^+(Q(s_1)),\\
\label{e:mas-local-minus1} \Mas_-\{\la(s),\mu(s);s\in[s_1,s_2]\}&=m^-(Q(s_1))-m^-(Q(s_2)),\\
\label{e:dim-intersection-local1} \dim(\la(s)\cap\mu(s))&=m^0(Q(s)),
\end{align}
where $Q(s)(x,y):=\omega(s)(x,(A_1(s)-B_1(s))y)$ for all $x,y\in\la_0(s)$ and
$s\in(-\delta,\delta)$.
\end{proposition}

\begin{proof} Since $V(s)\subset X_0(s)$ is isotropic, $P_0(\mu(s))$ and $V(s)$ is Lagrangian in $(X_0(s),\w_l(s))$.
We have $X_0(s)=P_0(\mu(s))\oplus V(s)$ and $Q(s)$ is an Hermitian form. For each $x\in\la_0(s)$, we
have
\begin{align*}&x+A_1(s)x=x+B_1(s)x+(A_1(s)-B_1(s))x,\text{ and}\\
&\w_l(s)(x+A_1(s)x,(A_1(s)-B_1(s))x)\\
&\qquad=\w(s)(x+A_1(s)x,(A_1(s)-B_1(s))x)=Q(s)(x,x).
\end{align*}
By Theorem \ref{t:mas-local}, Lemma \ref{l:local-mas-finite} and Proposition
\ref{p:maslov-properties}.b, our results follow.
\end{proof}

We now calculate $Q(s)$.

\begin{lemma}\label{l:cal-Q} Let $(X,\w)$ be a symplectic vector space with Lagrangian subspaces $\la$, $\mu$,
isotropic spaces $\alpha_0$, $V$ and a linear subspace $\la_0$. Assume that
$\dim\alpha_0=\dim\la_0=\dim V<+\infty$. Set $\la_1:=V^{\w}\cap\la$ and $\mu_1:=V^{\w}\cap\mu$.
Let $\alpha_1,\beta_1\subset V^{\w}$ be isotropic subspaces. Assume that
\[
X=\alpha_0\oplus V\oplus\alpha_1\oplus\beta_1=\la_0\oplus V\oplus\la_1\oplus\mu_1.
\]
Assume that $\la=\Graph(A)=\Graph({\wt A})$ and $\mu=\Graph(B)$, where
\begin{align}
\label{e:local-la1}A&=\left(\begin{array}{cc}A_{11} & A_{12} \\A_{21} & A_{22} \\\end{array}\right)\colon \alpha_0\oplus\alpha_1\to V\oplus \beta_1,\\
\label{e:local-la2}\wt A&=\left(\begin{array}{cc}A_1 & 0 \\A_2 & 0 \\\end{array}\right)\colon \la_0\oplus\la_1\to V\oplus \mu_1,\\
\label{e:local-mu1}B&=\left(\begin{array}{cc}B_{11} & B_{12} \\B_{21} & B_{22}
\\\end{array}\right)\colon \alpha_0\oplus\beta_1\to V\oplus\alpha_1.
\end{align}
Then the following holds.
\newline (a) $V^{\w}=V\oplus\alpha_1\oplus\beta_1$ and $\la_1=\ran(A_{12}+I_{\alpha_1}+A_{22})$.
\newline (b) If $\mu=\alpha_0\oplus\beta_1$ and $\la_0=\alpha_0$ hold, we have $\mu_1=\beta_1$, $A_1=A_{11}$
and $A_2=A_{22}$.
\newline (c) Set
\[f:=I_{\alpha_0}+B_{11}+B_{21}\colon\alpha_0\to X,\quad g\colon=B_{12}+B_{22}+I_{\beta_1}\colon \beta_1\to X.\]
Assume that $\la_0=f(\alpha_0)$ and $I_{\beta_1}-A_{22}B_{22}$ is invertible. Then we have
\begin{align}\label{e:aab1}A_1f=&A_{11}-B_{11}+A_{12}B_{21}-(B_{12}-A_{12}B_{22})\\
\nonumber&(I_{\beta_1}-A_{22}B_{22})^{-1}(A_{21}+A_{22}B_{21}),\\
 \label{e:aab2}A_2f=&g(I_{\beta_1}-A_{22}B_{22})^{-1}(A_{21}+A_{22}B_{21}).
\end{align}
\end{lemma}

\begin{proof} (a), (b) by definition.
\newline (c) Let $x\in\alpha_0$. By (a) we have $\ran g=\mu_1$. Set $\wt x:=f(x)$ and $w:=g^{-1}(A_2\wt x)$.
Then we have $\wt x+A_1\wt x+A_2\wt x\in\la$. Since $\la=\Graph(A)$, we have
\[\left(\begin{array}{c}B_{11}x+A_1\wt x+B_{12}w \\w \\\end{array}\right)=\left(\begin{array}{cc}A_{11} & A_{12} \\A_{21} & A_{22} \\
\end{array}\right)\left(\begin{array}{c}x\\B_{21}x+B_{22}w \\\end{array}\right).\]
By direct calculations we get (\ref{e:aab1}) and (\ref{e:aab2}).
\end{proof}

We generalize the notion of crossing forms of \auindex{Robbin,\ J.}\auindex{Salamon,\
D.}\cite{RoSa93} to the case of  $C^1$ varying symplectic structures.

Let $(X,\w)$ be a symplectic Banach space and let $\la=\{\lambda(s)\}_{s\in[0,1]}$ be a $C^1$ curve
of Lagrangian subspaces. Assume that $\la(t)$ is complemented. Let $W$ be a fixed Lagrangian
complement of $\lambda(t)$. The form
\begin{equation}\label{e:crossing1}
Q(\lambda,t):= Q(\lambda,W,t)=\frac{d}{ds}|_{s=t}Q(\la(t),W;\la(s))
\end{equation}
on $\la(t)$ is independent of the choice of $W$, where $Q(\alpha,\beta;\gamma)$ is defined by
Proposition \ref{p:cal-red}.e (see \auindex{Duistermaat,\ J.J.}\cite[(2.3)]{Du76}).

If $X=\alpha\oplus W=\la(s)\oplus W$ for two Lagrangian subspaces $\alpha$ and $W$ and $|s-t|<<1$,
then, by Lemma \ref{l:connectness-local}, there exists a path $A(s)\in\Bb(\alpha,W)$ with
$\la(s)=\{x+A(s)x;x\in\alpha\}$. By definition we have
\begin{equation}\label{e:cal-crossing}\symindex{Q@$Q,Q(s),Q(\la,t)$ intersection form}\subindex{Intersection form}
Q(\la,W,t)(x+A(t)x,y+A(t)y)=\frac{d}{ds}|_{s=t}Q(\alpha,W;\la(s))(x,y).
\end{equation}

\begin{lemma}[Crossing form independence]\label{l:crossing-form-independence} Let $(X,\w(s))$ $s\in(-\e,\e)$, be a $C^1$ path
of symplectic Banach spaces with two $C^1$ families of Lagrangian subspaces $\alpha(s)$,
$\beta(s)$. Assume that $X=\alpha(s)\oplus\beta(s)$. Let $x(s),y(s)\in\alpha(s)$ be two $C^1$
paths. Let $A(s),B(s),C(s)\in\Bb(\alpha(s),\beta(s))$ and $D(s)\in\Bb(\beta(s),\alpha(s))$ are
$C^1$ families of bounded linear maps with $A(0)=B(0)=C(0)=0$. Set $\la(s):=\Graph(A(s))$,
$\mu(s):=\Graph(B(s))$, $\wt\alpha(s):=\Graph(C(s))$ and $\wt\beta(s):=\Graph(D(s))$. As above we set $\lambda:=\{\lambda(s)\}_{s\in[0,1]}$ and $\mu:=\{\mu(s)\}_{s\in[0,1]}$. Then the following holds.
\newline (a) There exists a $\delta\in(0,\e)$ such that for all $s\in(-\delta,\delta)$, we have
\[X=\wt\alpha(s)\oplus\wt\beta(s)=\la(s)\oplus\wt\beta(s)=\mu(s)\oplus\wt\beta(s).\]
\newline (b) For $s\in(-\delta,\delta)$, let $u(s),v(s)\in\beta(s)$ be such that
\[y(s)+C(s)y(s)+w_1(s)\in\la(s),y(s)+C(s)y(s)+w_2(s)\in\mu(s),\]
where $w_1(s):=u(s)+D(s)u(s)$, $w_2(s):=v(s)+D(s)v(s)$. Then we have
\begin{align}\symindex{\Gamma@$\Gamma(\la,\mu,t),\Gamma(\la,\mu,\la_0,V,t)$ crossing form}\subindex{Crossing!crossing form}
\nonumber\Gamma(\la,\mu,0)&(x(0),y(0)):\\
\label{e:crossing-form-definition}&=\frac{d}{ds}|_{s=0}\w(s)(x(s)+C(s)x(s),w_1(s)-w_2(s))\\
\label{e:crossing-form-independence}&=\frac{d}{ds}|_{s=0}\w(0)(x(0),(A(s)-B(s))y(s)).
\end{align}
\newline (c) The form $\Gamma(\la,\mu,0)$
 is an Hermitian form on $\la(0)=\mu(0)=\alpha(0)$ if $\la(s)$ and $\mu(s)$ are Lagrangian subspaces of $(X,\w(s))$.
\end{lemma}

\begin{proof} (a) By the continuity of the given families.
\newline (b) By the definitions we have $y(s)+D(s)u(s)+u(s)+C(s)y(s)\in\la(s)$. Then we have $u(s)+C(s)y(s)=A(s)(y(s)+D(s)u(s))$, and
\[u(s)=(I_{\beta(s)}-A(s)D(s))^{-1}(A(s)-C(s))y(s).\]
Since $\alpha(s)$ and $\beta(s)$ are Lagrangian subspaces, we have
\begin{align}
\nonumber\frac{d}{ds}|_{s=0}&\w(s)(x(s)+C(s)x(s),w_1(s))\\
\nonumber&=\frac{d}{ds}|_{s=0}\w(s)(x(s),u(s))\\
\nonumber&=\frac{d}{ds}|_{s=0}\w(s)(x(s),(A(s)-C(s))y(s))\\
\label{e:crossing-form-independence-la}&=\frac{d}{ds}|_{s=0}\w(0)(x(0),(A(s)-C(s))y(s)).
\end{align}
Applying (\ref{e:crossing-form-independence-la}) for $\mu(s)$, we obtain
(\ref{e:crossing-form-independence}).
\newline (c) If $\la(s)$ and $\mu(s)$ are Lagrangian subspaces of $(X,\w(s))$,
then the two forms $\w(s)(x(s),A(s)y(s))$ and $\w(s)(x(s),B(s)y(s))$ are Hermitian and the form
$\Gamma(\la,\mu,0)$ becomes Hermitian.
\end{proof}

\begin{lemma}[Crossing form calculation]\label{l:crossing-form}
\subindex{Crossing!crossing form!calculation}
We use the data of Theorem \ref{t:mas-local}. Assume that $\w(s)$, $V(s)$, $\la(s)$, $\mu(s)$ with
$s\in(-\delta,\delta)$ are $C^1$ families and $\la_0(0)=\la(0)\cap\mu(0)$. Then the following
holds.
\newline (a) The families $V(s)^{\w(s)}$, $\la_1(s):=V(s)^{\w(s)}\cap\la(s)$,
$\mu_1(s):=V(s)^{\w(s)}\cap\mu(s)$ are $C^1$ families.
\newline (b) The form
\begin{equation}\label{e:crossing3}\symindex{\Gamma@$\Gamma(\la,\mu,t),\Gamma(\la,\mu,\la_0,V,t)$ crossing form}\subindex{Crossing!crossing form}
\Gamma(\la,\mu,\la_0,V,0)(x(0),y(0)):=\frac{d}{ds}|_{s=0}\omega(s)(x(s),(A_1(s)-B_1(s))y(s))
\end{equation}
on $\la_0(0)$ is Hermitian, where $x(s),y(s)\in\la_0(s)$ are two $C^1$ paths.
\newline (c) The form $\Gamma(\la,\mu,0):=\Gamma(\la,\mu,\la_0,V,0)$ does not depend on the
choices of paths $\la_0$ and $V$. It coincides with the form $\Gamma(\la,\mu,0)$ defined above in \eqref{e:crossing-form-definition} if $\la(0)=\mu(0)$. In particular,
\begin{itemize}
\item[(i)] if $V(s)$ is a $C^1$ isotropic path and $Q(s)$ is defined as in Proposition \ref{p:mas-local1}, we have
\begin{equation}\label{e:crossing4}
\Gamma(\lambda,\mu,0)(x(0),y(0))=\frac{d}{ds}|_{s=0}Q(s)(x(s),y(s));
\end{equation}
\item[(ii)] for $A_{11}(s)$ and $B_{11}(s)$ as defined in Lemma \ref{l:cal-Q}, we have
\begin{equation}\label{e:crossing5}
\Gamma(\lambda,\mu,0)(x(0),y(0))=\frac{d}{ds}|_{s=0}\w(s)(x(s),(A_{11}(s)-B_{11}(s))y(s));
\end{equation}
\item[(iii)] if $\w(s)\equiv\w$ is fixed, we have
\begin{equation}\label{e:crossing2}
\Gamma(\lambda,\mu,0)=(Q(\lambda,0)-Q(\mu,0))|_{\lambda(0)\cap\mu(0)}.
\end{equation}
\end{itemize}
\end{lemma}

\begin{proof} (a) By Lemma \ref{l:finite-bot-ck}.
\newline (b) By (a), the families $A_1(s),A_2(s),B_1(s),B_2(s)$ are of class $C^1$.
Note that $A_1(0)=A_2(0)=B_1(0)=B_2(0)=0$. Then we have
\begin{align*}
0=&\frac{d}{ds}|_{s=0}\w_l(s)(x(s)+B_1(s)x(s),y(s)+B_1(s)y(s)))\\
=&\frac{d}{ds}|_{s=0}\w(s)((B_1(s)-A_1(s))x(s),y(s))
\\&+\frac{d}{ds}|_{s=0}\w(s)(x(s),(B_1(s)-A_1(s))y(s)).
\end{align*}
Since $\w(s)$ is symplectic, we get our result.
\newline (c) Fix a $C^1$ path $\wt\la_0(s)\subset\mu(s)$ with $\wt\la_0(0)=\la_0(0)$.
Consider the $C^1$ symplectic space $(V(s)+\mu(s))/\mu_1(s)$. By Proposition \ref{p:cal-red} we have
\begin{equation}\label{e:crossing-form-invariant0}
\Gamma(\la,\mu,\la_0,V,0)=\Gamma(P_0(\la),P_0(\mu),0;\w_r)=(T_r(0)^{-1})^*\Gamma(R_{V+\mu}(\la),R_{V+\mu}(\mu),0),
\end{equation}
where $T_r(0)$ denotes the isomorphism $T_r(0)\colon  (V(0) + \mu(0))/ \mu_1(0) \to X_0$ of Theorem \ref{t:mas-local}; similarly, go back to Definition \ref{d:maslov-banach} for checking the definitions of $P_0(\la),P_0(\mu)$ and to Proposition \ref{p:cal-red} for $R_{V+\mu}(\la)$, $R_{V+\mu}(\mu)$, $T_r$, $\w_r$.
By Lemma \ref{l:crossing-form-independence} and (\ref{e:crossing-form-invariant0}), we have
\begin{equation}\label{e:crossing-form-invariant1}
\Gamma(\la,\mu,\la_0,V,0)=\Gamma(\la,\mu,\wt\la_0,V,0).
\end{equation}
Take a $C^1$ path of Lagrangian subspaces $\wt V(s)$ of symplectic subspaces $\wt
X_0(s):=\wt\la_0(s)+V(s)$. Denote by $\tilde w_l$ the symplecic form defined by Proposition \ref{p:cal-red}.d for $\tilde\la_0$ and $\tilde V$. Then we have $\tilde\w_r(s)=\w(s)|_{\wt X_0(s)}$. By Lemma
\ref{l:crossing-form-independence} and (\ref{e:crossing-form-invariant0}), we have
\begin{equation}\label{e:crossing-form-invariant2}
\Gamma(\la,\mu,\wt\la_0,V,0)=\Gamma(\la,\mu,\wt\la_0,\wt V,0).
\end{equation}
Fix a $C^1$ isotropic path $\overline V(s)$ with $X(0)=\overline V(0)\oplus(\la(0)+\mu(0))$. Fix
$C^1$ paths $\wt x(s),\wt y(s)\in\wt\la_0(s)$. Let $\wt A_1(s),\wt A_2(s)$ and $\overline
A_1(s),\overline A_2(s)$ be $C^1$ paths defined by Theorem \ref{t:mas-local} for
$(\la,\mu,\wt\la_0,\wt V)$ and $(\la,\mu,\wt\la_0,\overline V)$. Since $\wt x(s)+\wt A_1(s)\wt
x(s)+\wt A_2(s)\wt x(s), \wt y(s)+\overline A_1(s)\wt y(s)+\overline A_2(s)\wt y(s)\in\la(s)$, we
have
\begin{align*}
0=&\frac{d}{ds}|_{s=0}\w(s)(\wt x(s)+\wt A_1(s)\wt x(s)+\wt A_2(s)\wt x(s),\wt y(s)+\overline A_1(s)\wt y(s)+\overline A_2(s)\wt y(s))\\
=&\frac{d}{ds}|_{s=0}\w(s)(\wt x(s), \wt y(s)+\overline A_1(s)\wt y(s)+\overline A_2(s)\wt y(s))\\
&+\frac{d}{ds}|_{s=0}\w(s)(\wt x(s)+\wt A_1(s)\wt x(s)+\wt A_2(s)\wt x(s), \wt y(s))\\
=&\frac{d}{ds}|_{s=0}\left(\w(s)(\wt x(s),\overline A_1(s)\wt y(s))+\w(s)(\wt A_1(s)\wt x(s),\wt y(s))\right)\\
=&\Gamma(\la,\mu,\wt\la_0,\overline V,0)(x(0),y(0))-\Gamma(\la,\mu,\wt\la_0,\wt V,0)(x(0),y(0)).
\end{align*}
By (\ref{e:crossing-form-invariant1}) and (\ref{e:crossing-form-invariant2}), we obtain
\begin{equation}\label{e:crossing-form-invariant3}
\Gamma(\la,\mu,\la_0,V,0)=\Gamma(\la,\mu,\wt\la_0,\overline V,0).
\end{equation}
For the special cases, (i) is clear, (ii) by taking $\la_0(s)=\alpha_0(s)$, and (iii) by taking
$\la_0(s)\equiv\la_0(0)$ and $V(s)\equiv V(0)$ to be an isotropic subspace.
\end{proof}

Let $p\colon \mathbb{X}\to[0,1]$ be a $C^1$ Banach bundle with $p^{-1}(s)=X(s)$. Let
$\{(\lambda(s),\mu(s))\}$, $0\le s\le 1$ be a curve of Fredholm pairs of Lagrangian subspaces of
$C^1$ family $(X(s),\w(s))$ of index $0$. By Corollary \ref{c:complementary-lagrangian},
$\lambda(s)$ and $\mu(s)$ are complemented. For $t\in[0,1]$, the \subindex{Crossing!crossing
form}{\em crossing form} $\Gamma(\lambda,\mu,t)$ is an Hermitian form on $\lambda(t)\cap\mu(t)$
defined by Lemma \ref{l:crossing-form}.

A \subindex{Crossing}{\em crossing} is a time $t\in[0,1]$ such that $\lambda(t)\cap\mu(t)\ne\{0\}$.
A crossing is called \subindex{Crossing!regular}{\em regular} if $\Gamma(\lambda,\mu,t)$ is
non-degenerate. It is called \subindex{Crossing!simple}{\em simple} if it is regular and
$\lambda(t)\cap\mu(t)$ is one-dimensional. As before, we shall denote by \symindex{m@$m^{\pm}(Q)$,
$m^0(Q)$ positive (negative) Morse index and nullity of Hermitian form $Q$}$m^+,m^-,m^0$
the positive Morse index, the negative Morse index and the nullity of the form
respectively.

Now we give a method for using the crossing form to calculate Maslov indices (see \auindex{Robbin,\
J.}\auindex{Salamon,\ D.}\cite{RoSa93} for the fixed finite-dimensional symplectic vector space
case, \auindex{Boo{\ss}--Bavnbek,\ B.}\auindex{Furutani,\ K.}\cite[Theorem 2.1]{BoFu98} and
\auindex{Zhu,\ C.}\cite[Proposition 4.1]{Zh05} for the fixed strong symplectic Hilbert space case).

\begin{proposition}[Calculation of the Maslov index] \label{p:cal-mas}
\subindex{Calculation of the Maslov index}
Let $(X(s),\omega(s))$ be a $C^1$ family of symplectic Banach space
and $\{(\lambda(s),\mu(s))\}$, $0\le s\le 1$ be a $C^1$ curve of Fredholm pairs of Lagrangian
subspaces of $X$ of index $0$ with only regular crossings. Then we have
%
\begin{equation} \label{e:cal-mas}\symindex{sig@$\sign$ signature of Hermitian form}
\Mas\{\lambda,\mu\}=m^+(\Gamma(\lambda,\mu,0))
-m^-(\Gamma(\lambda,\mu,1))+\sum_{0<t<1}\sign(\Gamma(\lambda,\mu,t)).
\end{equation}
\end{proposition}

\begin{proof} For each crossing $t\in[0,1]$, we consider the path $(\la(s+t),\mu(s+t))$ for $|s|<<1$.
By Proposition \ref{p:reduction-prop}.c, we can take an isotropic $V$ with
$X=V\oplus(\la(t)+\mu(t))$. Then the assumptions of Lemma \ref{l:cal-Q} can be satisfied. Let
$Q(s)$ be defined in Proposition \ref{p:mas-local1}. By Lemma \ref{l:crossing-form}.c we have
\[
\frac{d}{ds}|_{s=0}Q(s)\ =\ \Gamma(\lambda,\mu,t).
\]
Since the crossing $t$ is regular, for $0<|s|<<1$, by (\ref{e:crossing4}) $Q(s)$ and
$s\Gamma(\lambda,\mu,t)$ are non-degenerate and they have the same positive (negative) Morse index.
Thus the set of crossings is discrete (and then finite, for $[0,1]$ is compact). By Proposition
\ref{p:mas-local1} and Proposition \ref{p:maslov-properties}.b, our results hold.
\end{proof}

We recall (see \auindex{Bott,\ R.}\cite[Definition 3.1]{Bott:1956} for the finite-dimensional
case).

\begin{definition}\label{d:positive-curve}
Let $(X,\w)$ be a symplectic Banach space and let $\la:=\{\la(s)\}_{0\in [0,1]}$, be a $C^1$ curve
of complemented Lagrangian subspaces. We call the \subindex{Lagrangian subspaces!(semi-)positive
path}\subindex{Path in Banach bundle!(semi-)positive path of Lagrangian subspaces}curve $\la$ {\em
(semi-)positive} at $t\in[0,1]$, if $Q(\lambda,t)$ is positive definite, respectively semi-positive
definite. The curve $\{\la(s)\}$ is called {\em (semi-)positive} if it is (semi-)positive at all
$t\in[0,1]$, respectively.
\end{definition}

\begin{lemma}\label{l:positive} Let $(X,\w)$ be a symplectic Banach space and
$\{\la(s)\}_{0\in [0,1]}$ a $C^1$ curve with a Lagrangian complement $W$. Then $\{\la(s)\}$ is
(semi-) positive if and only if the path of Hermitian forms $Q(\la(0)),W;\la(s))$ is strictly
increasing (respectively, increasing).
\end{lemma}

\begin{proof} By (\ref{e:cal-crossing}).
\end{proof}

\begin{lemma}[Morse-positive path]\label{l:morse-positive-path} Let $X$ be a finite-dimensional Hilbert space
and%
\[
A\colon (-\e,\e)\too\Bb^{\sa}(X)%
\]
a family of self-adjoint operators. Assume that $A(s_1)\le A(s_2)$ for all $-\e<s_1\le s_2<\e$.
Then the following holds.
\newline (a) There exists a $\delta\in(-\e,\e)$ such that the functions
$m^{\pm,0}(A(s))$ are constant for $s\in(-\delta,0)$ or $s\in(0,\delta)$.
\newline (b) Assume that $A(s)$ is continuous at $s$. Then we have
\begin{equation}
\label{e:minus-side} m^+(A(s))=m^+(A(0)),m^-(A(s))-m^-(A(0))=m^0(A(0))-m^0(A(s))
\end{equation}
for $s\in(-\delta,0)$, and
\begin{equation}\label{e:positive-side}
m^-(A(s))=m^-(A(0)),m^+(A(s))-m^+(A(0))=m^0(A(0))-m^0(A(s))
\end{equation}
for $s\in(0,\delta)$.
\end{lemma}

\begin{proof} (a) For all linear subspace $V$ of $X$, and $-\e<s_1\le s_2<\e$,
we have that $A(s_1)|_V>0$ implies $A(s_2)>0$. So $m^+(A(s))$ is an increasing function. Similarly,
$m^-(A(s))$ is a decreasing function. Since the two functions are bounded integer valued, they have
finitely many discontinuous points. Since $m^0(A(s))=\dim X-m^+(A(s))-m^-(A(s))$, the same result
holds for $m^0(A(s))$. So we obtain (a).
\newline (b) Since $A(s)$ is continuous at $s$, we have $m^{\pm}(A(0))\le m^{\pm}(A(s))$. Note that $m^+(A(s))$ is
an increasing function and $m^-(A(s))$ is a decreasing function. Then the first equalities of
(\ref{e:minus-side}) and (\ref{e:positive-side}) follow. The second equalities of
(\ref{e:minus-side}) and (\ref{e:positive-side}) follow from the first ones.
\end{proof}

\begin{proposition}[Counting rule for (semi-)positive paths] \label{p:cal-mas-positive} Let $(X,\omega)$ be a symplectic Banach space
and let $\{(\lambda(s),\mu)\}$, $0\le s\le 1$ be a $C^1$ curve of Fredholm pairs of Lagrangian
subspaces of $X$ of index $0$ with a semi-positive path $\la$ and a constant path $\mu$. Then
$\dim(\la(s)\cap\mu)$ is locally constant except for finitely many points $s\in[0,1]$, and we have
\begin{equation} \label{e:cal-mas-positive}
\Mas\{\lambda,\mu\}=\sum_{0\le t< 1}(\dim(\lambda(t)\cap\mu)-\lim_{s\to
t^+}\dim(\lambda(s)\cap\mu)).
\end{equation}
\end{proposition}

\begin{proof} Let $t\in[0,1]$ and consider the path $(\la(s+t),\mu)$ for $|s|<<1$. By
Proposition \ref{p:reduction-prop}, there exists an isotropic $V$ such that
$X=V\oplus(\la(s)+\mu)$. Set $\la_1(t+s)=V^{\w}\cap\la(t+s)$, $\mu_1:=V^{\w}\cap\mu$ and
$W:=V+\la_1$. Then $W$ is a Lagrangian subspace and $X=\la(t+s)\oplus W$. Let $Q(s)$ be defined by
Proposition \ref{p:mas-local1} in our case. Then we have
\[
Q(s)=Q(\la(t),W;\la(t+s))|_{\la(t)\cap\mu}.
\]
By Lemma \ref{l:positive}, the family of forms $Q(s)$ is increasing. By Proposition
\ref{p:mas-local1} and Lemma \ref{l:morse-positive-path}, we obtain our results.
\end{proof}

The following theorem axiomatizes the well-known technique of identifying the Maslov indices of two
one-parameter curves of Fredholm pairs of Lagrangian subspaces by suitable two-parameter extensions
to topological trivial families with vanishing {\em top} edge and calculable {\em side} edges (see
\auindex{Boo{\ss}--Bavnbek,\ B.}\auindex{Furutani,\ K.}\cite[Section 5]{BoFu98}).

\begin{theorem}[Comparing two-parameter families]\label{t:comparision}
Let $p\colon \mathbb{X}\to [0,1]$ and $\wt p\colon \wt{\mathbb{X}}\to[0,1]$ be two Banach bundles
with $X(s):=p^{-1}(s)$, $\wt X(s):=\wt p^{-1}(s)$ for each $s\in[0,1]$. Let
$\{\w(s)\},\{\wt\w(s)\}$ be paths of symplectic forms for $X(s)$, respectively, $\wt X(s)$, $0\le
s\le 1$. For $0\le a\le \delta$, $\delta>0$, we are given continuous two-parameter families
\begin{align}\label{e:two-parameter}
&\{(\la(s,a),\mu(s))\in \Ll^2(X(s),\w(s))\}\text{ and}\\
\nonumber&\{(\wt\la(s,a),\wt\mu(s))\in\Ll^2(\wt X(s),\wt\w(s))\}.
\end{align}
We assume that
\begin{align}
\label{e:comparision0} &(\la(s,0),\mu(s))\in \Ff\Ll_0(X(s))\text{ and }
(\wt\la(s,0),\wt\mu(s))\in\Ff\Ll_0(\wt X(s)),\\
\label{e:comparisionI} &\{\la(s,a)\} \text{ differentiable in $a$ and semi-positive for fixed $s$},\\
\label{e:comparisionII} &\{\wt\la(s,a)\} \text{ differentiable in $a$ and positive for fixed $s$},\\
\label{e:comparisionIII} &\dim(\la(s,a)\cap\mu(s)) - \dim(\wt\la(s,a)\cap\wt\mu(s)) = c(a).
\end{align}
Then we have
\begin{equation}\label{e:comparisionIV}
\Mas\{\la(s,0),\mu(s);\w(s)\} = \Mas\{\wt\la(s,0),\wt\mu(s);\wt\w(s)\} .
\end{equation}
\end{theorem}

\begin{proof} Since $[0,1]$ is compact, after making $\delta$ smaller,
we may assume that the given two families \eqref{e:two-parameter} are families of Fredholm pairs of
index $0$.

Fix $t\in[0,1]$. Since $\wt\la(t,a)$ is differentiable in $a$ and positive, by Proposition
\ref{p:cal-mas} there exists a $\delta(t)\in(0,\delta)$ such that $\wt\la(t,a)\cap\wt\mu(t)=\{0\}$
for $a\in(0,\delta(t)]$. From the continuity of our family $(\wt\la(s,a),\wt\mu(s))$, there exists
an $\e(t)>0$ such that
\begin{equation}\label{e:cap-wt}\wt\la(s,\delta(t))\cap\wt\mu(s)=\{0\}\;\text{for
$s\in(t-\e(t),t+\e(t))\cap[0,1]$}.\end{equation} By compactness of $[0,1]$, there exists a
partition $0=s_0<s_1<\ldots<s_n=1$ of $[0,1]$ and $t_1,\ldots t_n\in[0,1]$ with
$s_{k-1},s_k\in(t_k-\e(t_k),t_k+\e(t_k))$ for $k=1,\ldots,n$.

We now prove the formula \eqref{e:comparisionIV} for a small interval $[s_{k-1},s_k]$. We consider
the two-parameter families \eqref{e:two-parameter} for $s\in[s_{k-1},s_k]$ and
$a\in[0,\delta(t_k)]$. Because of the homotopy invariance of Maslov index, both integers
$\Mas\{\la(s,a),\mu(s)\}$ and $\Mas\{\wt\la(s,a),\wt\mu(s)\}$ must vanish for the boundary loop
going counter clockwise around the rectangular domain from the corner point $(s_{k-1},0)$ via the
corner points $(s_k,0)$, $(s_k,\epsilon(t_k))$, and $(s_{k-1},\epsilon(t_k))$ back to
$(s_{k-1},0)$.

Moreover, by \eqref{e:comparisionIII} and \eqref{e:cap-wt}, for all $s\in[s_{k-1},s_k]$ we have
\[\dim(\la(s,\delta(t_k))\cap\mu(s))=c(\delta(t_k))\tand \wt\la(s,\delta(t_k))\cap\wt\mu(s)=\{0\}.\]
Hence, our two Maslov indices must vanish on the top segment of our box.

Finally, by Proposition \ref{p:cal-mas-positive} and \eqref{e:comparisionIII} we have
\begin{align*}&\Mas\{\la(s_{k-1},a),\mu(s_{k-1});a\in[0,\delta(t_k)]\}\\
&\quad-\Mas\{\wt\la(s_{k-1},a),\wt\m(s_{k-1});a\in[0,\delta(t_k)]\}\\
&=\sum_{0\le a< \delta(t_k)}(\dim(\la(s_{k-1},a)\cap\mu(s_{k-1})-\lim_{b\to a^+}\dim(\la(s_{k-1},b)\cap\mu(s_{k-1}))\\
&\quad-\sum_{0\le a< \delta(t_k)}\dim(\wt\la(s_{k-1},a)\cap\wt\mu(s_{k-1}))\\
&=\sum_{0\le a< \delta(t_k)}(c(a)-\lim_{b\to a^+}c(b))\\
&=\sum_{0\le a< \delta(t_k)}(\dim(\la(s_k,a)\cap\mu(s_k))-\lim_{b\to a^+}\dim(\la(s_k,b)\cap\mu(s_k)))\\
&\quad-\sum_{0\le a< \delta(t_k)}\dim(\wt\la(s_k,a)\cap\wt\mu(s_k))\\
&=\Mas\{\la(s_k,a),\mu(s_k);a\in[0,\delta(t_k)]\}\\
&\quad-\Mas\{\wt\la(s_k,a),\wt\mu(s_k);a\in[0,\delta(t_k)]\}.
\end{align*}

By additivity under catenation, the formula \eqref{e:comparisionIV} holds for the small interval
$[s_{k-1},s_k]$. Again by additivity under catenation, the formula \eqref{e:comparisionIV} holds
for the whole interval $[0,1]$.
\end{proof}

\section{Invariance of the Maslov index under symplectic operations}\label{ss:sympl-invariance}
In this section we show that the Maslov index is invariant under symplectic reduction and
symplectic embedding under natural conditions.

The first theorem generalizes our previous set-up of Choices and Notations
\ref{cn:formal-definition} and our previous Definition \ref{d:maslov-banach}. We begin with a
lemma. It transfers the purely algebraic arguments of Definition \ref{d:sympl-red} to the case of
Banach spaces and gives a sufficient condition for the symplectic reduction becoming a symplectic
Banach space.

\begin{lemma}\label{l:sym-red-banach} Let $(X,\w)$ be a symplectic Banach space and $W\subset X$ a co-isotropic subspace.
Assume that $W$ is a Banach space (not necessarily induced by the norm on $X$) such that the
injection $j\colon W\into X$ is bounded. Then the symplectic reduction
\symindex{W@$W/W^{\omega}$ reduced symplectic space} \symindex{\omega tilde@$\tilde\omega$ induced
form on symplectic reduction}\subindex{Symplectic reduction!in symplectic Banach space}
$(W/W^{\w},\wt\w)$ is a symplectic Banach space.
\end{lemma}

\begin{proof} Note that $W^{\w}$ is closed in $X$. Since $j$ is continuous, $W^{\w}$ is closed in $W$. So
the quotient $W/W^{\w}$ with the norm induced by $W$ is a Banach space. Since $j$ is bounded, the
induced form $\wt \w$ is bounded on $W/W^{\w}$. Then $(W/W^{\w})^{\wt\w}=W^{\w}/W^{\w}=\{0\}$. So
$\wt\w$ is non-degenerate. Since $W$ is a Banach space, $(W/W^{\w},\wt\w)$ is a symplectic Banach
space.
\end{proof}

In the following we shall parameterize the symplectic reduction.

\begin{ass}\label{a:mas-red} We make the following assumptions.
\newline (1)
\symindex{W@$\mathbb{W},\mathbb{W}_0,\tilde{\mathbb{W}}$ Banach subbundles for parametrization of
symplectic reduction}
$q_0\colon \mathbb{W}_0\to[0,1]$, $q\colon \mathbb{W}\to [0,1]$, and $p\colon \mathbb{X}\to [0,1]$ with fibers $q_0^{-1}(s):=W_0(s)$, $q^{-1}(s):=W(s)$,
and $p^{-1}(s):=X(s)$ for each $s\in[0,1]$, respectively. Assume that we
have Banach subbundle maps%
\[
\mathbb{W}_0\ \into\ \mathbb{W}, \quad \mathbb{W}\
\into\ \mathbb{X}.
\]
\newline (2) We are given a path of symplectic structures $\{\w(s)\}_{s\in[0,1]}$ on $X(s)$.
\newline (3) We have a path $\{(\lambda(s),\mu(s))\}_{s\in[0,1]}$ of Fredholm pairs of Lagrangian
subspaces of $(X(s),\omega(s))$, all of index $0$.
\newline (4) We assume that%
\begin{itemize}
\item  $W_0(s)=W(s)^{\w(s)}$,
\item $W(s)^{\w(s)}\subset\la(s)\subset W(s)$,
\item $\dim (W(s)^{\w(s)}\cap\mu(s))=k$ for each $s\in[0,1]$, and
\item $\{W(s)+\mu(s)\}_{s\in[0,1]}$ is a path of closed subspaces of $X(s)$ (it holds automatically if
$W(s)$ is closed in $X(s)$).
\end{itemize}
\end{ass}

\begin{theorem}[Invariance under symplectic reduction]\label{t:mas-red}
\subindex{Maslov index!invariance under symplectic reduction}\subindex{Symplectic
reduction!invariance of the Maslov index}
Under Assumption \ref{a:mas-red}, we have the following:
\newline (a) For each $s\in[0,1]$, we have $\dim X(s)/(W(s)+\mu(s))=k$ and
$W(s)+\mu(s)=W(s)^{\w(s)\w(s)}+\mu(s)$.
\newline (b) The family
\[
\bigl\{\bigl(R_W(s)^{\w(s)}(\la(s)),R_W(s)^{\w(s)}(\mu(s))\bigr)\bigr\}_{s\in[0,1]}
\]
is a path of Fredholm pairs of Lagrangian subspaces of
\[(W(s)/W(s)^{\w(s)},\wt\omega(s))\]
of index $0$.
\newline (c) We have
\begin{equation}\label{e:mas-red}
\Mas_{\pm}\{\la(s),\mu(s)\}=\Mas_{\pm}\{R_{W(s)}^{\w(s)}(\la(s)),R_{W(s)}^{\w(s)}(\mu(s))\}.
\end{equation}
\end{theorem}

\begin{proof} We divide the proof into three steps.

Step 1. By Proposition \ref{p:red-fredholm}, (a) and (b) hold.

Since the map $W(s)\into X(s)$ is continuous, $\la(s)$, $\mu(s)\cap W(s)$, and $W(s)^{\w(s)}$ are closed in
$W(s)$ for each $s\in[0,1]$. Since $R_W(s)^{\w(s)}(\mu(s))$ is closed in $W(s)/W(s)^{\w(s)}$,
$\mu(s)\cap W(s)+W(s)^{\w(s)}$ is closed in $W(s)$ for each $s\in[0,1]$. By Corollary
\ref{c:continuous-lift}, $\{\la(s)\}_{s\in[0,1]}$ is a path in $\Ss(W(s))$. Here we use the symbol
\symindex{S@$\Ss(X)$ set of closed linear subspaces in Banach space $X$}$\Ss(X)$ to denote the set
of all closed subspaces of a Banach space $X$, equipped with the gap topology (see our Appendix
\ref{ss:gap-topology}). By Lemma \ref{l:quotient-delta}, $\{R_W(s)^{\w(s)}(\la(s))\}_{s\in[0,1]}$
is a path in $\Ss(W(s)/W(s)^{\w(s)})$.

Since $W(s)+\mu(s)$ is a finite-dimensional extensions of the closed subspace
$\la(s)+\mu(s)\subset X(s)$, by Proposition \ref{p:finite-extension} they are closed in $X(s)$.

By Corollary \ref{c:finite-close-to},
\[
\{\mu(s)\cap W(s)\}_{s\in[0,1]}, \quad \{\mu(s)\cap W(s)^{\w(s)}\}_{s\in[0,1]}, \tand \{\mu(s)\cap
W(s)+W(s)^{\w}(s)\}_{s\in[0,1]}%
\]
are paths of closed subspaces of $W(s)$. By Lemma \ref{l:quotient-delta},
$\{R_W(s)^{\w(s)}(\mu(s))\}_{s\in[0,1]}$ is a path in $\Ss(W(s)/W(s)^{\w(s)})$. Therefore the
Maslov index on the right hand side of (\ref{e:mas-red}) is well-defined.

Step 2. Reduce to the case of $W(s)+\mu(s)=X(s)$.

Since $W(s)+\mu(s)$, $s\in[0,1]$ is a path of closed subspaces of $X(s)$ of finite codimension, we
have $W(s)+\mu(s)\in\Ss^c(X(s))$. Recall from Remark \ref{r:complemented} that we denote the space
of complemented subspaces of a Banach space $X$ by $\Ss^c(X)$. By Lemma \ref{l:ck-complemented}
(see also \auindex{Neubauer,\ G.}\cite[Lemma 0.2]{Ne68}), $\bigcup_{s\in[0,1]}(W(s)+\mu(s))$ is a
Banach bundle over $[0,1]$, and there exists a finite-dimensional Banach subbundle
$\bigcup_{s\in[0,1]}V(s)$ of $\mathbb{X}$ such that $V(s)\oplus(W(s)+\mu(s))=X(s)$.

We use the notations from Corollary \ref{c:inner-red}. Set
$\mathbb{X}_j:=\bigcup_{s\in[0,1]}X_j(s)$, $j=0,1$. Then
$\{X_0(s):=W(s)^{\w(s)}\cap\mu(s)\oplus V(s)\}_{s\in[0,1]}$ is a path of $\Ss(X(s))$. By Proposition \ref{p:reduction-prop}, we have $X(s)=X_0(s)\oplus X_1(s)$,
$V(s)^{\w(s)}+W(s)^{\w(s)}\cap\mu(s)=X$, and
\[X_1(s):=V(s)^{\w(s)}\cap W(s)+V(s)^{\w(s)}\cap\mu(s)=V(s)^{\w(s)}\cap(W(s)+\mu(s)).\]
Then $W(s)^{\w(s)}\cap\mu(s)$, $X_0(s)$ and $X_1(s)$ are paths of $\Ss^c(X(s))$. By Lemma
\ref{l:ck-complemented} (see also \auindex{Neubauer,\ G.}\cite[Lemma 0.2]{Ne68}), $\mathbb{X}_0$ and $\mathbb{X}_1$ are
Banach subbundles of $\mathbb{X}$. Set $W_{01}(s):=W(s)^{\w(s)}\cap V(s)^{\w(s)}$,
$W_1(s):=V(s)^{\w(s)}\cap W(s)$, and $\mathbb{W}_j(s):=\bigcup_{s\in[0,1]}W_j(s)$ for $j=01,1$. By Lemma \ref{l:w-quotient}, we have
\[\dim W_0(s)/W_{01}(s)=\dim W(s)/W_1(s)=\dim V(s)=k.\]
By Lemma \ref{l:finite-bot-ck}, $\mathbb{W}_{01}$ is a Banach subbundle of $\mathbb{W}_0$,
$\mathbb{W}_1$ is a Banach subbundle of $\mathbb{W}$, and $\mathbb{W}_{01}$ is a Banach subbundle
of $\mathbb{W}_1$. Then we can replace $X_1(s)$ for $X(s)$.

Set $\la_0(s)=\mu_0(s):=W(s)^{\w(s)}\cap\mu(s)$, $\la_1(s):=\la(s)\cap X_1(s)$ and
$\mu_1(s):=\mu(s)\cap X_1(s)$. By Proposition \ref{p:maslov-properties}, for a local path
$s\in[s_1,s_2]\subset(t-\delta(t),t+\delta(t))$ we have
\begin{align*}\Mas_{\pm}\{\la(s),\mu(s)\}&=\Mas_{\pm}\{\la_0(s),\mu_0(s)\}+\Mas_{\pm}\{\la_1(s),\mu_1(s)\}\\
&=\Mas_{\pm}\{\la_1(s),\mu_1(s)\}.
\end{align*}
Then our result follows from the compactness of $[0,1]$ and Definition \ref{d:maslov-banach}.

Step 3. The case of $W(s)+\mu(s)=X(s)$.

Fix $t\in[0,1]$. Let $V_1(t)\subset W(t)$ be a linear subspace such that
$X(s)=V_1(t)\oplus(\la(t)+\mu(t))$. Let $L(t,s)\colon W(t)\to W(s)$ be the local frame of the
bundle $\mathbb{W}$. Set $V_1(t,s):=L(t,s)V_1(t)\subset W(s)$. By Lemma \ref{l:red-transitive} and
Theorem \ref{t:mas-local}, for a local path $s\in[s_1,s_2]\subset(t-\delta(t),t+\delta(t))$ we have
\begin{align*}\Mas_{\pm}\{\la(s),\mu(s)\}&=\Mas_{\pm}\{R_{V_1(t,s)+\la(s)}(\la(s)),R_{V_1(t,s)+\la(s)}(\mu(s))\}\\
&=\Mas_{\pm}\{R_{W(s)}(\la(s)),R_{W(s)}(\mu(s))\}.
\end{align*}
Then our result follows from the compactness of $[0,1]$ and Definition \ref{d:maslov-banach}.
\end{proof}

We have the following side results:

\begin{corollary}\label{c:closed-W-reduction} The equation (\ref{e:mas-red}) holds if we replace Assumption \ref{a:mas-red}.1 of Theorem \ref{t:mas-red} by
a new Assumption (5):
\newline (5) $\{W(s)\}_{s\in[0,1]}$ is a path in
$\Ss^c(X(s))$ and $\{W(s)^{\w(s)}\}_{s\in[0,1]}$ a path in $\Ss^c(W(s))$.
\end{corollary}

\begin{proof} Set $\mathbb{W}:=\bigcup_{s\in[0,1]}W(s)$ and $\mathbb{W}_0:=\bigcup_{s\in[0,1]}W(s)^{\w(s)}$. By Lemma \ref{l:ck-complemented}
(see also \auindex{Neubauer,\ G.}\cite[Lemma 0.2]{Ne68}), $\mathbb{W}$ is a Banach subbundle of
$\mathbb{X}$, and $\mathbb{W}_0$ is a Banach subbundle of $\mathbb{W}$. By Theorem \ref{t:mas-red},
our result follows.
\end{proof}

\begin{corollary}\label{c:finite-dimension-reduction} The equation (\ref{e:mas-red})
holds for $W(s)=V(s)\oplus\la(s)$ if the path $\{V(s)\}_{s\in[0,1]}$ is a path of finite-dimensional linear
subspaces of $X(s)$ and it holds that
\begin{equation}\label{e:sum-condition}
X(s)=V(s)+\la(s)+\mu(s).
\end{equation}
\end{corollary}

\begin{proof} Since $V(s)\cap\la(s)=\{0\}$, by Lemma \ref{l:sum-whole-space}
we have $V(s)^{\w(s)}+\la=X(s)$. Note that $X(s)=W(s)+\mu(s)$ and $V(s)\cap\la(s)=\{0\}$. By Lemma
\ref{l:w-quotient}, we have%
\[
\dim W(s)/\la(s)\ =\ \dim \la(s)/W(s)^{\w(s)}\ =\ \dim V(s).%
\]
By Proposition \ref{p:finite-extension}, $W(s)\in\Ss^c(X)$. Clearly, $W(s)^{\w(s)}\in\Ss^c(W(s))$
since it is closed and of finite codimension. Since $\{V(s)^{\w(s)}\}_{s\in[0,1]}$ and
$\{\la(s)\}_{s\in[0,1]}$ are paths and $V(s)^{\w(s)}+\la(s)=X$, by Proposition \ref{p:close-to},
$\{V(s)^{\w(s)}\cap\la(s)\}_{s\in[0,1]}$ is also a path. By Theorem \ref{t:mas-red}, our result
follows.
\end{proof}

\begin{rem}\label{r:multiform-sympl-reduction}
In Section \ref{ss:symplectic-red1} we established the invariance and transitivity of symplectic
reduction in purely algebraic terms. That was more general - and simpler than our present situation
in Section \ref{ss:maslov-definition} and in this section. Here topological considerations come
into play.
\newline a) Proposition \ref{p:cal-red} is the model for the symplectic reductions in Theorem \ref{t:main},
based on the Choices and Notations \ref{cn:formal-definition} and Definition \ref{d:maslov-banach}.
We have the uniform decomposition $X(s) = L(t,s) X_0(t) \oplus X_1(t,s)$ of Equation
\eqref{e:direct-sum} and make a reduction to the finite-dimensional spaces $X_0(s)$. To consider
$R_W(\cdot)$ we have two choices of $W$: We may take $W_l:= L(t,s) V(t)\oplus\la(s)$ and $W_r:=
L(t,s) V(t)\oplus\mu(s)$. In both cases we have $W/W^{\w(t)}=L(t,s)X_0(t)$. Their symplectic
structures are the same (see Proposition \ref{p:cal-red}.d).
\newline b) The $W(s)$ of the present section is just $W_l$. In the preceding Corollary
\ref{c:finite-dimension-reduction}, however, we impose an alternative condition: $\dim
W(s)/W(s)^{\w(s)}$ is finite. Note that the corresponding Condition \eqref{e:sum-condition} is
weaker than the former conditions \eqref{e:plus-condition2}-\eqref{e:plus-condition3}.
\end{rem}

\smallskip

Our second theorem (Theorem \ref{t:mas-emb}) shows the invariance of the Maslov index under
symplectic embedding. It generalizes \auindex{Boo{\ss}--Bavnbek,\ B.}\auindex{Zhu,\ C.}\cite[Lemma
12]{BooZhu:2013}. We make some preparations for it.

\begin{lemma}\label{l:seperate-lagrange} Let $(X,\w)$ be a symplectic vector space and $X_0, X_1$
two symplectic subspaces with $X=X_0\oplus X_1$ and $X_0=X_1^{\w}$. Let $\la\subset X$ be a
Lagrangian subspace of $X$. Assume that $\la\cap X_0$ is a Lagrangian subspace of $X_0$. Then
$\la\cap X_1$ is a Lagrangian subspace of $X_1$, and we have
\begin{equation}\label{e:seperate-lagrange}
\la=\la\cap X_0\oplus\la\cap X_1.
\end{equation}
\end{lemma}

\begin{proof} By Lemma \ref{l:sym-subspace}, we have $X_1=X_0^{\w}$. So there holds $(\lambda\cap X_0)^{\omega}\supset X_0^{\omega}=X_1$.
Since $X=X_0\oplus X_1$, By Corollary \ref{c:whole-space} we have
$$(\lambda\cap X_0)^{\omega}=(\lambda\cap X_0)^{\omega}\cap X_0+X_1.$$
Since $\la\cap X_0$ is a Lagrangian subspace of $X_0$, we have $(\la\cap X_0)^{\w}\cap X_0=\la\cap X_0\subset\la$. Then we
have $\lambda=\lambda^{\omega}\subset(\lambda\cap X_0)^{\omega}=\lambda\cap X_0+X_1$. Thus there holds
\begin{align*}
\la&=\la\cap(\la\cap X_0+X_1)=\la\cap X_0+\la\cap X_1, \text{ and}\\
\la^{\w}&=(\la\cap X_0)^{\w}\cap X_0+(\la\cap X_1)^{\w}\cap X_1.
\end{align*}
Consequently, $(\la\cap X_1)^{\w}\cap X_1=\la\cap X_1$ and $\la\cap X_1$ is a Lagrangian subspace
of $X_1$.
\end{proof}

\begin{proposition}[Decomposition of the Maslov index into parts]\label{p:seperate-mas}
\subindex{Maslov index!decomposition into parts}
Let $p\colon \mathbb{X}\to [0,1]$ be a Banach bundle with $X(s):=p^{-1}(s)$ for each $s\in[0,1]$. Let $\w(s)$ be a path of symplectic structures on $X(s)$. Let
$(\lambda(s),\mu(s))$ be a path of Fredholm pairs of Lagrangian subspaces of $(X,\omega(s))$ of
index $0$. Let $p_j\colon \mathbb{X}_j\to [0,1]$ be two Banach subbundles of $p\colon \mathbb{X}\to
[0,1]$ with $X_j(s):=p_j^{-1}(s)$, $s\in[0,1]$, $j=1,2$. We assume that
\begin{enumerate}
\item[(1)] $\w(s)|_{X_j(s)}$ are continuously varying for $j=0,1$,
\item[(2)] $X(s)=X_0(s)\oplus X_1(s)$ and $X_0(s)=X_1(s)^{\w(s)}$, and
\item[(3)] $\{(\la(s)\cap X_0(s),\mu(s)\cap X_0(s))\}_{s\in[0,1]}$ is a path of pairs of Lagrangian subspaces in $(X_0(s),\w(s)|_{X_0(s)})$.
\end{enumerate}
Then $(\la(s)\cap X_j(s),\mu(s)\cap X_j(s))$ is a path of Fredholm pairs of Lagrangian subspaces in
$(Y(s),\w(s)|_{X_j(s)})$ of index $0$, $j=0,1$, and
\begin{eqnarray}\label{e:seperate-mas}
\Mas_{\pm}\{\la(s),\mu(s)\}&=&\Mas_{\pm}\{\la(s)\cap X_0(s),\mu(s)\cap X_0(s)\}\\
\nonumber& &+\Mas_{\pm}\{\la(s)\cap X_1(s),\mu(s)\cap X_1(s)\}.
\end{eqnarray}
\end{proposition}

\begin{proof} By (2), we have an injective continuous map%
\[
f(s)\colon \Ss(X_0(s))\times \Ss(X_1(s))\too \Ss(X(s)),\quad \text{defined by $f(M,N)\ :=\ M+N$},
\]
for all $s\in[0,1]$, and $f(s)$ is a homeomorphism onto its image. By Lemma
\ref{l:seperate-lagrange}, $\{(\la(s)\cap X_j(s),\mu(s)\cap X_j(s))\}_{s\in[0,1]}$ is a path of
Fredholm pairs of Lagrangian subspaces in $(Y(s),\w(s)|_{X_j(s)})$, $j=0,1$, and
\begin{eqnarray*}
\Index(\la(s),\mu(s))& =&\Index(\la(s)\cap X_0(s),\mu(s)\cap X_0(s))\\
& &+\Index(\la(s)\cap X_1(s),\mu(s)\cap X_1(s))
\end{eqnarray*}
for all $s\in[0,1]$. By Lemma \ref{l:negative-index}, we have $\Index(\la(s)\cap X_j(s),\mu(s)\cap
X_j(s))\le 0$, $j=0,1$. So we have $\Index(\la(s)\cap X_j(s),\mu(s)\cap X_j(s))=0$, $j=0,1$. Then
the Equation (\ref{e:seperate-mas}) follows from Proposition \ref{p:maslov-properties}.c.
\end{proof}

\smallskip

The following theorem is the second major result of this section. It strengthens the claim of the
preceding proposition in the following sense: The second term on the right hand side of Euqation
\ref{e:seperate-mas} is vanishing, if the intersection dimensaions of the Fredholm pairs diminishes
by a constant under the decomposition.In general, the reduced intersection dimensions will
\textit{not} become constant. Otherwise the Maslov index must vanish by Proposition
\ref{p:maslov-properties}.e and its analogue Theorem \ref{t:main}.

Intuitively, the claim of Theorem \ref{t:mas-emb} seems unquestionable. As often in
infinite-dimensional analysis, however, intuition can not be trusted. As a matter of fact, the
following rigorous proof of the theorem may appear quite involved and is definitely not straight
forward.

\begin{theorem}[Invariance under symplectic embedding]\label{t:mas-emb}
\subindex{Maslov index!invariance under symplectic embedding}\subindex{Symplectic embedding}
Let $p\colon \mathbb{X}\to [0,1]$ be a Banach bundle. Denote by $X(s):=p^{-1}(s)$ the fiber of $p$
at $s\in[0,1]$. Let $\w(s)$ be a path of symplectic structures on $X(s)$. Let $(\lambda(s),\mu(s))$
be a path of Fredholm pairs of Lagrangian subspaces of $(X,\omega(s))$ of index $0$. Let $p_1\colon
\mathbb{Y}\to [0,1]$ be a second Banach bundle which is a linear subbundle of $p\colon
\mathbb{X}\to [0,1]$ (in general the inclusion $Y(s)\hookrightarrow X(s)$ is neither continuous nor
dense), where $Y(s):=p_1^{-1}(s)$. We assume that
\begin{itemize}
\item[(1)] $\w(s)|_{Y(s)}$ is continuously varying,
\item[(2)] $\{(\la(s)\cap Y(s),\mu(s)\cap Y(s))\}_{s\in[0,1]}$ is a path of Fredholm pairs of Lagrangian subspaces in $(Y(s),\w(s)|_{Y(s)})$ of index $0$, and
\item[(3)] $\dim(\la(s)\cap\mu(s))-\dim(\la(s)\cap\mu(s)\cap Y(s))$ is  a constant $k$.
\end{itemize}
Then we have
\begin{equation}\label{e:mas-emb}
\Mas_{\pm}\{\la(s),\mu(s)\}=\Mas_{\pm}\{\la(s)\cap Y(s),\mu(s)\cap Y(s)\}.
\end{equation}
\end{theorem}

\begin{proof} Since $[0,1]$ is compact, by Definition \ref{d:maslov-banach} we need only consider the local case.
In this case the bundles $\mathbb{X}$ and $\mathbb{Y}$ are both trivial, i.e., we can assume that
$X(s)=X$ and $Y(s)=Y$.

Fix $t\in[0,1]$. Set $\la_Y(s):=\la(s)\cap Y$, $\mu_Y(s):=\mu(s)\cap Y$, and
$\w_Y(s):=\w(s)|_{Y\times Y}$ for all $s\in[0,1]$. By the Fredholm properties, there exist
finite-dimensional linear subspaces $V_1\subset Y$ and $V_2\subset X$ such that
\[Y=V_1\oplus (\la(t)\cap Y+\mu(t)\cap Y),\quad X=V_2\oplus(Y+\la(t)+\mu(t)).\]
Set $V:=V_1\oplus V_2$. Then we have
\[V+\la(t)+\mu(t)=V_1+V_2+\la(t)+\mu(t)=V_2+Y+\la(t)+\mu(t)=X.\]

Note that
\begin{align*}
V\cap\la(t)&=(V_1+V_2)\cap(V_1+\la(t))\cap\la(t)\\
&=(V_1+V_2\cap(V_1+\la(t))\cap\la(t)\\
&=V_1\cap\la(t)=V_1\cap Y\cap\la(t)=\{0\}.
\end{align*}
By Appendix \ref{ss:continuity-of-operations}, there exists a $\delta>0$ such that for
$s\in(-\delta,\delta)\cap[0,1]$, we have $V+\la(s)+\mu(s)=X$, $V_1+\la_Y(s)+\mu_Y(s)=Y$,
and $V\cap\la(s)=V_1\cap\la_Y(s)=\{0\}$.
 By Corollary \ref{c:finite-dimension-reduction}, for all paths $[s_1,s_2]\subset(-\delta,\delta)\cap[0,1]$ we have
\begin{align*}
\Mas_{\pm}\{\la(s),\mu(s)\}&=\Mas_{\pm}\{R_{V+\la(s)}^{\w(s)}(\la(s)),R_{V+\la(s)}^{\w(s)}(\mu(s))\},\\
\Mas_{\pm}\{\la_Y(s),\mu_Y(s)\}&=\Mas_{\pm}\{R_{V_1+\la_Y(s)}^{\w_Y(s)}(\la(s)),R_{V+\la_Y(s)}^{\w_Y(s)}(\mu(s))\}.
\end{align*}

Consider the symplectic linear maps
\begin{equation}\label{dia:composition}
\xymatrix@C=1cm{ \frac{V_1+\la_Y(s)}{V_1^{\w_Y(s)}\cap\la(s)} \ar^{f(s)}[r] &
\frac{V_1+\la(s)}{V_1^{\w(s)}\cap\la(s)} \ar^{g(s)}[r] & \frac{V+\la(s)}{V^{\w(s)}\cap\la(s)},}
\end{equation}
where $f(s)$ is induced by the embedding $\la_Y(s)\into \la(s)$ and $g(s)$ is induced by the
embedding $V_1\into V$. Note that a symplectic linear map is an injection. By comparing dimensions,
each $f(s)$ is an isomorphism and each $g(s)$ is an injection. For any linear subspace $M$ of $X$,
we have
\begin{equation}\label{e:subspace}
f(s)(R_{V_1+\la_Y(s)}^{\w_Y(s)}(M\cap Y))\subset R_{V_1+\la(s)}^{\w(s)}(M).
\end{equation}
If $M=\la(s)$ or $M=\mu(s)$, then
\begin{itemize}
\item $R_{V_1+\la_Y(s)}^{\w_Y(s)}(M\cap Y)$ is a Lagrangian subspace in the reduced space
$(V_1+\la_Y(s)) / (V_1^{\w_Y(s)}\cap\la(s))$, and

\item $R_{V_1+\la(s)}^{\w(s)}(M)$ is a Lagrangian subspace in $(V_1+\la(s))/(V_1^{\w(s)}\cap\la(s))$.
\end{itemize}
So (\ref{e:subspace}) is an equality of two Banach bundles of finite fiber dimension for each
$s\in[0,1]$. Then we can apply Lemma \ref{l:red-transitive}, Lemma \ref{l:red-index} and
Proposition \ref{p:red-fredholm}. Our problem is then reduced to a path of symplectic embeddings
$g(s)\circ f(s)$, which replace the linear embedding of the bundles.

Now we are in the finite-dimensional case, i.e., we can assume that $\dim X<+\infty$. In this case,
the embedding is always continuous, and $X=Y(s)\oplus Y(s)^{\w(s)}$. By Proposition
\ref{p:seperate-mas} and Proposition \ref{p:maslov-properties}.e we have
\begin{align*}
\Mas_{\pm}\{\la(s),\mu(s)\}\ &=\ \Mas_{\pm}\{\la(s)\cap Y(s),\mu(s)\cap Y(s)\}\\
&\qquad+\ \Mas_{\pm}\{\la(s)\cap Y(s)^{\w(s)},\mu(s)\cap Y(s)^{\w(s)}\}\\
&=\ \Mas_{\pm}\{\la(s)\cap Y(s),\mu(s)\cap Y(s)\}.
\end{align*}
Our result is then proved.
\end{proof}
\section{The H{\"o}rmander index}\label{ss:hormander-index}

In this section we fix the symplectic Banach space $(X,\w)$. Firstly we give some topological and
calculatory preparations. Recall from Definition \ref{d:fredholm-lag-pair}: for $k,m\in\Z$ and
$\mu\in\Ll(X)$, we define
\begin{align*}
\Ff\Ll_k(X)\ :&=\ \{(\lambda,\mu)\in\Ff\Ll(X); \Index(\la,\mu)=k\}, \\
\Ff\Ll_k(X,\mu)\ :&=\ \{\la\in\Ll(X);(\la,\mu)\in\Ff\Ll_k(X)\},\\
\Ff\Ll_0^m(X,\mu)\ :&=\ \{\la\in\Ff\Ll_0(X,\mu);\dim(\la\cap\mu)=m\}.
\end{align*}

\begin{lemma}\label{l:pi0-fl-mu} Let $\mu\in\Ll(X)$. Then we have that
\subindex{Fredholm Lagrangian Grassmannian!topology}
\newline (a) $\Ff\Ll_0^0(X,\mu)$ is an affine space (hence contractible),
\newline (b) $\Ff\Ll_0^0(X,\mu)$ is dense in $\Ff\Ll_0(X,\mu)$ and $\Ff\Ll_0(X,\mu)$ is path connected.
\end{lemma}

\begin{proof} (a) Let $\la\in\Ff\Ll_0^0(\mu)$. By Lemma \ref{l:connectness-local}, we have
\begin{align*}\Ff\Ll_0^0&(X,\mu)=\{\Graph(A)\in\Ll(X);A\in\Bb(\la,\mu)\}\\
&=\{\Graph(A);A\in\Bb(\la,\mu),\w(x,Ay)+\w(Ax,y)=0,\forall x,y\in\la\}.
\end{align*}
So (a) is proved.
\newline (b) Let $\la\in \Ff\Ll_0(X,\mu)$. By Proposition \ref{p:reduction-prop}, we have $X=X_0\oplus X_1$,
where $X_0:= V\oplus\la_0$, $X_1:=\la_1\oplus\mu_1$, $\la_0:=\la\cap\mu$, $\la_1=V^{\w}\cap\la$,
$\mu_1=\mu^{\w}\cap\mu$, and $V$ is chosen to be isotropic. We have $X_0=X_1^{\w}$ is of finite
dimension, and $X_0$, $X_1$ are symplectic. Note that $V,\la_0\in\Ll(X_0)$ and
$\la_1,\mu_1\in\Ll(X_1)$. Let $A\colon \la_0\to V$ be a linear isomorphism with
$\w(x,Ay)+\w(Ax,y)=0,\forall x,y\in\la_0$. Set $c_1(s):=\Graph(sA)$, $s\in[0,1]$ and
$c(s)=c_1(s)\oplus\la_1$. The $c(0)=\la$ and $c(s)\in\Ff\Ll_0^0(X,\mu)$. By (a), we get (b).
\end{proof}

\begin{lemma}\label{l:pi0-cp-la} Let $\la,\mu\in\Ll(X)$ and $X=\la\oplus\mu$. Let $\Cc\Pp_0(X,\la)$ be defined in Corollary \ref{c:pi0-fpc0}.
Then we have
\newline (a) $\Ff\Ll_0^0(X,\mu)\cap\Cc\Pp_0(X,\la)$ is an affine space (hence contractible),
\newline (b) $\Ff\Ll_0^0(X,\mu)\cap\Cc\Pp_0(X,\la)$ is dense in $\Ff\Ll_0(X,\mu)\cap\Cc\Pp_0(X,\la)$,
and $\Ff\Ll_0(X,\mu)\cap\Cc\Pp_0(X,\la)$ is path connected.
\end{lemma}

As explained in the Appendix (Corollary \ref{c:pi0-fpc0}), the set $\Cc\Pp_0(X,\la)$ consists of
all complemented finite changes of $\la$ in $X$ of vanishing relative index. For the notion of
\textit{finite change} and \textit{relative index} see Definition \ref{d:compact-perturb}.
\subindex{Finite change}\subindex{Relative index}\subindex{Index!relative}
\symindex{Zz@$\sim^c,\sim^f$ compact, finite change}\symindex{Zz@$[\cdot-\cdot]$ relative index}
\begin{proof} The proof of Lemma \ref{l:pi0-cp-la} is similar to that of Corollary \ref{c:pi0-fpc0} and we omit it.
\end{proof}

\begin{corollary}\label{c:mas-loop-independence}
Let $\mu_1,\mu_2\in\Ll(X)$ such that $\mu_1\sim^c\mu_2$ and $[\mu_1-\mu_2]=0$. Let
$\{\la_j(s)\}_{s\in[0,1]}$  be two paths in $\Ff\Ll_0(X,\mu_1)$ with the same endpoints for
$j=1,2$. Then $\{\la_j(s)\}_{s\in[0,1]}$ is a path in $\Ff\Ll_0(X,\mu_2)$ and we have
\begin{equation}\label{e:mas-loop-independence}
\Mas\{\la_1,\mu_2\}-\Mas\{\la_1,\mu_1\}=\Mas\{\la_2,\mu_2\}-\Mas\{\la_2,\mu_1\}.
\end{equation}
\end{corollary}

\begin{proof} By Lemma \ref{l:pi0-cp-la}, there is a path $\mu(s)$ with $\mu(0)=\mu_1$ and $\mu(1)=\mu_2$, $\mu(s)\sim^c\mu_2$ and $[\mu_1-\mu(s)]=0$.
By Proposition \ref{p:compact-perturb}.g, we have $\Ff_{0,\mu_1}(X)=\Ff_{0,\mu_2}(X)$ and
$\Ff\Ll_0(X,\mu_1)=\Ff\Ll_0(X,\mu(s))$. Then we have $(\la_j(s),\mu(s))\in\Ff\Ll_0(X)$.
 Then we have two homotopies $(\la_j(s),\mu(t))\in\Ff\Ll_0(X)$, $(s,t)\in[0,1]$ for $j=1,2$. By Proposition \ref{p:maslov-properties} we have
\begin{align*}
\Mas\{\la_1,\mu_2\}&-\Mas\{\la_1,\mu_1\}=\Mas\{\la_1(1),\mu\}-\Mas\{\la_1(0),\mu\}\\
&=\Mas\{\la_2(1),\mu\}-\Mas\{\la_2(0),\mu\}\\
&=\Mas\{\la_2,\mu_2\}-\Mas\{\la_2,\mu_1\}.\qedhere
\end{align*}
\exendproof

Now we are in the position of defining the H{\"o}rmander index for a quadruple of Lagrangian
subspaces in symplectic Banach space, so in particular in \textit{weak} symplectic Hilbert space.
Formally, our definition reminds the definition given by K. Furutani and the first author in
\auindex{Boo{\ss}--Bavnbek,\ B.}\auindex{Furutani,\ K.}\cite[Definition 5.2]{BoFu99} for strong
symplectic Hilbert space. The novelty of the following definition is that we need two additional
conditions for the weak symplectic case, namely
\[
\Index(\la_1,\mu_1)\ =\ \Index(\la_2,\mu_1)\ = \ 0 \ \tand\  [\mu_1-\mu_2]\ =\ 0.
\]
Note that these conditions are always satisfied in the strong symplectic case.

\begin{definition} Let $\la_1,\la_2,\mu_1,\mu_2\in\Ll(X)$ be Lagrangian subspaces of $X$.
Assume that $\la_1,\la_2\in \Ff\Ll_0(X,\mu_1)$, $\mu_1\sim^c\mu_2$ and $[\mu_1-\mu_2]=0$. By Lemma
\ref{l:pi0-fl-mu}, there is a path $\la\colon [0,1]\to\Ff\Ll_0(X,\mu_1)$ with $\la(0)=\la_1$,
$\la(1)=\la_2$. By Lemma \ref{c:mas-loop-independence}, we can define the {\em H{\"o}rmander index}
$\sigma(\la_1,\la_2;\mu_1,\mu_2)$ by
\subindex{Index!H{\"o}rmander index}\subindex{H{\"o}rmander index}%
\symindex{s@$s(\cdot,\cdot;\cdot,\cdot)$ H{\"o}rmander index}
\begin{equation}\label{e:hormander}
s(\la_1,\la_2;\mu_1,\mu_2)=\Mas\{\la,\mu_2\}-\Mas\{\la,\mu_1\}.
\end{equation}
\end{definition}

We note that the condition $[\mu_1-\mu_2]=0$ is automatically satisfied under the assumptions of
the following lemma.

\begin{lemma}\label{l:0-relative-index}

Let $(X,\w)$ be a symplectic Banach space with three closed subspaces $\la,\mu_1,\mu_2$. Let
$(\la,\mu_1),(\la,\mu_2)\in\Ff\Ll(X)$. Assume that $\mu_1\sim^c\mu_2$.  If
$\Index(\la,\mu_1)=\Index(\la,\mu_2)=0$, we have $[\mu_1-\mu_2]=0$.
\end{lemma}
Note that by Proposition \ref{p:compact-perturb}.g,%
\[
(\la,\mu_1)\in\Ff\Ll(X)\ \tand\ \mu_1\sim^c\mu_2\ \Longrightarrow\ (\la,\mu_2)\in\Ff\Ll(X).
\]

\begin{proof}
The Lemma is just a special case of Proposition \ref{p:compact-perturb}.g.
\end{proof}

\begin{rem}\label{r:de-Gosson}
In \cite{Gosson:2009}, M. de Gosson\auindex{de Gosson,\ M.} gave a very elegant definition of the
H{\"o}rm\-ander index in finite dimensions in great generality. His definition differs slightly
from ours. By admitting half-integer indices, it  yields more simple proofs, but may be more
difficult to apply in concrete applications in Morse theory.
\end{rem}

\addtocontents{toc}{\medskip\noi}
\part{Applications in global analysis}\label{part:2}
\addtocontents{toc}{\medskip\noi}
\chapter{The desuspension spectral flow formula}\label{s:gsff}

In this section, we study self-adjoint Fredholm extensions of symmetric operators, and prove a
general spectral flow formula under the assumption of a certain weak inner unique continuation
property (wiUCP).

\section{Short account of predecessor formulae}\label{s:history}

To begin with, we describe the topological and analytic background of our applications.

\subsection{The spectral flow}
In various branches of mathematics one is interested in the calculation of the \subindex{Spectral
flow!wider context}spectral flow of a continuous family of closed densely defined (not necessarily
bounded) self-adjoint \subindex{Fredholm operator!curves of closed self-adjoint densely defined
Fredholm operators}Fredholm operators in a fixed Hilbert space. Roughly speaking, the spectral flow
is an intersection number between the spectrum and the real line and counts the net number of
eigenvalues changing from the negative real half axis to the nonnegative one.

The spectral flow for a one parameter family of linear self-adjoint Fredholm operators was
introduced by \auindex{Atiyah,\ M.F.}\auindex {Patodi,\ V.K.}\auindex{Singer,\ I.M.}M. Atiyah, V.
Patodi, and I. Singer \cite{AtPaSi76} in their study of \subindex{Index!index theory on manifolds
with boundary}index theory on manifolds with boundary. Since then, other significant applications
have been found; many of them were inspired by \auindex{Vafa,\ C.}\auindex{Witten,\ E.}C. Vafa and
E. Witten's use of the spectral flow to estimate uniform bounds for the \subindex{Spectral
gap}spectral gap of \subindex{Dirac type operators}Dirac operators in \auindex{Vafa,\
C.}\auindex{Witten,\ E.}\cite{VaWi84}. The spectral flow was implicit already in Atiyah and Singer
\auindex{Atiyah,\ M.F.}\auindex{Singer,\ I.M.}\cite{Atiyah-Singer:1969} as the isomorphism from the
fundamental group of the non-trivial connected component of bounded self-adjoint Fredholm operators
in complex Hilbert space onto the integers. Later this notion was made rigorous for not necessarily
closed curves of bounded self-adjoint Fredholm operators in \auindex{Phillips,\ J.}J. Phillips
\cite{Ph96} and for \subindex{Continuously varying!gap-continuous curves of self-adjoint (generally
unbounded) Fredholm operators in Hilbert spaces}\subindex{Gap topology}gap-continuous curves of
self-adjoint (generally unbounded) Fredholm operators in Hilbert spaces in
\auindex{Boo{\ss}--Bavnbek,\ B.}\auindex{Lesch,\ M.}\auindex{Phillips,\ J.}\cite{BoLePh01} by the
Cayley transform. The notion was generalized to the higher dimensional case in \auindex{Dai,\
X.}\auindex{Zhang,\ W.}X. Dai and W. Zhang \cite{DaZh98} for \subindex{Continuously
varying!Riesz-continuous curves of self-adjoint (generally unbounded) Fredholm operators in Hilbert
spaces}\subindex{Riesz topology}Riesz-continuous families, and to more general operators by
\auindex{Wojciechowski,\ K.P.}K.P. Wojciechowski and \auindex{Zhu,\ C.}\auindex{Long,\ Y.}C. Zhu
and Y. Long in \cite{Wo85,Zh01,ZhLo99}.

\subsection{Switch between symmetric and symplectic category}
In this section we derive \subindex{Spectral flow!spectral flow formula $\to$ \textit{Desuspension
spectral flow formula}}\subindex{Desuspension spectral flow formula!switch between symmetric and
symplectic category}spectral flow formulae in the following sense. We are given a continuous curve
of self-adjoint Fredholm operators, or more generally, a \subindex{Fredholm
relation}\subindex{Continuously varying!self-adjoint Fredholm relations}continuous curve of
self-adjoint Fredholm relations. Such curves arise typically from a family of elliptic operators
over a compact manifold with boundary with smoothly varying coefficients and smoothly varying
regular boundary conditions. Then we consider two mutually related invariants: within the {\em
symmetric} \subindex{Symmetric category}\subindex{Symmetric category!$\to$ \textit{Morse
index}}\subindex{Symmetric category!$\to$ \textit{Spectral flow}}category, we have the number of
negative eigenvalues or, more generally, the spectral flow; that is our first invariant. Basically,
it is an intersection number of the \subindex{Spectral lines}spectral lines with the real axis. It
is well defined, but, being a spectral invariant, difficult to determine in general. To define the
second invariant, we switch from the symmetric category to the {\em symplectic} category. We notice
that self-adjoint extensions are characterized by Lagrangian subspaces in corresponding symplectic
Hilbert spaces coming from the domains, i.e., from the boundary values. That consideration yields
another \subindex{Symplectic structures!symplectic invariant}intersection number, the Maslov index.
The Maslov index does not arise from the spectrum, but can be calculated directly from the boundary
values of the solutions. Speaking roughly, the Maslov index counts the changes of the intersection
dimensions of two curves of Lagrangians. In our case, the one curve is made of the continuously
varying Lagrangians coming from the \subindex{Domain!Fredholm}Fredholm  domains. The other curve is
made of the \subindex{Cauchy data space}Cauchy data spaces, which also form Lagrangians and vary
continuously under suitable assumptions. Then the type of spectral flow formulae we are interested
in are formulae where the spectral flow of a given curve of self-adjoint Fredholm relations or
operators is expressed by a related Maslov index. Here the point is that the calculation of the
Maslov index is different from the calculation of the spectral flow, and, in general, easier.

\subsection{Origin and applications in Morse theory}
The first spectral flow formula was the classical \subindex{Desuspension spectral flow
formula!origin and applications in Morse theory}Morse index theorem (cf. \auindex{Morse,\ M.}M.
Morse \cite{Mo}) for \subindex{Geodesic}geodesics on \subindex{Riemannian manifold}Riemannian
manifolds. It was extended by \auindex{Ambrose,\ W.}W. Ambrose \cite{Am} in 1961 to more general
\subindex{Boundary value problems!boundary condition}boundary conditions, which allowed the two
endpoints of the geodesics to vary in two submanifolds of the manifold. In 1976,
\auindex{Duistermaat,\ J.J.}J.J. Duistermaat \cite{Du76} completely solved the problem of
calculating the Morse index for the \subindex{Variational calculus!one-dimensional
problems}one-dimensional variational problems, where the positivity of the second order terms was
required. In 2000-2002, \auindex{Piccione,\ P.}\auindex{Tausk,\ D.V.}P. Piccione and D.V. Tausk
\cite{PiT1,PiT2} were able to prove the Morse index theorem for \subindex{Semi-Riemannian
manifold}semi-Riemannian manifolds for the same boundary conditions as in \auindex{Ambrose,\
W.}\cite{Am}, and certain non-degeneracy conditions were needed. In 2001, the second author
\auindex{Zhu,\ C.}\cite{Zh01} was able to solve the general problem for the calculation of the
Morse index of index forms for \subindex{Dynamical systems!regular Lagrangian system}regular
Lagrangian systems. See also the work of \auindex{Musso,\ M.}\auindex{Pejsachowicz,\
J.}\auindex{Portaluri,\ A.}M. Musso, J. Pejsachowicz, and A. Portaluri on a Morse index theorem for
perturbed geodesics on semi-Riemannian manifolds in \cite{MuPePo05} which has in particular lead
\auindex{Waterstraat,\ N.}N. Waterstraat to a $K$-theoretic proof of the Morse Index Theorem in
\cite{Wa12}.

\subsection{From ordinary to partial differential equations}
In 1988, \auindex{Floer,\ A.}A. Floer \cite{Fl88} emphasized that the notion of a Morse index of a
function on a finite-dimensional manifold cannot be generalized directly to the \subindex{Morse
theory!symplectic action function}symplectic action function \symindex{\alpha@$\alpha$ symplectic
action function}$\alpha$ on the \subindex{Geodesic!loop space of a manifold}loop space of a
manifold. He defined for any pair of \subindex{Critical points!of symplectic action
function}critical points of $\alpha$ a \subindex{Morse theory!Morse index!relative}relative Morse
index, which corresponds to the difference of the two Morse indices in finite dimensions. It is
based on the spectral flow of the Hessian of $\alpha$. That paper opened another line of studying
spectral flow formulae, namely for partial differential operators:

Let \symindex{M@$M, M(s)$ smooth compact manifold!with boundary $\Si,\Si(s)$}
$\{A(s):\Ci(M;E)\to\Ci(M;E)\}_{s\in [0,1]}$ be a family of continuously varying formally
self-adjoint linear \subindex{Elliptic differential operators!on smooth manifolds with
boundary}elliptic differential operators of first order over a smooth compact Riemannian manifold
$M$ with boundary $\Si$, acting on sections of a Hermitian vector bundle $E$ over $M$. Fixing a
unitary bundle isomorphism between the original bundle and a product bundle in a collar
neighborhood $N$ of the boundary, the operators $A(s)$ can be written in the form
\begin{equation}\label{e:intro-dirac}
\symindex{J@$J,J_t,J_{s,t}$ skew-self-adjoint bundle isomorphisms}
\symindex{B@$B,B_t,B_{s,t}$ tangential operators}\symindex{t@$t$ inward normal coordinate}
\symindex{s@$s$ variational parameter}
A(s)|_N = J_{s,t}(\frac{\dpa}{\dpa t}+B_{s,t})
\end{equation}
with skew-self-adjoint bundle isomorphisms $J_{s,t}$ and first order elliptic
differential operators $B_{s,t}$ on $\Si$. Here $t$ denotes the inward normal coordinate in $N$.
For details see \auindex{Boo{\ss}--Bavnbek,\ B.}\auindex{Lesch,\ M.}\auindex{Zhu,\ C.}B.
Boo{\ss}-Bavnbek, M. Lesch and C. Zhu \cite[Section 1]{BoLeZh08}.

Following another seminal paper by \auindex{Floer,\ A.}A. Floer \cite{Fl88b}, in 1991
\auindex{Yoshida,\ T.}T. Yoshida \cite{Yo91} \subindex{Desuspension spectral flow
formula!predecessors}elaborated on \subindex{Floer homology!of 3-manifolds}Floer homology of
3-manifolds by studying a curve \symindex{A@$\{A(s)\}$ curve!of Dirac operators}$\{A(s)\}$ of Dirac
operators with invertible ends, such that over the boundary, the bundle isomorphisms $J_{s,t}=J(s)$
are unitary and the tangential operators $B_{s,t}=B(s)$ symmetric in the preceding notation. In
1995, \auindex{Nicolaescu,\ L.}L. Nicolaescu \cite{Ni95} generalized Yoshida's
results\subindex{Yoshida-Nicolaescu's Theorem} to arbitrary $\dim M$. One year later,
\auindex{Cappell,\ S.E.}\auindex{Lee,\ R.}\auindex{Miller,\ E.Y.}S.E. Cappell, R. Lee, and E.Y.
Miller \cite[Theorem G]{CaLeMi96b} found a somewhat intricate spectral flow formula for curves of
arbitrary elliptic operators of first order under the conditions of constant coefficients close to
the boundary in normal direction and symmetric induced tangential operators. In 2000,
\auindex{Daniel,\ M.}M. Daniel \cite{Da00} removed the nondegenerate conditions in
\auindex{Nicolaescu,\ L.}\cite{Ni95}. In 1998-2001, the first author, jointly with
\auindex{Boo{\ss}--Bavnbek,\ B.}\auindex{Furutani,\ K.}\auindex{Otsuki,\ N.}K. Furutani and N.
Otsuki \cite{BoFu99,BoFuOt01} proved the case that the $A(s)$ differ by 0th order operators, and
the boundary condition is fixed. In 2001, \auindex{Kirk,\ P.}\auindex{Lesch,\ M.}P. Kirk and M.
Lesch \cite[Theorem 7.5]{KiLe00} proved the case that $A(s)$ is of Dirac type, $J_{s,t}$ is fixed
unitary, and $B_{s,t}=B(s)$ symmetric. Later in this section we shall only assume that each $A(s)$
satisfies  \subindex{Weak inner Unique Continuation Property (wiUCP)}weak inner unique continuation
property (wiUCP), i.e., $\ker A(s)|_{H_0^1(M;E)}=\{0\}$.

The formulae are  of varying generality: Some deal with a fixed (elliptic) differential operator
with varying self-adjoint extensions (i.e., varying boundary conditions); others keep the boundary
condition fixed and let the operator vary. An example for a path of operators with fixed principal
symbol is a curve of Dirac operators on a manifold with fixed Riemannian metric and Clifford
multiplication but varying defining connection (varying background field which is a zero-order
perturbation and as such does not inflict the principal symbol). See also the results by the
present authors in \auindex{Boo{\ss}--Bavnbek,\ B.}\auindex{Zhu,\ C.}\cite{BoZh05} for varying
operator and varying boundary conditions but fixed maximal domain. Recently, \auindex{Prokhorova,\
M.}M. Prokhorova \cite{Pro13} considered a path of Dirac operators on a two-dimensional disk with a
finite number of holes subjected to local elliptic boundary conditions of \subindex{Boundary value
problems!chiral bag type}chiral bag type. She obtained a beautiful explicit formula for the
spectral flow (respectively, the Maslov index) which recently was re-proved and generalized by
\auindex{Katsnelson,\ M.I.}\auindex{Nazaikinskii,\ V.E.}M. Katsnelson, V. Nazaikinskii,
\auindex{Gorokhovsky,\ A.}\auindex{Lesch,\ M.}A. Gorokhovsky, and M. Lesch in
\cite{KaNa12,GorLes13}.

\subsection{Our contribution in this Memoir}
In this Memoir we have substantially expanded and settled the validity range of the predecessor
formulae. Roughly speaking, we have achieved the following results:
\begin{enumerate}
\subindex{Desuspension spectral flow formula!novelty of this Memoir}
\item[(1)] In the language of Banach bundles we present the list of assumptions on operator families
that yield an abstract general spectral flow formula. The list can be found in Assumption
\ref{a:asff} and the obtained formulae in Equations \eqref{e:asff1} and \eqref{e:asff2} of Theorem
\ref{t:asff}. This result expands substantially the validity of the functional analytic spectral
flow formula of \auindex{Boo{\ss}--Bavnbek,\ B.}\auindex{Furutani,\ K.} \cite{BoFu98}. The novelty
of the approach is the replacement of a fixed strong symplectic Hilbert space (the $\beta$-space of
the quotients of maximal and minimal domain of a closed symmetric operator) by the quotient of a
fixable intermediate domain with the minimal domain, equipped with varying weak symplectic forms.

\item[(2)] In Section \ref{ss:gsff} we turn to the geometric setting. We consider a smooth
family of formally self-adjoint elliptic differential operators of fixed order acting on sections
of varying vector bundles over varying manifolds with boundary and impose varying well-posed
self-adjoint boundary conditions. Under a technical condition that generalizes weak inner UCP, we
obtain an array of general spectral flow formulae in Equations \eqref{e:gsff1} and \eqref{e:gsff2}
of Theorem \ref{t:gsff} in all Sobolev spaces over the boundaries of non-negative order. This
result removes the restriction of previous formulae to curves of Dirac type operators or curves
with only lower order variation.

\item[(3)] In Theorem \ref{t:split} we give the conditions for the validity of two formulae for the spectral flow of a curve of formally
self-adjoint elliptic differential operators over a curve of closed partitioned manifolds
$M(s)=M(s)^+\cup_{\Si(s)} M(s)^-$ with separating hypersurfaces $\Si(s)$, $s\in [0,1]$. The first
Formula \eqref{e:sf-partition} expresses the spectral flow over the whole manifold(s) in terms of a
spectral flow of a canonically associated curve of well posed boundary problems over one part. The
second Formula \eqref{e:sf-mas-partition} expresses the spectral flow over the whole manifold(s) by
the Maslov index of the corresponding Cauchy data spaces from both sides along the separating
hypersurface(s). This result generalizes the splitting formulae\subindex{Yoshida-Nicolaescu's
Theorem} of \auindex{Yoshida,\ T.}T. Yoshida \cite{Yo91} and \auindex{Nicolaescu,\ L.}L. Nicolaescu
\cite{Ni95} and determines the limits of their validity.
\end{enumerate}

\subsection{Spectral flow formulae also for higher order operators}\label{sub:higher-oder}
Usually one is only interested in the spectral flow of elliptic differential operators of first
order.  \subindex{Desuspension spectral flow formula!for higher order operators}Typically, elliptic
differential operators of second order like the various Laplacians are essentially positive (or
essentially negative). For such operators the spectral flow of loops must vanish and the spectral
flow of curves is just the difference between the number of negative (respectively positive)
eigenvalues at the endpoints, hence trivial. However, the formula $\SF=\Mas$ is not trivial and can
give radically new insight also in the case of second order operators. Therefore, we have not
restricted our treatment to operators of order one.

For first order operators, in most applications the formula $\SF=\Mas$ will be read as a
\textit{desuspension}-type formula, namely expressing the spectral flow (a kind of quantum variable
arising from the spectrum) over a manifold by the Maslov index (a kind of classical variable
arising from solution spaces) over a submanifold of codimension 1. Then for essentially positive
second order elliptic differential operators, in the applications we have in mind (e.g., a higher
order Morse index theorem) the formula $\SF=\Mas$ should be read as a  \subindex{Desuspension
spectral flow formula!read as suspension type}\textit{suspension}-type formula, namely expressing
the a-priori unknown Maslov index by the in that case trivial spectral flow via the introduction of
an additional parameter.

\subsection{Partitioned manifolds in topology, geometry, and analysis}
In topology, the interest in \subindex{Partitioned manifold}partitioned manifolds is connected to
the name of \auindex{Heegaard,\ P.} P. Heegaard who in his dissertation \cite{Heegaard:1898}
introduced ways of splitting 3-manifolds and gaining corresponding graphs for algebraic
investigation. In that way he could point to essential differences between homology and homotopy
theory that had been missed by \auindex{Poincar{\'e},\ J.H.} H. Poincar{\'e} (see, e.g., the
elementary presentation by \auindex{Scharlemann,\ M.}M. Scharlemann in \cite{Scharlemann:2003}).
Later his ideas were lavishly generalized in the concepts of cobordism, surgery, and cutting and
pasting of the 1950-60s, see \auindex{Wall,\ C.T.C.}C.T.C. Wall \cite{Wall:1999}. In spite of the
great expectations, the concept of partitioned manifolds has not proved valuable for proving
\auindex{Poincar{\'e},\ J.H.} \subindex{Poincar{\'e}'s Conjecture}Poincar{\'e}'s Conjecture. Years
before \auindex{Perelman,\ G.}G. Perelman's final proof of the Conjecture, \auindex{Floer,\ A.}A.
Floer expressed in \cite{Fl88b} his expectation that the approach via \subindex{Heegaard
splitting}Heegaard splittings or more general decompositions most probably would not solve the
Poincar{\'e} Conjecture but would support the complementary topological program, namely to
determine all groups that can show up as fundamental groups of 3-manifolds.

In geometry, the interest in partitioned manifolds is connected both to the concept of
\subindex{Global analysis!coarse geometry}coarse geometry and to the geometry of \subindex{Singular
spaces}singular spaces. In the first case one separates arduous, but topologically uninteresting
parts out of complete (non-compact) manifolds, e.g., in the relative index theorems of
\auindex{Gromov,\ M.|bind}\auindex{Lawson,\ H.B.}M. Gromov and H.B. Lawson \cite{Gromov-Lawson}. In
the second case one focuses on the geometry around singularities by separating them out.

In analysis, the interest in partitioned manifolds is connected with the \subindex{Riemann-Hilbert
Problem}Riemann-Hilbert Problem of complex analysis. Classically, one looks for pairs of functions
where one is holomorphic inside, and the other outside the disc and that are linearly conjugated by
a transmission condition along the circle, see, e.g., \auindex{Muskhelishvili,\ N.I.} N.I.
Muskhelishvili \cite{Mu77}. In \cite{Bojarski2} \auindex{Bojarski,\ B.}B. Bojarski conjectured the
\subindex{Bojarski Conjecture}general validity of a Riemann-Hilbert type index formula for elliptic
operators on even-dimensional closed partitioned manifolds in terms of the index of the Fredholm
pair of Cauchy data spaces along the separating hypersurface. That Bojarski Conjecture was proved
by K.P. Wojciechowski and the first author in \auindex{Boo{\ss}--Bavnbek,\
B.}\auindex{Wojciechowski,\ K.P.}\cite[Chapter 24]{BoWo93}. We missed the odd-dimensional
case\subindex{Yoshida-Nicolaescu's Theorem} which was then treated by \auindex{Nicolaescu,\ L.}L.
Nicolaescu in \cite{Ni95}. While his result is restricted to \subindex{Dirac type operators}Dirac
type operators it served as the model for the present treatise.

There is a remarkable difference between the topological and the analytic approach to invariants of
partitioned manifolds. The topological approach is characterized by the ease of achieving
additivity formulae for topological invariants like the Euler characteristic or the signature
solely by means of singular homology. Deriving the same results by analytic means, e.g., via the
\subindex{Index!Atiyah-Singer Index Theorem}Atiyah-Singer Index Theorem is much more demanding. For
finer topological invariants and spectral or differential invariants, homology theory may not
suffice and harder means are demanded either from homotopy theory or, after all, from analysis. On
the analysis level, there is clearly no recognizable splitting of the spectrum of a Dirac or
Laplace operator on a partitioned manifold in its components from the parts. For the index (the
chiral multiplicity of the zero-eigenvalues) we have both topological and analytical
\subindex{Partitioned manifold!splitting formulae}splitting formulae (\auindex{Boo{\ss}--Bavnbek,\
B.}\auindex{Wojciechowski,\ K.P.}\cite[Chapters 23-25]{BoWo93}). For the analytic torsion we have a
topological splitting formula by \auindex{L{\"u}ck,\ W.}W. L{\"u}ck in \cite{Luck:1993}. For the
$\eta$-invariant we have an analytic splitting formula by \auindex{Wojciechowski,\ K.P.}K.P.
Wojciechowski in \cite{Wo99}. Similarly, our Theorem \ref{t:split} should be considered an
\subindex{Analytic splitting formula}analytic splitting formula for the spectral flow.

\subsection{Wider perspectives} In this section, we focus solely on the intertwining of the symmetric
category (here the spectral flow) and the anti-symmetric category (here symplectic analysis).
Clearly, each side deserves independent investigations and poses puzzles of their own.

\noi\textit{Symplectic error terms in global analysis of singular manifolds}
\subindex{Symplectic geometry!error terms}\subindex{Singular space}
One such puzzle is to find the correct place of symplectic invariants (like the Maslov index and
the H{\"o}rmander index) in the hierarchy of invariants in global analysis, compared with the
index, the $\eta$-invariant, and the $\zeta$-regularized determinant. We meet the {\em index} of
Fredholm operators as the \subindex{Index!of elliptic problems}index of elliptic problems on closed
manifolds, on manifolds with boundary, and on manifolds with singularities. From the viewpoint of
global analysis, however, the index of elliptic problems on \textit{closed} manifolds is
distinguished because there the index can be expressed by an integral over an integrand that is
locally expressed by the coefficients of the operator. The \textit{$\eta$-invariant} arises in
\textit{boundary value problems}. It is not given by an integral, not by a local formula. It
depends, however, only on finitely many terms of the symbol of the resolvent and will not change
when one changes or removes a finite number of eigenvalues. Its derivative is local.

Keeping this difference in mind, we meet a question repeatedly put forward by \auindex{Gelfan'd,\
I.M.}I.M. Gelfand: ``what comes next?" To this, \auindex{Singer,\ I.M.}I.M. Singer remarked in
personal communication \cite{Singer:1999}: ``Just as $\eta$ arises in boundary value problems for
smooth boundaries, I think the next level will come from corner contributions when the boundary has
corners." Indeed, when the boundary has corners, a third term, the \auindex{H{\"o}rmander,\
L.}\subindex{Symplectic structures!symplectic invariant!H{\"o}rmander index}\textit{H{\"o}rmander
index} of symplectic analysis appears, see C.T.C. Wall \cite{Wall:1969}\auindex{Wall,\ C.T.C.}. He
noticed the non-additivity of the signature for the Hopf bundle with fibre $D^2$ over $S^2$ having
signature $\pm 1$ depending on the choices of sign: this is the union of the induced bundles over
the upper and lower hemispheres of $S^2$, each of which (being contractible) has signature zero.
Wall's observation was in striking contradiction to the common wisdom in topology, first observed
by \auindex{Novikov,\ S.P.}S.P. Novikov: If two manifolds are glued by an orientation-preserving
diffeomorphism of their boundaries, then the signature of their union is the sum of their
signatures. So, Wall found that this additivity property does not hold for the more general
situation where one glues two $4k$-manifolds $Y_\pm$ along a common submanifold $X_0$ of the
boundaries, which itself has a boundary $Z$.  That yields an abstract Zaremba problem (see also
\auindex{Schulze,\ B.-W.|bind}\auindex{Chang,\ D.-C.|bind}\auindex{Habal,\ N.|bind} B.-W. Schulze,
C.-C. Chang, and N. Habal \cite{Schulze:2014}). Wall determined the non-additivity term as the
H{\"o}rmander index of three associated Lagrangian subspaces of an induced finite-dimensional
symplectic vector space. His result was extended to the gluing of $\eta$-invariants by U. Bunke in
\cite{Bunke:1995, Bunke:2009}\auindex{Bunke,\ U.}.

That supports the claim of a \textit{hierarchy of asymmetry invariants}, placing the index of
elliptic problems on closed manifolds at the bottom; placing the $eta$-invariant a little higher,
namely as an error term for smooth boundary value problems; and placing symplectic invariants even
higher, namely as error terms for boundary value problems with corners. In this Memoir, we shall
not follow that line of thoughts any further and content with placing the Maslov index on the level
of smoothly partitioned manifolds for now.

\subsection{Other approaches to the spectral flow}\label{sub:other-approaches}
It may be worth mentioning that there is a multitude of {\em other} formulae involving spectral
flow, e.g., as error term under cutting and pasting of the index (see the first author with K.P.
Wojciechowski \auindex{Boo{\ss}--Bavnbek,\ B.}\auindex{Wojciechowski,\ K.P.}\cite[Chapter
25]{BoWo93}) or under pasting of the eta-invariant as in \auindex{Kirk,\ P.}\auindex{Lesch,\
M.}\cite{KiLe00}. Whereas these \subindex{Suspension spectral flow formula}formulae typically
relate the spectral flow of a family on a closed manifold of dimension $n-1$ to the index or the
eta-invariant of a single operator on a manifold of dimension $n$, this Memoir addresses the
opposite direction, namely how to express the spectral flow of a family over a manifold of
dimension $n$ by objects (here by the Maslov index) defined on a hypersurface of dimension $n-1$.
See also our Subsection \ref{sub:higher-oder} above.

\section{Spectral flow for closed self-adjoint Fredholm relations}\label{ss:sa-relations}
In this section, we show that one can easily obtain a formula expressing the \subindex{Spectral
flow!of curves of closed linear self-adjoint Fredholm relations}spectral flow for curves of linear
self-adjoint Fredholm relations in Hilbert spaces by the Maslov index on a very basic and abstract
level (Proposition \ref{p:basic-sff}). This leads us to a general definition of the spectral flow
for curves of closed linear self-adjoint Fredholm relations of index $0$ in Banach spaces
(Definition \ref{d:sf}).

\subsection{Basic facts and notions of linear relations}\subindex{Linear relation!basic facts and notions}\label{ss:linear-relations}
First we have to explain the terms \textit{linear relation, closed linear relation, self-adjoint
linear relation}, and \textit{Fredholm relation}.

Recall that a \subindex{Linear relation}{\em linear relation} $A$ between two linear spaces $X$ and
$Y$ is a linear subspace of $X\times Y$. We use the notions of linear relations and spectral flow
in \auindex{Boo{\ss}--Bavnbek,\ B.}\auindex{Zhu,\ C.}\cite[Appendix A.2, A.3]{BooZhu:2013}. For
additional details on linear relations, see \auindex{Cross,\ R.}Cross \cite{Cr98}. In this Memoir,
however, we are not directly interested in the many applications of the concept of linear relations
in the fields of multi-valued functions and singular spaces. We rather employ the concept of linear
relations to clarify the \textit{algebraic character} of geometric incidence numbers like the
Maslov index and the spectral flow.

To begin with, we summarize the basic algebraic notions of linear relations: As usual, we define
the \subindex{Domain!of linear relation}{\em domain}, the \subindex{Range!of linear relation}{\em
range}, the \subindex{Kernel!of linear relation}{\em kernel} and the \subindex{Indeterminant
part!of linear relation}{\em indeterminant part} of $A$ by
\begin{eqnarray*}
\symindex{dom@$\dom A$ domain of relation $A$}
\symindex{im@$\ran A$ range of relation $A$}
\symindex{ker@$\ker A$ kernel of relation $A$}
\dom(A)&\ :=\ &\{x\in X;  \;\mbox{there exists}\;y\in Y
\;\mbox{such that}\; (x,y)\in A\}, \\
\ran A&\ :=\ &\{y\in Y;  \;\mbox{there exists}\;x\in X \;\mbox{such
that}\; (x,y)\in A\},\\
\ker A&\ :=\ &\{x\in X; \;(x,0)\in A\},\\
A(0)&\ :=\ &\{y\in Y; \;(0,y)\in A\},
\end{eqnarray*}
respectively.

Then the \textit{sum} $A+B$ and the \textit{composition} $C\cdot A$ are defined by
\begin{align}
A+B\ &:=\ \{(x,y+z)\in X\times Y; \;(x,y)\in A, (x,z)\in B\},\label{e:a-add}\\
C\cdot A\ &:=\ \{(x,z)\in X\times Z; \;\exists y\in Y\;\mbox{such that}\; (x,y)\in A, (y,z)\in
C\}\label{e:a-comp}.
\end{align}
Here $X,Y,Z$ are three vector spaces, $A,B$ are linear relations between $X$ and $Y$, and $C$ is a
linear relation between $Y$ and $Z$.

The \subindex{Inverse!of linear relation}{\em inverse} $A^{-1}$ of a linear relation $A$ is always
defined. It is the linear relation between $Y,X$ defined by
\begin{equation}\label{e:inverse-clr}
A^{-1}=\{(y,x)\in Y\times X;(x,y)\in A\}.
\end{equation}

A linear relation $A$ is (the graph of) an \textit{operator} if and only if $A(0)=\{0\}$. In that
case we identify $A$ and the \symindex{g@$\Graph(A)$ graph of operator $A$}graph of $A$.

We get more interesting relations by incorporating topological aspects. Let $X,Y$ be two Banach
spaces. A \subindex{Closed linear relation}\subindex{Closed linear relation!$\to$ also
\textit{Linear relation}}{\em closed linear relation} between $X,Y$ is an element of
\symindex{S@$\Ss(X)$ set of closed linear subspaces in Banach space $X$}$\Ss(X\times Y)$, i.e., a
closed linear subspace of $X\times Y$. It is called \subindex{Closed linear relation!bounded
invertible}{\em bounded invertible}, if $A\ii$ is the graph of a bounded operator from $Y$ to $X$,
shortly: $A^{-1}\in\Bb(Y,X)$. Then the \subindex{Resolvent set!of closed linear
relation}\textit{resolvent set} \symindex{\rho@$\rho(A)$ resolvent set of closed linear
relation}$\rho(A)$ of a closed linear relation $A$ consists of all $z\in\C$ where $A-z$ is bounded
invertible. As usual, we define the \subindex{Spectrum!of closed linear relation}\textit{spectrum}
by \symindex{\sigma@$\s(A)$ spectrum of closed linear relation}$\s(A):=\C\setminus\rho(A)$.

By definition, the graph of a closed operator is a closed linear relation.

A closed linear relation $A\in\Ss(X\times Y)$ is called \subindex{Fredholm relation}{\em Fredholm},
if $\dim\ker A<+\infty$, $\ran A$ is closed in $Y$ and $\dim (Y/\ran A)<+\infty$. In this case, we
define the \subindex{Index!of Fredholm relation}{\em index} of $A$ to be
\begin{equation}\label{e:index-clr}
\symindex{index@$\Index A$ index of Fredholm relation}
\Index A\ :=\ \dim\ker A-\dim (Y/\ran A).
\end{equation}

By \auindex{Boo{\ss}--Bavnbek,\ B.}\auindex{Zhu,\ C.}\cite[Lemma 16 (a)]{BooZhu:2013}, a closed
linear relation $A$ is Fredholm if and only if $(A,X\times\{0\})$ is a Fredholm pair of elements of
$\Ss(X\times Y)$. By Remark \ref{r:redholm-pairs}, the closedness of $\ran A$ follows from its finite codimension in $Y$ in
combination with the closedness of $A$. In that case, we have
\begin{equation}\label{e:clr-index}
\Index A=\Index(A,X\times\{0\}),
\end{equation}
where the index on the left is that defined in \eqref{e:index-clr} and the index on the right that
defined in \eqref{e:fp-lag-index}. Moreover, a closed linear relation $A$ between $X$ and $Y$ is
bounded invertible, if and only if $X\times Y$ is the direct sum of $A$ and $X\times\{0\}$.
These two results follow from the fact that
\begin{multline}\label{e:fredholm}
A\cap (X\times\{0\})\ =\ \ker A\times\{0\},\ \tand\\ A+\left(X\times\{0\}\right)\ =\
\left(\{0\}\times\range A\right)+\left(X\times\{0\}\right).
\end{multline}
We define a purely \subindex{Fredholm relation!algebraic}\textit{algebraic Fredholm relation} by
dropping the requirement that $A$ is closed and that $\ran A$ is closed. Then $\Index A$ is still
well defined in \eqref{e:index-clr} and the equations \eqref{e:clr-index} and \eqref{e:fredholm}
remain valid.

\subsection{Induced symplectic forms on product spaces}\label{ss:product-spaces}
Our next goal is a purely algebraic, respectively, symplectic characterization of symmetric and
self-adjoint relations. Inspired by \auindex{Bennewitz,\ Ch.}C. Bennewitz \cite{Be72} and
\auindex{Ekeland,\ I.}I. Ekeland \cite{Ekeland:1990}, we define

\begin{definition}\label{d:symm-rel}
Let $X$, $Y$ be two complex vector spaces and \symindex{\Omega@$\Omega:X\times Y\to\C$
non-degenerate sesquilinear map}$\Omega\colon X\times Y\to \C$ be a non-degenerate sesquilinear
map. Set $Z:=X\times Y$ and
\begin{equation}\label{e:Omega-omega}
\symindex{\omega@$\omega: Z\times Z\to\CC$ induced symplectic form on $Z:=X\times Y$}
\subindex{Induced symplectic forms on product spaces}
\omega((x_1,y_1),(x_2,y_2))\ :=\ \Omega(x_1,y_2)-\overline{\Omega(x_2,y_1)}.
\end{equation}
Then $(Z,\omega)$ is a symplectic vector space with two canonical Lagrangian subspaces
$X\times\{0\}$ and $\{0\}\times Y$. We call $\omega$ on $Z$ the {\em symplectic structure induced
by $\Omega$}. The \subindex{Linear relation!adjoint relation}{\em adjoint} of a linear relation
$A\<Z$ is defined to be the \subindex{Annihilator!defining the adjoint relation}annihilator
$A^{\w}$. A linear relation $A\<Z$ is called \subindex{Linear relation!symmetric}{\em symmetric},
respectively \subindex{Linear relation!self-adjoint}{\em self-adjoint}, if $A\< (Z,\w)$ is
isotropic, respectively Lagrangian. Note that we admit that $Y$ is different of $X$. If $A$ is
symmetric, we define a form $Q(A)$ on $\dom(A)$ by $Q(A)(x_1,x_2):=\Omega(x_1,z)$ for all
$x_1,x_2\in\dom(A)$ and $z\in A(x_2):=\{y\in Y;(x_2,y)\in A\}$. Since $A$ is an isotropic subspace of $X\times Y$, the value
$\Omega(x_1,z)$ is independent of the choice of $z\in A(x_2)$. Moreover, the form
\symindex{Q@$Q(A)$ Hermitian form associated to a symmetric relation}\subindex{Hermitian form
associated to a symmetric relation}$Q(A)$ is a well-defined Hermitian form. We call the form $Q(A)$
the {\em Hermitian form associated to $A$}.
\end{definition}

Note that $\Omega$ corresponds to a conjugate-linear injection
\symindex{\tau@$\tau:Y\to X^*$ conjugate-linear injection}
$\tau\colon Y\to X^*$ such that $\bigcap_{y\in Y}\ker\tau(y)=\{0\}$ by
\begin{equation}\label{e:Omega-tau}
\Omega(x,y)\ =\ (\tau(y))(x),\text{ for all }x\in X, y\in Y.
\end{equation}

We consider the case when $X$ is a Banach space and $\tau\colon Y\to X^*$ is an $\R$-linear
isomorphism. Note that so far, $Y$ is only a vector space. So $Y$ is identified with $X^*$ by
$\tau$. In the real case we set $Y:=X^*$ and $\tau=I_{Y}$. If $X$ is a complex Hilbert space, we
set $Y:=X$ and $\tau(y)(x):=\lla x,y\rra$. The space $Y$ becomes a Banach space with the norm
$\|y\|_Y:=\|\tau(y)\|_{X^*}$. Then $(X\times Y,\w)$ is a symplectic Banach space for
\begin{equation}\label{e:X+Y-symplectic}
\w((x_1,y_1),(x_2,y_2))\ :=\ (\tau(y_2))(x_1)-\overline{(\tau(y_1))(x_2)}.
\end{equation}
Note that $X\times\{0\}$ and $\{0\}\times Y$ are two natural Lagrangian subspaces of $X\times Y$.
The symplectic structure $\w$ is strong if and only if $X$ is reflexive (see \auindex{Swanson,\
R.C.}R.C. Swanson, \cite{Swanson:1980}). In this case we call $(X\times Y,\w)$ \auindex{Darboux,\
J.G.}\subindex{Darboux property}{\em Darboux}, following \auindex{Weinstein,\ A.}A. Weinstein,
\cite{Weinstein:1971}.

\begin{rem}\label{r:cross-adjoint}
(a) There is an alternative and fully natural way to introduce the adjoint relation $A^*\in
\Ss(Y^*\times X^*)$ of a closed linear relation $A\in\Ss(X\times Y)$, where  $X,Y$ are Banach
spaces and $X^*,Y^*$ denote the norm-dual spaces, see \auindex{Cross,\ R.} \cite[Chapter
III.1]{Cr98}. Contrary to that definition, our construction of the adjoint relation as the
annihilator $A^\w$ stays within $\Ss(X\times Y)$. It is less general, though, since it depends of
the choice of $\W$ inducing $\w$, respectively, the existence of a conjugate-linear injection
$\tau:Y\to X^*$. Of course, such $\tau$ is naturally given for $X=Y$ a Hilbert space. In that case,
we have $A^*=A^{\w_{\can}}$ with the canonical strong symplectic form $\w_{\can}$ defined in
\eqref{e:canonical-product-sympform}. In this Memoir, our applications deal with elliptic operators
on manifolds with boundary which induce other symplectic forms than the strong form $\w_{\can}$ for
dealing with adjoints. In so far, defining adjoint relations by $A^\w$ is more flexible than the
standard definition $A^*$.

\noi (b) As explained in Section \ref{ss:ann}, Example \ref{ex:double-annihilator} the double
adjoint $A^{\w\w}$ of a closed linear relation $A$ is not necessarily the original $A$, unless the
form$\w$ is strong symplectic.

\noi (c) For strong symplectic $\w$, the adjoint $A^\w$ of a Fredholm relation $A$ is again a
Fredholm relation, and we have $\Index A +\Index A^\w=0$. For weak symplectic forms radically new
features appear: (i) The adjoint $A^\w$ of a Fredholm relation $A$ is not necessarily a Fredholm
relation. (ii) Even when it is a Fredholm relation, we have not necessarily $\Index A +\Index
A^\w=0$. (iii) As explained in Section \ref{s:basic-concepts}, Example \ref{ex:negative-index},
the index of a Fredholm operator or relation that is self-adjoint relative to a weak form $\w$ does
not necessarily vanish.
\end{rem}

\subsection{Natural coincidence of spectral flow and Maslov index}\label{ss:coincidence}

Let $X$ be a Hilbert space. Clearly, a closed operator $A\colon X\supset\dom(A)\to X$ is
{symmetric}, respectively {self-adjoint}, if and only if $\Graph(A)$ is a {symmetric}, respectively
{self-adjoint} closed linear relation. Inspired by \auindex{Boo{\ss}--Bavnbek,\ B.}\auindex{Lesch,\
M.}\auindex{Phillips,\ J.}\cite[Theorem 1.1b and Remark 1.4]{BoLePh01} we have the following
results about the Cayley image of various types of closed relations in Hilbert space.

\begin{lemma}\label{l:inverse-bounded}
Let $X$ be a Hilbert space and $A\in \Ss(X\times X)$ a symmetric relation. Then we have:
\newline (a) The {\em Cayley transform}%
\subindex{Cayley transform}\symindex{\kappa@$\kappa$ Cayley transform}
\[
\kappa(A)\ :=\ (A-iI_X)(A+iI_X)\ii%
\]
is a well-defined operator; it is the graph of a partial isometry on $X$. Moreover, we have
$\sigma(\kappa(A))=\kappa(\sigma(A))$, where $\kappa\colon z\mapsto\frac{z-i}{z+i}$.
\newline (b) If $A$ is self-adjoint, then $\kappa(A)\in\Bb(X)$ and it is unitary.
\newline (c) Moreover, in that case we have $\sigma(A)\subset\R$.
\newline (d) The number $-1$ is discrete in $\sigma(\kappa(A))\cup\{-1\}$ if $A$ is self-adjoint Fredholm.
\end{lemma}

\begin{proof} (a) and (b) We have a symplectic decomposition $X\times X=X^-\oplus X^+$ with
\begin{equation*}
X^{\mp}:=\{(x,\pm ix);x\in X\}.
\end{equation*}
Define the relation $U:=\{(y,z)\in X\times X;(y,iy) + (z,-iz)\in A\}$. By
\auindex{Boo{\ss}--Bavnbek,\ B.}\auindex{Zhu,\ C.}\cite[Lemma 3]{BooZhu:2013}, $U$ is a partial
isometry. If $A$ is self-adjoint, $\dom(U)=\ran U=X$ and $U\in\Bb(X)$ is unitary.

Given $(x,x')\in A$\/, $x,x'\in X$, we decompose
\[
(x,x') = (y,iy) + (z,-iz) \text{ with suitable $y,z\in X$}.
\]

We obtain at once $y=\frac{x-ix'}2\tand z=\frac{x+ix'}2$, or, equivalently in the language of
linear relations,
\begin{equation}\label{e:basic-relations}
(x,y)\in \frac{I_X-iA}2 \tand (x,z)\in \frac{I_X+iA}2 \,.
\end{equation}
Inverting the $y$--formula in \eqref{e:basic-relations} yields
\begin{equation}\label{e:U-expressing}
(y,z)\in (I_X-iA)\ii (I_X+iA).
\end{equation}
Conversely, if (\ref{e:U-expressing}) holds, we get $(x,x')\in A$. Then we have $U=(I_X-iA)\ii
(I_X+iA)=-\kappa(A)$. By functional calculus we have $\sigma(\kappa(A))=\kappa(\sigma(A))$.
\newline (c) and (d) By (a), (b).
\end{proof}

We have the following \subindex{Basic coincidence of spectral flow and Maslov index}basic
coincidence of spectral flow and Maslov index.

\begin{proposition}\label{p:basic-sff}
Let $\mathbb{X}\to [0,1]$ be a Hilbert bundle with $X(s):=p^{-1}(s)$, continuous varying inner
product $\lla\cdot,\cdot\rra_s$ on $X(s)$, and $\{s\mapsto A(s)\in \Ss^c(X(s)\times
X(s))\}_{s\in[0,1]}$ a continuous curve of self-adjoint Fredholm relations of $X(s)$. Then we have
\[
\SF\{A(s)\}\ =\ \Mas_-\left\{A(s),X(s)\times\{0\};\w_{\can}(s)\right\},
\]
where $\w_{\can}(s)$ denotes the canonical strong symplectic form on $X(s)\times X(s)$.
\end{proposition}

\begin{proof} Note that $X(s)\times\{0\}=\Graph(0)$. By Lemma \ref{l:inverse-bounded} and Definition \ref{d:mas-hilbert} we have
\begin{align*}
\SF\{A(s)\}\ &=\ \SF_{\ell_+}\{-\kappa(A(s))\}\ =\ \SF_{\ell_+}\left\{(-\kappa(A(s))(-\kappa(0))^{-1}\right\}\\
&=\ \Mas_-\left\{A(s),X(s)\times\{0\};\w_{\can}(s)\right\}.\qedhere
\end{align*}
\exendproof

The proposition leads to the following definition.

\begin{definition}\label{d:sf} Let $p\colon \mathbb{X}\to[0,1]$, $q\colon \mathbb{Y}\to[0,1]$ be
Banach bundles with fibers $X(s):=p^{-1}(s)$, $Y(s):=q^{-1}(s)$ for each $s\in[0,1]$ respectively.
Let $\Omega(s)\colon X(s)\times Y(s)\to\C$ be a path of bounded non-degenerate sesquilinear forms,
and let $\omega(s)$ denote the weak symplectic structure on $Z(s):=X(s)\times Y(s)$ induced by
$\Omega(s)$. Let $A(s)$, $s\in[0,1]$ be a path of linear self-adjoint Fredholm relations of index
$0$. By \auindex{Boo{\ss}--Bavnbek,\ B.}\auindex{Zhu,\ C.}\cite[Lemma 16]{BooZhu:2013}, we have
$\Index(A(s),X(s)\times\{0\})=0$. The {\em spectral flow} of $A(s)$ is defined by
\begin{equation}\label{e:definition-sf-selfadjoint}
\symindex{sf@$\SF\{A(s)\}$ spectral flow of a curve of self-adjoint Fredholm relations}
\subindex{Spectral flow!of a curve of self-adjoint Fredholm relations}
\SF\{A(s)\}\ :=\ \Mas_-\left\{A(s),X(s)\times\{0\};\w(s)\right\}.
\end{equation}
\end{definition}

\begin{rem}[Reversing the order between spectral flow and Maslov index]\label{r:order-reversing}
\subindex{Maslov index!and spectral flow!logical order}\subindex{Spectral flow!and Maslov
index!logical order}
(a) In modern times, any rigorous definition
of the spectral flow begins with the careful partitioning and local hedging of eigenvalues of
\auindex{Phillips,\ J.}J. Phillips \cite{Ph96}. From that, in \auindex{Boo{\ss}--Bavnbek,\
B.}\auindex{Furutani,\ K.}\cite{BoFu98} and followers, a rigorous definition of the Maslov index is
derived for curves of Fredholm pairs of Lagrangian subspaces in strong symplectic Hilbert space via
the spectral flow of an associated curve of unitary generators. That also is the path we chose in
Chapter \ref{s:maslov-hilbert}, followed by our definition of the Maslov index in weak symplectic
Banach spaces in Chapter \ref{s:maslov-general} via symplectic reduction to the finite-dimensional
strong case. So far, we followed the usual view which considers the concept of the spectral flow
for more fundamental than the concept of the Maslov index which is derived from it.

However, with the preceding definition we reverse the order: Now the spectral flow for curves of
self-adjoint Fredholm relations is defined via the general Maslov index as it was introduced in
Chapter \ref{s:maslov-general}. In this way it turns out at the bottom line that the concept of the
Maslov index now becomes more fundamental than the spectral flow which has to be defined via the
Maslov index.

\noi (b) It is informative to compare the present Definition \ref{d:sf} of the spectral flow with
its intricate definition in \auindex{Boo{\ss}--Bavnbek,\ B.}\auindex{Zhu,\
C.}\cite[Appendix]{BooZhu:2013}. In one respect, our present definition is more general than our
previous one since we dealt in \cite[Appendix]{BooZhu:2013} only with relations in $X\times X$
while we admit now relations in $X\times Y$ with possibly $Y\ne X$. In another perhaps more
relevant respect, our present definition is much less general since we restrict ourselves to
self-adjoint relations relative to the choice of a curve $\W(s)$ of non-degenerate sesquilinear
forms. This restriction explains why we here can avoid the intricate considerations regarding the
spectral admissibility of the curves of Fredholm relations of our previous paper.

As shown in Proposition \ref{p:basic-sff}, the two definitions coincide for curves of self-adjoint
Fredholm relations in complex Hilbert space.
\end{rem}

\section{Symplectic analysis of operators and relations}\label{ss:sa-extensions}

We are now ready to follow the \auindex{Neumann, von,\ J.}famous \subindex{von Neumann program}von
Neumann program \cite{Neu} of investigating all self-adjoint extensions for our special case of
Fredholm relations. We are inspired by the \auindex{Birman,\ M.S.}\auindex{Krein,\
M.G.}\auindex{Vishik,\ M.I.}\subindex{Birman-Kre\u\i n-Vishik theory}{B}irman-{K}re\u\i n-{V}ishik
theory of \subindex{Extension!self-adjoint}self-adjoint extensions of semi-bounded operators (see
the review \cite{Alonso-Simon} by \auindex{Alonso,\ A.}\auindex{Simon,\ B.}A. Alonso and B. Simon),
modified in \auindex{Boo{\ss}--Bavnbek,\ B.}\auindex{Furutani,\ K.}\auindex{Otsuki,\
N.}\cite{BoFu98,BoFu99,BoFuOt01} for the spectral theory of curves of self-adjoint Fredholm
extensions of symmetric operators on a Hilbert space. As explained in the Introduction (see also
our Remark \ref{r:why-we-drop-beta} below), we need a slightly broader setting for our
applications, based on our concept of $\w$-symmetric linear relations.

To begin with, we shall determine a few preconditions for our calculus with adjoint Fredholm
relations in a weak symplectic setting. Let $X$, $Y$ be two complex vector spaces and let
$\Omega\colon X\times Y\to \C$ be a non-degenerate sesquilinear map. Set $Z:=X\times Y$. Let $\w$
denote the symplectic form induced by $\W$ and defined by (\ref{e:Omega-omega}). As observed there,
$(Z,\omega)$ is a symplectic vector space with two Lagrangian
subspaces $X\times\{0\}$ and $\{0\}\times Y$. We shall determine conditions to transfer the classical dualities of kernel and cokernel
from adjoint Fredholm operators to adjoint linear relations. As emphasized above (see also Remark
\ref{r:vanishing-index}), contrary to the easy calculus with strong symplectic forms, for a weak
symplectic form $\w$ the double adjoint $A^{\w\w}$ of a closed linear relation $A$ is not
necessarily the original $A$, and the adjoint $A^\w$ of a Fredholm relation $A$ is not necessarily
a Fredholm relation. Even when it is a Fredholm relation, in the weak symplectic case we have not
necessarily $\Index A +\Index A^\w=0$ and the index of a self-adjoint Fredholm relation does not
necessarily vanish. Finally, we must recall from Remark \ref{r:vanishing-index}.c, that the index
of a Fredholm operator in Banach space that is self-adjoint relative to a weak form $\w$ does
neither necessarily vanish.

That explains why we need special assumptions to exclude intractable complications with index
calculations for self-adjoint Fredholm relations. Fortunately it turns out that these special
assumptions are naturally satisfied in our applications. That is what this section is about.

\begin{proposition}[Calculus with adjoint Fredholm relations]\label{p:sum-0}
\subindex{Fredholm relation!dimension calculus with adjoint Fredholm relation}
Let $A\subset W\subset Z$ be two linear relations. Assume that $A,A^\w$ are algebraic Fredholm
relations with $\Index A+\Index A^{\w}=0$. Then we have $\dim\ker W^{\w}=\dim Y/\ran W$ and $\ran
W=\ran W^{\w\w}$.
\end{proposition}

\begin{proof} We apply Proposition \ref{p:sum-index-0}, taking our $Z$ as the underlying
symplectic vector space and setting $\la:=A$ and $\mu:=X\times\{0\}$. Then we have $\dim(W^{\w}\cap
(X\times\{0\}))=\dim(Z/(W+X\times\{0\}))$ and $W+X\times\{0\}=W^{\w\w}+X\times\{0\}$. By
\eqref{e:clr-index} and \eqref{e:fredholm}, our results follow.
\end{proof}

\smallskip

Let $X,Y$ be two complex Banach spaces and $\W\colon X\times Y\to\C$ a bounded non-degenerate
sesquilinear map with induced symplectic form $\w$ on $X\times Y$. Let $\Cc(X,Y)$ denote the space
of closed linear operators from $X$ to $Y$ and let \symindex{A@$A_m$ closed symmetric
operator}$A_m\in\Cc(X,Y)$ with $(\dom(A_m))^{\Omega,r}=\{0\}$ with this right annihilator defined
in \ref{e:r-annihilator}. (If $X=Y$ a Hilbert space with $\W:=\lla\cdot,\cdot\rra$, the vanishing
of the right annihilator of the domain of a closed linear operator $A_m$ means just that $A_m$ is
densely defined. We shall come back to this condition in our Assumption \ref{a:reduced-space}.2
for a revised von-Neumann setting of abstract boundary value problems and Assumption \ref{a:asff}.2
for the proof of our abstract spectral flow formula.) We consider $A_m$ as a closed linear
relation. Then the adjoint relation \symindex{A@$A_m^\w$ adjoint operator (maximal
extension)}$A_m^{\w}$ is (the graph of) a closed operator. We assume that $A_m$ is symmetric, i.e.,
$\Graph (A_m)$ is an isotropic subspace of the symplectic product space $(X\times Y,\w)$; in
operator notation that means shortly $A_m^{\w}\>A_m$. Then $A_m^{\w}$ is a closed co-isotropic
subspace of $(X\times Y,\w)$ and by Lemma \ref{l:sym-red-banach}, the quotient space
$\left(\Graph(A_m^{\w})/\Graph(A_m^{\w\w});\wt\w\right)$ is a naturally symplectic Banach space
with the reduced form $\wt\w$ induced by $\w$. We denote the domains of $A_m$ by \symindex{Dm@$D_m$
minimal domain}$D_m$ (the {\it minimal} domain) and of $A_m^{\w}$ by \symindex{Dm@$D_{\mmax}$
maximal domain}$D_{\mmax}$ (the {\it maximal} domain). For these data, we have, as in
\auindex{Boo{\ss}--Bavnbek,\ B.}\auindex{Furutani,\ K.}\cite{BoFu98,BoFu99}:

\begin{enumerate}

\item[(1)]
The space $D_{\mmax}$ is a Banach space with the \subindex{Graph norm}graph norm
\begin{equation}\label{e:graph-norm}
\symindex{x@$\norm{x}_{\GGG}$ graph norm}
\Vert x\Vert_{\GGG}\ :=\ \|x\|_X + \|A_m^{\w}x\|_Y \quad\text{ for $x\in D_{\mmax}$}\,.
\end{equation}

\item[(2)] The space $D_m$ is a closed subspace in the graph norm and the quotient
space
\subindex{von-Neumann space of abstract boundary values}\symindex{Dm@$D_{\mmax}/D_m$ von-Neumann
space of abstract boundary values}
$D_{\mmax}/D_m$ is a Banach space with the minus Green's form
\begin{equation}\label{e:symplectic-minus-green1}
\symindex{\omega s@$\w(s)_{\operatorname{Green}}$ by Green's form induced symplectic form}
-{\w}_{\operatorname{Green}}(x+D_m,y+D_m)\ :=\
\Omega(x,A_m^{\w}y)-\overline{\Omega(y,A_m^{\w}x)}\text{ for }x,y\in D_{\mmax}\,,
\end{equation}
where both $A_m$ and $\W$ enter into the definition. The form is symplectic if and only if
$A_m^{\w\w}=A_m$.

\item[(3)] Let $B\supset A_m$ be an \subindex{Extension!self-adjoint}extension of $A_m$. By Lemma \ref{l:red-lagrangian}, the operator
$B$ is self-adjoint if and only if $A_m^{\w\w}\subset B\subset A_m^{\w}$, and for the symplectic
reduction of $\Graph(B)\<X\times Y$ via the co-isotropic $\Graph(A_m^{\w})$, there holds
\[
R_{\Graph(A_m^{\w})}(\Graph(B))\ \in\ \Ll\left(\Graph(A_m^{\w})/\Graph(A_m^{\w\w});\wt\w\right).
\]

\item[(4)] We denote by \subindex{Trace map}\symindex{\gamma@$\g$ abstract trace map}$\g$ the natural projection
\[
\g\colon D_{\mmax}\too D_{\mmax}/D_m\/.
\]
For any linear subspace $D\subset X$, we set
\[\g(D)\ :=\ (D\cap D_{\mmax}+D_m)/D_m.\]
\end{enumerate}

\begin{rem}\label{r:why-we-drop-beta}
In our applications, we consider families of self-adjoint Fredholm operators with \textit{varying
domain} and \subindex{Continuously varying!maximal domains}\textit{varying maximal domain}. To us,
there is no natural way to identify the different symplectic spaces and to define continuity of
Lagrangian subspaces and continuity of symplectic forms in these varying symplectic spaces.
Fortunately, in most applications the minimal domain is fixed and also an intermediate Hilbert
space \symindex{Dm@$D_M$ intermediate Banach space}$D_M$\,, typically the Sobolev space $H^d$ for
elliptic differential operators of order $d$.
\end{rem}

We shall show that meaningful modifications of the preceding statements can be obtained when we
replace $D_{\mmax}$ by this intermediate space $D_M$ under the following assumptions.

\begin{ass}[Revised von-Neumann setting for abstract boundary value problems]\label{a:reduced-space}
\subindex{Revised von-Neumann setting}

\noi (1) As in the preceding statements, we let $X$, $Y$ be two Banach spaces and $\Omega\colon
X\times Y\to \C$ a bounded non-degenerate sesquilinear map. We set $Z:=X\times Y$ and let $\w$ be
defined by (\ref{e:Omega-omega}).
\newline (2) Our data are now four Banach spaces with continuous inclusions
\[
\symindex{Dm@$D_M$ intermediate Banach space}
D_m\into D_M \into D_{\mmax} \into X,
\]
where the Banach space structure is given on $D_{\mmax}$ and $D_m$ by the graph inner product of a
fixed closed symmetric operator $A_m\in\Cc(X,Y)$ with $\dom(A_m)=D_m$. Assume that
$(D_m)^{\Omega,r}=\{0\}$.
\newline (3) We assume that the $(A_M)^{\w}=A_m$, where \symindex{Am@$A_M$ intermediate extension}
\subindex{Extension!fixed intermediate domain}$A_M:=A_m^{\w}|_{D_M}$.
\newline (4) Finally, we assume that there exists a \subindex{Extension!self-adjoint Fredholm}self-adjoint
Fredholm extension \symindex{AD@$A_D$ extension (realization) of operator $A$ with domain
$D$}\subindex{Extension!self-adjoint Fredholm}$A_D$ of $A_m$ of index $0$ with domain $D_m\< D \<
D_M$\,.
\end{ass}

Assumption \ref{a:reduced-space}.2 implies
\begin{equation}\label{e:am-bounded}
\norm{x}_{\GGG} = \norm{x}_X + \norm{A_Mx}_X\leq C_1\norm{x}_{D_M} \text{ for all }x\in D_M\,.
\end{equation}
In particular, it follows that $A_M \colon  D_M\to X$ is bounded.

We have an injection $j:D_M\to D_{\max}$ with $D_m\subset D_M$ and $D_m$ closed in $D_{\max}$\/.
Therefore, by Proposition \ref{p:embedding-continuous1}.a, $D_m$ is closed in $D_M$, and on $D_m$
the graph norm and the norm induced by the Banach space $D_M$ are equivalent. Then we have the
opposite estimate to \eqref{e:am-bounded}, namely a \subindex{G{\aa}rding's Inequality}G{\aa}rding
type inequality, known from the study of \subindex{Elliptic differential
operators!regularity}elliptic regularity:
\begin{equation}\label{e:am-and-gaarding}
\norm{x}_{D_M} \leq C_2\bigl( \norm{x}_X + \norm{A_Mx}_X\bigr) = C_2\norm{x}_{\GGG} \text{ for all
}x\in D_m\,.
\end{equation}

By (\ref{e:three-Omega}), Assumption \ref{a:reduced-space}.3 implies $A_m^{\w\w}=A_m$.

\begin{lemma}\label{l:weak-symplectic-boundary-space}
Under Assumption \ref{a:reduced-space}.1-3, the quotient space \symindex{Dm@$D_M/D_m$
reduced weak symplectic von-Neumann space}\subindex{von-Neumann space of abstract boundary
values!reduced}$D_M/D_m$ is a weak symplectic Banach space with the symplectic form
$-\w_{\operatorname{Green}}$ induced by the minus Green's form on $D_{\mmax}$ defined in
\eqref{e:symplectic-minus-green1}.
\end{lemma}

\begin{proof} By Lemma \ref{l:sym-red-banach}.
\end{proof}

The lemma shows that any intermediate space $D_M$ satisfying Assumption \ref{a:reduced-space}.1-3
is big enough to permit a meaningful symplectic analysis on the reduced quotient space
$D_M/D_m$\,. The point of this construction is that the norm in $D_M/D_m$ does not come from the
graph norm in $D_{\mmax}$ but from the norm of $D_M$\,. Therefore, it can be kept fixed even when
our operator varies. The {\em symplectic structure} of $D_M/D_m$\,, however, is induced by the
minus Green's form and therefore will change with varying operators.

In \auindex{Boo{\ss}--Bavnbek,\ B.}\auindex{Furutani,\ K.}\cite[Proposition 3.5]{BoFu98}, in the
spirit of the classical von-Neumann program, self-adjoint Fredholm extensions were characterized by
the property that their domains, projected down into the strong symplectic space
\symindex{\beta@$\beta(A)$ space of abstract boundary values}$\bbb(A_m):=D_{\mmax}/D_m$ of abstract
boundary values, make Fredholm pairs of Lagrangian subspaces with the \subindex{Cauchy data
space!abstract (or reduced)}abstract, reduced \subindex{Cauchy data space!abstract (or
reduced)}Cauchy data space $(\ker A_m^*+D_m)/D_m$\/. Immediately, this does not help for operator
families with varying maximal domain. Surprisingly, however, the arguments generalize to the weak
symplectic space \symindex{Dm@$D_M/D_m$ reduced weak symplectic von-Neumann
space}\subindex{von-Neumann space of abstract boundary values!reduced} intrinsically, i.e., without
additional topological conditions.

For the following Lemma \ref{l:red-cauchy} and Proposition \ref{p:cauchy} we exploit our Assumption
\ref{a:reduced-space}.4, i.e., the existence of a self-adjoint Fredholm extension $A_D$\/.

\begin{lemma}\label{l:red-cauchy} Denote by $P_X\colon  Z=X\times Y\to X$ the projection onto the first component.
Set $W:=\Graph(A_M)$, $\la:=\Graph(A_D)$ and $\mu:=X\times\{0\}$. Then $P_X$ induces a symplectic
Banach isomorphism $\wt P_X\colon  \bigl(W/W^{\w},\wt \w\bigr)\to \bigl(D_M/D_m,-\w_{\Green}\bigr)$,
and we have
\[\wt P_X(R_W(\la))\ =\ \g(D),\qquad \wt P_X(R_W(\mu))\ =\ \g(\ker A_M).\]
Here $\w$ denotes the symplectic structure on $X\times Y$ induced by the given non-degenerate
sesquilinear form $\W\colon X\times Y\to\C$ and defined in \eqref{e:Omega-omega}, $\wt\w$ denotes
the corresponding symplectic structure on the reduced space, and $-\w_{\Green}$ denotes the form
defined in Equation \eqref{e:symplectic-minus-green1} and established as symplectic form on
$D_M/D_m$ in Lemma \ref{l:weak-symplectic-boundary-space}.

\end{lemma}

\begin{proof} By definition and direct calculation.
\end{proof}

\begin{proposition}[Abstract regularity]\label{p:cauchy}
\subindex{Abstract regularity}\subindex{Cauchy data space!abstract (or reduced)} Under
Assumption \ref{a:reduced-space}, the quotient space $D/D_m$ and the {\em reduced Cauchy data
space} $(\ker A_M+D_m)/D_m$ form a Fredholm pair of Lagrangian subspaces of the (weak) symplectic
Banach space $(D_M/D_m,\w_{\Green})$ with index $0$, and $\dim\ker A_m=\dim Y/(\ran A_M)$.
Moreover, it follows that $\range A_M=\range A_m^{\omega}$\,.
\end{proposition}

\begin{proof} Take $Z$ as the symplectic vector space, $W:=\Graph(A_M)$, $\la:=\Graph(A_D)$ and $\mu:=X\times\{0\}$.
By Lemma \ref{l:red-cauchy} and Proposition \ref{p:red-fredholm}, our results follow.
\end{proof}

\begin{rem}\label{r:alg}
(a) By definition it is clear that $\range A_M\<\range A_m^{\omega}$\,. The point of the preceding
Proposition is the opposite inclusion. That inclusion is well known for any elliptic differential
operator $A$ acting from sections of a bundle $E$ over a compact manifold $M$ to sections of a
bundle $F$ over $M$, say of order $d=1$: Let us denote the maximal extension of $A$ on $L^2(X;E)$
by $A_{\mmax}$.

Then, as one aspect of the Lifting Jack (the existence of parametrices) for \subindex{Elliptic
differential operators!regularity}elliptic operators, for any $w\in\range A_{\mmax}$ we have by
definition a $v\in D_{\mmax}$ and therefore by elliptic regularity a $v'\in D_M$ such that
$A_{\mmax} v=A_M v'=w$. In general, there will be many different $v \in D_{\mmax}$ with $A_{\mmax}
v = w$, and not all such $v$ belong to $D_M$, but some of them will. In the classical theory of
well-posed elliptic boundary value problems, those are typically elements that satisfy an extra
condition at the boundary (as specified in \auindex{Boo{\ss}--Bavnbek,\ B.}\auindex{Wojciechowski,\
K.P.}\cite[Theorem 19.6]{BoWo93} for elliptic operators of first order and in \auindex{Frey,\ C.}
\cite[Theorem 2.2.1]{Frey:2005} for elliptic operators of higher order).

\noi (b) If we remove the topological requirements in Assumption \ref{a:reduced-space}, the
algebraic results of Lemma \ref{l:weak-symplectic-boundary-space}, Lemma \ref{l:red-cauchy} and
Proposition \ref{p:cauchy} still hold.
\end{rem}

\section{Proof of the abstract spectral flow formula}\label{ss:proof-asff}

In this section we prove an abstract spectral flow formula by Theorem \ref{t:mas-red}. We shall
make the following new assumptions. They are all natural in our applications, as we shall see later
in Section \ref{ss:gsff}.

\begin{ass}\label{a:asff}
\symindex{G@$\mathbb{G},\mathbb{G}_0,\mathbb{X},\mathbb{Y}$ Banach bundles over $[0,1]$}
(1) Let $r_0\colon \mathbb{G}_0\to[0,1]$, $r\colon \mathbb{G}\to [0,1]$, $p\colon \mathbb{X}\to
[0,1]$, and $q\colon \mathbb{Y}\to [0,1]$ be Banach bundles with fibers $r_0^{-1}(s):=D_m(s)$,
$r^{-1}(s):=D_M(s)$, $p^{-1}(s):=X(s)$ and $q^{-1}(s):=Y(s)$ for each $s\in[0,1]$ respectively.
Assume that we have Banach subbundle maps $\mathbb{G}_0\to \mathbb{G}$, $\mathbb{G}\to \mathbb{X}$.
\newline (2) Let
\[
\left\{\Omega(s)\colon X(s)\times Y(s)\too\C\right\}_{s\in [0,1]}
\]
be a path of bounded non-degenerate sesquilinear forms, and denote by $\omega(s)$ the weak
symplectic structure on $Z(s):=X(s)\times Y(s)$ induced by $\Omega(s)$ and defined by
\eqref{e:Omega-omega}. Assume that $(D_m(s))^{\Omega,r}=\{0\}$, which implies that the adjoint of
an operator with domain $D_m(s)$ is an operator.
\newline (3) Let
\[
\left\{A_m(s)\colon X(s)\supset D_m(s)\too Y(s)\right\}_{s\in [0,1]}
\]
be a family of closed symmetric operators such that the norm on $D_m(s)$ is equivalent to the graph
norm of $D_m(s)$ defined by $A_m(s)$. Assume that there exists a constant integer $k$ such that
$\dim\ker A_m(s)=k$, i.e., we assume \subindex{Weak inner Unique Continuation Property (wiUCP)}weak
inner unique continuation property (wiUCP) up to a finite constant dimension. Set
$A_M(s):=A_m(s)^{\omega(s)}|_{D_M(s)}$. Assume that $(A_M(s))^{\omega(s)}=A_m(s)$, and
\[
\left\{A_M(s)\in\Bb(D_M(s),Y(s))\right\}_{s\in [0,1]}
\]
is a path of bounded operators.
\newline (4) Let $\left\{D(s)\in\Ss(D_M(s))\right\}_{s\in [0,1]}$ be a path of closed subspaces with $D_m(s)\subset D(s)\subset D_M(s)$.
Assume that $\left\{A(s,D(s)):=A_M(s)|_{D(s)}\right\}_{s\in [0,1]}$ is a family of self-adjoint
Fredholm operators of index $0$.
\end{ass}

\begin{theorem}[Abstract spectral flow formula]\label{t:asff}
\subindex{Desuspension spectral flow formula!abstract}
Under Assumption \ref{a:asff}, we have the following.
\newline (a) We have $\ran A_M(s)=\ran (A_m(s))^{\omega(s)}$ and $\dim Y(s)/(\ran A_M(s))=k$ holds for each $s\in[0,1]$.
\newline (b) The family
\[
\left\{\bigl(\gamma(D(s)),\gamma(\ker A_M(s))\bigr)\right\}_{s\in [0,1]}
\]
is a path of Fredholm pairs of Lagrangian subspaces of the symplectic Banach space
$\bigl(D_M(s)/D_m(s),\w_{\Green}(s)\bigr)$ of index $0$.
\newline (c) We have
\begin{align}
\label{e:asff1}\SF\{A(s,D(s))\}&=-\Mas\{\gamma(D(s)),\gamma(\ker A_M(s));\w_{\Green}(s)\}\\
\label{e:asff2}&=-\Mas\{\gamma(\ker A_M(s),\gamma(D(s)));-\w_{\Green}(s)\}.
\end{align}
\end{theorem}

\begin{proof} Let $s\in[0,1]$. By \auindex{Boo{\ss}--Bavnbek,\ B.}\auindex{Zhu,\ C.}\cite[Lemma 16]{BooZhu:2013}, the pair
\[(\Graph(A(s,D(s)),X(s)\times\{0\})\]
is a Fredholm pair of Lagrangian subspaces of the symplectic Banach space $Z(s)$ with the form
$\omega(s)$. By Lemma \ref{l:weak-symplectic-boundary-space}, the quotient space $D_M(s)/D_m(s)$ is
a weak symplectic Banach space with the symplectic form induced by the minus Green's form on
$D_{\mmax}(s)$. By Proposition \ref{p:cauchy}, (a) holds and the pair
\[
\bigl(\gamma(D(s)), \gamma(\ker A_M(s))\bigr)%
\]
is a Fredholm pair of Lagrangian subspaces of the symplectic Banach space%
\[
(D_M(s)/D_m(s),\w_{\Green}(s))\text{ of index $0$ for each $s\in[0,1]$}.%
\]
By Proposition \ref{p:embedding-continuous1}, the norm on $D_m(s)$ is uniformly equivalent to the
graph norm of $D_m(s)$ defined by $A_m(s)$.
\newline By Lemma \ref{l:quotient-delta}, $\left\{\gamma(D(s))\subset\Ss(D_M(s)/D_m(s))\right\}_{s\in
[0,1]}$ is a continuous path.
\newline By Corollary \ref{c:continuous-operator}, $\left\{A(s,D(s))\in\Cc(X(s),Y(s))\right\}_{s\in
[0,1]}$ is a continuous family.

We shall use the following notations:
\begin{gather*}
W_0(s):=\Graph(A_m(s)),\ W(s):=\Graph(A_M(s)),\ W_1(s):=D_m(s)\times Y(s)\\
\wt W(s):=D_W(s)\times Y(s),\quad\mu(s):=X(s)\times\{0\},\\
\mathbb{W}_0:=\bigcup_{s\in[0,1]}W_0(s),\quad \mathbb{W}:=\bigcup_{s\in[0,1]}W(s),\quad \mathbb{W}_1:=\bigcup_{s\in[0,1]}W_1(s),\\
\wt{\mathbb{W}}:=\bigcup_{s\in[0,1]}\wt W(s),\quad\mathbb{Z}:=\bigcup_{s\in[0,1]}Z(s).
\end{gather*}
Then we have $W_0(s)$ is closed and complemented in $W_1(s)$, in short, $W_0(s)\in\Ss^c(W_1(s))$,
and similarly $W(s)\in\Ss^c(\wt W(s))$. By Lemma \ref{l:ck-complemented}.b (see also
\auindex{Neubauer,\ G.}\cite[Lemma 0.2]{Ne68}), $\mathbb{W}_0$ is a subbundle of $\mathbb{W}_1$,
and $\mathbb{W}$ is a subbundle of $\wt{\mathbb{W}}$. The bundle $\wt{\mathbb{W}}$ is a subbundle
of $\mathbb{Z}$, and we have $\wt W(s)+\mu(s)=Z(s)$. By Definition \ref{d:sf}, Theorem
\ref{t:mas-red}, Lemma \ref{l:red-cauchy} and Proposition \ref{p:maslov-properties}.d,f, we have
\begin{align*}
\SF&\{A(s,D(s))\}=\Mas_-\{\Graph(A(s,D(s)),X(s)\times\{0\})\}\\
&=\Mas_-\{R_W(s)^{\omega(s)}(\Graph(A(s,D(s))),R_W(s)^{\omega(s)}(X(s)\times\{0\})\}\\
&=\Mas_-\{\gamma(D(s)),\gamma(\ker A_M(s));-\w_{\Green}(s)\}\\
&=-\Mas\{\gamma(\ker A_M(s)),\gamma(D(s));-\w_{\Green}(s)\}\\
&=-\Mas\{\gamma(D(s)),\gamma(\ker A_M(s));\w_{\Green}(s)\}.\qedhere
\end{align*}
\exendproof

\begin{rem}\label{r:simple_case} Consider the special case when $X(s)$ is a Hilbert space with Hilbert structure $\Omega(s)$ for each $s$. If
we assume that $\dim\ker A_m(s)=0$ for each $s$, the method in the proof of \cite[Theorem
2.13]{BoZh04}, \cite[Theorem 1.3]{BoZh05} works here. The method is not applicable in the case that
$\dim\ker A_m(s)=k>0$ for each $s$.
\end{rem}
\medskip

\section[Desuspension spectral flow formula for elliptic problems]{An application:
A general desuspension formula for the spectral flow of families of elliptic boundary value
problems}\label{ss:gsff} Having expanded weak symplectic linear algebra and analysis to some length
and detail in the two preceding sections, we shall turn to the geometric setting and the geometric
applications.

\subsection{Parametrization of vector bundles over manifolds with boundary and domains in Sobolev chains}\label{sss:sobolev_spaces}
Consider a (big) Hermitian vector bundle \symindex{E@$\EE$ (big) Hermitian vector bundle
over a (big) compact Hausdorff space $\MM$}$\EE$ over a \symindex{M@$\MM$ (big) compact Hausdorff
space}(big) compact Hausdorff space $\MM$. We assume that $\MM$ itself is a fiber bundle over the
interval $[0,1]$ such that $\MM$ is a continuous family of compact smooth Riemannian manifolds
\symindex{M@$M,M(s)$ smooth compact manifold!with boundary $\Si,\Si(s)$}\symindex{j@$j,j(s)$
embedding}$j(s)\colon M(s) \into \MM$ with boundary \symindex{\Si@$\Si,\Si(s)$ smooth closed
boundary part}$\Si(s)$, \symindex{s@$s$ variational parameter}$s\in [0,1]$. We require that the
\subindex{Vector bundle structure}vector bundle structure is compatible with the boundary part.
More precisely, we shall have a \subindex{Trivialization}trivialization
\[\symindex{\f@$\f$ trivialization}
\f\colon M(0)\times [0,1]\simeq \MM
\]
such that \symindex{\pi@$\pi$ natural projection}$\pi\circ\f\ii\circ j(s)\colon  (M(s),\Si(s))\to
(M(0),\Si(0))$ is a diffeomorphism. Here $\pi$ denotes the natural projection $\pi\colon M(0)\times
[0,1]\to M(0)$\,. We do not assume that $M(s)$ or $\Si(s)$ are connected. Note that the
trivialization defines smooth structures on $\ran\f|_{(M(0)\setminus \Si(0))\times [0,1]}$ and so
on $\MM$ and $\EE$.

Let $E(s)\to M(s)$ be the induced bundle, i.e., the pull back $j(s)^*(\EE)$. Denote by
$\Ci_0(M(s);E(s))$ the space of smooth sections with support in the interior \symindex{M@$M(s)^0$
interior of manifold with boundary}$M(s)^0:=M(s)\setminus \Si(s)$ of $M(s)$\,. Assume that
\symindex{d@$d$ order of differential operator}$d>0$ is a positive integer and
\symindex{\sigma@$\sigma$ order of Sobolev space}$\sigma\ge 0$ a non-positive real (on manifolds
with boundary, Sobolev spaces of negative order are a nuisance and shall be avoided here). We
define the Hilbert (before choices rather ``Hilbertable") space
\[\symindex{H@$H^{\sigma}(M;E), H_0^{\sigma}(M;E)$ Sobolev spaces}\subindex{Sobolev spaces}
  H^{\sigma}_0(M(s);E(s))\ :=\ \ol{\Ci_0(M(s);E(s))}^{H^{\sigma}(M(s);E(s))}\,.
\]
Here $H^{\sigma}(M(s);E(s))$ denotes the Sobolev space of order $\sigma$ defined in
\auindex{Boo{\ss}--Bavnbek,\ B.}\auindex{Wojciechowski,\ K.P.}\cite[Chapter 11]{BoWo93} as the
restrictions to $M(s)$ of sections belonging to $H^{\sigma}(\wt{M(s)};\wt{E(s)})$, where
$\wt{E(s)}\to \wt{M(s)}$ is a smooth extension of the given vector bundle ${E(s)}\to {M(s)}$ over a
smooth closed extension $\wt{M(s)}$ of $M(s)$, e.g., the closed double. C. Frey\auindex{Frey,\ C.}
\cite[p. 14]{Frey:2005} has shown that these definitions of $H^{\sigma}(M(s);E(s)),
H^{\sigma}_0(M(s);E(s))$ coincide with the definitions given in \auindex{Lions,\
J.-L.}\auindex{Magenes,\ E.}J.-L. Lions and E. Magenes \cite[Chapter 9]{Lions-Magenes}. The inner
product is given by the Sobolev inner product. Set
\begin{gather*}
\symindex{D@$D_m(s;\sigma),D_M(s;\sigma)$ two-parameter domains in Sobolev chain}
\symindex{S@$S(s;\sigma)$ two-parameter boundary traces in Sobolev chain}
D_m(s;\sigma):= H^{\sigma+d}_0(M(s);E(s)), \quad D_M(s;\sigma) := H^{\sigma+d}(M(s);E(s)),\\
X(s)=Y(s) := L^2(M(s);E(s)), \quad S(s;\sigma):= \sum_{j=1}^dH^{\sigma+d-j}(\Si(s);E(s)|_{\Si(s)}),\\
\mathbb{G}_0 := \bigcup_{s\in[0,1]}D_m(0,s), \quad \mathbb{G} := \bigcup_{s\in[0,1]}D_M(0,s),\quad
\mathbb{X} := \bigcup_{s\in[0,1]}X(s).
\end{gather*}
Then $\mathbb{G}_0$, $\mathbb{G}$, $\mathbb{X}$ have Banach bundle structures over $[0,1]$, and the
natural inclusions $\mathbb{G}_0\to \mathbb{G}$, $\mathbb{G}\to \mathbb{X}$ are Banach subbundle
maps. So Assumption \ref{a:asff}.1 holds.

By the \subindex{Trace Theorem}trace theorem, for $\sigma>-1/2$ we have
\begin{equation}\label{e:strong-cauchy}
S(s;\sigma+1/2)\ \cong \ D_M(s;\sigma)\/ /\/ D_m(s;\sigma).
\end{equation}

Let \symindex{\W@$\W,\Omega(s)$ $L^2$ inner product}$\Omega(s)\colon X(s)\times X(s)\to\C$ denote
the $L^2$ inner product, and let $\omega(s)$ be the strong symplectic structure on
$Z(s):=X(s)\times X(s)$ induced by $\Omega(s)$. Since $D_m(0,s)$ is dense in $X(s)$, we have
$(D_m(0,s))^{\Omega,r}=(D_m(0,s))^{\bot}=\{0\}$. So Assumption \ref{a:asff}.2 holds. For any
closed operator $T\in\Cc(X(s))$, we have $T^*=T^{\w(s)}$.

\subsection{Curves of elliptic differential operators and their Calder{\'o}n projections and Cauchy data spaces}
We consider a smooth linear differential operator \symindex{A@$\AA$ (big) smooth linear
differential operator}$\AA\colon \Ci(\MM;\EE)\to\Ci(\MM;\EE)$ which induces a smooth family of
elliptic differential operators $A(s)$ of order $d>0$
\begin{equation}\label{e:as-original}
A(s)\colon \Ci_0(M(s);E(s)) \too \Ci(M(s);E(s)).
\end{equation}
For each $s\in[0,1]$ and $\sigma\ge 0$, the operator $A(s)$ extends to a bounded operator
\begin{align}\symindex{A@$A_m(s;\sigma),A_M(s;\sigma)$ bounded two-parameter extensions in Sobolev chain of elliptic differential operator}
\label{e:as-hd1}&A_m(s;\sigma)\colon D_m(s;\sigma) \to H^{\sigma}(M(s);E(s)),\\
\label{e:as-hd2}& A_M(s;\sigma)\colon D_M(s;\sigma) \to H^{\sigma}(M(s);E(s)).
\end{align}
For each $s\in[0,1]$, by the \subindex{Elliptic differential operator!interior elliptic estimate}interior elliptic estimate for $A(s)$, the operator%
\[
A_m(s;0)\colon X(s)\supset D_m(s;0)\ \too\ X(s)%
\]
is a closed operator, and the graph norm on $D_m(s;0)$ defined by $A_m(s;0)$ is equivalent to the
Sobolev norm (see, e.g., \auindex{Boo{\ss}--Bavnbek,\ B.}\auindex{Wojciechowski,\
K.P.}\cite[Proposition 20.7]{BoWo93} for the first order case and \auindex{Frey,\
C.}\cite[Proposition 1.1.1]{Frey:2005} in the higher order case). The family%
\[
\left\{A_M(s;0)\in\Bb(D_M(s;0),X(s))\right\}_{s\in[0,1]}
\]
is a path of bounded operators.

We have the following \cite[Proposition 1.1.2]{Frey:2005}.

\begin{proposition}\label{p:minus_Green_form_diff} Denote by $\nu$ the inner normal vector field on $\Sigma(s)$,
$\nu^{\flat}$ the dual of $\nu$, and $\hat {A(s)}$ the principal symbol of $A(s)$. Then the minus Green's form on $S(s;\sigma)$ is given by
\begin{equation}\label{e:minus_Green_form_diff}-\w_{\Green}(s)(u,v)=\lla J(s)u,v\rra_{L^2(\Si(s);E(s)^d|_{\Si(s)})},
\end{equation}
where $J(s)$ is a matrix of differential operators $J_{k,j}(s))$ of order $d+1-k-j$,
$k,j=1,\ldots,d$. Moreover,
\begin{equation}\label{e:minus_Green_form_J}
J_{k,j}(s)\ =\ \begin{cases} i^d(-1)^{d+1-k}\hat{A(s)}(\nu^{\flat}), & \text{if
$k+j=d+1$},\\
0, & \text{if $k+j>d+1$}. \end{cases}
\end{equation}
\end{proposition}

On $S(s;\sigma)$ we have the family of the minus Green's form
$\left\{-\w_{\Green}(s)\right\}_{s\in[0,1]}$ which is a continuous path of bounded non-degenerate
sesquilinear forms for $\sigma\ge\frac{1-d}{2}$. The forms are invertible for $\sigma=\frac{1-d}{2}$, and they are not
well-defined for $\sigma<\frac{1-d}{2}$. We have $(A_M(s;0))^*=A_m^t(s;0)$, where \symindex{A@$A^t,A^t(s)$
formal adjoint of differential operator}\subindex{Formal adjoint of differential operator}$A^t(s)$
denotes the formal adjoint of $A(s)$.

Let $Q(s;\sigma)\colon S(s;\sigma)\to S(s;\sigma)$ denote the projection defined by the orthogonal
pseudo-differential \subindex{Calder{\'o}n projection}\symindex{Q@$Q,Q(s),Q(s;\sigma)$ Calder{\'o}n
projection}Calder{\'o}n projection $Q(s)$ belonging to the operator $A(s)$. For the construction of
the Calder{\'o}n projection (first depending on choices and then orthogonalized), we refer to
\auindex{Seeley,\ R.T.}R.T. Seeley \cite[Section 4]{Seeley:1966}, \cite[Theorem 1]{Seeley:1968}.
There it is shown that it is a pseudo-differential idempotent with the Cauchy data space as its
range. Recall that the \subindex{Cauchy data space}Cauchy data space of a differential operator $A$
of order $d$ consists of the closure in $S(\sigma)$ of the array of all derivatives of sections in
$\ker A$ up to order $d-1$ in normal direction along the boundary. For operators of Dirac type,
Seeley's definition was worked out and made canonical in \cite[Chapter 12]{BoWo93}, see also
\auindex{Frey,\ C.}\cite[Section 2.3]{Frey:2005} for elliptic differential operators of arbitrary
order.

\begin{proposition}[Continuity of Calder{\'o}n projection]\label{p:ccp} Assume that%
\[
\dim\ker A_m(s;0)=k \tand \dim\ker A_m^t(s;0)=l%
\]
are independent of $s$ and $\sigma\ge 1/2-d$. Then the family $\left\{\ran
Q(s;\sigma)\right\}_{s\in[0,1]}$ is continuous.
\end{proposition}

\begin{rem}\label{r:calderon-dirac}
(a) If $d=1$, $A(s)$ is of Dirac type and $M(s)$ has product structure near $\Si(s)$ compatible with
the fiber structure for each $s$, then we have weak Unique Continuation Property (UCP, see
\auindex{Boo{\ss}--Bavnbek,\ B.}\auindex{Wojciechowski,\ K.P.}\cite[Chapter 8]{BoWo93}) and, by the
construction, the projectors $Q(s)$ form a continuous family of pseudo-differential projectors and
the family $\{\ran Q(s;\sigma)\}_{s\in[0,1]}$ is continuous for all real $\sigma$.
\newline (b) In \auindex{Schwartz,\ L.}\cite{Schw:56} L. Schwartz conjectured $l=k$. As a student,
the first author gave arguments in support for the \subindex{Schwartz Conjecture}Schwartz
Conjecture in \cite{Boo:65}. A rigorous proof (or a striking counterexample), though, is still
missing.
\end{rem}

\begin{proof} The complete proof of our result will appear in \auindex{Boo{\ss}--Bavnbek,\ B.}\auindex{Deng,\ J.}\auindex{Deng,\
    J.}\auindex{Zhou,\ Y.}\auindex{Zhu,\ C.}\cite{BDZZ}. Here we prove some special cases.

By (\ref{e:strong-cauchy}), for $\sigma>0$ we have
\begin{equation}\label{e:Q-gamma}\ran Q(s;\sigma)=\gamma(\ker A_M(s;\sigma-1/2)).\end{equation}
Let $\sigma\ge 1/2$. Since $A(s)$ is elliptic,%
\[
\ker A_m(s;\sigma-1/2)\ =\ \ker A_m(s;0)\ \tand\ \ker A_m^t(s;\sigma-1/2)\ =\ \ker A_m^t(s;0)%
\]
consist of smooth sections. Since $\dim\ker A_m^t(s;0)=l$ is constant, by Corollary
\ref{c:continuous-ker}, $\{\ker A_M(s;\sigma-1/2)\}_{s\in[0,1]}$ is a continuous family . Since
$\ran Q(s;\sigma)$ is closed, the subspace $\ker A_M(s;\sigma-1/2)+D_m(s;\sigma-1/2)$ is closed in
$D_M(s;\sigma-1/2)$. Since $\dim\ker A_m(s;0)=k$ is constant, by Corollary \ref{c:quotient}, the
family $\{\ran Q(s;\sigma)\}_{s\in[0,1]}$ is continuous.

By our \auindex{Boo{\ss}--Bavnbek,\ B.}\auindex{Chen,\ G.}\auindex{Lesch,\
    M.}\auindex{Zhu,\ C.}\cite[Theorem 5.3]{BCLZ}, based on our \cite[Theorem 7.2b]{BoLeZh08}, the Cauchy data family
$\{\ran Q(s;\sigma)\}_{s\in[0,1]}$ is continuous for $\sigma\in[1/2-d,1/2]$. Then $\{\ran
Q(s;\sigma)\}_{s\in[0,1]}$ is a continuous family for all $\sigma\ge 1/2-d$. Note that the
continuity is only proved there for the case when $d=1$ and $k=l=0$. The proof can be somewhat
simplified by \auindex{Grubb,\ G.}G. Grubb \cite{Grubb:2012a}, and the proof can be
transferred to the general case.
\end{proof}

We assume that all $A(s)$ are \subindex{Formally self-adjoint}formally self-adjoint, i.e.,
$A_m(s;0)\< (A_m(s;0))^*$. Note that we make no assumptions about product structures near the
boundary $\Si(s)$. Then $(A_M(s;0))^*=A_m(s;0)$. Assume that there exists a constant integer $k$
such that $\dim\ker A_m(s;0)=k$. Then Assumption \ref{a:asff}.3 holds.

\subsection{Desuspension spectral flow formula for curves of self-adjoint well-posed elliptic
boundary value problems}
For each $s$ we choose a \subindex{Well-posed self-adjoint boundary condition}well-posed
self-adjoint boundary condition $P(s)\in\Grass_{\sa}(A(s))$ in the sense of \auindex{Seeley,\
R.T.}R.T. Seeley \cite[Definition 3 and Theorem 7]{Seeley:1968}, worked out in our
\auindex{Boo{\ss}--Bavnbek,\ B.}\auindex{Wojciechowski,\ K.P.}\cite[Definition 18.1 and Proposition
20.3]{BoWo93} and further expanded by \auindex{Br{\"u}ning,\ J.}\auindex{Lesch,\ M.}J. Br{\"u}ning
and M. Lesch \cite{BrLe01} for the first order case and by \auindex{Frey,\ C.}C. Frey
\cite[Definition 1.2.5]{Frey:2005}   for $d\ge 1$. It is a self-adjoint projection $P(s;\sigma)\colon
S(s;\sigma)\to S(s;\sigma)$ defined by a pseudo-differential projection $P(s)$ satisfying a
certain conjugacy condition between $I-P(s)$ and $P(s)$ such that it yields a
\subindex{Extension!self-adjoint Fredholm}\symindex{A@$A(s,P(s))$ self-adjoint Fredholm
extension}self-adjoint Fredholm extension $A(s,P(s))$ in $X(s)$ with
\[
\dom A(s,P(s))\ =\ D(s)\ :=\ \{x\in D_M(s;0); P(s;1/2)(\gamma(x))=0\}.
\]

Fix $\sigma\ge\frac{1-d}{2}$. We assume $\{P(s;1/2)\}_{s\in[0,1]}$ and $\{P(s;\sigma)\}_{s\in[0,1]}$ are
continuous families. By Lemma \ref{l:quotient-delta}, $\{D(s)\}_{s\in[0,1]}$ is a continuous
family. Then Assumption \ref{a:asff}.4 holds.

We then have the following spectral flow formula.

\begin{theorem}[Desuspension spectral flow formula]\label{t:gsff} Under the above assumptions, we have the following.
\newline (a) We have that
\[
\ran A_M(s;0)=\ran \left(A_m(s;0)\right)^* \tand \dim X(s)/\left(\ran A_M(s;0)\right)=k%
\]
holds for each $s\in[0,1]$.
\newline (b) The family
\[
\left\{\bigl(\ker P(s;\sigma),\ran Q(s;\sigma)\bigr)\right\}_{s\in[0,1]}
\]
is a path of Fredholm pairs of Lagrangian subspaces of the symplectic Banach space
$(S(s;\sigma),\w_{\Green}(s))$ of index $0$.
\newline (c) We have
\begin{align}
\label{e:gsff1}\SF\{A(s,D(s))\}&=-\Mas\{\ker P(s;\sigma),\ran Q(s;\sigma);\w_{\Green}(s)\}\\
\label{e:gsff2}&=-\Mas\{\ran Q(s;\sigma),\ker P(s;\sigma);-\w_{\Green}(s)\}.
\end{align}
\end{theorem}

\begin{rem}\label{r:L2-spectral-flow-formula}
To us, the preceding Formulae \eqref{e:gsff1} and \eqref{e:gsff2} are most natural for $\sigma =
1/2$, i.e., when we evaluate the Maslov index on the right side of the formulae in the weak
symplectic quotient spaces
\symindex{H@$H^{1/2}(\Si;E\mid_{\Sigma})$ weak symplectic Sobolev space}
$S(s;1/2) = H^{d}\bigl(M(s);E(s)\bigr)$
/$H_0^{d}\bigl(M(s);E(s))$. In that case the arguments are most easily derived from the abstract
spectral flow formula in the preceding section. Note, however, that the two formulae remain valid
for all $\sigma \ge\frac{1-d}{2}$, so, in particular also for $\sigma=\frac{1-d}{2}$, i.e., for calculating the Maslov
index in the continuous family of the common strong symplectic Hilbert spaces
$S(s;\frac{1-d}{2})$. The arguments are getting more involved, though, as
indicated by the double continuity requirement for the boundary projections $P(s)$, namely
requiring continuity both in $S(s;1/2)$ and $S(s;\sigma)$.
\end{rem}

\begin{proof} Since $X(s)$ is a Hilbert space and $A(s,P(s))$ is a self-adjoint Fredholm operator,
$\Index A(s,P(s))=0$ and $P(s)$ is a well-posed boundary value condition for $A(s)$ in the sense of
\auindex{Frey,\ C.}\cite[Definition 1.2.5]{Frey:2005}. By Theorem \ref{t:asff}, (b), (c) hold for
$\sigma=1/2$ and (a) holds.

By \auindex{Frey,\ C.}\cite[Theorem 2.1.4]{Frey:2005} and the regularity theory for elliptic
operators, we have
\begin{equation}
\dim \left(\ker P(s;\sigma)\cap\ran Q(s;\sigma)\right) \label{e:sigma-intersection1}\ =\
\dim\left(\ker P(s;1/2)\cap\ran Q(s;1/2)\right)
\end{equation}
and
\begin{multline}
\dim  S(s;\sigma)/\left(\ker P(s;\sigma)\cap\ran Q(s;\sigma)\right) \label{e:sigma-intersection2}\\
=\dim S(s;1/2)/\left(\ker P(s;1/2)\cap\ran Q(s;1/2)\right).
\end{multline}
Then the pair $\bigl(\ker P(s;\sigma),\ran Q(s;\sigma)\bigr)$ is a Fredholm pair of isotropic
subspaces of the symplectic Banach space $S(s;\sigma)$ of index $0$, so it is a Lagrangian pair by
\auindex{Boo{\ss}--Bavnbek,\ B.}\auindex{Zhu,\ C.}\cite[Proposition 1]{BooZhu:2013}. Then (b)
holds.

Note that we have $S(s;\sigma)\subset S(s;1/2)$ for $\sigma\ge 1/2$ and $S(s;\sigma)\supset
S(s;1/2)$ for $\sigma\in[\frac{1-d}{2}, 1/2]$. By (\ref{e:sigma-intersection1}) and
(\ref{e:sigma-intersection2}), we can apply Theorem \ref{t:mas-emb} and obtain (c).
\end{proof}
\medskip

\subsection{General spectral splitting formula on partitioned manifolds}
Now we assume that the manifold \symindex{M@$M, M(s)$ smooth compact manifold!closed and
partitioned with hypersurface $\Si,\Si(s)$}$M(s)=M(s)^+\cup_{\Si(s)} M(s)^-$ is a partitioned
closed manifold with a hypersurface $\Si(s)$. Let $\sigma\ge \frac{1-d}{2}$. We denote the
restrictions of $A(s)$ to the parts by $A(s)^{\pm}$\,. Note that we now have a pair of Calder{\'o}n
projections $(Q(s)^+,Q(s)^-)$ for each $s\in [0,1]$ with $\range Q(s;\sigma)^\pm$ Lagrangian
subspaces in $S(s;\sigma)$ with symplectic form again defined by the minus Green's form
$-\w_{\Green}(s)$. Then we have the following generalization of the Yoshida-Nicolaescu splitting
formula\subindex{Yoshida-Nicolaescu's Theorem} for the spectral flow:

\begin{theorem}[General Yoshida-Nicolaescu Splitting Formula]\label{t:split}
\subindex{Partitioned manifold!desuspension spectral flow formula}\subindex{Desuspension spectral
flow formula}
For the partitioned case we assume that $\sigma\ge \frac{1-d}{2}$ and
\[\dim\ker A^{\pm}_m(s;0)=k^{\pm}.\]
Then we have
\begin{align}
\label{e:sf-partition}\SF\{A(s)\} &= \SF\{A^-(s,I-Q^+(s))\}\\
\label{e:sf-mas-partition} &= -\Mas\{\range Q^-(s,\sigma),\range Q^+(s,\sigma);\w_{\Green}(s)\}.
\end{align}
\end{theorem}

\begin{proof} Let \symindex{M@$M^{\sharp},M^{\sharp}(s)$ cut partitioned manifold}$M^{\sharp}(s)$ denote the compact manifold
\[M^+(s)\sqcup M^-(s)=
\bigl(M(s)\setminus \Si(s)\bigr)\,\cup\, \bigl((\Si(s)\sqcup(-\Si(s))\bigr)\] with boundary
\[\dpa M^{\sharp}(s) =\dpa M^+(s)\sqcup\dpa M^-(s)=\Si(s)
\sqcup(-\Si(s))=:\Si^{\sharp}(s)\] and $E^{\sharp}(s)\to M^{\sharp}(s)$ the corresponding Hermitian
bundle. Over $M(s)$, $M^a(s)$  and $\Si^a(s)$ with $a\in\{\pm,\sharp\}$ we have specified section
spaces with the notations $X(s)$, $X^a(s)$, $D_M(s;\sigma)$, $D_M^a(s;\sigma)$, $D_m(s;\sigma)$,
$D_m^a(s;\sigma)$, $S(s;\sigma)$, and $S^{\sharp}(s;\sigma)$. Fixing $\Si(s)$ induces a
decomposition
\[
X(s)\cong X^+(s)\oplus X^-(s) =X^{\sharp}(s),
\]
and for the Sobolev space
\[
D_M^+(s;\sigma)\oplus D_M^-(s;\sigma)=D_M^{\sharp}(s;\sigma),\;D_m^+(s;\sigma)\oplus
D_m^-(s;\sigma)=D_m^{\sharp}(s;\sigma).
\]
We have the symplectic decomposition
\[(S^{\sharp}(s;\sigma),\w_{\Green}^{\sharp}(s))=(S(s;\sigma)\times S(s,\sigma),\w_{\Green}(s)\oplus(-\w_{\Green}(s))).\]
Correspondingly, we obtain an operator \symindex{A@$A^{\sharp},A(s)^{\sharp}$ induced operator on
cut manifold}$A(s)^{\sharp}$ for each $s\in [0,1]$ which is a formally self-adjoint elliptic
differential operator of order $d$ according to the assumptions made for Theorem \ref{t:split}.

For the Calder{\'o}n projection of $A^{\sharp}$ we have
\begin{equation}\label{e:calderon-partition}
\range Q^{\sharp}(s)=\range Q^+(s)\oplus \range Q^-(s).
\end{equation}

Let \symindex{\D@$\D(s;\sigma)$ diagonal in $S(s;\sigma)\times S(s;\sigma)$}$\D(s;\sigma)$
denote the diagonal in $S(s;\sigma)\times S(s;\sigma)$. By Lemma \ref{l:boxplus}, for each
$s\in[0,1]$,  the diagonal $\D(s;\sigma)$ is a Lagrangian subspace of $S^{\sharp}(s,\sigma)$ with
respect to $\w_{\Green}^{\sharp}(s)$ and makes a Fredholm pair with each $\ran Q^{\sharp}(s)$\/. By
\auindex{Frey,\ C.}\cite[Theorem 2.1.4]{Frey:2005}, the projection of $S^{\sharp}(s)$ onto $\D(s)$
is well-posed for $A^{\sharp}(s)$ in the sense of \cite[Definition 1.2.5]{Frey:2005} (even if it is
not a pseudo-differential operator over the manifold $\Si^{\sharp}(s)$, as noticed in
\auindex{Kirk,\ P.}\auindex{Lesch,\ M.}\cite[Section 5]{KiLe00} in the $d=1$ case).

Consequently, we have on the manifold $M^{\sharp}(s)$ a natural self-adjoint elliptic boundary
condition (in the sense of our Theorem \ref{t:gsff}) defined for $A^{\sharp}(s)$ by the {\it
pasting} domain
\begin{align}\label{e:pasting}
\symindex{D@$D^\sharp(s)$ pasting domain for induced operator on cut partitioned manifold}
D^\sharp(s) :&=\{(x,y)\in D_M^{\sharp}(s;0) ; (\g^+(s))(x)=(\g^-(s))(y)\}\\
&= \{(x,y)\in D_M^{\sharp}(s;0)  ;  (\g^{\sharp}(s))(x,y)\in\D(s;1/2)\},
\end{align}
where $\g^a(s)\colon D_M^{\sharp}(s;0)\to S^{\sharp}(s,1/2)$ denotes the trace maps for $s\in[0,1]$
and $a=\pm,\sharp$. Let $A^{\sharp}(s,D^{\sharp}(s))$ denote the operator which acts like
$A^{\sharp}(s)$ and has domain $D^{\sharp}(s)$\/.

By these definitions and applying Proposition \ref{p:boxplus}.b
 and Theorem \ref{t:gsff} to the
operator family $\{A^{\sharp}(s,D^{\sharp}(s))\}$ we obtain
\begin{align*}
\SF&\{A(s)\}= \SF\{A^{\sharp}(s,D^{\sharp}(s))\}\\
&\fequal{Th. \ref{t:gsff}}-\Mas\{\D(s;\sigma),\range Q^+(s;\sigma)\oplus \range Q^-(s,\sigma);\omega_{\Green}(s))\oplus(-\omega_{\Green}(s))\}\\
&\fequal{\eqref{e:maslov4}} -\Mas\{\range Q^-(s;\sigma), \range Q^+(s;\sigma);\omega_{\Green}(s)\}\\
&\fequal{\eqref{e:maslov3}} -\Mas\{\range Q^+(s;\sigma), \range Q^-(s;\sigma);-\omega_{\Green}(s)\} \\
&\fequal{Th. \ref{t:gsff}}  \SF\{A^-(s,I-Q^+(s))\}.\qedhere
\end{align*}
\exendproof

\begin{rem}\label{r:sf-partition}
If one is only interested in the equality (\ref{e:sf-partition}), one need not argue with the
Maslov index, as we do, but can find a direct proof in \auindex{Kirk,\ P.}\auindex{Lesch,\
M.}\cite[Corollary 5.6]{KiLe00} based solely on the homotopy invariance of the spectral flow of a
related two-parameter family.
\end{rem}
%
%
%


%


%
%
%


\pdfbookmark[-1]{Backmatter}{someuniquename}
{\appendix
%
%
%


\addtocontents{toc}{\medskip\noi}
\chapter{Perturbation of closed subspaces in Banach
spaces}\label{s:closed-subspaces}

This appendix serves as an introduction to the topology of closed linear subspaces in Banach spaces
with applications to families of closed operators with nested domains and perturbations of Fredholm
pairs. Denote by \symindex{S@$\Ss(X)$ set of closed linear subspaces in Banach space $X$}$\Ss(X)$
(\symindex{Sc@$\Ss^c(X)$ set of closed complemented subspaces in Banach space $X$}$\Ss^c(X)$) the
set of all (complemented) closed linear subspaces of a Banach space $X$. Denote by $\Bb(X,Y)$
($\Cc(X,Y)$) the set of all bounded operators (closed, not necessarily bounded operators) between
Banach spaces $X$ and $Y$. Let $\Ss(X),\Ss^c(X)$ be equipped with the gap topology (see below
Section \ref{ss:gap-topology}). Then, we shall solve the following problems:

\begin{enumerate}
\renewcommand{\labelenumi}{(\Roman{enumi})}
\item Under what conditions do the elementary linear operations (intersection, sum and
making quotients) become continuous for pairs of closed subspaces?

\item Under what conditions do we obtain a continuous mapping $(A,D)\mapsto A_D$, where the
operator $A$ varies continuously in $\Bb(X,Y)$, the domain $D$ varies continuously in $\Ss(X)$, and
$A_D$ denotes the restriction of $A$ to the domain $D$ and varies in $\Cc(X,Y)$?

\item How can we control changes of a space of Fredholm pairs under finite or compact perturbation
of one factor?
\end{enumerate}

Question (I) will be answered in Propositions \ref{p:closed-spaces-dimensions} and
\ref{p:close-to}. Question (II) will be answered in Corollary \ref{c:continuous-operator}. Question
(III) will be answered in Proposition \ref{p:compact-perturb}.

These results will be formulated and proved in general terms. We shall emphasize, however, the
various applications to solving variational problems of the global analysis of elliptic operators
on manifolds with boundary. Problem (I) has two immediate applications: The first application is
the local stability of weak inner UCP, see Corollary \ref{c:ls-wiucp}. The second application is
the continuous variation of the Cauchy data spaces under variation of the operator under the
assumption of weak inner UCP (or fixed dimension of the inner solution spaces), see Corollary
\ref{c:continuous-ker}.

Problem (II) settles the intricate delicacies of independent variation of operator and boundary
condition, yielding continuous variation of the induced Fredholm extension.

Problem (III) addresses the changes, roughly speaking, when we replace one boundary condition by
another one under \textit{small} perturbation. Here \textit{small} means by finite or compact
change of the domain, to be defined rigorously below. To give an idea of what kind of changes we
are dealing with, we refer to the Grassmannian of pseudo-differential projections with the same
principal symbols, that define large classes of well-posed and mutually intimately related boundary
problems, as in \auindex{Boo{\ss}--Bavnbek,\ B.}\auindex{Wojciechowski,\ K.P.}\cite{BoWo93}.

This program requires rather detailed investigations of the topology of graphs and domains of
closed operators. Our topological approach is based on the \subindex{Gap topology}\textit{gap}
$\wh{\delta}\colon \Ss(X)\times \Ss(X)\to \R_+$ and the \subindex{Angular distance}\textit{angular
distance} $\wh{\gamma}\colon \Ss(X)\times \Ss(X)\to [0,1]$ (also called \textit{minimum gap}), see
Definition \ref{d:closed-distance} below. According to \auindex{Berkson,\ E.}E. Berkson in
\cite{Berkson:1963}, the concept of \textit{opening} (as the \textit{gap} was called in the 1940s
and 1950s) was first introduced in Hilbert space in 1947 by \auindex{Krein,\
M.G.}\auindex{Krasnosel'skii,\ M.A.}M. G. Krein and M.A. Krasnosel'ski in \cite{Krein-Krasno:1947}.
The definition was one year later extended to arbitrary Banach spaces in
\cite{Krein-Krasno-Mil:1948} by \auindex{Krein,\ M.G.}\auindex{Krasnosel'skii,\
M.A.}\auindex{Mil'man,\ D.P.}M.G. Krein, M.A. Krasnosel'ski, and D.P. Mil'man. Ten years later, it
was supplemented by the definition of the \textit{minimum gap/angular distance}  $\wh{\gamma}$  in
\cite{Gohberg-Markus:1959} by \auindex{Gohberg,\ I.}\auindex{Markus,\ A.S.}I. Gohberg and A.S.
Markus.

We shall use \auindex{Kato,\ T.}T. Kato's \cite[Chapter IV]{Ka95} as our general reference. We
shall apply considerable diligence to the estimates to guarantee the sharpest versions of our
invariance results. Some of the results, often in different and weaker form, can be found in the
quoted original papers and the classical treatises
\cite{CorLab,Gohberg-Krein:1957,Massera-Sch:1958,Ne65,Ne68,Ne51} by \auindex{Cordes,\
H.O.}\auindex{Labrousse,\ J.-P.} H.O. Cordes and J.-P. Labrousse, \auindex{Gohberg,\
I.}\auindex{Krein,\ M.G.}I. Gohberg and M.G. Krein, \auindex{Massera,\ J.L.}\auindex{Sch{\"a}ffer,\
J.J.} J.L. Massera and J.J. Sch{\"a}ffer, \auindex{Neubauer,\ G.}G. Neubauer, and
\auindex{Newburgh,\ J.D.} J.D. Newburgh.

\section{Some algebra facts}\label{ss:linear-algebra}

We have the following elementary fact of algebra.

\begin{lemma}\label{l:lin-alg} Let $X$ be an additive group and $V_1,V_2,V_3$ three subgroups. If $V_1\subset V_3$, we have
\begin{equation}\label{e:sum-cap}
(V_1+V_2)\cap V_3=V_1+V_2\cap V_3.
\end{equation}
\end{lemma}

\begin{corollary}\label{c:whole-space} Let $X$ be an additive group and $V, X_0,X_1$ three subgroups with $X=X_0\oplus X_1$.
Denote by $P_0\colon X\to X_0$ the projection defined by the decomposition $X=X_0\oplus X_1$.
Assume that $V\supset X_1$. Then we have $V=P_0V+X_1$ and $P_0V=V\cap X_0$. In particular, we have $V=X$ if $P_0V=X_0$.
\end{corollary}

\begin{proof} Since $V\supset X_1$, by Lemma \ref{l:lin-alg} we have $V=V\cap(X_0+X_1)=V\cap X_0+X_1$.
So we have $P_0V=V\cap X_0$, and $V=P_0V+X_1$. If $P_0V=X_0$, we have $V=X$.
\end{proof}
\section{The gap topology}\label{ss:gap-topology}

Let $X$ be a Banach space. Let $M,N$ be two closed linear subspaces of $X$, i.e., $M,N\in\Ss(X)$.
Denote by \symindex{SM@$\romS_M$ unit sphere of linear subspace in Banach space}$\romS_M$ the unit
sphere of $M$. We recall three common definitions of distances in $\Ss(X)$ (see also
\auindex{Kato,\ T.}\cite[Sections IV.2.1 and IV.4.1]{Ka95}):
\begin{itemize}

\item  the \textit{Hausdorff metric} $\hat d$;

\item the \textit{aperture (gap distance)} $\hat\delta$, that is not a metric since it does not in general
satisfy the triangle inequality, but defines the same topology as the metric $\hat d$, called
\textit{gap topology}, and is easier to estimate than $\hat d$; and

\item the \textit{angular distance (minimum gap)} $\wh\gamma$, that is useful in our estimates, though not
defining any suitable topology.

\end{itemize}

\begin{definition}[The gap between subspaces]\label{d:closed-distance}
\subindex{Distances in $\Ss(X)$}
(a) We set \begin{multline*}\subindex{Hausdorff metric}\symindex{d@$d$ Hausdorff metric}%
 d(M,N)\ =\
d(\romS_M,\romS_N)\\
:= \begin{cases} \max\left\{\begin{matrix}\sup_{u\in \romS_M}\dist(u,\romS_N),\\ \sup_{u\in
\romS_N}\dist(u,\romS_M)\end{matrix}\right\},&\text{ if both $M\ne 0$ and $N\ne 0$},\\
0,&\text{ if $M=N=0$},\\
2,&\text{ if either $M= 0$ and $N\ne 0$ or vice versa}.
\end{cases}
\end{multline*}
\newline (b) We set
\begin{eqnarray*}\subindex{Gap topology!gap distance}\symindex{dhat@$\hat{d}$ gap distance}
\delta(M,N)\ &:=&\ \begin{cases}\sup_{u\in \romS_M}\dist(u,N),& \text{if $M\ne\{0\}$},\\
0,& \text{if $M=\{0\}$},\end{cases}\\
\hat\delta(M,N)\ &:=&\ \max\{\delta(M,N),\delta(N,M)\}.
\end{eqnarray*}
$\hat\delta(M,N)$ is called the {\em gap} between $M$ and $N$.
\newline (c) We set
\begin{eqnarray*}\subindex{Minimum gap (angular distance)}\symindex{\gamma@$\hat{\gamma}$ minimum gap (angular distance)}
\gamma(M,N)\ :&=&\ \begin{cases} \inf_{u\in M\setminus N}\frac{\dist(u,N)}{\dist(u,M\cap N)}\ (\le
1), & \text{if $M\nsubseteq N$},\\
1, & \text{if $M\subset N$},\end{cases}\\
\hat\gamma(M,N)\ :&=&\ \min\{\gamma(M,N),\gamma(N,M)\}.
\end{eqnarray*}
$\hat\gamma(M,N)$ is called the {\em minimum gap} between $M$ and $N$. If $M\cap N=\{0\}$, we have
\[\gamma(M,N)\ =\ \inf_{u\in \romS_M}\dist(u,N).\]
\end{definition}

In this Memoir we shall impose the gap topology on the space $\Ss(X)$ of all closed linear
subspaces of a Banach space $X$ and its subset $\Ss^c(X)$ of complemented subspaces.

\smallskip

We recall the following two results on finite-dimensional variation. For the second see
\auindex{Brezis,\ H.}\cite[Proposition 11.4]{Brezis:2011}. For the first see (\auindex{Kato,\
T.}\cite[Lemma III.1.9]{Ka95}.

\begin{proposition}[Finite extension]\label{p:finite-extension} Let $X$ be a Banach space and $M$ be a closed subspace of $X$.
Let $M^{\prime}\supset M$ be a linear subspace of $X$ with $\dim M^{\prime}/M<+\infty$. Then we
have
\newline (a) $M^{\prime}$ is closed, and
\newline (b) $M'\in\Ss^c(X)$ if and only if $M\in\Ss^c(X)$.
\end{proposition}

\begin{definition}\label{d:fredholm-pair}
(a) The space of (algebraic) \emph{Fredholm pairs} of linear subspaces of a vector space $X$ is
defined by
\begin{equation}\label{e:fp-alg}
\Ff^{2}_{\operatorname{alg}}(X):=\{(M,N);  \dim (M\cap N)  <+\infty\;\; \text{and} \dim
X/(M+N)<+\infty\}
\end{equation}
with
\begin{equation}\label{e:fp-index}
\Index(M,N):=\dim(M\cap N) - \dim X/(M+N).
\end{equation}

\noi (b) In a Banach space $X$, the space of (topological) \emph{Fredholm pairs} is defined by
\begin{multline}\label{e:fp}
\Ff^{2}(X):=\{(M,N)\in\Ff^2_{\operatorname{alg}}(X); M,N, \tand M+N \subset X \text{ closed}\}.
\end{multline}
A pair $(M,N)$ of closed subspaces is called {\em semi-Fredholm} if $M+N$ is closed, and at least
one of $\dim (M\cap N)$ and $\dim\left(X/(M+N)\right)$ is finite.

\noi (c) Let $X$ be a Banach space, $M\in\Ss(X)$ and $k\in\Z$. We define
\begin{align}
\label{e:fp-M}\Ff_M(X):&=\{N\in\Ss(X);(M,N)\in\Ff^2(X)\},\\
\label{e:fp-kM}\Ff_{k,M}(X):=&\{N\in\Ss(X);(M,N)\in\Ff^2(X),\Index(M,N)=k\}.
\end{align}
\end{definition}

\begin{rem}\label{r:redholm-pairs}
Actually, in Banach space the closedness of $\la+\mu$ follows from its finite codimension in $X$ in
combination with the closedness of $\la,\mu$ (see \auindex{Boo{\ss}--Bavnbek,\
B.}\auindex{Furutani,\ K.}\cite[Remark A.1]{BoFu99} and \auindex{Kato,\ T.}\cite[Problem
IV.4.7]{Ka95}).
\end{rem}

The following lemma is from \auindex{Kato,\ T.}\cite[Problem IV.4.6]{Ka95}.

\begin{lemma}\label{l:finite-diff-index} Let $X$ be a vector space and $M^{\prime},M,N$ be linear subspaces. Assume that
$M^{\prime}\supset M$ and $\dim M^{\prime}/M=n<+\infty$. Then we have
$\Index(M^{\prime},N)=\Index(M,N)+n$.
\end{lemma}

We give the following elementary fact.

\begin{lemma}\label{l:fp-complemented} Let $X$ be a Banach space and $(M,N)\in\Ff^{2}(X)$. Then we have $M,N\in\Ss^c(X)$.
\end{lemma}

\begin{proof} Since $(M,N)\in\Ff^{2}(X)$, there exist closed linear subspaces $M_1\subset M$, $N_1\subset N$ and a
finite-dimensional linear subspace $V\subset X$ such that
\[M=M\cap N\oplus M_1, N=M\cap N\oplus N_1, X=V\oplus(M+N).\]
Then we have $N_1\cap M=N_1\cap M\cap N=\{0\}$, and
\begin{equation}\label{e:directsum-decomposition}
X=M\cap N\oplus M_1\oplus N_1\oplus V.
\end{equation}
So $M,N\in\Ss^c(X)$ holds.
\end{proof}
\section{Continuity of operations of linear subspaces}\label{ss:continuity-of-operations}
We study the continuity of $M/L$, $M\cap N$ and $M+N$ for varying closed subspaces $M$ and $N$ and
fixed closed subspace of a Banach space $X$.

For the quotient space, we have the following lemma.

\begin{lemma}\label{l:quotient-delta} Let $X$ be a Banach space with closed subspaces $M,N,L\in\Ss(X)$
such that $M,N\supset L$. Denote by $p$ the natural map $p\colon X\to X/L$.  Then we have
\newline (a) $d(p(u),p(M))=d(u,M)$ for $u\in X$,
\newline (b) $\gamma(p(M),p(N))=\gamma(M,N)$,
\newline (c) $d(u,N)\le d(u,L)\delta(M,N)$ for $u\in M$, and
\newline (d) $\delta(M,N)=\delta(p(M),p(N))$.
\end{lemma}

\begin{proof} (a), (b) By the last paragraph of the proof of \auindex{Kato,\ T.}\cite[Theorem IV.4.2]{Ka95}.
\newline (c) Let $\e\in(0,1)$. By \auindex{Kato,\ T.}\cite[Lemma III.1.12]{Ka95}, for any $u\in M$, there exists a $v\in L$ such that
$d(u,L)\ge (1-\e)\|u-v\|$. Since $L\subset N$, we have
\[d(u,N)=d(u-v,N)\le\|u-v\|\delta(M,N)\le(1-\e)^{-1}d(u,L)\delta(M,N).\]
Let $\e\to 0$, and we have $d(u,N)\le d(u,L)\delta(M,N)$.
\newline (d) If $M=L$, we have $\delta(M,N)=\delta(p(M),p(N))=0$. Assume that $M\ne L$. By definition and the first equality we have
\begin{align*}\delta(M,N)&=\sup\{d(u,N);u\in \romS_M\}\\
&\le\sup\{d(u,N);u\in M,d(u,L)\le 1\}=\delta(p(M),p(N)).
\end{align*}
By (a) and (c) we have
\[\delta(p(M),p(N))=\sup\{d(u,N);u\in M,d(u,L)=1\}\le\delta(M,N).\]
Thus we obtain (d).
\end{proof}

Firstly, we consider the case of $\dim(M\cap N)<+\infty$. We need the following uniform estimate of
the given Banach norm by the coefficients with regard to a basis for finite-dimensional subspaces.

\begin{lemma}\label{l:ball-distances}
Let $X$ be a complex  Banach space and $u_1,\ldots,u_n\in \operatorname{S}_X$. Set%
\[
V_k:=\begin{cases} \{0\},& \text{for $k=0$},\\
 \Span\{u_1,\dots,u_k\}, & \text{for $k=1\ldots,n$}.
 \end{cases}
 \]
 Assume that $\dist(u_k,V_{k-1})\ge \delta$ for $k=1,\ldots,n-1$ and $\delta>0$. Then we
have $\delta\le 1$, $\dim V_k=k$, and%
\[
\frac{1}{n}\biggl(\frac{\delta}{1+\delta}\biggr)^{n-1}\sum_{k=1}^n|a_k|\ \le\ \|\sum_{k=1}^na_ku_k\|\ \le\ \sum_{k=1}^n|a_k|%
\]
for all $a_1,\ldots,a_k\in\C$.
\end{lemma}

\begin{proof} Only the left inequality needs a proof. It is a Banach space variant of
\subindex{Bessel's Inequality}Bessel's Inequality of harmonic analysis (see, e.g., \cite[Section
I.6.3]{Ka95}\auindex{Kato,\ T.}). Certainly, our precise version of the left inequality will be
well known in functional analysis. For the convenience of the reader we give, however, an
elementary proof.

Clearly we have $1=\|u_1\|=\dist(u_1,V_0)\ge \delta$. Since $\delta>0$, we have $u_k\notin
V_{k-1}$, and by induction we have $\dim V_k=k$.
\newline Also by induction:  $\|a_1u_1\|=|a_1|$, and so%
\begin{align*}
\|a_1u_1+a_2u_2\|\ &\ge\ \max\{\delta\,|a_2|,|a_1|-|a_2|\}\\
&\ge\ \max\biggl\{\frac{\delta}{1+\delta}|a_1|,\delta|a_2|\biggr\},\ldots,\\
\|a_1u_1+\ldots+a_nu_n\|\ &\ge
\max\biggl\{\biggl(\frac{\delta}{1+\delta}\biggr)^{n-1}|a_1|,\biggl(\frac{\delta}{1+\delta}\biggr)^{n-k}\delta|a_k|;\\
&\qquad\qquad\qquad\qquad\qquad\qquad\qquad k=2,\ldots,n\biggr\}\\
&\ge\ \frac{1}{n}\biggl(\frac{\delta}{1+\delta}\biggr)^{n-1}\sum_{k=1}^n|a_k|.
\end{align*}
Since $u_1,\ldots,u_n\in \romS_X$, we have $\|\sum_{k=1}^na_ku_k\|\ \le\ \sum_{k=1}^n|a_k|$.
\end{proof}

In general, the distances $\d(M,N)$ and $\d(N,M)$ can be very different and, even worse, behave
very differently under small perturbations. However, for finite-dimensional subspaces of the same
dimension in a Hilbert space we can estimate $\d(M,N)$ by $\d(N,M)$ in a uniform way. We can give
the following generalization of \auindex{Boo{\ss}--Bavnbek,\ B.}\auindex{Zhu,\ C.}\cite[Lemma
14]{BooZhu:2013}, which is different from \auindex{Neubauer,\ G.}\cite[Lemma 1.7]{Ne68}:

\begin{lemma}\label{l:finite-dim-delta} Let $X$ be a Banach space and $M,N$ be two linear subspaces with $\dim M=\dim N=n$.
Then we have
\[
\delta(M,N)\ \le\ \frac {2^{n-1}n\delta(N,M)} {(1-\delta(N,M))^n}\,,
\]
if $1-\delta(N,M)>0$.
\end{lemma}

\begin{proof}
Take $\e\in(0,1-\delta(N,M))$. By induction and \auindex{Kato,\ T.}\cite[Lemma IV.2.3]{Ka95}, there
exist
$v_1,\ldots,v_n\in \romS_N$ and $u_1,\ldots u_n\in M$ such that%
\[
\dist(v_k,V_{k-1})\ =\ 1 \tand \|u_k-v_k\|\ \le\ \delta(N,M)+\e
\]
for
\[
V_k:=\begin{cases} \{0\},& \text{for $k=0$},\\
 \Span\{u_1,\dots,u_k\},  &\text{for $k=1\ldots,n$}.
 \end{cases}
\]
Then $1-\delta(N,M)-\e\le\|u_k\|\le 1+\delta(N,M)+\e$ and $\dist(u_k,V_{k-1})\ge 1-\delta(N,M)-\e$.
By Lemma \ref{l:ball-distances}, $V_n=M$. For any $u\in \romS_M$, there exist $a_1,\ldots,a_n\in\C$
with $u=\sum_{k=1}^na_ku_k$. By Lemma \ref{l:ball-distances}, we also have
\begin{align*}
1\ &=\ \|\sum_{k=1}^n a_ku_k\|\ \ge\
\frac{1}{n}\biggl(\frac{1-\delta(N,M)-\e} {2}\biggr)^{n-1}\,\sum_{k=1}^n\,|a_k|\|u_k\|\\
&\ge\ \frac{(1-\delta(N,M)-\e)^n} {2^{n-1}n}\,\sum_{k=1}^n|a_k|.%
\end{align*}
Set $v:=\sum_{k=1}^na_kv_k$. Then we have:
\begin{align*}
\|u-v\|\ &=\ \|\sum_{k=1}^na_k(u_k-v_k)\|\ \le\ \sum_{k=1}^n|a_k|\delta(N,M)\\
&\le\ \frac{2^{n-1}n\delta(N,M)}{(1-\delta(N,M)-\e)^n}\/.%
\end{align*}
So $\delta(M,N)\le \frac{2^{n-1}n\delta(N,M)}{(1-\delta(N,M)-\e)^n}$\/. Let $\e\to 0$, then we have
$\delta(M,N) \le \frac{2^{n-1}n\delta(N,M)}{(1-\delta(N,M))^n}$.
\end{proof}

The diligence with the preceding estimates pays back with the following Proposition
\ref{p:closed-spaces-dimensions} that confines possible changes of the dimensions of intersections
and the co-dimensions of sums of pairs of closed linear subspaces under variation. For that, we
shall use the concepts of approximate nullity (approximate deficiency) defined by \auindex{Kato,\
T.}\cite[\S{IV.4}]{Ka95}:

\begin{definition}
Let $M,N$ be closed linear manifolds (i.e., closed subspaces) of a Banach space $Z$.
\newline (a) We define the \textit{approximate nullity} of the pair $M,N$,  denoted by $\nuli'(M,N)$, as the
least upper bound of the set of integers $m$ ($m=+\infty$ being permitted) with the property that,
for any $\e>0$, there is an $m$-dimensional closed linear subspace $M_{\e}\subset M$ with
$\delta(M_{\e},N)<\e$.
\newline (b) We define the \textit{approximate deficiency} of the pair $M,N$, denoted by $\defi'(M,N)$, by
$\defi'(M,N):=\nuli'(M^{\bot},N^{\bot})$.
\end{definition}

\begin{note} While $\nuli(M,N):=\dim M\cap N$ and $\defi(M,N):=\dim Z/(M+N)$ are defined in a
purely algebraic fashion, the definitions of $\nuli'(M,N)$ and $\defi'(M,N)$ depend on the
topology of the underlying space $Z$. Moreover, it is easy to show (see l.c., \auindex{Kato,\ T.}\cite[Theorems IV.4.18 and
IV.4.19]{Ka95}) that
\begin{align*}
\nuli'(M,N) \ &=\ \begin{cases} \nuli(M,N),&\text{for $M+N$ closed,}\\ +\infty,
&\text{else,}\end{cases}%
\tand\\%
\defi'(M,N)\ &=\ \begin{cases} \defi(M,N),&\text{for $M+N$ closed,}\\ +\infty, &\text{else.}\end{cases}
\end{align*}
\end{note}

We are now ready for the first main result of this appendix:

\begin{proposition}\label{p:closed-spaces-dimensions}
Let $Z$ be a Banach space and $M,N,M^{\prime},N^{\prime}$ be closed linear subspaces. Assume that
$M+N$ is closed. Then $\gamma(M,N)>0$ by \auindex{Kato,\ T.}\cite[Theorem IV.4.2]{Ka95}, and we
have

\begin{enumerate}
\renewcommand{\labelenumi}{(\alph{enumi})}

\item $\delta(M^{\prime}\cap N^{\prime},M\cap N)\ \le\ \frac{2}{\gamma(M,N)}
(\delta(M^{\prime},M)+\delta(N^{\prime},N))$,

\item $\dim(M^{\prime}\cap N^{\prime})\le\nuli^{\prime}(M^{\prime},N^{\prime})\le\dim(M\cap N)$ \linebreak if
$\delta(M^{\prime},M)(1+\gamma(M,N))+\delta(N^{\prime},N)< \gamma(M,N)$,

\item $\dim Z/(M^{\prime}+N^{\prime})\le\defi^{\prime}(M^{\prime},N^{\prime})\le\dim Z/(M+N)$  \linebreak if
$\delta(M,M^{\prime})+\delta(N,N^{\prime})(1+\gamma(M,N))< \gamma(M,N)$, and

\item $M^{\prime}\cap N^{\prime}\to M\cap N$ if $\dim(M^{\prime}\cap N^{\prime})=\dim(M\cap
N)<+\infty$ and $\delta(M^{\prime},M)+\delta(N^{\prime},N)\to 0$.
\end{enumerate}
\end{proposition}

\begin{proof} (a) If $M^{\prime}\cap N^{\prime}=\{0\}$, we have%
\[
\delta(M^{\prime}\cap N^{\prime},M\cap N)=0 \le\ \frac{2}{\gamma(M,N)}
(\delta(M^{\prime},M)+\delta(N^{\prime},N)).%
\]
If $M^{\prime}\cap N^{\prime}\ne\{0\}$, (a) follows from \auindex{Kato,\ T.}\cite[Lemma
IV.4.4]{Ka95}.
\newline (b) and (c) Similar to the proof of \auindex{Kato,\ T.}\cite[Theorem IV.4.24]{Ka95}.
\newline (d) By Lemma \ref{l:finite-dim-delta}.
\end{proof}

We give a first application of the preceding proposition.

\begin{ass}\label{a:ell}Assume that the following data are given:
\begin{itemize}
\item a compact smooth Riemannian manifold $(M,g)$ with smooth boundary $\Sigma:=\partial M$,
\item Hermitian vector bundles $(E,h^E)$ and $(F,h^F)$ over $M$,
\item an order $d>0$ elliptic differential operator
\begin{equation}\label{e:diff-op}
A\colon \Ci(M;E)\too\Ci(M;F),
\end{equation}
\item $A^t$ denotes the formal adjoint of $A$ with respect to the metrices $g$, $h^E$, $h^F$.
\item Let $\sigma\ge 0$. Then $A_{m,\sigma}$ denotes the operator $A\colon H_0^{d+\sigma}(M;E)\to H^{\sigma}(M;E)$ (see Subsection \ref{sss:sobolev_spaces} for the definition of the Sobolev spaces on the manifold with boundary), and $A_{M,\sigma}$
denotes the operator $A\colon H^{d+\sigma}(M;E)\to H^{\sigma}(M;E)$.
\end{itemize}
\end{ass}

The following lemma is standard in elliptic operator theory.

\begin{lemma} Let $A$ satisfy Assumption \ref{a:ell}. Then $A_{m,\sigma}$ and $A_{M,\sigma}$ are semi-Fredholm operators,
$\ker A_{m,\sigma}=\ker A_{m,0}$ consists of smooth sections, and we have $\dim
(H^{\sigma}(M;E))/(\ran A_{M,\sigma})=\dim\ker A^t_{m,0}$.
\end{lemma}

\begin{proof} By \auindex{Frey,\ C.}\cite[Proposition A.1.4]{Frey:2005} and G{\aa}rding's inequality, $A_{m,\sigma}$ is left-Fredholm,
i.e., $\dim\ker A_{m,\sigma}<+\infty$ and $\ran A_{m,\sigma}$ is closed in $H^{\sigma}(M;E)$. By
the regularity, $\ker A_{m,\sigma}$ consists of smooth sections and hence $\ker A_{m,\sigma}=\ker
A_{m,0}$.

Denote by $C_+(A)$ the Calder{\'{o}}n projection of $A$. Denote by $\gamma$ the trace map. Set
\[D_{\sigma}:=\{u\in H^{d+\sigma}(M;E); C_+(A)(\gamma(u))=0\}.\]
Denote by $A_{D_{\sigma}}$ the operator $A\colon D_{\sigma}\to H^{\sigma}(M;E)$. Then
$A_{D_{\sigma}}$ is a Fredholm operator. Since $\ran A_{M,\sigma}\supset\ran A_{D_{\sigma}}$, the
space $\ran A_{D_{\sigma}}$ is closed and we have $\dim (H^{\sigma}(M;E))/(\ran
A_{D_{\sigma}})<+\infty$. Then we have
\[\dim (H^{\sigma}(M;E))/(\ran A_{M,\sigma})=\dim\ker A^t_{m,\sigma}=\dim\ker A^t_{m,0}.\qedhere
\]
\exendproof

\begin{definition}\label{d:wiucp} Let $A$ satisfy Assumption \ref{a:ell}. The elliptic operator $A$ is said to have
{\em weak inner unique continuation property (UCP)} if $\ker A_{m,0}=\{0\}$.
\end{definition}

\begin{corollary}[Local stability of weak inner UCP]\label{c:ls-wiucp} Let $X,Y$ be Banach spaces and $A\in\Bb(X,Y)$ a bounded operator.
Assume that $\ker A=\{0\}$ and $\ran A$ is closed in $Y$. Then there exists a $\delta>0$ such that
for all $A^{\prime}\in\Bb(X,Y)$ and $\|A^{\prime}-A\|<\delta$, we have $\ker A=\{0\}$.
\end{corollary}

\begin{proof} Set $Z:=X\times Y$, $M:=\Graph(A)$, $M^{\prime}:=\Graph(A^{\prime})$ and $N=N':=X\times\{0\}$.
By Proposition  \ref{p:closed-spaces-dimensions}.b and the proof of \auindex{Boo{\ss}--Bavnbek,\
B.}\auindex{Zhu,\ C.}\cite[Lemma 16]{BooZhu:2013}, our result follows.
\end{proof}

Now we refine our estimates to investigate the deformation behavior a bit further.

\begin{lemma}\label{l:complement-estimate} Let $X$ be a Banach space and $M,N$ be closed subspaces of $X$.
Assume that $M\nsubseteq N$. Then for any $\varepsilon\in(0,1)$ and $u\in M\setminus N$, there
exists a $u_0\in M\setminus N$ such that $\dist(u_0,N)=\dist(u,N)$ and $\dist(u_0,M\cap
N)=\dist(u,M\cap N)\ge(1-\varepsilon)\|u_0\|$.
\end{lemma}

\begin{proof} There exists $v\in M\cap N$ such that $\dist(u,M\cap N)\le(1-\varepsilon)\|u-v\|$. Set $u_0:=u-v$.
\end{proof}

We have the following estimate. See \auindex{Neubauer,\ G.}\cite[(1.4.2)]{Ne68} for a different
estimate.

\begin{lemma} Let $X$ be a Banach space and $M,N,M^{\prime}, N^{\prime}$ be closed linear subspaces.
Assume that $\frac{1-\delta(M^{\prime}\cap N^{\prime},M\cap N)}{1+\delta(M^{\prime}\cap
N^{\prime},M\cap N)}>\delta(M,M^{\prime})$. Then we have
\begin{equation}\label{e:gamma-estimate}
\gamma(M^{\prime},N^{\prime})\ \le\
\frac{(1+\delta(N,N^{\prime}))\gamma(M,N)+\delta(M,M^{\prime})+\delta(N,N^{\prime})}{\frac{1-\delta(M^{\prime}\cap
N^{\prime},M\cap N)}{1+\delta(M^{\prime}\cap N^{\prime},M\cap N)}-\delta(M,M^{\prime})}.%
\end{equation}
\end{lemma}

\begin{proof} \textbf{1.} If $M\subset N$, we have $\gamma(M^{\prime},N^{\prime})\le 1=\gamma(M,N)$. So (\ref{e:gamma-estimate}) holds.
\newline \textbf{2.} Assume that $M\nsubseteq N$. Then for any $\varepsilon>0$ and $u\in M\setminus N$, there exists
$u^{\prime}\in M^{\prime}$ such that $\|u-u^{\prime}\|\le\|u\|(\delta(M,M^{\prime})+\varepsilon)$.
By \auindex{Kato,\ T.}\cite[Lemma IV.2.2]{Ka95} we have
\begin{eqnarray*}\dist(u^{\prime},M^{\prime}\cap N^{\prime})&\ge&\dist(u,M^{\prime}\cap N^{\prime})-\|u-u^{\prime}\|\\
&\ge&\frac{\dist(u,M\cap N)-\|u\|\delta(M^{\prime}\cap N^{\prime},M\cap N)}{1+\delta(M^{\prime}\cap N^{\prime},M\cap N)}\\
& &\qquad\qquad\qquad\qquad\qquad -\|u\|(\delta(M,M^{\prime})+\varepsilon).
\end{eqnarray*}
If the right side of the inequality is larger than $0$, we have $u^{\prime}\in M^{\prime}\setminus
N^{\prime}$. By \auindex{Kato,\ T.}\cite[Lemma IV.2.2]{Ka95} we also have
\begin{multline*}\gamma(M^{\prime},N^{\prime})\ \le\ \frac{\dist(u^{\prime},N^{\prime})}{\dist(u^{\prime},M^{\prime}\cap N^{\prime})}\\
\le\
\frac{(1+\delta(N,N^{\prime}))\dist(u,N)+\|u\|\delta(N,N^{\prime})+\|u\|(\delta(M,M^{\prime})+\varepsilon)}{\frac{\dist(u,M\cap
N)-\|u\|\delta(M^{\prime}\cap N^{\prime},M\cap N)}{1+\delta(M^{\prime}\cap N^{\prime},M\cap
N)}-\|u\|(\delta(M,M^{\prime})+\varepsilon)}.
\end{multline*}
Let $\varepsilon\to 0$, and we have
$$\gamma(M^{\prime},N^{\prime})\le\frac{(1+\delta(N,N^{\prime}))\dist(u,N)+\|u\|(\delta(M,M^{\prime})+\delta(N,N^{\prime}))}
{\frac{\dist(u,M\cap N)-\|u\|\delta(M^{\prime}\cap N^{\prime},M\cap N)}{1+\delta(M^{\prime}\cap
N^{\prime},M\cap N)}-\|u\|\delta(M,M^{\prime})}.
$$

By Lemma \ref{l:complement-estimate}, for any%
\[
\varepsilon_1\in(0,1-\delta(M^{\prime}\cap N^{\prime},M\cap
N)-\delta(M,M^{\prime})(1+\delta(M^{\prime}\cap N^{\prime},M\cap N)),%
\]
there exists $v\in M$ such that%
\[
\dist(v,N)\le(\gamma(M,N)+\varepsilon_1)\dist(v,M\cap N)\le\|v\|(\gamma(M,N)+\varepsilon_1)%
\]
and $\dist(v,M\cap N)\ge(1-\varepsilon_1)\|v\|$. Then we have
\begin{align*}\gamma&(M^{\prime},N^{\prime})\\
&\le\ \frac{(1+\delta(N,N^{\prime}))\|v\|(\gamma(M,N)+\varepsilon_1)
+\|v\|(\delta(M,M^{\prime})+\delta(N,N^{\prime}))}
{\frac{\|v\|(1-\varepsilon_1)-\|v\|\delta(M^{\prime}\cap N^{\prime},M\cap N)}{1+\delta(M^{\prime}\cap N^{\prime},M\cap N)}-\|v\|\delta(M,M^{\prime})}\\
&=\ \frac{(1+\delta(N,N^{\prime}))(\gamma(M,N)+\varepsilon_1)
+\delta(M,M^{\prime})+\delta(N,N^{\prime})} {\frac{1-\varepsilon_1-\delta(M^{\prime}\cap
N^{\prime},M\cap N)}{1+\delta(M^{\prime}\cap N^{\prime},M\cap N)}-\delta(M,M^{\prime})}.
\end{align*}
Let $\varepsilon_1\to 0$, and we have
\[
\gamma(M^{\prime},N^{\prime})\le \frac{(1+\delta(N,N^{\prime}))\gamma(M,N)
+\delta(M,M^{\prime})+\delta(N,N^{\prime})} {\frac{1-\delta(M^{\prime}\cap N^{\prime},M\cap
N)}{1+\delta(M^{\prime}\cap N^{\prime},M\cap N)}-\delta(M,M^{\prime})}.\qedhere
\]
\exendproof

By \auindex{Kato,\ T.}\cite[Theorem IV.4.2 and Lemma IV.4.4]{Ka95}, we have

\begin{corollary}
Let $X$ be a Banach space and $M,N$ be closed linear subspaces of $X$. Then we have
\newline (a) $\limsup_{M^{\prime}\to M, N^{\prime}\to N}\gamma(M^{\prime},N^{\prime})\le\gamma(M,N)$ if $M+N$ is closed, and
\newline (b) $\lim_{M^{\prime}\to M, N^{\prime}\to N,M^{\prime}\cap N^{\prime}\to M\cap N}\gamma(M^{\prime},N^{\prime})=\gamma(M,N)$.
\end{corollary}

Now we are ready to investigate the deformation behavior, following some lines of
\auindex{Neubauer,\ G.}\cite[Lemma 1.5 (1), (2)]{Ne68}:

\begin{proposition}\label{p:close-to}
Let $(M'_j)_{j=1,2,\dots}$ be a sequence in $\Ss(X)$ converging to $M\in\Ss(X)$ in the gap
topology, shortly $M'\to M$, let similarly $N'\to N$  and let $M+N$ be closed. Then $M'\cap N'\to
M\cap N$ if and only if $M'+N'\to M+N$. Differently put, we prove in the gap topology that
$\cap\colon\Ss(X)^2_{\operatorname{cl}}\to\Ss(X)$ is continuous if and only if
$+\colon\Ss(X)^2_{\operatorname{cl}}\to\Ss(X)$ is continuous, where $\Ss(X)^2_{\operatorname{cl}}$
denotes the set $\{(M,N)\in\Ss(X)\times\Ss(X);\, M+N \text{ closed}\}$.
\end{proposition}

\begin{proof}
By \auindex{Kato,\ T.}\cite[Theorem IV.4.8]{Ka95}, we have $\gamma(N^{\bot},M^{\bot})=\gamma(M,N)$.
Here $N^{\bot},M^{\bot}\< X^*$ denote the annihilators in the dual space $X^*$\/. By
the cited theorem, $M^{\bot}+N^{\bot}$ is closed. So the proposition
follows from the above lemma.
\end{proof}

\begin{corollary}\label{c:finite-close-to}
Let $M'\to M$, $N'\to N$ and let $M+N$ be closed. Assume that $\dim(M'\cap N')=\dim(M\cap
N)<+\infty$ or $\dim X/(M'+N')=\dim X/(M+N)<+\infty$. Then we have $M'\cap N'\to M\cap N$ and
$M'+N'\to M+N$.
\end{corollary}

\begin{proof} If $\dim(M'\cap N')=\dim(M\cap N)<+\infty$, by Proposition \ref{p:closed-spaces-dimensions}.d we have
$ M'\cap N'\to M\cap N$. By Proposition \ref{p:close-to} we have $M'+N'\to M+N$. If $\dim
X/(M'+N')=\dim X/(M+N)<+\infty$, by \auindex{Kato,\ T.}\cite[Section IV.4.11]{Ka95} we have
$\dim((M')^{\bot}\cap (N')^{\bot})=\dim(M^{\bot}\cap N^{\bot})<+\infty$. By \auindex{Kato,\
T.}\cite[Theorem 4.2.9, Theorem 4.4.8]{Ka95} and the above arguments we have $(M')^{\bot}\cap
(N')^{\bot}\to M^{\bot}\cap N^{\bot}$ and $(M')^{\bot}+(N')^{\bot}\to M^{\bot}+N^{\bot}$. By
\auindex{Kato,\ T.}\cite[Theorem 4.2.9, Theorem 4.4.8]{Ka95} we have $M'\cap N'\to M\cap N$ and
$M'+N'\to M+N$.
\end{proof}

\begin{corollary}\label{c:quotient} Let $X$ be a Banach space with closed subspaces $M,M',L\in\Ss(X)$.
Denote by $p$ the natural map $p\colon X\to X/L$. Let $(M'_j)_{j=1,2,\dots}$ be a sequence in $\Ss(X)$ converging to $M\in\Ss(X)$ in the gap
topology, shortly $M'\to M$. Assume that $M+L$ is closed, and $M'\to M$. If $M'\cap L\to M\cap L$, or $\dim(M'\cap L)=\dim(M\cap
N)<+\infty$, or $\dim X/(M'+L)=\dim X/(M+L)<+\infty$, we have $p(M')\to p(M)$.
\end{corollary}

\begin{proof} By Proposition \ref{p:close-to} and Corollary \ref{c:finite-close-to}, our condition implies that $M'+L\to M+L$. By Lemma \ref{l:quotient-delta} we have $p(M')=p(M'+L)\to p(M+L)=p(M)$.
\end{proof}

Combined with the preceding corollary, the following corollary generalizes
\auindex{Boo{\ss}--Bavnbek,\ B.}\auindex{Furutani,\ K.}\cite[Theorem 3.8]{BoFu98} and yields
Proposition \ref{p:ccp}.

\begin{corollary}[Continuity of the family of the inner solution spaces and the solution spaces]\label{c:continuous-ker}
Let $X$, $Y$ be Banach spaces and $A',A\in\Ss(X\times Y)$ be closed linear relations with $A'\to A$
and $\ran A$ closed. If $\dim\ker A'=\dim\ker A<+\infty$ or $\dim Y/\ran A'=\dim Y/\ran A<+\infty$,
we have $\ker A'\to\ker A$ and $\ran A\to\ran A'$.
\end{corollary}

\begin{proof} We have
\[\ker A\times\{0\}=A\cap(X\times\{0\}),\quad X\times\ran A=A+X\times\{0\}.\]
By Corollary \ref{c:finite-close-to}, our results follows.
\end{proof}

Similar to the proof in \auindex{Kato,\ T.}\cite[Section IV.4.5]{Ka95}, we have (see \cite[Remark
IV.4.31]{Ka95} for discussions):

\begin{proposition}\label{p:kato-perturb} Let $X$ be a Banach space and let $(M,N)$ be a (semi-)Fredholm pair. Then there is a $\delta>0$
such that $\hat\delta(M^{\prime},M)+\hat\delta(N^{\prime},N)<\delta$ implies that
$(M^{\prime},N^{\prime})$ is a (semi-)Fredholm pair and
\[\Index(M^{\prime},N^{\prime})\ =\ \Index(M,N).\]
\end{proposition}
\section{Smooth family of closed subspaces in Banach spaces}\label{ss:facts-banach}
We begin with the definition.

\begin{definition}\label{d:smooth} Let $X$ be a Banach space and $B$ a $C^k$ manifold, $k$ is a nonnegative integer or $+\infty$ or $\w$.
A map $f\colon B\to\Ss(X)$ is called {\em $C^k$ at $b_0\in B$} if there exist a neighborhood $U$ of
$b_0$ and a $C^k$ map $L\colon U\to\Bb(X)$ such that $L(b)$ is invertible and $L(b)f(b_0)=f(b)$ for
each $b\in U$. $f$ is called a {\em $C^k$ map} if and only if $f$ is $C^k$ at each point $b\in B$.
For the $C^0$ case we need $B$ to be a topological space only.
\end{definition}

By the definition we have

\begin{lemma}\label{l:c0-continuous} Let $X$ be a Banach space and $B$ a topological space. Let $f\colon B\to\Ss(X)$ be a map.
If $f$ is $C^0$ at $b_0\in B$, $f$ is continuous at $b_0$.
\end{lemma}

The converse is not true in general (see \auindex{Neubauer,\ G.}\cite[Lemma 0.2]{Ne68}).

Recall from Remark \ref{r:complemented} that $\Ss^c(X)$ denotes the set of complemented subspaces
of a Banach space $X$. We omit the proof of the following standard facts.

\begin{lemma}\label{l:basic-complemented} Let $X$ be a Banach space and $M\in\Ss(X)$. We have
\newline (a) $M\in\Ss^c(X)$ if and only if there exists a $P\in\Bb(X)$ such that $P^2=P$ and $\ran P=M$, and
\newline (b) $M\in\Ss^c(X)$ if either $\dim M<+\infty$ or $\dim X/M<+\infty$.
\end{lemma}

\begin{lemma}\label{l:connectness-finite} Let $X$ be a Banach space. Let $n\ge 0$ be an integer. Set $G(n,X):=\{V\in\Ss(X);\dim V=n\}$. Then
the set $G(n,X)$ is open and path connected in $\Ss(X)$.
\end{lemma}

\begin{proof} By \auindex{Kato,\ T.}\cite[Corollary IV.2.6]{Ka95}, the set $G(n,X)$ is open. Let $V_1$ and $V_2$ be in $G(n,X)$. Since $G(n,V_1+V_2)$ is
path connected, $V_1$ and $V_2$ can be joined by a path in $G(n,V_1+V_2)$. So our result follows.
\end{proof}

\begin{lemma}\label{l:connectness-local} Let $X$ be a Banach space with a closed linear subspace $X_1$. Set $G(X,X_1):=\{M\in\Ss(X);X=M\oplus X_1\}$
(can be an empty set). Then the set $G(X,X_1)$ is an open affine subspace of $\Ss(X)$.
\end{lemma}

\begin{proof} If $G(X,X_1)=\emptyset$, our results hold. Now we assume that $G(X,X_1)\ne\emptyset$. By \auindex{Kato,\ T.}\cite[Lemma IV.4.29]{Ka95}, the
set $G(X,X_1)$ is open. Let $X_0\in G(X,X_1)$. Denote by $\Graph(A):=\{x+Ax;x\in X_0\}$ for all
bounded operators $A\in\Bb(X_0,X_1)$. By the \subindex{Closed Graph Theorem}closed graph theorem,
we have
\begin{equation}\label{e:set-H}
G(X,X_1)=\{\Graph(A);A\in\Bb(X_0,X_1)\}.
\end{equation}
Since the topology of $G(X,X_1)$ coincides with that of $\Bb(X_0,X_1)$, our results follow.
\end{proof}

\begin{corollary}\label{c:complemented-manifold} The set $\Ss^c(X)$ is a \subindex{Banach manifold}Banach manifold. The local chart at $X_0\in\Ss^c(X)$ is
defined by the equation (\ref{e:set-H}).
\end{corollary}

\begin{corollary}\label{c:pi0-fp0} Let $X$ be a Banach space with a closed linear subspace $X_1$. Then the set $G(X,X_1)$ is
dense in $\Ff_{0,X_1}(X)$, and the set $\Ff_{0,X_1}(X)$ is path connected.
\end{corollary}

\begin{proof} Let $M\in\Ff_{0,X_1}(X)$. Then there exist closed subspaces $M_1$, $X_2$ and a finite-dimensional subspace $V$
such that $M=M\cap X_1\oplus M_1$, $X_1=M\cap X_1\oplus X_2$, and $X=V\oplus(M+X_1)$. By
(\ref{e:directsum-decomposition}) we have
\[X=M\cap X_1\oplus M_1\oplus X_2\oplus V.\]
Since $\Index(M,X_1)=0$, we have $\dim M\cap X_1=\dim V$. Let $A\colon M\cap X_1\to V$ be a linear
isomorphism. Set $c_1(s):=\Graph(sA)$ for $s\in[0,1]$. Then the path $c_1\colon [0,1]\to \Ss(M\cap
X_1\oplus V)$ satisfies that $c_1(0)=M\cap X_1$ and $M\cap X_1\oplus V=M\cap X_1\oplus c_1(s)$ for
each $s\in(0,1]$. Set $c(s):=c_1(s)\oplus M_1$. Then we have $c(0)=M$ and $c(s)\in G(X,X_1)$ for
$s\in(0,1]$. So the set $G(X,X_1)$ is dense in $\Ff_{0,X_1}(X)$. By Lemma
\ref{l:connectness-local}, the set $\Ff_{0,X_1}(X)$ is path connected.
\end{proof}

By \cite[Theorem IV.4.8]{Ka95} we have

\begin{lemma}\label{l:perp-complemented}Let $X$ be a Banach space with two closed linear subspaces $M$, $N$. Then $X=M\oplus N$ if and only if $X^*=M^{\bot}+N^{\bot}$. In this case, the projection on $M$ defined by $X=M\oplus N$ is $P$ if and only if the projection on $M^{\bot}$ defined by $X^*=M^{\bot}+N^{\bot}$ is $I-P^*$.
\end{lemma}

By \cite[Lemma I.4.10]{Ka95} we have

\begin{corollary}\label{c:perp-complemented-continuous} Let $X$ be a Banach space and $B$ be a topological space with a family $M\colon B\to\Ss^c(X)$. Then $M$ is a continuous family if and only if $M^{\bot}$ is a continuous family. In this case both of them are $C^0$ families.
\end{corollary}

The "if" part of the following Lemma \ref{l:ck-complemented}.b is \auindex{Neubauer,\
G.}\cite[Lemma 0.2]{Ne68}.

\begin{lemma}\label{l:ck-complemented} Let $X$ be a Banach space and $B$ a $C^k$ manifold. For the $C^0$ case we need $B$ to be a
topological space only. Let $f\colon B\to\Ss(X)$ be a map. Let $b_0\in B$ be a point. Assume that
$f(b_0)$ is complemented in $X$. Then
\newline (a) $f$ is $C^k$ at $b_0$ if and only if there exist a neighborhood $U$ of $b_0$ and a $C^k$ map $P\colon U\to\Bb(X)$
such that $P(b)^2=P(b)$ and $\ran P(b)=f(b)$ for each $b\in U$, and
\newline (b) $f$ is $C^0$ at $b_0$ if and only if $f$ is continuous at $b_0$.
\end{lemma}

\begin{proof} (a) Since $f(b_0)$ is complemented in $X$, there exists a projection $P_0\in\Bb(X)$ such that $f(b_0)=\ran P_0$.
If $f$ is $C^k$ at $b_0$, there exist a neighborhood $U$ of $b_0$ and a $C^k$ map $L\colon
U\to\Bb(X)$ such that $L(b)$ is invertible and $L(b)f(b_0)=f(b)$ for each $b\in U$. Define
$P(b):=L(b)P_0L(b)^{-1}$. Then $P\colon U\to\Bb(X)$ is of class $C^k$ and $\ran P(b)=L(b)\ran
P_0=f(b)$ for each $b\in U$.

Conversely, if there exists a neighborhood $U$ of $b_0$ and a $C^k$ map $P\colon U\to\Bb(X)$ such
that $P(b)^2=P(b)$ and $\ran P(b)=f(b)$ for each $b\in U$, there exists a neighborhood $U_1\subset
U$ of $b_0$ such that $\|P(b)-P(b_0)\|<1$. By \auindex{Kato,\ T.}\cite[Lemma I.4.10]{Ka95}, there
exist a $C^k$ map $L\colon U_1\to\Bb(X)$ such that $L(b)$ is invertible, $L(b_0)=I$, and
$L(b)P(b_0)=P(b)L(b)$ for each $b\in U_1$. So for $b\in U_1$, we have $L(b)f(b_0)=f(b)$.
\newline (b) By Lemma \ref{l:c0-continuous} and \auindex{Neubauer,\ G.}\cite[Lemma 0.2]{Ne68}.
\end{proof}

\begin{lemma}\label{l:finite-bot-ck} Let $X$ be a Banach space and $B$ a $C^k$ manifold. For the $C^0$ case we need $B$
to be a topological space only. Let $f\colon B\to\Ss(X^*)$ be a $C^k$ map. Assume that $\dim
f(b)=n<+\infty$ for each $b\in B$. Then the map $b\mapsto f(b)^{\bot}$ is of class $C^k$.
\end{lemma}

\begin{proof} Fix $b_0\in B$. Let $x_1^*,\ldots x_n^*$ be a base of $f(b_0)$. Since $f$ is $C^k$, there exist a
neighborhood $U$ of $b_0$ and a $C^k$ map $L\colon U\to\Bb(X)$ such that $L(b)$ is invertible and
$L(b)f(b_0)=f(b)$. Set $x_k^*(b):=L(b)x_k^*$ for $k=1,\ldots,n$ and $b\in U$. Then $x_k^*\colon
U\to X^*$ is a $C^k$ map for each $k=1,\ldots,n$.

Since $\dim f(b_0)=n$, there exist $x_1,\ldots,x_n\in X$ such that the matrix $M(b_0)$ is
invertible, where $M(b):=((x_j^*(b))(x_k))_{j,k=1,\ldots,n}$. The map $M\colon U\to\gl(n,\C)$ is of
class $C^k$. Then there exists a neighborhood $U_1\subset U$ of $b_0$ such that $\det M(b)\ne 0$.
Set $N(b,x):=((x_k^*(b))(x)x_j)_{j,k=1,\ldots,n}$. Define $P(b)\in\Bb(X)$ by
$P(b)x=x-M(b)^{-1}N(b,x)$. Then $P\colon U_1\to\Bb(X)$ is $C^k$, $P(b)^2=P(b)$, and
$f(b)^{\bot}=\ran P(b)$ for each $b\in U_1$. By Lemma \ref{l:ck-complemented}, the map $b\mapsto
f(b)^{\bot}$ is of class $C^k$.
\end{proof}

\section{Basic facts about symplectic Banach bundles}\label{ss:banach-bundles}
The central concept of this Memoir is the continuous variation (i.e., the parametrization) of
Fredholm pairs of Lagrangian subspaces in varying Hilbert or Banach spaces with varying symplectic
forms, see the preceding Sections \ref{ss:maslov-unitary} (dealing with strong symplectic Hilbert
bundles), \ref{ss:maslov-definition} and \ref{ss:cal-maslov} (dealing with symplectic Banach
bundles), \ref{ss:proof-asff} (proving the abstract desuspension spectral flow formula), and
\ref{ss:gsff} (dealing with curves of well-posed elliptic boundary value problems). The concept of
\textit{symplectic Banach bundles} is a natural generalization of the familiar concept of vector
bundles. It provides a suitable frame for making the notion of \textit{continuous variation}
rigorous. We summarize the essential properties in the following list. We refer to
\auindex{Zaidenberg,\ M.G.}\auindex{Krein,\ S.G.}\auindex{Kucment,\ P.A.}\auindex{Pankov,\
A.A.}\cite{ZKKP:1975} for more details regarding the concept of Banach bundles.

\begin{note} Typically in this Memoir the base space of the considered Banach bundles is the
interval, hence contractible. Then the total space can be written as a trivial product, i.e., the
fibres can always be identified.
\end{note}

\begin{properties}[Basic properties of Banach bundles]\label{prop:banach-bundles}
(1) Let $B$ be a topological space and  $p\colon \mathbb{X}\to B$ a Banach bundle with fibers
$p^{-1}(b)=X(b)$ for each $b\in B$. For simplicity, we shall restrict ourselves to the case
$B=[0,1]$ and write shortly $\{X(s)\}_{s\in[0,1]}$ instead of $p\colon \mathbb{X}\to [0,1]$. This
means that there exists an open covering $\{I(t)\}_{t\in A}$ of $[0,1]$, with $t\in I(t)$, and $A$
is a given subset of $[0,1]$, such that there exists a Banach isomorphism $\varphi(t,s):X(t)\to
X(s)$ for each $s\in I(t)$. It is called of class $C^k$ if $\varphi(t_2,s)^{-1}\varphi(t_1,s)$ is
of class $C^k$ for $s\in I(t_1)\cap I(t_2)$ and all $t_1,t_2\in A$.

(2) A family of forms $\{\w(s)\}_{s\in [0,1]}$ on the fibers $X(s)$ is called continuous ($C^k$) iff
all the families $\{\varphi(t,s)^*(\w(s))\}_{s\in I(t)}$ are continuous ($C^k$) for each $t\in A$.

(3) A family of closed subspaces $\{M(s)\}_{s\in [0,1]}$ of $X(s)$ is called continuous ($C^k$) iff
all the families $\{\varphi(t,s)^{-1}(M(s))\}_{s\in I(t)}$ are continuous ($C^k$) for each $t\in
A$.

(4) For simplicity we identify all the fibers $X(s)$ of $\mathbb{X}$ with one fixed Banach space
$X$. Let $\{X=X(s)^+\oplus X(s)^-\}_{s\in [0,1]}$ be a family of splittings. It is called
continuous iff the family of projections $\{P(s)\}_{s\in [0,1]}$ from $X$ onto $X(s)^+$ along
$X(s)^-$ is continuous. Then for all $s,t\in [0,1]$ with $\|P(s)-P(t)\|<1$, there exists an
invertible operator $U(t,s)$, such that $P(s)U(t,s)=U(t,s)P(t)$. Then the family $\{X,P(s)\}_{s\in
[0,1]}$ forms a bundle. So we may fix $X(s)^{\pm}=:X^{\pm}$ and $P(s)=:P$ locally.

(5) Let $\{M(s)\}_{s\in [0,1]}$ be a family of closed subspaces of a fixed Banach space $X$, as in
Property 4. For all $s$, let $M(s)$ be the graph of a suitable $U(s)\colon X^+\to X^-$. Then the
family $\{M(s)\}$ is continuous in the gap topology iff the operator family $\{U(s)\}$ is
continuous in the graph norm (by definition). If all $U(s)$ are bounded linear maps, then the
family $\{U(s)\}$ is continuous in the graph norm iff the family $\{U(s)\}$ is continuous in the
operator norm.
\end{properties}

\begin{rem}\label{r:banach-bundles}
(a) The preceding list becomes very simple in the special case of a Hilbert bundle $p\colon
\mathbb{H}\to [0,1]$, considered above in Section \ref{ss:maslov-unitary}. We identify the
underlying vector spaces of the fibers $p^{-1}(s)=H(s)=:H$ of Hilbert spaces for all $s\in [0,1]$
and require that the bounded invertible operators $A_{s,0}$ defined by
\[
\lla x,y\rra_s\ = \ \lla A_{s,0} x,y\rra_0\/, \text{ for all $x,y\in H$}
\]
form a continuous family. That is the reformulation of Property 1. It explains what we mean by
$\{H,\lla\cdot,\cdot\rra_s\}_{s\in [0,1]}$ being a continuous family of Hilbert spaces.
\newline (b) Similarly, we can reformulate Property 2 in that case: A family of symplectic forms
$\{\w(s)(x,y)=\lla J(s)x,y\rra\}$ is continuous iff the family of injective operators $\{J(s)\}$ is
continuous in the operator norm (in the case of strong symplectic forms) or in the gap topology (in
the case of weak symplectic forms). Actually, that definition generalizes to symplectic forms in
Banach bundles, namely requiring that the family of injective operators $\{J(s)\colon X(s)\to
X(s)^*\}$ given by $\w(x,y)=(J(s)(x))(y)$ is continuous.
\newline (c) For strong symplectic Hilbert bundles, the continuity of the canonical splitting
$\{H=\ker(J(s)- iI)\oplus\ker(J(s)+ iI)\}$ is just Property 3.
\newline (d) Properties 4 and 5 explain the equality of the two natural topologies of the
Fredholm Lagrangian Grassmannian in the presence of a symplectic splitting (as canonically given in
strong symplectic Hilbert spaces and assumed in \cite{BooZhu:2013}): a curve of Lagrangian
subspaces is continuous in the gap topology iff the curve of the unitary generators of the
Lagrangians is continuous.
\end{rem}

\section{Embedding Banach spaces}\label{ss:embedding-banach}

Let $j\colon W\to X$ be a Banach space embedding. In this subsection we study the continuous family
of closed subspaces in $W$ which is also closed in $X$.

\begin{proposition}\label{p:embedding-continuous1} Let $W$, $X$ be Banach spaces. Let $j\in\Bb(W,X)$ be an
injective bounded linear map. Let $M\subset W$ be such that $j(M)\in\Ss(X)$. Then the following
hold.
\newline (a) $M$ is closed in $W$, and the linear map $(j|_M)^{-1}\colon j(M)\to M$ is bounded.
\newline (b) Denote by $C(M)=\|(j|_M)^{-1}\|$. Let $N\in \Ss(W)$ be a closed linear subspace.
Assume that $M\ne \{0\}$ and $\delta(N,M)<(1+\|j\|C(M))^{-1}$. Then we have $j(N)\in\Ss(X)$ and
\begin{equation}\label{e:uniform-CN}C(N)\le C(M)(1-(1+\|j\|C(M))\delta(N,M))^{-1}.
\end{equation}
\newline (c) Under the assumptions of (b), we have
\begin{equation}\label{e:delta-embedding}
\hat\delta(j(M),j(N))\le C(M)\|j\|(1-(1+\|j\|C(M))\delta(N,M))^{-1}\hat\delta(M,N).
\end{equation}
\end{proposition}

\begin{proof} (a) By the continuity of $j$, $M=j^{-1}j(M)$ is closed. By the closed graph theorem,
the linear map $(j|_M)^{-1}\colon j(M)\to M$ is bounded.
\newline (b) Since $M\ne \{0\}$, we have $C(M)>0$. Let $x\in N$ and $\e\in(0,1-(1+\|j\|C(M))\delta(N,M))$. Then there exists a $y\in M$ such that
\[(1-\e)(\|x\|_W-\|y\|_W)\le(1-\e)\|x-y\|_W\le d(x,M)\le\delta(N,M)\|x\|_W.\]
So we have
\begin{equation}\label{e:x-y-W}
\|x\|_W\le\frac{(1-\e)\|y\|_W}{1-\e-\delta(N,M)}.
\end{equation}
Note that
\begin{align*}
C(M)^{-1}\|y\|_W&-\|j(x)\|_X\le\|j(y)\|_X-\|j(x)\|_X\\
&\le\|j(x)-j(y)\|_X\le\|j\|\|x-y\|_W.
\end{align*}
By (\ref{e:x-y-W}) we have
\begin{align*}
\|j(x)\|_X&\ge C(M)^{-1}\|y\|_W-\|j\|\|x-y\|_W\\
&\ge\frac{C(M)^{-1}(1-\e-\delta(N,M))\|x\|_W}{1-\e}-\frac{\|j\|\delta(N,M)\|x\|_W}{1-\e}\\
&=C(M)^{-1}\|x\|_W\left(1-\frac{(1+\|j\|C(M))\delta(N,M)}{1-\e}\right).
\end{align*}
Let $\e\to 0$, and we have
\[\|j(x)\|_X\ge C(M)^{-1}\|x\|_W\left(1-(1+\|j\|C(M))\delta(N,M)\right).\]
Since $N$ is closed in Banach space $W$, we have $j(N)\in\Ss(X)$ and the equation
(\ref{e:uniform-CN}) holds.
\newline (c) By the definition of the gap we have
\begin{align}
\label{e:delta-M-N}\delta(j(M),j(N))&\le C(M)\|j\|\delta(M,N),\\
\label{e:delta-N-M}\delta(j(N),j(M))&\le C(N)\|j\|\delta(N,M).
\end{align}
Then we have
\begin{equation}\label{e:hat-delta-M-N}
\hat\delta(j(M),j(N))\le\max\{C(M),C(N)\}\|j\|\hat\delta(M,N).
\end{equation}
By (b), our result follows.
\end{proof}

\begin{corollary}\label{c:continuous-lift} Let $B$ be a topological space. Let $q\colon F\to B$
and $p\colon E\to B$ be Banach bundles with fibers $q^{-1}(b):=W(b)$ and
$p^{-1}(b):=X(b)$ for each $b\in B$ respectively. Assume that we have a Banach subbundle map
$F\to E$, and there is a family $M(b)\in\Ss(X(b))$, $b\in B$ such that
$M(b)\subset W(b)$ for each $b\in B$, and the family $M(b)\in\Ss(W(b))$, $b\in B$ is continuous.
Then the family $M(b)\in\Ss(X(b))$, $b\in B$ is continuous.
\end{corollary}

\begin{proof} By Proposition \ref{p:embedding-continuous1}.a, we have $M(b)\in\Ss(W(b))$
for each $b\in B$. By Proposition \ref{p:embedding-continuous1}.c, the family $M(b)\in X(b)$, $b\in B$ is continuous.
\end{proof}

The following corollary is the second main result of this appendix. It generalizes
\auindex{Nicolaescu,\ L.}\cite[Proposition B.1]{Ni97} and \auindex{Boo{\ss}--Bavnbek,\
B.}\auindex{Lesch,\ M.}\auindex{Zhu,\ C.}\cite[Theorem 7.16]{BoLeZh08}.

\begin{corollary}\footnote{Added in proof: By slightly different methods, a related result was obtained by M. Prokhorova \auindex{Prokhorova,\ M.}\cite{Pro17} in the meantime.}[Continuity of the operator family]\label{c:continuous-operator} Let $B$ be a topological space and
 $p\colon \mathbb{X}\to B$, $p_1\colon \mathbb{X}_1
\to B$, $q\colon \mathbb{Y}\to B$  three Banach bundles with fibers $p^{-1}(b)=X(b)$,
$p_1^{-1}(b)=D_M(b)$ and $q^{-1}(b)=Y(b)$ for each $b\in B$ respectively. Assume that
$\mathbb{X}_1$ is a subbundle of $\mathbb{X}$, and we have two continuous families
$A(b)\in\Bb(D_M(b),Y(b)))$ and $D(b)\in\Ss(D_M(b))$, $b\in B$ such that $A(b)|_{D(b)}\colon
X(b)\supset D(b)\to Y(b)$ is a closed operator for each $b\in B$. Then the family of operators
$\left\{A(b)|_{D(b)}\in\Cc(X(b),Y(b))\right\}_{b\in B}$ is continuous.
\end{corollary}

\begin{proof} Set $W(b):=\Graph(A(b)|_{D_M(b)})$, $\wt W(b):=D_M(b)\times Y(b)$, $Z(b):=X(b)\times Y(b)$, $F:=\bigcup_{b\in B}W(b)$,
$\wt F:=\bigcup_{b\in B}\wt W(b)$, and $E_2:=\bigcup_{b\in B}Z(b)$. Then we have a subbundle map
$\wt F\to E_2$. Since the family $A(b)\in\Bb(D_M(b),Y(b)))$, $b\in B$ is continuous and $\tilde W(b)=W(b)\oplus\{0\}\times Y(b)$, we have subbundles $F\to \wt F$ and $\tilde F \to E_2$. Then our result follows from Corollary
\ref{c:continuous-lift}.
\end{proof}

\begin{rem}\label{r:continuous-operator} (a) By Lemma \ref{l:quotient-delta}, our condition means $P_n\to P$
in $H^{\sigma}(\Sigma)\to H^{\sigma}(\Sigma)$ (in \cite{Ni97} \auindex{Nicolaescu,\ L.}Nicolaescu
uses the notation $L_{\sigma}^2$ for the Sobolev space $H^{\sigma}(\Sigma)$) for $\sigma=1/2$. We
do not require the condition for $\sigma=0$.
\newline (b) In \auindex{Boo{\ss}--Bavnbek,\ B.}\auindex{Lesch,\ M.}\auindex{Phillips,\ J.}\cite[Theorem 3.9 (d)]{BoLePh01},
it is assumed that $P_n\to P$ in $H^{\sigma}(\Sigma)\to H^{\sigma}(\Sigma)$ for $\sigma=0$. The
proof is incomplete there. For the correct proof, see \auindex{Boo{\ss}--Bavnbek,\
B.}\auindex{Lesch,\ M.}\auindex{Zhu,\ C.}\cite[Theorem 7.16]{BoLeZh08}.
\end{rem}

\section{Compact perturbations of closed subspaces}\label{ss:compact-perturb}
Let $X$ be a Banach space and $M$ a closed subspace of $X$. In this subsection we study compact
perturbations of $M$.

We recall the notion of relative index between projections.

\begin{definition}\label{d:relative-index} Let $P,Q\in\Bb(X)$ be projections and $QP\colon \ran P\to \ran Q$ is Fredholm.
The {\em relative index} $[P-Q]$ is defined by
\begin{equation}\label{e:relative-index}
[P-Q]:=\Index(QP\colon \ran P\to \ran Q).
\end{equation}
\end{definition}

The relative index has the following properties.

\begin{lemma}\label{l:relative-index} Let $X$ be a Banach space and $P,Q,R,P_1,Q_1\in \Bb(X)$ projections.
\newline (a) we have $[P-Q]=\Index(\ran P,\ker Q)=[(I-Q)-(I-P)]$.
\newline (b) If $P-Q$ is compact, $QP\colon \ran P\to \ran Q$ is Fredholm.
\newline (c) If $P-Q$ or $Q-R$ is compact, we have $[P-Q]+[Q-R]=[P-R]$. In particular, we have $[P-Q]=-[Q-P]$ if $P-Q$ is compact.
\newline (d) If $PP_1=P_1P=0$, $QQ_1=Q_1Q=0$ and $P-Q$ (or $P_1-Q_1$) is compact, we have $[(P+P_1)-(Q+Q_1)]=[P-Q]+[P_1-Q_1]$.
\newline (e) If $T\in\Bb(X,Y)$ is invertible, we have $[TPT^{-1}-TQT^{-1}]=[P-Q]$.
\end{lemma}

\begin{proof} (a) We have $\ker(QP\colon \ran P\to \ran Q)=\ran P\cap\ker Q$. Note that $\ran P+\ker Q=Q(\ran P)+\ker Q=\ran (QP)+\ker P$. Then we have
\[X/(\ran P+\ker Q)=(\ran Q+\ker Q)/(\ran P+\ker Q)\simeq\ran Q/\ran (QP).\]
So we have
\begin{align*}[P-Q]&=\Index(\ran P,\ker Q)=\Index(\ran(I-Q),\ker(I-P))\\
&=[(I-Q)-(I-P)].
\end{align*}
\newline (b)-(c) See \auindex{Zhu,\ C.}\auindex{Long,\ Y.}\cite[Lemma 2.2,2.3]{ZhLo99}.
Note that (c) follows from the proof of \auindex{Zhu,\ C.}\auindex{Long,\ Y.}\cite[Lemma
2.3]{ZhLo99}.
\newline (d) Note that $Q_1P=(Q_1-P_1)P=Q_1(P-Q)$ and $QP_1=(Q-P)P_1=Q(P_1-Q_1)$ are compact. So we have
\begin{align*}[(P+P_1)-(Q+Q_1)]&=\Index((Q+Q_1)(P+P_1):\ran(P+P_1)\to\ran(Q+Q_1))\\
&=\Index(QP+Q_1P_1:\ran(P+P_1)\to\ran(Q+Q_1))\\
&=\Index(QP:\ran P\to\ran Q)+\Index(Q_1P_1:\ran P_1\to\ran Q_1)\\
&=[P-Q]+[P_1-Q_1].
\end{align*}
\newline (e) It follows from the definition.
\end{proof}

\begin{definition}\label{d:compact-perturb} Let $X$ be a Banach space and $M$, $N$ be closed subspaces of $X$.
\newline (a) We define $M\sim^f N$ if $\dim M/(M\cap N), \dim N/(M\cap N)<+\infty$, and call $N$ a
{\em finite change} of $M$ (see \auindex{Neubauer,\ G.}\cite[p. 273]{Ne68}).
\newline (b) We define $M\sim^{c} N$ if there exist closed subspaces $M_1\subset M$, $N_1\subset N$ and
a compact operator $K\in\Bb(X)$ such that $I+K$ is invertible, $N_1=(I+K)M_1$ and $\dim M/M_1, \dim
N/N_1<+\infty$, and call $N$ a {\em compact perturbation} of $M$. In this case we define the {\em
relative index} $[M-N]:=\dim M/M_1-\dim N/N_1$.
\end{definition}

\begin{lemma}\label{l:finite-codimension} Let $X$ be a vector space and $M, N, W$ three linear subspaces. If $N+W\subset M$, we have
$\dim M/(N\cap W)\le\dim M/N+\dim M/W$.
\end{lemma}

\begin{proof} We have
\begin{align*}\dim M/(N\cap W)&=\dim M/N+\dim N/(N\cap W)\\
&=\dim M/N+\dim(N+W)/W\\
&\le\dim M/N+\dim M/W. \qedhere
\end{align*}
\exendproof

\begin{lemma}\label{l:zero-relative-index} Let $X$ be a Banach space and $M$, $M_1$, $M_2$ be closed linear subspaces.
Assume that $M_1$, $M_2$ are subspaces of $M$ with finite codimension in $M$, and there exists a
compact operator $K\in\Bb(M_1,X)$ such that $(I_{M_1}+K)M_1=M_2$ and $I_{M_1}+K\in\Bb(M_1,M_2)$ is
invertible. Then there exists an invertible operator $L\in\Bb(M)$ such that $L-I_M$ is compact and
$L|_{M_1}=I_{M_1}+K$. In particular, we have $\dim M/M_1=\dim M/M_2$.
\end{lemma}

\begin{proof} Let $V_1$, $V_2$ be finite-dimensional subspaces of $M$ such that $M=M_1\oplus V_1=M_2\oplus V_2$.
Let $A\in\gl(V_1,V_2)$ be a linear map. Set $L:=(I_{M_1}+K)\oplus A$. Then $L\in\Bb(M)$ is Fredholm
and
\[\Index L=\Index A=\dim V_1-\dim V_2.\]
For any bounded set $B$ of $M$, the sets $\{x\in M_1;x+v\in B\text{ for some }v\in V_1\}$ and
$\{v\in V_1;x+v\in B\text{ for some }x\in M_1\}$ are bounded. Since $K$ is compact, the set
$(L-I_M)(B)=\{Kx+(A-I_M)v;x\in M_1, v\in V_1, x+v\in B\}$ is a sequentially compact set. Thus $L-I_M$ is
compact and $\Index L=0$. So we have $\dim V_1-\dim V_2=\dim M/M_1-\dim M/M_2=0$. Then we can
choose $A$ such that $A$ is invertible. In this case $L$ is invertible.
\end{proof}

Now we are ready for the third main result of this appendix:

\begin{proposition}\label{p:compact-perturb} (a) The relations $\sim^f$ and $\sim^c$ are equivalence relations.
\newline (b) If $M\sim^c N$ holds, the relative index $[M-N]$ is well-defined. In the case of $[M-N]=0$,
there exists a compact operator $K\in\Bb(M,X)$ such that $(I_M+K)M=N$ and $I_M+K\in\Bb(M,N)$ is
invertible.
\newline (c) If $M\sim^c N$ and $\dim M_1,\dim N_1<+\infty$ hold, we have $[M-N]=-[N-M]$ and $[M_1-N_1]=\dim M_1-\dim N_1$.
\newline (d) If $M\sim^c N\sim^c W$ holds, we have $[M-N]+[N-W]=[M-W]$.
\newline (e) If $M\cap M_1=N\cap N_1=\{0\}$, $\dim M_1,\dim N_1<+\infty$, $M\sim^c N$ holds if and only
if $M+M_1\sim^c N+N_1$. In this case we have $[(M+M_1)-(N+N_1)]=[M-N]+[M_1-N_1]$.
\newline (f) Assume that $M\in\Ss^c(X)$. Then $M\sim^c N$ holds if and only if $N\in\Ss^c(X)$,
and there exist projections $P,Q\in\Bb(X)$ such that $P-Q$ is compact, $\ran P=M$ and $\ran Q=N$.
In this case we have $[M-N]=[P-Q]$. In the case of $[P-Q]=0$, there exists a compact operator
$K\in\Bb(X)$ such that $I+K$ is invertible and $(I+K)M=N$.
\newline (g) If $M\sim^c N$ and $M\in\Ss^c(X)$, we have $\Ff_{k+[M-N],M}(X)=\Ff_{k,N}(X)$.
\newline (h) If $M\sim^c N$, $M_1\sim^c N_1$, and $M, M_1\in\Ss^c(X)$, we have $\Index(M_1,N_1)-\Index(M,N)=[M_1-M]+[N_1-N]$.
\end{proposition}

\begin{proof} (a) (i) If $M\sim^f N$, we have $N\sim^f M$. If $M\sim^f N\sim^f W$, we have
\begin{align*}
\dim (M\cap N)/(M\cap N\cap W)&=\dim(M\cap N+N\cap W)/(N\cap W)\\
&\le\dim N/(N\cap W)<+\infty.
\end{align*}
Then we have $\dim M/(M\cap N\cap W),\dim N/(M\cap N\cap W) <+\infty$. Similarly, we have $\dim
W/(M\cap N\cap W)<+\infty$ and $M\sim^f W$.

(ii) If $M\sim^{c} N$, there exist closed subspaces $M_1\subset M$, $N_1\subset N$ and a compact
operator $K\in\Bb(X)$ such that $I+K$ is invertible, $N_1=(I+K)M_1$ and $\dim M/M_1, \dim
N/N_1<+\infty$. Then $(I+K)^{-1}-I$ is compact and $M_1=(I+K)^{-1}N_1$. So we have $M\sim^{c} N$.

(iii) If $M\sim^{c} N\sim^{c} W$, there exist closed subspaces $M_1$, $N_1$, $N_2$, $W_2$ and
compact operators $K,L\in\Bb(X)$ such that $I+K$,  $I+L$ is invertible, $N_1=(I+K)M_1$,
$W_2=(I+L)N_2$ and
\[\dim M/M_1, \dim N/N_1,\dim N/N_2,\dim W/W_2<+\infty.\]
By Lemma \ref{l:finite-codimension} we have $\dim N/(N_1\cap N_2)<+\infty$. Set
\begin{align*}
M_3:&=(I+K)^{-1}(N_1\cap N_2)\subset M_1\subset M, \\
W_3:&=(I+L)(N_1\cap N_2)\subset W_2\subset W.
\end{align*}
Then we have that $(I+L)(I+K)-I$ is compact, $W_3=(I+L)(I+K)M_3$, $\dim M_1/M_3=\dim N_1/(N_1\cap
N_2)<+\infty$, and $\dim W_1/W_3=\dim N_2/(N_1\cap N_2)<+\infty.$ Thus we have $\dim M/M_3,\dim
W/W_3<+\infty$ and $M\sim^{c} W$.
\newline (b) Assume that we are given closed subspaces $M_j\subset M$, $N_j\subset N$ and
compact operators $K_j\in\Bb(X)$, $j=1,2$ such that $I+K_j$ is invertible, $N_j=(I+K_j)M_j$ and
$\dim M/M_j, \dim N/N_j<+\infty$ for $j=1,2$. Since $(I+K_2)(I+K_1)^{-1}-I$ is compact and
$(I+K_j)(M_1\cap M_2)\subset N$, by Lemma \ref{l:zero-relative-index} we have
\[\dim N/((I+K_1)(M_1\cap M_2))=\dim N/((I+K_2)(M_1\cap M_2)).\]
So we have
\begin{align*}
\dim M/M_1&-\dim M/M_2=\dim M_2/(M_1\cap M_2)-\dim M_1/(M_1\cap M_2)\\
&=\dim ((I+K_2)M_2)/((I+K_2)(M_1\cap M_2))\\
&\quad-\dim((I+K_1)M_1)/((I+K_1)(M_1\cap M_2))\\
&=\dim N/((I+K_2)M_1)-\dim N/((I+K_1)M_2)\\
&=\dim N/N_1-\dim N/N_2,
\end{align*}
and therefore $\dim M/M_1-\dim N/N_1=\dim M/M_2-\dim N/N_2$. Then $[M-N]$ is well-defined when
$M\sim^{c} N$. If $[M-N]=0$, by the proof of Lemma \ref{l:zero-relative-index} we get the desired
$K$.
\newline (c) By definition.
\newline (d) We use the notations of the proof of (a) (iii). Then we have
\begin{align*}
[M-N]&-[N-W]=(\dim M/M_3-\dim N/(N_1\cap N_2))\\
&\quad+(\dim N/(N_1\cap N_2)-\dim W/W_3)\\
&=\dim M/M_3-\dim W/W_3=[M-W].
\end{align*}
\newline (e) Since $M\sim^{c}M+M_1$ and $N\sim^{c}N+N_1$, by (b), $M\sim^c N$ holds if and only
if $M+M_1\sim^c N+N_1$. By (d), in this case we have
\begin{align*}
[(M+M_1)&-(N+N_1)]=[(M+M_1)-M]+[M-N]\\
&\quad +[N-(N+N_1)]\\
&=\dim M_1+[M-N]-\dim N_1\\
&=[M-N]+[M_1-N_1].
\end{align*}
\newline (f) (i) If $M\sim^{c} N$, there exist closed subspaces $M_1\subset M$, $N_1\subset N$
and a compact operator $K\in\Bb(X)$ such that $I+K$ is invertible, $N_1=(I+K)M_1$ and $\dim M/M_1,
\dim N/N_1<+\infty$. Then we have $\dim((I+K)M)/N_1, \dim N/N_1<+\infty$. Since $M\in\Ss^c(X)$,
$(I+K)M\in\Ss^c(X)$. By Proposition \ref{p:finite-extension}, $N_1\in\Ss^c(X)$ and $N\in\Ss^c(X)$.

Let $V_1$, $V_2$ be finite-dimensional subspaces such that $(I+K)M=N_1\oplus V_1$ and $N=V_2\oplus
N_1$. Set $W_1:=(I+K)\ker P$, $W_2:=\ker Q$ and $Y:=(N_1+W_1)\cap(N_1+W_2)$. By Lemma
\ref{l:finite-codimension} we have $\dim X/Y<+\infty$. Set $W_3:=W_1\cap(N_1+W_2)$ and
$W_4:=W_2\cap(N_1+W_1)$. Then $W_3$ and $W_4$ are closed, and we have
\[Y=N_1\oplus W_3=N_1\oplus W_4.\]
So $W_1/W_3\simeq(N_1+W_1)/Y$ and $W_2/W_4\simeq(N_1+W_2)/Y$ are finite-dimensional. Let
$V_3\subset W_1$, $V_4\subset W_2$ be finite-dimensional subspaces such that
\[W_3\oplus V_3=W_1\text{ and }W_4\oplus V_4=W_2.\]
So we have
\begin{align}
\label{e:splitting1}X&=((I+K)M)\oplus W_1=(N_1\oplus V_1)\oplus(W_3\oplus V_3)\\
\nonumber&=N\oplus W_2=(N_1\oplus V_2)\oplus(W_4\oplus V_4)\\
\label{e:splitting2}&=(N_1\oplus V_2)\oplus(W_3\oplus V_4).
\end{align}
The projection of $X$ on $(I+K)M$ defined by (\ref{e:splitting1}) is $\wt P:=(I+K)P(I+K)^{-1}$.
Denote by $\wt Q$ the projection of $X$ on $N$ defined by (\ref{e:splitting1}). Then $\ran(\wt
P-\wt Q)=\wt P(V_2\oplus V_4)$ is finite-dimensional. Since $P-\wt P$ is compact, $P-\wt Q$ is
compact.

(ii) If there exist projections $P,Q\in\Bb(X)$ such that $\ran P=M$ and $\ran Q=N$, we set
$R:=QP+(I-Q)(I-P)$. Assume that $R$ is Fredholm with index $0$. Set
\begin{align*}
V_5:&=\ker(QP\colon \ran P\to\ran Q),\\
W_6:&=\ran(QP\colon \ran P\to\ran Q),\\
V_7:&=\ker((I-Q)(I-P)\colon \ker P\to\ker Q),\\
W_8:&=\ran((I-Q)(I-P)\colon \ker P\to\ker Q).
\end{align*}
Let $W_5$, $W_7$ be closed subspaces and $V_6$, $V_8$ finite-dimensional subspaces such that
\[\ran P=V_5\oplus W_5,\ran Q=V_6\oplus W_6,\ker P=V_7\oplus W_7,\ker Q=V_8\oplus W_8.\]
Then we have a bounded invertible linear map
\[\wt R:=R|_{W_5+W_7}\colon W_5+W_7\to W_6+W_8\]
with $\wt R(W_5)=W_6$, $\wt R(W_7)=W_8$. Since $\Index R=0$, we have $\dim V_5+\dim V_7=\dim
V_6+\dim V_8$. Let $A\in\GL(V_5+V_7,V_6+V_8)$ be invertible. Set $L:=\wt R\oplus A$. Then
$L\in\Bb(X)$ is invertible, $L(W_5)=W_6$ $L(W_7)=W_8$. Note that $\dim M/W_5,N/W_6<+\infty$.

In this case, we have
\[[M-N]=\dim M/W_5-\dim N/W_6=\dim V_5-\dim V_6=[P-Q].\]

If $[P-Q]=0$, we can require $A(V_5)=V_6$, $A(V_7)=V_8$ and then $L(M)=N$.

Now assume that $P-Q$ is compact. Then $R-I=(Q-P)(2P-I)$ is compact and $\Index R=0$. So we can
apply the above argument. In this case $L-I$ is compact and our result is obtained.
\newline (g) By Lemma \ref{l:fp-complemented}, we have $\Ff_{k+[M-N],M}(X),\Ff_{k,N}(X)\subset\Ss^c(X)$.
Let $W\in\Ss^c(X)$. By (f), there exist projections $P,Q,R\in\Bb(X)$ such that $\ran P=M$,
 $\ran Q=N$, $\ran R=W$, and $P-Q$ is compact. By (f) and Lemma \ref{l:relative-index}.a,c we have
\begin{align*}
\Index(M,W)&=[P-(I-R)]=[P-Q]+[Q-(I-R)]\\
&=[M-N]+\Index(N,W),
\end{align*}
and one side of each equality is well-defined if and only if the other side is. Thus we have
$\Ff_{k+[M-N],M}(X)=\Ff_{k,N}(X)$.
\newline (h) By (g) we have
$$\Index(M_1,N_1)=\Index(M_1,N)+[N_1-N]=\Index(M,N)+[M_1-M]+[N_1-N].
$$
\end{proof}

\begin{corollary}\label{c:pi0-fpc0} Let $X$ be a Banach space with a complemented closed linear subspace $M$.
Let $P\in\Bb(X)$ be a projection onto $M$ with $N:=\ker P$. Set
\begin{align}
\Cc\Pp_0(X,M):&=\{(W\in\Ss^c(X);W\sim^c M, [W-M]=0\},\\
\Cc\Pp_0(P):&=\{W\in \Cc\Pp_0(X,M);X=W\oplus N\}.
\end{align}
Then
\newline (a) the set $\Cc\Pp_0(P)$ is an affine space (hence contractible), and
\newline (b) the set $\Cc\Pp_0(P)$ is dense in $\Cc\Pp_0(X,M)$, and the set $\Cc\Pp_0(X,M)$ is path connected.
\end{corollary}

\begin{proof} (a) Let $W\in \Cc\Pp_0(P)$. By Lemma \ref{l:connectness-local} we have $W=\Graph(A)$ for some $A\in\Bb(M,N)$.
Denote by $P_W$ the projection of $X$ onto $W$ along $N$, and we have $P_W(x+y)=x+Ax$ for $x\in M$,
$y\in N$. By Proposition \ref{p:compact-perturb}.f, $A$ is compact. Conversely, for a given compact
operator $A\in\Bb(M,N)$, the space $W:=\Graph(A)\in\Cc\Pp_0(P)$. So we have
\begin{equation}\label{e:cp0}
\Cc\Pp_0(P)=\{\Graph(A);A\in\Bb(M,N)\text{ is compact}\}
\end{equation}
and the set $\Cc\Pp_0(P)$ is an affine space (hence contractible).

Let $W\in\Cc\Pp_0(X,M)$. By the proof of Corollary \ref{c:pi0-fp0}, there exists a path $c\colon
[0,1]\to\Ss^c(M)$ such that $c(0)=W$, $c(s)\sim^c W$, $c(s)\in G(X,N)$ and  $[W-c(s)]=0$ for
$s\in(0,1]$. Since $W\sim^c M$ and $[W-M]=0$, by Proposition \ref{p:compact-perturb}.a,d we have
$c(s)\sim^c M$ and $[c(s)-M]=0$. So we have $c(s)\in\Cc\Pp_0(X,M)\cap G(X,N)=\Cc\Pp_0(P)$ for
$s\in(0,1]$. Our results then follow from (a).
\end{proof}


%
%
%


%
%
%

}


\bibliography{GlobalAnalysis}
\bibliographystyle{amsplain-jl}

%
\indexcomment{Except a few standard notations, all symbols are explained at their first occurrence.
We recall a few very standard notations and then we provide an index to the used more peculiar
symbols.
{\small
\begin{tabular}{rl}
$\Bb(X,X'), \Cc(X,X'),\Ff(X,X')$ & Bounded, closed, Fredholm operators\\
$C(\dots),\Ci(\dots)$ & Continuous resp. smooth functions \\
$\Ci(M;E)$ & Smooth sections of the vector bundle $E$ over $M$ \\
$\Ci_0(M;E)$ & Smooth sections with support in the interior $M\setminus \partial M$ of $M$ \\
$D^n$ & Unit ball in Euclidean $n$-space\\
$\GL(n,\Aa)$ & Invertible $n\times n$ matrices with entries in $\Aa$\\
$\gl(n,\Aa)$ & $n\times n$ matrix algebra over the algebra $\Aa$\\
$\Hh$ & Generic name for a Hilbert space\\
$I$ & Identity operator\\
$L^2(M;E)$ & $L^2$ sections of the Hermitian vector bundle $E$\\
$M, \Si$ & Generic names for Riemannian manifolds w/wo boundary\\
$\NN,\ZZ,\RR,\CC$ & Natural, integer, real, complex numbers \\
$\RR_+$ & Nonnegative real numbers \\
$S^n$ & Unit sphere in Euclidean $(n+1)$-space\\
$\Span(x_1,..,x_k)$ & Linear span\\
$\supp(f)$ & Support of the section (or distribution) $f$ \\
$V^\bot$ & Orthogonal complement of subspace in Hilbert space\\
$X$ & Generic name for a complex vector space or Banach space\\
$X^*$ & Dual space\\
$\ZZ_+$ & Nonnegative integers\\
$[\cdot]$ & Integer part of a real number
\end{tabular}
}}
\cleardoublepage
\Printindex{symbolindex}{List of Symbols}
\indexcomment{\center{Cursive numbers refer to the bibliography.}}
\cleardoublepage
\Printindex{authorindex}{Index of Names/Authors}%
\cleardoublepage
\Printindex{subjectindex}{Subject Index}
\addcontentsline{toc}{chapter}{\numberline{}Index}%

\end{document}